\documentclass[11pt]{amsart}

\usepackage[text={400pt,675pt},centering]{geometry}

\usepackage{color}
\usepackage{esint,amssymb}
\usepackage{graphicx}
\usepackage{MnSymbol}
\usepackage{mathtools}
\usepackage[colorlinks=true, pdfstartview=FitV, linkcolor=blue, citecolor=blue, urlcolor=blue,pagebackref=false]{hyperref}
\usepackage{microtype}

\usepackage{bm}
\usepackage{scalerel} 
\usepackage{dsfont}
\usepackage{mathrsfs}
\usepackage{eufrak}
\usepackage[font={footnotesize}]{caption}

\definecolor{darkblue}{rgb}{0,0,0.7}

\newtheorem{theorem}{Theorem}
\newtheorem{proposition}{Proposition}
\newtheorem{lemma}[proposition]{Lemma}
\newtheorem{corollary}[proposition]{Corollary}

\theoremstyle{remark}

\theoremstyle{definition}
\newtheorem{definition}[proposition]{Definition}
\newtheorem{remark}[proposition]{Remark}

\numberwithin{equation}{section}
\numberwithin{proposition}{section}

\newcommand{\Z}{\mathbb{Z}}
\newcommand{\N}{\mathbb{N}}
\newcommand{\R}{\mathbb{R}}

\newcommand{\E}{\mathbb{E}}
\renewcommand{\P}{\mathbb{P}}
\newcommand{\F}{\mathcal{F}}
\newcommand{\Zd}{\mathbb{Z}^d}
\newcommand{\Rd}{{\mathbb{R}^d}}
\newcommand{\Bd}{\mathcal{B}_d}
\newcommand{\Ed}{E_d}

\newcommand{\ep}{\varepsilon}

\renewcommand{\a}{\mathbf{a}}
\newcommand{\ahom}{{\overbracket[1pt][-1pt]{\a}}}  

\renewcommand{\subset}{\subseteq}

\newcommand{\cu}{{\scaleobj{1.2}{\square}}}

\renewcommand{\fint}{\strokedint}

\DeclareMathOperator{\dist}{dist}

\DeclareMathOperator{\var}{var}

\DeclareMathOperator{\diam}{diam}

\DeclareMathOperator{\sign}{sign}

\DeclareMathOperator{\supp}{supp}

\DeclareMathOperator{\size}{size}
\DeclareMathOperator{\intr}{int}
\DeclareMathOperator{\cl}{cl}
\DeclareMathOperator{\card}{Card}

\renewcommand{\bar}{\overline}
\renewcommand{\tilde}{\widetilde}

\newcommand{\indc}{\mathds{1}}

\newcommand{\C}{\mathscr{C}}
\newcommand{\Pa}{\mathcal{P}}
\newcommand{\Pas}{\mathcal{P}_*}
\newcommand{\Qa}{\mathcal{Q}}

\newcommand{\Ra}{\mathcal{R}}
\newcommand{\G}{\mathcal{G}}
\renewcommand{\S}{\mathcal{S}}
\newcommand{\M}{\mathcal{M}}

\renewcommand{\O}{\mathcal{O}}
\newcommand{\A}{\mathcal{A}}
\newcommand{\zbar}{\overline{z}}
\newcommand{\pc}{\mathfrak{p}_{\mathfrak{c}}}
\newcommand{\p}{\mathfrak{p}}
\newcommand{\e}{\mathfrak{e}}
\newcommand{\aconn}{\leftrightarrow_\a}
\newcommand{\T}{\mathcal{T}}
\newcommand{\g}{\mathbf{G}}
\newcommand{\uhom}{u_\mathrm{hom}}
\newcommand{\tuhom}{\tilde{u}_{\mathrm{hom}}}
\newcommand{\X}{\mathcal{X}}

\begin{document}

\title[Regularity and homogenization on percolation clusters]{Elliptic regularity and quantitative homogenization on percolation clusters}

\begin{abstract}
We establish quantitative homogenization, large-scale regularity and Liouville results for the random conductance model on a supercritical (Bernoulli bond) percolation cluster. The results are also new in the case that the conductivity is constant on the cluster. The argument passes through a series of renormalization steps: first, we use standard percolation results to find a large scale above which the geometry of the percolation cluster behaves (in a sense made precise) like that of Euclidean space. Then, following the work of Barlow~\cite{Ba}, we find a succession of larger scales on which certain functional and elliptic estimates hold. This gives us the analytic tools to adapt the quantitative homogenization program of Armstrong and Smart~\cite{AS} to estimate the yet larger scale on which solutions on the cluster can be well-approximated by harmonic functions on~$\Rd$. This is the first quantitative homogenization result in a porous medium and the harmonic approximation allows us to estimate the scale on which a higher-order regularity theory holds. The size of each of these random scales is shown to have at least a stretched exponential moment. As a consequence of this regularity theory, we obtain a Liouville-type result that states that, for each $k\in\N$, the vector space of solutions growing at most like $o(|x|^{k+1})$ as $|x|\to \infty$ has the same dimension as the set of harmonic polynomials of degree at most~$k$, generalizing a result of Benjamini, Duminil-Copin, Kozma, and Yadin~\cite{BDKY} from~$k\leq1$ to $k\in\N$.
\end{abstract}

\author[S. Armstrong]{Scott Armstrong}
\address[S. Armstrong]{Courant Institute of Mathematical Sciences, New York University, 251 Mercer St., New York, NY 10012}
\email{scotta@cims.nyu.edu}

\author[P. Dario]{Paul Dario}
\address[P. Dario]{Universit\'e Paris-Dauphine, PSL Research University, CNRS, UMR [7534], CEREMADE, Paris, France}
\email{paul.dario@dauphine.eu}

\subjclass[2010]{}
\date{\today}

\maketitle

\setcounter{tocdepth}{1}
\tableofcontents

\section{Introduction}
\label{s.introduction}

\subsection{Motivation and informal summary of results}
\label{ss.summary}

Consider the random conductance model on the infinite percolation cluster for supercritical bond percolation on the graph $(\Zd,\Bd)$ in dimension $d\geq 2$. Here $\Bd$ is the set of \emph{bonds}, that is, unordered pairs $\{x,y\}$ with $x,y\in\Zd$ satisfying $|x-y|=1$. We are given $\lambda \in (0,1)$ and a function 
\begin{equation*} \label{}
\a :\Bd \longrightarrow \{0 \} \cup [\lambda,1].
\end{equation*}
We call $\a(\{x,y\})$ the \emph{conductance} of the bond $\{x,y\} \in \Bd$ and we assume that $\{ a(e) \}_{e \in \Bd}$ is an i.i.d. ensemble. We assume that the Bernoulli random variable $\indc_{\{ \a(e)\neq 0\}}$ has parameter $\p > \pc(d)$, where $\pc(d)$ is the bond percolation threshold for the lattice $\Zd$. It follows that the graph $(\Zd,\mathcal{E}(\a))$, where $\mathcal{E}(\a)$ is the set of edges $e\in\Bd$ for which $\a(e)\neq 0$, has a unique infinite connected component, which we denote by $\C_\infty = \C_\infty(\a)$.

\smallskip

Our interest in this paper is in the elliptic finite difference equation
\begin{equation}
\label{e.eq}
-\nabla \cdot \a \nabla u = 0 \quad \mbox{in} \ \C_\infty.
\end{equation}
Here the elliptic operator $-\nabla\cdot \a\nabla$ is defined on functions $u:\C_\infty\to \R$ by
\begin{equation} \label{e.op}
\left( -\nabla \cdot \a \nabla u\right)(x):= \sum_{y\sim x} \a((x,y)) \left( u(x) - u(y) \right). 
\end{equation}
The  operator $-\nabla \cdot \a \nabla$ is the generator of a continuous-time Markov chain~$\left\{ X(t) \right\}_{t\geq 0}$ which can be briefly described as follows. Each edge $e\in\Bd$ is endowed with a clock which rings after exponential waiting times with expectation $\a^{-1}(e)$. The random walker begins at the origin, i.e., $X(0)= 0$. When $X(t) = x\in\Zd$, the random walker waits until one of the clocks at an adjacent edge to $x$ rings, and then instantly moves across the edge to the neighboring point.  The reader may choose to focus on the special case that the conductance~$\a$ takes only the values $\{0,1\}$ and the model reduces to the simple random walk on the supercritical percolation cluster, with the generator being the Laplacian. The results in this paper are new even in this simpler situation.

\smallskip

Of primary interest is the scaling limit of this random walk (conditioned on the event that $0 \in \C_\infty$), and more generally its long-time behavior. A \emph{quenched invariance principle} for the random walk $X(t)$ in the case~$\a \in\{0,1\}$ was first proved in dimensions $d\geq 4$ by Sidoravicius and Sznitman~\cite{SS} and later, in every dimension $d\geq 2$ by Berger and Biskup~\cite{BB} and, independently, Mathieu and Piatnitski~\cite{MP}. It states that, $\P\left[\cdot \,\vert\, 0 \in \C_\infty\right]$--a.s., the process $\left\{ \ep X \left( \ep^{-2} t\right)\right\}_{t\geq 0}$ converges in law, as $\ep \to 0$, to a non-degenerate Brownian motion with covariance matrix $\sigma I_d$. This result was extended to the setting considered here (and to even greater generality) by Biskup and Prescott~\cite{BP} and Mathieu~\cite{M} (see also Andres, Barlow, Deuschel and Hambly~\cite{ABDH}). We refer to the survey of Biskup~\cite{biskupsurvey} and the references therein for more on the many recent works on this problem. 

\smallskip

The quenched invariance principle for the process $\{ X_t\}_{t\geq 0}$ is closely related to questions of \emph{homogenization}, that is, the study of the solutions of~\eqref{e.eq} on large length scales. The basic qualitative homogenization result states that, $\P$--a.s., a solution $u_r$ of~\eqref{e.eq} in $\C_\infty \cap B_r$ converges, as $r \to \infty$, to solutions of the (continuum) partial differential equation
\begin{equation*}
-\nabla \cdot \ahom \nabla \bar{u}_r = 0 \quad \mbox{in} \ B_r
\end{equation*}
in the sense that 
\begin{equation}
\label{e.DPlim}
\limsup_{r\to\infty} \frac1{r^2} \sum_{x\in \C_\infty \cap B_r} \left| u_r(x) - \bar{u}_r(x)\right|^2 = 0,
\end{equation}
where $u_r$ and $\bar{u}_r$ are given the same Dirichlet boundary condition $f_r(x) = r f\left( \frac xr  \right)$ for a fixed function $f:\partial B_1 \to \R$. The matrix $\ahom = \frac12 \sigma^2 I_d$, where $\sigma>0$ is the covariance of the limiting Brownian motion from the invariance principle. The study of~\eqref{e.eq} can be motivated independently, from the PDE perspective, by the desire to extend the theory of elliptic homogenization to random porous media (supercritical bond percolation being a very natural model of a random porous medium). However, from the probability point of view, the important point is that a homogenization result is essentially equivalent to an invariance principle. Certainly a quenched invariance principle implies a qualitative homogenization result, while \emph{quantitative} homogenization results give \emph{quantitative} invariance principles. Indeed, perhaps the main difficulty  in proving a quenched invariance principle is establishing the sublinear growth of correctors (cf.~\cite{BB,biskupsurvey}), which is nothing but a quantitative homogenization estimate. (The definition of the correctors is given in the comments following Theorem~\ref{t.reg}, below.) 

\smallskip

It is a well-known open problem to obtain quantitative information (for instance, rates of convergence) for this model, both in terms of the quenched invariance principle as well as homogenization. It is mentioned for example in~\cite[Section 4.4]{biskupsurvey},~\cite[below Theorem 1.2]{BH} and~\cite{LNO}. The obstacle is that the various qualitative proofs of the quenched invariance principle rely on an appeal to the ergodic theorem which is difficult to quantify. On the other hand, in the uniformly elliptic setting (when all bonds are open and $\a \in [\lambda,1]$), there is now a very precise quantitative theory due to Gloria and Otto~\cite{GO1,GO2} (see also~\cite{GNO}) with implications to random walks explained in~\cite{EGMN}. However, it is not obvious to see how to extend the methods of these papers to the case of percolation clusters since they rely heavily on uniform ellipticity and seem to require the geometry of the random environment to posses some homogeneity down to the smallest scales. In a recent work, Lamacz, Neukamm and Otto~\cite{LNO} adapt these methods to the case of Bernoulli bond percolation which is modified so that all bonds in a fixed unit direction are open. However, this model has the property that every lattice point of~$\Zd$ belongs to the infinite cluster and is still quite far from the setting of supercritical percolation clusters. 

\smallskip

In this paper, we prove the first quantitative homogenization results for the random conductance model on supercritical percolation clusters (see Theorems~\ref{t.homog} and~\ref{t.reg} below for the precise statements). In particular, we give explicit bounds on the sublinear growth of correctors and rates of convergence for the limit~\eqref{e.DPlim}. We also prove a higher-order \emph{regularity theory}, extending recent results in the uniformly elliptic case~\cite{AS,AM,GNO} to this setting. In particular, we prove Liouville-type results of every order, which characterize the set of solutions on the infinite cluster which exhibit polynomial growth. Such a regularity theory also provides important gradient estimates which are an essential ingredient for obtaining an \emph{optimal} quantitative theory and obtaining scaling limits for correctors. We expect that the results in this paper will open the way for the development of such a theory on percolation clusters and to resolve several open problems mentioned for example in~\cite{BB,biskupsurvey}. Indeed, in a forthcoming sequel~\cite{AD2} to this paper, we establish optimal bounds on the scaling of correctors as well as the decay of the gradient of the Green's function. 

\smallskip

A main source of our inspiration comes from the work of the first author and Smart~\cite{AS}, who recently introduced an alternative approach to quantitative theory of homogenization in the uniformly elliptic setting. Their method is based on studying certain subadditive quantities related to the variational formulation of the equation (i.e., the Dirichlet form) and quantifying their convergence by an iteration argument. At each step of the iteration, one passes information from a certain (large) length scale to a multiple of the length scale, showing that the error contracts by a factor less than~$1$. In this way the method resembles a renormalization argument. Critically, information at the smallest scales can be ``forgotten" and we only need to ensure that the model behaves like a uniformly elliptic equation, in some sense, on large scales. This adds some flexibility and robustness to the approach and, as we show here, it is well-suited to handling difficulties encountered in attempting a generalization to percolation clusters.

\smallskip

The ideas are therefore relatively straightforward, even if the details are many and the proofs are long. We begin in Sections~\ref{s.partition} and~\ref{s.functional} by finding a large (random) scale on which the percolation cluster has geometric properties which are close to those of~$\Rd$. Since this random scale is not uniformly bounded (there will be some large regions where the cluster is badly behaved) we partition $\Zd$ into triadic cubes of different sizes such that every cube is well-connected in the sense of Antal and Pisztora~\cite{AP}, using a Calder\'on-Zygmund-type stopping time argument. In regions where this partition is rather coarse, the geometry of the cluster is less well-behaved and where it is finer, the cluster is well-connected. Inspired by the work of Barlow~\cite{Ba} (which was itself inspired by the earlier work of Mathieu and Remy~\cite{MR}), we continue to coarsen the graph in stages, obtaining functional and elliptic inequalities on larger and larger scales: first we pass to a larger scale on which a Sobolev-Poincar\'e holds, then obtain a scale on which solutions of~\eqref{e.eq} satisfy a reverse H\"older inequality, and again to get a scale on which the gradients of solutions satisfy a Meyers-type higher integrability estimate.

\smallskip

This provides us with the elliptic estimates needed to run the arguments of~\cite{AS}, which takes up the bulk of the analysis in the paper. In Section~\ref{s.quantities}, we introduce analogues of the subadditive energy quantities and show that they possess similar properties to the ones in the uniformly elliptic setting, at least on large scales and with high probability. The main part of the analysis comes in Section~\ref{s.convergence}, where we show convergence of the subadditive quantities to their deterministic limits. The main difficulty compared to the analysis of~\cite{AS} is to deal with the possibility that the energy density of the solutions may be very large in regions of the cluster in which the connectivity is poor (i.e., where the cube partition mentioned above is quite coarse). The resolution comes by using the gain of integrability from Meyers estimate to show that \emph{spatial averages} of the gradients of the solutions cannot concentrate in small regions, which gives us just what we need.

\smallskip

At this stage in the development of our theory, the difference in difficulty between the uniformly elliptic case and the percolation cluster has basically vanished. We conclude by showing first in Section~\ref{s.dirichlet}, by a deterministic argument resembling a numerical analysis exercise, that the convergence of the subadditive quantities implies control of the error in homogenization for the Dirichlet problem. This concludes the proof of our first main result and gives us the harmonic approximation we need to run the arguments of~\cite{AS,AM,GNO,AKM2} to obtain the quantitative $C^{k,1}$ regularity theory and, in particular, the Liouville results. The latter is summarized in Section~\ref{s.regularity}.  

\smallskip

We continue in the next two subsections by giving the precise assumptions and then the statements of the main results. 

\subsection{Notation and assumptions}
\label{ss.notation}

Let~$\Zd$ be the standard~$d$-dimensional hypercube lattice and~$\Bd:=\left\{ \{ x,y\}\,:\, x,y\in\Zd, |x-y|=1 \right\}$ denote the set of nearest neighbor, non-oriented edges. We denote the standard basis in $\Rd$ by $\{\e_1,\ldots,\e_d\}$. For~$x,y\in\Zd$, we write~$x\sim y$ if~$x$ and~$y$ are nearest neighbors. We usually denote a generic edge by~$e$. We fix a parameter $\lambda \in (0,1]$ and denote by~$\Omega$ the set of all functions $\a: \Bd \to \{ 0\} \cup [\lambda,1]$, in other words,~$\Omega = \left( \{ 0\} \cup [\lambda,1] \right)^{\Bd}$ and we let $\a$ denote the canonical element of~$\Omega$. The Borel~$\sigma$-algebra on $\Omega$ is denoted by $\F$. For each $U \subseteq \Zd$, we let $\F(U)\subseteq\F$ denote the smallest $\sigma$-algebra such that each of the random variables $\a\mapsto \a(\{x,y\})$, for $x,y\in U$ with $x \sim y$, is $\F(U)$-measurable.

\smallskip

We fix an i.i.d.~probability measure~$\P$ on $(\Omega,\F)$, that is, a measure of the form $\P = \P_0^{\Bd}$ where $\P_0$ is the law of a random variable~$\a(0)$ taking values in $\{0\}\cup[\lambda,1]$ with the property that, for a fixed edge $e$,
\begin{equation*} \label{}
\p := \P_0 \left[ \a(e) \neq 0 \right] > \pc(d)
\end{equation*}
and $\pc(d)$ is the bond percolation threshold for the lattice $\Zd$. We denote by~$\E$ the expectation with respect to~$\P$. 

\smallskip

Given $\a\in\Omega$, we say that an edge $e\in\Bd$ is \emph{occupied} if $\a(e)>0$ and \emph{vacant} if $\a(e)=0$. Given vertices $x,y\in\Zd$, a \emph{path connecting $x$ and $y$} is a sequence of occupied edges of the form $\{x,z_1\}, \{z_1,z_2\}, \ldots, \{z_{n},z_{n+1}\},\ldots,\{z_N,y\}$.  We say that \emph{$x$ and $y$ are connected} and write $x \aconn y$ if there exists a path connecting $x$ and $y$. A \emph{cluster} is a subset $\C\subseteq \Zd$ with the property that, for every $x,y\in\C$, there exists a path connecting $x$ and $y$ consisting only of edges between elements of $\C$. The \emph{$\a$-interior} of a subset $U \subseteq\Zd$  is the subset 
\begin{equation*} \label{}
\intr_\a(U) := \left\{ x\in U\,:\, \mbox{$y\sim x$ and $\a(\{ x,y\})\neq 0$} \implies y\in U\right\}
\end{equation*}
The \emph{$\a$-boundary of $U$} is $\partial_\a U:= U \setminus \intr_\a (U)$. The interior and boundary with respect to the nearest-neighbor lattice $(\Zd,\Bd)$ are denoted by
\begin{equation*} \label{}
\intr(U):= \left\{ x\in U \,:\, y\sim x \implies y\in U \right\} \quad \mbox{and}  \quad \partial U := U \setminus \intr(U). 
\end{equation*}
We write $x \aconn \infty$ if $x$ belongs to an unbounded cluster and we denote by
\begin{equation*} \label{}
\C_\infty:= \left\{ x\in \Zd \,:\, x \aconn \infty \right\}
\end{equation*}
the maximal unbounded cluster, which $\P$--almost surely exists and is unique~\cite{BK}. 

\smallskip

We next introduce our notation for vector fields and Dirichlet forms. We let $\Ed:= \left\{ (x,y)\,:\, x,y\in\Zd, x\sim y \right\}$ denote the set of oriented nearest-neighbor pairs and $\Ed(U):=\left\{ (x,y)\,:\, x,y\in U, x\sim y \right\}$ denote the edges lying in a subset $U \subseteq \Zd$. A \emph{vector field} $G$ on $U$ is a function 
\begin{equation*} \label{}
G: \Ed(U) \to \R
\end{equation*}
which is antisymmetric, that is, $G(x,y) = -G(y,x)$ for every $(x,y)\in \Ed(U)$. If $u: U \to \R$, then $\nabla u$ is the vector field defined by
\begin{equation*} \label{}
(\nabla u)(x,y) := u(x)-u(y)
\end{equation*}
and $\a\nabla u$ is the vector field defined by
\begin{equation*} \label{}
\left( \a \nabla u\right)(x,y) := \a\left( \{x,y\} \right) \left( u(x)-u(y) \right). 
\end{equation*}
If $q \in\Rd$, we also let $q$ denote the constant vector field given by 
\begin{equation*} \label{}
q(x,y) := q\cdot (x-y). 
\end{equation*}
If $F$ is a vector field on $U$, then we define, for  each $x\in U$, 
\begin{equation} 
\label{e.magF}
\left| F \right|(x):= \left( \frac12 \sum_{y\in U,\, y \sim x} \left| F(x,y) \right|^2 \right)^{\frac12}.
\end{equation}
We note that the definition of $|F|$ depends on the underlying domain $U$, but we do not display this dependence explicitly since it is always clear from the context. We put an inner product $\left\langle \cdot,\cdot \right\rangle_U$ on the space of vector fields on $U$, defined by
\begin{equation*} \label{}
\left\langle F,G \right\rangle_U 
:= \sum_{x,y\in U, \, x\sim y} F(x,y) G(x,y).
\end{equation*}
We denote by $\left \langle F \right\rangle_U$ the unique vector in $\Rd$ such that, for every $p\in\Rd$, 
\begin{equation*}
p\cdot \left \langle F \right\rangle_U = \left\langle p,F \right\rangle_U.
\end{equation*}
Given $\a\in\Omega$ and two functions $u,v:U \to \R$, the (bilinear) Dirichlet form can be written in this notation as
\begin{equation*} \label{}
\left\langle \nabla u,\a\nabla v \right\rangle_U 
= \frac12 \sum_{x,y\in U, \, x\sim y} \left( u(x)- u(y) \right) \a\left(\{x,y\} \right) \left( v(x) - v(y) \right). 
\end{equation*}
The elliptic operator $-\nabla\cdot\a\nabla$ is defined for each $u:U\to \R$ and $x\in U$ by 
\begin{equation*} \label{}
\left( -\nabla\cdot\a\nabla u\right)(x):= \sum_{y\in U, \, x\sim y} (\a\nabla u)(x,y) =  \sum_{y\in U, \, x\sim y}\a(\{x,y\})(u(x)-u(y)). 
\end{equation*}
We denote set of solutions (i.e., $\a$-harmonic functions) on a subset $U \subseteq \Zd$ by
\begin{equation}
\label{e.defA}
\mathcal{A}(U):= \left\{ u: U \to \R \,:\, -\nabla\cdot\a\nabla u(x) = 0 \ \ \mbox{for every} \  x \in \intr_\a(U) \right\}. 
\end{equation}
We denote by $\mathcal{C}_0^\a(U)$ the set of functions $w:U \to \R$ with compact support and satisfying $w=0$ on $\partial_\a U$. Then it is easy to check that
\begin{equation}
\label{e.Avarchar0}
u\in \A(U) 
\iff 
\left\langle \nabla w, \a\nabla u \right\rangle_U = 0 \quad \mbox{for every} \ w \in \mathcal{C}_0^\a(U). 
\end{equation}

\smallskip

We next introduce our notation for keeping track of the sizes and stochastic integrability of random variables. Given $s,\theta >0$ and a random variable $X$ on $(\Omega,\F)$, we write 
\begin{equation*} \label{}
X \leq \O_s(\theta) \iff \E \left[  \exp\left( \left( \frac{X}{\theta} \right)^s \right)   \right] \leq 2.  
\end{equation*}
Note that by Markov's inequality, $X\leq \O_s(\theta)$ implies that, for every $t >0$, 
\begin{equation*} \label{}
\P \left[ X \geq \theta t  \right]  \leq 2\exp\left( -t^s \right), 
\end{equation*}
so roughly the notation means that $X$ has characteristic size at most $\theta$ with the tails of the distribution of~$\theta^{-1}X$ decaying at most like $\lesssim\exp\left( -t^s\right)$. If~$Y$ is another random variable and $a>0$, we also write 
\begin{equation*} \label{}
X \leq Y + a \O_s(\theta) \iff X-Y \leq \O_s(a\theta)
\end{equation*}
and, if $Y$ is nonnegative, 
\begin{equation*}
X \leq  \O_s(\theta) Y  \iff \frac{X}{Y} \leq \O_s(\theta).\end{equation*}
This notation is transitive in the sense that (cf.~\cite[Lemma 2.3(i)]{AKM2}), for a universal constant $C$ depending only on $s$, which may be taken to be~$1$ if $s\geq 1$, \begin{equation}
\label{e.Osums}
X \leq \O_s(\theta_1) \ \mbox{and} \ Y \leq \O_s(\theta_2) \implies X+Y \leq \O_s(C(\theta_1+\theta_2))
\end{equation}
Moreover, by~\cite[Lemma 2.3(ii)]{AKM2}, for any $s>0$ there exists $C(s)<\infty$ such that, for every measure space $(X,\F,\mu)$ and measurable function $f:E \to \R_+$ and jointly measurable family $\{ X(z) \}_{z\in E}$ of nonnegative random variables,
\begin{equation}
\label{e.Oavgs}
\forall z\in E, \ X(z) \leq \O_s(1) \implies \int_E X(z) \, d\mu(z) \leq \O_s(C). 
\end{equation}
Young's and H\"older's inequalities imply (cf.~\cite[Remark 2.2]{AKM2}) 
\begin{equation}
\label{e.Oprods}
|X| \leq \O_{s_1}(\theta_1) \ \mbox{and} \ |Y| \leq \O_{s_2}(\theta_2) \implies |XY| \leq \O_{\frac{s_1s_2}{s_1+s_2}}\left( \theta_1\theta_2 \right). 
\end{equation}
It is easy to check from the Young's inequality that, for every $s_1,s_2 \in (0,\infty)$, every $t \in [0,1]$ and random variable $X$, we have for $\alpha = \frac{ts_1}{ts_1 + (1-t) s_2}$
\begin{equation}
\label{e.improves}
\ X \le \O_{s_1}(\theta_1)  \mbox{ and }  \ X \le \O_{s_2}(\theta_2)
\implies 
X \le \O_{ts_1 + (1-t) s_2}(\theta_1^{\alpha}\theta_2^{1-\alpha}).
\end{equation}

\smallskip

For a finite $U\subseteq \Zd$ and $w:U \to \R$, we often denote sums by integrals; for example,
\begin{equation}
\label{e.integralconvention}
\mbox{we often write} \quad \int_{U} w(x)\,dx 
\quad
\mbox{in place of} 
\quad 
\sum_{x\in U} w(x). 
\end{equation}
If $U$ is a finite set, we denote its cardinality by $|U|$. Sometimes we also use $|V|$ to denote the Lebesgue measure of a subset $V\subseteq \Rd$, but the meaning will always be clear from context. The normalized integral for a function $w:U\to \R$ for a finite subset $U\subseteq\Zd$ is denoted 
\begin{equation*}
\fint_U w(x)\,dx = \frac{1}{|U|} \int_U w(x)\,dx  = \frac{1}{|U|} \sum_{x\in U} w(x). 
\end{equation*}
For $p\in [1,\infty)$, we denote the $L^p$ and normalized $L^p$ norms of $w$ by
\begin{equation*}
\left\| w \right\|_{L^p(U)} := \left( \int_U \left|w(x)\right|^p\,dx\right)^{\frac1p} 
\quad \mbox{and} \quad 
\left\| w \right\|_{\underline{L}^p(U)} := \left( \fint_U \left|w(x)\right|^p\,dx\right)^{\frac1p}
\end{equation*}
and $\left\| w \right\|_{L^\infty(U)} := \sup_{x\in U} \left| w(x)\right|$. We define the distance function $\dist$ with respect to the $\ell_\infty$ norm on the coordinates, i.e., $\dist(x,y) = \max_{i=1,\ldots,d} |x_i-y_i|$ and extend this to subsets of $\Rd$ by $\dist(U,V) = \inf_{x\in U, y\in V} \dist(x,y)$. 

\smallskip

A \emph{cube} is a subset of $\Zd$ of the form 
\begin{equation*} \label{}
\Zd \cap \left( z + \left[ 0, N \right)^d \right), \quad z\in \Zd, \ N \in \N. 
\end{equation*}
We define the \emph{center} and \emph{size} of the cube given in the previous display above to be $z$ and $N$, respectively, and denote the size of a cube $\cu$ by $\size(\cu)$. For a cube $\cu$ and $r>0$, we use the nonstandard convention of denoting by $r\cu$ the cube with size $\left\lfloor r\size(\cu) \right\rfloor$ and \emph{having the same center as $\cu$}. A \emph{triadic cube} is a cube of the form
\begin{equation*}
\cu_m(z):= \Zd \cap \left( z + \left( -\frac12 3^m, \frac12 3^m \right)^d \right), \quad z\in 3^m\Zd, \ m \in\N. 
\end{equation*}
We also write $\cu_m = \cu_m(0)$. Observe that $\size(\cu_m) = 3^m$. For every $m,n\in\N$ with $n\leq m$, each triadic cube $\cu_m(z)$ may be uniquely partitioned into exactly $3^{d(m-n)}$ disjoint triadic cubes of the form $\cu_n(y)$, $y\in 3^n\Zd$. Moreover, any two triadic cubes (of possibly different sizes) are either disjoint or else one is a subset of the other. We denote the collection of triadic cubes by $\T$ and the set of triadic cubes of size $3^n$ by $\T_n$. Note that $\T_n:=\left\{ z+\cu_n \,:\, z\in 3^n\Zd \right\}$. For each $\cu \in \T$, the \emph{predecessor of $\cu$} is the unique triadic cube $\tilde \cu\in\T$ satisfying 
\begin{equation} 
\label{e.predecessor}
\cu \subseteq \tilde \cu \quad \mbox{and} \quad \frac{\size(\tilde\cu) }{ \size(\cu)} = 3.
\end{equation}
If~$\tilde\cu$ is the predecessor of~$\cu$, then we also say that~$\cu$ is a \emph{successor} of~$\tilde\cu$. Note that, since we are working with subsets of the discrete lattice $\Zd$, disjoint triadic cubes will be separated by a distance of at least~$1$. In fact, two disjoint cubes $\cu$, $\cu'$ are neighbors if and only if~$\dist(\cu,\cu')=1$.

\subsection{Statement of the main results}
\label{ss.results}

The first main result gives an estimate of the length scale on which the homogenization approximation holds, up to an algebraic error threshold. In the statement, we use the notation $\C_*(\cu_m)$ which is not defined until Section~\ref{s.partition}, but roughly denotes the largest connected component of $\C_\infty\cap \cu_m$ (which is also the same as $\C_\infty\cap \cu_m$, up to a small number of vertices near the boundary of~$\cu_m$). 

\begin{theorem}[Quantitative homogenization]
\label{t.homog}
Fix an exponent $p>2$. Then there exist $s(p,d,\p,\lambda)>0$ and $\alpha(p,d,\p,\lambda)>0$, a constant $C(p,d,\p,\lambda)<\infty$, a symmetric matrix $\ahom$ such that 
\begin{equation*} \label{}
\frac{1}{C} I_d \leq \ahom\leq CI_d
\end{equation*}
and a nonnegative random variable $\X$ satisfying 
\begin{equation*} \label{}
\X \leq \O_s \left( C \right)
\end{equation*}
such that the following holds: for every $m\in\N$ such that $3^m \geq \X$ and function $u : \C_*(\cu_{m})\to \R$ satisfying
\begin{equation*} \label{}
-\nabla \cdot \a\nabla u = 0 \quad \mbox{in} \  \C_*(\cu_{m})\setminus \partial \cu_m,
\end{equation*}
there exists an $\ahom$-harmonic function $u_{\mathrm{hom}}$ on $\left[ -\frac12 (3^m-1), \frac12(3^m-1)\right]^d$ satisfying
\begin{equation} \label{e.bcs}
u (x) = u_{\mathrm{hom}} (x) \quad \mbox{for every} \ x \in \C_*(\cu_m) \cap \partial \left(  -\frac12 (3^m-1), \frac12(3^m-1) \right)^d,
\end{equation}
\begin{equation} \label{e.homogbnd}
\fint_{\left(-\frac12(3^m-1),\frac12(3^m-1)\right)^d} \left|\nabla u_{\mathrm{hom}}(x)\right|^2\,dx \leq \frac{1}{|\cu_m|} \sum_{x\in \C_*(\cu_m)} \left| \nabla u \indc_{\{\a\neq0\}} \right|^p(x) 
\end{equation}
and
\begin{multline} \label{e.homog}
3^{-m} \left( \frac{1}{|\cu_m|} \sum_{x \in \C_*(\cu_m)} \left| u(x) - u_{\mathrm{hom}}(x) \right|^2 \right)^{\frac12} 
\\
 \leq C 3^{-m\alpha} \left( \frac{1}{|\cu_m|} \sum_{x\in \C_*(\cu_m)} \left| \nabla u \indc_{\{\a\neq0\}} \right|^p(x)  \right)^{\frac1p} .
\end{multline}
\end{theorem}

In view of~\eqref{e.bcs}, Theorem~\ref{t.homog} can be thought of as an error estimate for the Dirichlet problem. Indeed, as we will see in the proof in Section~\ref{s.dirichlet}, the function $u_{\mathrm{hom}}$ is constructed by solving the Dirichlet problem with boundary data obtained by smoothing out $u$ itself near $\partial \cu_m$. The first estimate~\eqref{e.homogbnd} just says that the gradient of~$u_{\mathrm{hom}}$ in $L^p$ is no larger than that of $u$ itself. The second estimate~\eqref{e.homog} is the main part of the conclusion, which states that the (properly scaled) $L^2$ difference between $u$ and $u_{\mathrm{hom}}$ is smaller than the size of $\nabla u$ in $L^p$ by the factor $3^{-m\alpha}$, that is, some power of the length scale. 

\smallskip

We turn to an important consequence of Theorem~\ref{t.homog}, namely the large-scale regularity theory for solutions on the infinite cluster, which we denote by~$\mathcal{A}(\C_\infty)$. For each $k\in\N$, we also let $\mathcal{A}_k(\C_\infty)$ denote the subspace of $\mathcal{A}(\C_\infty)$ consisting of functions growing more slowly at infinity than a polynomial of degree $k+1$:
\begin{equation*} \label{}
\mathcal{A}_k(\C_\infty):= \left\{ u\in \mathcal{A}(\C_\infty) \,:\, 
\limsup_{R\to \infty} R^{-(k+1)} \left\| u \right\|_{\underline{L}^2(\C_\infty\cap B_R)} = 0  \right\}. 
\end{equation*}
Our second main result (see Theorem~\ref{t.reg} below) concerns the  structure of~$\mathcal{A}_k(\C_\infty)$.

\begin{theorem}[Regularity theory]
\label{t.reg}
There exist $s(d,\p,\lambda)>0$, $\delta(d,\p,\lambda)>0$ and a nonnegative random variable $\X$ satisfying
\begin{equation}
 \label{e.Xbound}
\X \leq \O_s(C(d,\p,\lambda))
\end{equation}
such that the following hold:
\begin{enumerate}

\item[(i)] 
For each $k\in\N$, there exists a constant $C(k,d,\p,\lambda)<\infty$ such that, for every $u\in \A_k(\C_\infty)$, there exists $p\in \overline{\A}_k$ such that, for every $r\geq \X$, 
\begin{equation}
\label{e.upcaptcha}
\left\| u - p \right\|_{\underline{L}^2(\C_\infty\cap B_r)} \leq C r^{-\delta} \left\| p \right\|_{\underline{L}^2(B_r)}.
\end{equation}

\smallskip
\item[(ii)]
For every $k\in\N$ and $p\in\overline{\A}_k$, there exists $u\in \A_k$ such that, for every $r\geq \X$, the inequality~\eqref{e.upcaptcha} holds. 

\smallskip
\item[(iii)]
For each $k\in\N$, there exists $C(k,d,\p,\lambda)<\infty$ such that, for every $R\geq 2X$ and $u\in \A(\C_\infty \cap B_R)$, there exists $\phi\in \A_k(\C_\infty)$ such that, for every $r\in \left[ \X ,\frac12 R \right]$, we have
\begin{equation*} \label{}
\left\| u - \phi \right\|_{\underline{L}^2(\C_\infty\cap B_r)} \leq C \left( \frac rR\right)^{k+1} \left\| u \right\|_{\underline{L}^2(\C_\infty\cap B_R)}.
\end{equation*}

\end{enumerate} 
\end{theorem}

A consequence of statements (i) and (ii) of Theorem~\ref{t.reg} is that the vector space $\A_k(\C_\infty)$ has the same dimension as~$\overline{\A}_k$, that is, the same dimension as the space of harmonic polynomials of order at most~$k$. This was previous proved in the case $k=1$ and $\a \in \{0,1\}$ by~Benjamini, Duminil-Copin, Kozma, and Yadin~\cite{BDKY}. For~$k>1$, it was previously proved (in greater generality) that the subspace of $\A(\C_\infty)$ of functions growing at most like $O(|x|^k)$ had finite dimension: see Sapozhnikov~\cite{Sap}.

\smallskip

The Liouville result summarized in (i) and (ii) imply, in the case $k=1$, that every element $u$ of $\mathcal{A}_1(\C_\infty)$ can be written as 
\begin{equation*}
u(x) = c + p\cdot x + \chi_p(x)
\end{equation*}
where $c\in\R$, $p\in\Rd$ and $\chi_p$ is a function satisfying, for every $r\geq \X$, 
\begin{equation}
\label{e.chi}
\left\| \chi_p \right\|_{\underline{L}^2(\C_\infty\cap B_r)} \leq C|p|r^{1-\delta}. 
\end{equation}
These functions $\{ \chi_p \,:\, p\in\Rd \}$ are called the \emph{correctors} and their sublinear growth is a very important property that was previously proved only qualitatively (cf.~\cite{BB,biskupsurvey}). Together with the bound on~$\X$ in~\eqref{e.Xbound}, this provides the first quantitative bound on the sublinearity of $\chi_p$.

\smallskip

Moreover, the qualitative Liouville result is quantified by the third statement~(iii), which tells us much more: any $\a$-harmonic function may be expanded to arbitrary order in terms of elements of $\A_k(\C_\infty)$ in the same way that analytic functions can be approximated by Taylor polynomials of degree~$k$. Even this statement for $k=0$ is new and, combining it with Caccioppoli inequality (see Lemma~\ref{l.caccioppoli}), gives the following gradient bound may be compared to the results of~\cite{AS,AM}: for every $\X \leq r \leq R$, 
\begin{equation}
\label{e.lipschitzestimate}
\left\| \nabla u \indc_{\{\a\neq0\}} \right\|_{\underline{L}^2(\C_\infty\cap B_r)} \leq C \left\| \nabla u \indc_{\{\a\neq0\}} \right\|_{\underline{L}^2(\C_\infty\cap B_R)}.
\end{equation}

Gradient estimates like~\eqref{e.lipschitzestimate} play an important role in obtaining optimal quantitative homogenization estimates, in particular estimates for the sublinearity of the correctors in the uniformly elliptic setting; see~\cite{GNO3}. In the forthcoming sequel~\cite{AD2} to this paper, we explore the consequences of~\eqref{e.lipschitzestimate} in the setting of supercritical percolation clusters and show that they allow us to prove optimal estimates on the decay of the Green's function, its gradient as well as optimal bounds on the scaling of the correctors. In particular, we improve~\eqref{e.chi} to the optimal sublinearity bound in all dimensions.

\begin{remark}
The gradient bound~\eqref{e.lipschitzestimate} allows us to  immediately upgrade the bound in~\eqref{e.chi} from $\underline{L}^2(\C_\infty\cap B_r)$ to $L^\infty(\C_\infty\cap B_r)$. To see this, we use the following interpolation inequality for $L^\infty$ between $L^2$ and $C^{0,1}$: there exists $C(d)<\infty$ such that, for every $w : \Zd \cap B_r \to \R$ with zero mean on $\Zd\cap B_r$, we have
\begin{equation*} \label{}
\left\| w  \right\|_{L^\infty(\Zd\cap B_r)} 
\leq 
C \left\| w \right\|_{\underline{L}^2(\Zd\cap B_r)}^{\frac 23} \left( r \left\| \nabla w \right\|_{L^\infty(\Zd\cap B_r)}\right)^{\frac13}. 
\end{equation*}
Next, we deduce from~\eqref{e.Xbound},~\eqref{e.chi} and~\eqref{e.lipschitzestimate} the bound 
\begin{equation*} \label{}
\left| \nabla \chi_p \indc_{\{\a\neq 0\}}\right|(0)   \leq C\X^{\frac d2} \left\| \nabla \chi_p \indc_{\{\a\neq0\}} \right\|_{\underline{L}^2(\C_\infty\cap B_\X)} \leq C |p| \X^{\frac d2} \leq \O_{2s/d} \left( C |p| \right).
\end{equation*}
Therefore, by stationarity and~\eqref{e.Oavgs}, for every $m>1$,
\begin{equation*} \label{}
\left\| \nabla \chi_p\indc_{\{\a\neq 0\}} \right\|_{L^\infty(\C_\infty \cap B_r)}
\leq 
Cr^{\frac dm} \left\| \nabla \chi_p\indc_{\{ \a\neq 0\} } \right\|_{\underline{L}^m(\C_\infty \cap B_r)}
\leq \O_{2s/d} \left( C r^{\frac dm} |p| \right).
\end{equation*}
Here the $C$ depends additionally on $m$. We deduce that, for every $\beta>0$, there exists $s(d,\p,\lambda)>0$ and $C(\beta,d,\p,\lambda)<\infty$ such that
\begin{equation} \label{e.poompoom}
\left\| \nabla \chi_p \indc_{\{ \a\neq 0 \}} \right\|_{L^\infty(\C_\infty \cap B_r)} \leq \O_s\left( C |p| r^{\beta} \right).
\end{equation}
We may now apply the interpolation inequality above to the coarsened function $\left[ \chi_p \right]_{\Pa}$ (see Definition~\ref{def.coarsenfunction}), using the previous estimates~\eqref{e.chi} and~\eqref{e.poompoom} and Lemmas~\ref{l.coarseLs} and~\ref{l.coarsegrads} to obtain, for some $s(d,\p,\lambda)>0$ and $C(d,\p,\lambda)<\infty$,
\begin{equation}
\label{e.chicoarse}
\left\| \left[ \chi_p \right]_{\Pa} \right\|_{{L}^\infty(B_r)} \leq \O_s\left( C|p|r^{1-\delta/2} \right). 
\end{equation}
Finally, we can then use Lemma~\ref{l.coarseLs} and the previous inequality to get
\begin{equation}
\label{e.chilinfty}
\left\|  \chi_p \right\|_{{L}^\infty(\C_\infty \cap B_r)} \leq \O_s\left( C|p|r^{1-\delta/2} \right),
\end{equation}
as desired. 
\end{remark}

We conclude with some comments regarding some of the parameters appearing in the main theorems. First, we do not obtain the optimal exponent~$s(d,\p,\lambda)>0$ for the stochastic integrability appearing in Theorems~\ref{t.homog} and~\ref{t.reg}. In the uniformly elliptic setting, it is proved in~\cite{AS} that we may take any $s \in (0,d)$, which is the optimal stochastic integrability (in the sense that the results are  false for $s>d$). The arguments here do lead to an explicit estimate of~$s$, although we expect that it is impossible to find the optimal exponent~$s$ (which should depend only on $d$) without a very deep understanding of the geometry of the percolation cluster which, at least for $\p$ close to $\p_c$ and $d>2$, remains elusive. We therefore have not made any attempt to optimize or even keep track of the explicit~$s$ we obtain. Likewise, it would be very interesting to show that the constant~$c(d,\p,\lambda)>0$ in the lower bound for the effective diffusivity we obtain in~\eqref{e.ahombounds} depends on~$\p$ like a power of~$\p-\p_c$ that is, for some $\beta >0$, 
\begin{equation*} \label{}
c(d,\p,\lambda) \geq c_0(d,\lambda) \cdot \left( \p - \p_c \right)^{\beta}.
\end{equation*}
Our arguments actually give such a bound provided we can quantify the constants $c(d,\p)$ and $C(d,\p)$ in Lemma~\ref{l.AP} below, which is proved in Antal and Pisztora~\cite{AP} and contains the basic geometric information about the supercritical percolation clusters that all of our renormalization argument rely on. We are not aware of any work which estimates these constants, even crudely, and such an estimate would obviously be of fundamental interest to the study of percolation clusters beyond its implications to random walks on the clusters. 

\smallskip

The basic theme we wish to emphasize is that the bottleneck to getting estimates of these parameters, and improving our quantitative understanding in other ways, lies not in improving our quantitative homogenization methodology but rather in obtaining a better quantitative, geometric understanding of supercritical percolation clusters (especially near criticality). 

\smallskip

As a final remark, we would like to mention that the arguments in this paper give similar results for other random graphs (besides the particular case of a supercritical percolation cluster), provided that we have some quantitative estimate, like the one in Lemma~\ref{l.AP}, which says that the graph behaves like Euclidean space above some random scale.

\section{Triadic partitions of good cubes}
\label{s.partition}

The geometry of the percolation cluster~$\C_\infty$ and, more generally, the behavior of solutions of~\eqref{e.eq}, is highly irregular on small scales but becomes more regular as we look on larger scales. A large part of our effort in this paper is to quantify this vague assertion for various notions of ``good behavior." To this end, we will find it very useful to partition~$\Zd$ and certain subsets of~$\Zd$ into ``good" cubes (in which the percolation cluster and solutions of~\eqref{e.eq} are well-behaved in some sense). These partitions will be random and in particular the sizes of the cubes will necessarily be non-uniform, but we will prove quantitative estimates on the size of a typical cube (which will depend on what ``good" means). The coarseness of the partition therefore provides a measure of the local scale above which the system is well-behaved.

\smallskip

In this section, we give a general scheme for creating such partitions. Then, as a first application, we partition~$\Zd$ into ``well-connected" cubes which greatly simplifies the geometry of the cluster and will allow us in the next section to prove functional inequalities (for example, a Sobolev inequality) for functions on subsets of $\C_\infty$. 

\subsection{A general scheme for partitions of good cubes}
The construction of the partition is accomplished by a stopping time argument reminiscent of a Calder\'on-Zygmund-type decomposition. We are given a notion of ``good cube" represented by an $\F$-measurable function which maps $\Omega$ into the set of all subsets of $\T$. In order words, for each $\a\in\Omega$, we are given a subcollection $\mathcal{G}(\a) \subseteq \T$ of triadic cubes. We think of $\cu\in\T$ as being a good cube if $\cu\in\mathcal{G}(\a)$. As usual, we typically drop the dependence on~$\a$ and just write $\mathcal{G}$. 

\begin{proposition}
\label{p.partitions}
Let $\G \subseteq \T$ be a random collection of triadic cubes, as above. Suppose that $\G$ satisfies, for every $\cu = z + \cu_n \in \T$,
\begin{equation}
\label{e.goodlocal}
\mbox{the event $\left\{ \cu \not\in\G \right\}$ is $\F(z + \cu_{n+1})$--measurable,}
\end{equation}
and, for some constants $K,s>0$,
\begin{equation}
\label{e.probabilitygood}
\sup_{z\in 3^n\Zd} \P \left[ z+ \cu_n \not\in\G \right] \leq K \exp\left( -K^{-1}3^{ns} \right).
\end{equation}
Then, $\P$--almost surely, there exists a partition $\S\subseteq \T$ of $\Zd$ into triadic cubes with the following properties:
\begin{enumerate}

\item[(i)] All predecessors of elements of $\S$ are good: for every $\cu,\cu'\in\T$, 
\begin{equation*}
\cu' \subseteq \cu \ \mbox{and} \ \cu'\in\S \implies \cu \in \G.
\end{equation*}

\item[(ii)] Neighboring elements of $\S$ have comparable sizes: for every $\cu,\cu'\in \S$ such that $\dist(\cu,\cu') \leq 1$, we have
\begin{equation*}
\frac13 \leq \frac{\size(\cu')}{\size(\cu)} \leq 3.
\end{equation*}

\item[(iii)] Estimate for the coarseness of $\S$: if we denote $\cu_\S(x)$ the unique element of $\S$ containing a point $x\in\Zd$, then there exists $C(s,K,d)<\infty$ such that, for every $x\in\Zd$,
\begin{equation*}
\size\left( \cu_\S(x) \right) = \O_s (C). 
\end{equation*}

\item[(iv)] Approximate locality of $\S$: for each $m\in\N$ with $m\geq n$ and $z\in 3^m\Zd$, there exists a constant $C(s,K,d)<\infty$ and a partition $\S_{\mathrm{loc}}(z+\cu_m) \subseteq \T$ of $z+\cu_m$ which is $\F\left(z+\cu_m\right)$-measurable, finer than~$\S$ and satisfies, for 
\begin{equation*} \label{}
\cu_m^{(n)} := \left\{ x\in \cu_m\,:\, \dist(x,\partial \cu_m)\geq 3^n \right\},
\end{equation*}
the estimate
\begin{equation*}
\qquad \P \left[ \exists x\in z+\cu_m^{(n)}, \, \cu_\S(x) \neq \cu_{\S_{\mathrm{loc}}(z+\cu_m)}(x) \right] \leq C3^{m-n} \exp\left( -C^{-1} 3^{ns}\right). 
\end{equation*}
\end{enumerate}
\end{proposition}
\begin{proof}
We begin by giving an algorithmic construction of $\S$. First, we shrink~$\G$ to make it closed under taking predecessors and neighbors of predecessors. This requires the following definition: for each $\cu\in\T$, we set
\begin{multline*} \label{}
\mathcal{K}(\cu):= \big\{ \cu' \in\T\,:\, \exists n\in\N \ \mbox{and} \ \cu^0,\ldots,\cu^n\in\T \ \mbox{such that} \
\cu^0 = \cu, \ \cu^n = \cu', \\
\forall m \in \lbrace 1, \cdots , n \rbrace, \, \size(\cu^m) = 3^m\size(\cu), \ \mbox{and} \  
\dist\left(\tilde{\cu}^m,\cu^{m-1} \right) \leq 1
\big\}.
\end{multline*}
In other words, $\mathcal{K}(\cu)$ is the collection of triadic cubes we can obtain in finitely many steps starting from $\cu$ and where, in the $m$th step, we move from a cube $\cu^{m-1}$ of size $3^{m-1}\size(\cu)$ to a cube $\cu^{m}$ of size $3^{m}\size(\cu)$ whose predecessor is the neighbor of, or contains,~$\cu^{m-1}$. Recall that the predecessor $\tilde{\cu}$ of $\cu$ is defined in the sentence ending in~\eqref{e.predecessor}. 

\smallskip

It is clear from the definition of $\mathcal{K}(\cu)$ that
\begin{equation}
\label{e.mclKclosed}
\cu' \in \mathcal{K}(\cu) \implies \mathcal{K}(\cu') \subseteq \mathcal{K}(\cu). 
\end{equation}
We now define
\begin{equation*} \label{}
\overline{\G}:= \left\{ \cu \in \G \,:\, \mathcal{K}(\cu) \subseteq \G \right\}. 
\end{equation*}
By~\eqref{e.mclKclosed} we see immediately that 
\begin{equation} 
\label{e.sillyincl}
\cu \in \overline{\G} \implies \mathcal{K}(\cu) \subseteq \overline{\G}.
\end{equation}
To estimate the probability that a cube belongs to $\overline{\G}$, we use a union bound,~\eqref{e.probabilitygood} and the fact that, for each $m\in\N$, there are at most~$C$  distinct cubes of size $3^m\size(\cu)$ belonging to $\mathcal{K}(\cu)$ (this is a consequence of \eqref{e.growthK}, below). The union bound gives, for each $\cu\in\G$,
\begin{equation*} \label{}
\P \left[ \cu \not\in \overline{\G} \right] \leq \sum_{m=0}^\infty C K \exp\left( - K^{-1} 3^{ms}\size(\cu)^s \right) \leq C \exp\left( -C^{-1} \size(\cu)^s \right). 
\end{equation*}
It follows immediately that, $\P$--a.s.,~every element of~$\Zd$ belongs to infinitely many elements of~$\overline{\G}$. In particular,~$\overline{\G}$ covers~$\Zd$. 

\smallskip

We then introduce the partition~$\S$ by defining, for each~$x\in\Zd$, the cube~$\cu_\S(x)$ to be the largest element of~$\overline{\G}$ containing~$x$ which has a successor which does not belong to~$\overline{\G}$. If there is no such cube, we set~$\cu_\S(x):=x+\cu_0 = \{x\}$, which must then ($\P$--a.s.) belong to~$\overline{\G}$. It is easy to see from this construction that~$\S$ is indeed a partition. Moreover, we can estimate, for each $n\in\N$,
\begin{equation*} \label{}
\P \left[ \size\left(\cu_\S(x) \right) = 3^n \right] \leq 3^d \sup_{z\in 3^{n-1}\Zd} \P\left[ z+\cu_{n-1} \not\in\overline{\G} \right] 
\leq C \exp\left( -C^{-1} 3^{(n-1)s} \right).
\end{equation*}
It follows that 
\begin{equation*} \label{}
\size(\cu_\S(x)) = \O_{s}\left( C \right),
\end{equation*}
which confirms property~(iii). 

\smallskip

To check property~(i), we note simply that $\S \subseteq \overline{\G} \subseteq \G$ and that~$\overline{\G}$ is closed under taking predecessors. 

\smallskip

To check property~(ii), consider an element~$\cu\in \S$ and another cube~$\cu'\in \T$ with~$\dist(\cu,\cu') =1$ and~$\size(\cu')\geq 9\size(\cu)$. To see that~$\cu'$ cannot belong to~$\S$, observe that, since $\dist(\cu',\cu)\leq 1$, each of the successors of $\cu'$ belongs to~$\mathcal{K}(\cu)$. Therefore each of the successors of $\cu'$ belongs~$\overline{\G}$ by~\eqref{e.sillyincl} and the fact that~$\cu\in \S \subseteq \overline{\G}$. Thus $\cu'\not\in\S$ by the definition of~$\S$.

\smallskip

We have left to check property~(iv), which is accomplished by localizing the construction above. First we observe that there exists $C(d)<\infty$ such that 
\begin{equation}
\label{e.growthK}
\cu \in \T \ \mbox{and} \ \cu' \in \mathcal{K}(\cu) \implies \dist(\cu,\cu') \leq C \size(\cu').
\end{equation}
To see this, suppose that $\cu,\cu^1\in \T$ are such that 
\begin{equation*} \label{}
\size(\cu_1)=3\size(\cu) \quad \mbox{and} \quad \dist\left(\tilde{\cu}_1,\cu \right) \leq 1.
\end{equation*}
This implies that 
\begin{equation*}
\dist(\cu^1,\cu) \leq \diam\left( \tilde{\cu}^1 \right) + \dist\left(\tilde{\cu}^1,\cu \right) \leq 9 \diam(\cu) + 1 \leq C \size(\cu).
\end{equation*}
Then if $\cu^0,\ldots,\cu^n$ are as in the definition of $\mathcal{K}(\cu)$, we obtain
\begin{equation*} \label{}
\dist(\cu,\cu') \leq C\sum_{j=1}^n \size(\cu^j) =C \size(\cu) \sum_{j=1}^n 3^j \leq C 3^n \size(\cu) = C\size(\cu').
\end{equation*}
This yields~\eqref{e.growthK}. 

\smallskip

The implication~\eqref{e.growthK} allows us to localize the previous construction of~$\S$ by looking only at the triadic cubes for which we can evaluate membership in $\mathcal{G}$ by only looking at the edges in $z+\cu_m$. All other triadic cubes will be considered to be ``good'' by default. This motivates the definition
\begin{equation*}
\G_{\mathrm{loc}}\left(z+\cu_m\right) := \G \cup \left( \cup \left\{ y + \cu_{n} \, : \, n\in\N,\, y\in 3^n\Zd, \, y + \cu_{n+1} \not\subseteq  z + \cu_m \right\} \right). 
\end{equation*}
Using the property~\eqref{e.goodlocal}, we see that for every $\cu' \in \T$, 
\begin{equation*}
\mbox{the event } \left\{ \a \in \Omega \, : \, \cu' \in \G_{\mathrm{loc}}\left(z+\cu_m\right) \right\} \mbox{ is } \mathcal{F}(z + \cu_m) \mbox{--measurable.}
\end{equation*}
We now define $\S_{\mathrm{loc}}\left(z+\cu_m\right)$ to be the partition obtained by applying the previous construction to $\G_{\mathrm{loc}}\left(z+\cu_m\right)$. We write $\S_{\mathrm{loc}}=\S_{\mathrm{loc}}\left(z+\cu_m\right)$ for short. 
It is clear from the construction that $\S_{\mathrm{loc}}$ is completely determined by the environment in $z+\cu_m$, that is, $\S_{\mathrm{loc}}$ is $\F(z+\cu_m)$--measurable. Moreover, we see immediately that $\cu_\S(x)$ and $\cu_{\S_{\mathrm{loc}}}(x)$ may differ only if there exists $\cu'\in \mathcal{K}(\cu_{\S_{\mathrm{loc}}}(x)) \setminus \G$ such that $\cu' = y + \cu_n$ and $ y + \cu_{n+1} \not \subseteq z+\cu_m$. In this case, we have
\begin{align*}
\dist\left(x,\partial (z+\cu_m)\right)
& \leq \size\left(\cu_{\S_{\mathrm{loc}}}(x) \right) + \dist\left(\cu_{\S_{\mathrm{loc}}}(x) ,\cu' \right) + 3 \diam(\cu')\\
& \leq \size\left(\cu_{\S_{\mathrm{loc}}}(x) \right) + C \size(\cu') \\
& \leq C\size(\cu'). 
\end{align*}
Thus there exists $k_0(d)\in\N$ such that, for every $x\in z+\cu_m$,
\begin{multline*} \label{}
\cu_S(x) \neq \cu_{\S_{\mathrm{loc}}}(x)  \\
\implies
\exists \cu,\cu'\in \T \ \mbox{s.t.} \ x\in \cu, \ \cu'\in \mathcal{K}(\cu)\setminus \G, \ \dist\left(x,\partial (z+\cu_m)\right) \leq 3^{k_0} \size(\cu').
\end{multline*}
By~\eqref{e.growthK}, for each $j\in\N$ with $j\leq m$, there are at most $C3^{m-j}$ elements of the set 
\begin{equation*} \label{}
\left\{ \cu' \,:\, \exists \cu\in \T, \ \cu \cap z+\cu_m \neq \emptyset,\ \cu'\in\mathcal{K}(\cu),\ \size(\cu') = 3^j \right\}
\end{equation*}
while, for $j\geq m$, then there are at most $C$ elements of this set. 
Thus by a union bound and~\eqref{e.probabilitygood}, 
\begin{align*}
\P \left[ \exists x\in z+\cu_m^{(n)}, \ \cu_S(x)\neq \cu_{\S_{\mathrm{loc}}}(x) \right] 
& \leq CK \sum_{j=n-k_0}^\infty (3^{m-j} +1)\exp\left( -K^{-1} 3^{js}\right)  \\
& \leq C 3^{m-n} \exp\left( -c3^{ns} \right). 
\end{align*}
This completes the proof of~(iv).
\end{proof}

Many times in this paper we will be required to estimate, with $\S$ as in the previous proposition and for a finite subset $U\subseteq \Zd$ and exponent $t\geq1$, the random variable
\begin{equation}
\label{e.defLambdas}
\Lambda_t(U,\S):= \frac{1}{|U|} \sum_{x\in \cl_\S(U)} \size(\cu_\S(x))^{t}   
=   \frac{1}{|U|} \sum_{\S \ni \cu \subseteq \cl_\S(U)} \size(\cu)^{d+t} ,
\end{equation}
where, given an arbitrary subset $U \subseteq \Zd$, we define the \emph{closure $\cl_\S(U)$ of $U$ with respect to a partition~$\S$} by 
\begin{equation*} \label{}
\cl_\Pa(U):= \bigcup_{z\in U} \cu_\S(z). 
\end{equation*}
As in the statement of Proposition~\ref{p.partitions}, we let $\cu_\S(x)$ denote the unique element of~$\S$ containing~$x\in\Zd$. 

\smallskip

An immediate consequence of~Proposition~\ref{p.partitions}(iii) and ~\eqref{e.Oavgs} is the estimate
\begin{equation}
\label{e.LambdaOsbound}
\Lambda_t(U,\S) \leq \O_{\frac st}(C),
\end{equation}
for some constant $C := C(t,s,K,d) < + \infty$.
While~\eqref{e.LambdaOsbound} is quite useful, some of our arguments require something slightly stronger: the existence of a \emph{random scale} $\mathcal{M}_t(\S)\in\N$, with good quantitative bounds on the size of $\mathcal{M}_t(\S)$, and a \emph{deterministic} constant $C(t,K,s,d,\lambda,\p)<\infty$ such that, for every $m\in\N$,
\begin{equation*}
m \geq \mathcal{M}_t(\S) \implies \Lambda_t(\cu_m,\S) \leq C.
\end{equation*}
The precise result is stated below in Proposition~\ref{p.minimalscales}.
To prove it, we need to use independence and thus the localization provided by Proposition~\ref{p.partitions}(iv). In the following lemma, we put the localization statement into a more convenient form, in terms of $\Lambda_t(U,\S)$.

\begin{lemma}
\label{l.localizeLambda}
Let $K,s>0$, $\S$ and $\S_{\mathrm{loc}}$ be as in the statement of Proposition~\ref{p.partitions}. Fix $t\in [1,\infty)$. Then, for every $t' > 1$, there exists $C(t',t,K,s,d)<\infty$ such that, for every $m\in\N$ and $z\in 3^m\Zd$, 
\begin{equation} 
\label{e.comparelocalS}
\Lambda_t(z+\cu_m,\S) 
\leq  \Lambda_t(z+\cu_{m},\S_{\mathrm{loc}}(z+\cu_{m+1})) + \O_{\frac s{t+t'}}\left( C3^{-mt'} \right).
\end{equation}
\end{lemma}
\begin{proof}
Fix $m\in\N$ and $z\in 3^m\Zd$ and write $\S_{\mathrm{loc}}$ in place of $\S_{\mathrm{loc}}(z+\cu_{m+1})$ to lighten the notation and let $D_m$ denote the event
\begin{equation*} \label{}
D_{m}:= \left\{ \exists x \in z+\cu_m \ \mbox{such that} \ \cu_\S(x) \neq \cu_{\S_{\mathrm{loc}}}(x)  \right\}.
\end{equation*}
According to Proposition~\ref{p.partitions}, since $\cu_m \subseteq \cu_{m+1}^{(m)}$, we have 
\begin{equation}
\label{e.probDm}
\P\left[ D_m \right] \leq C\exp\left( -c3^{ms} \right).
\end{equation}
We estimate
\begin{align*}
\Lambda_t(z+\cu_m,\S) \indc_{\Omega\setminus D_{m}} 
& = \left( \frac{1}{|\cu_m|} \sum_{x\in z+\cu_m} \size(\cu_\S(x))^{t} \right)  \indc_{\Omega\setminus D_{m}}   \\
& \leq \Lambda_t(z+\cu_{m},\S_{\mathrm{loc}})
\end{align*}
On the other hand,~\eqref{e.probDm} implies that, for every $t'> 1$, 
\begin{equation*} \label{}
\indc_{D_{m}} = \O_{\frac s{t'}}\left( C3^{-mt'} \right) 
\end{equation*}
and, therefore by~\eqref{e.Oprods},
\begin{equation*} \label{}
\Lambda_t(z+\cu_m,\S) \indc_{D_{m}}  
\leq \O_{\frac st} (C) \cdot \O_{\frac s{t'}}\left( C3^{-mt'} \right) 
\leq \O_{\frac s{t+t'}}\left( C3^{-mt'} \right). 
\end{equation*}
Combining the above displays yields~\eqref{e.comparelocalS}.
\end{proof}

We need the following technical lemma. 

\begin{lemma}
\label{l.OtoM}
Fix $K\geq 1$, $s>0$ and $\beta>0$ and suppose that $\{ X_n \}_{n\in\N}$ is a sequence of nonnegative random variables satisfying, for every $n\in\N$,
\begin{equation} 
\label{e.XnbdnsOM}
X_n \leq \O_s\left( K3^{-n\beta} \right). 
\end{equation}
Then there exists $C(s,\beta,K) <\infty$ such that the random scale
\begin{equation*}
M := \sup \left\{ 3^n\in\N \,:\, X_n \geq 1 \right\}
\end{equation*}
satisfies the estimate
\begin{equation} 
\label{e.OtoM}
M \leq \O_{s\beta}(C). 
\end{equation}
\end{lemma}
\begin{proof}
Chebyshev's inequality and~\eqref{e.XnbdnsOM} imply
\begin{equation*} \label{}
\P \left[ X_n \geq 1 \right] \leq 2 \exp\left( - \left( K^{-1}3^{-n\beta} \right)^s \right).
\end{equation*}
Fix $t>0$ and $\delta>0$ and compute
\begin{align*} \label{}
\E \left[ \exp\left( \delta M^t  \right) \right] 
& \leq 1 + \sum_{k=0}^\infty \delta t 3^{(k+1)t} \exp\left(\delta 3^{(k+1)t}  \right)\P \left[ M > 3^k\right] \\
& \leq 1 + \sum_{k=0}^\infty \delta t 3^{(k+1)t} \exp\left(\delta 3^{(k+1)t}  \right) \sum_{n=k}^\infty \P \left[ X_n \geq 1 \right] \\
& \leq 1 + \sum_{k=0}^\infty 2\delta t 3^{(k+1)t} \exp\left(\delta 3^{(k+1)t}  \right) \sum_{n=k}^\infty \exp\left( - \left( K^{-1}3^{-n\beta} \right)^s \right) \\
& \leq 1 + C\delta t\sum_{k=0}^\infty  3^{(k+1)t} \exp\left(\delta 3^{(k+1)t}  \right)  \exp\left( - \left( C^{-1}3^{-k\beta} \right)^s \right).
\end{align*}
Taking $t:=s\beta$ and imposing the condition $\delta < \frac{3^{-t}}{2}C^{-s}$, we get
\begin{equation*} \label{}
\E \left[ \exp\left( \delta M^t  \right) \right]  \leq 1 + C\delta s\beta\sum_{k=0}^\infty  3^{(k+1)t} \exp\left( - \frac12\left( C^{-1}3^{-k\beta} \right)^s \right) \leq 1 + C\delta s\beta. 
\end{equation*}
Taking $\delta< Cs\beta$ gives $\E \left[ \exp\left( \delta M^t  \right) \right]  \leq 2$, thus
\begin{equation*} \label{}
M \leq \O_{s\beta} \left( \delta^{-1} \right),
\end{equation*}
and the proof of~\eqref{e.OtoM} is complete.
\end{proof}

\begin{proposition}[{Minimal scales for $\S$}]
\label{p.minimalscales}
Let $K,s>0$ and $\S$ be as in Proposition~\ref{p.partitions}. Fix $t\in [1,\infty)$. Then there exist  $C(t,K,s,d,\p)<\infty$, an $\N$-valued random variable $\mathcal{M}_t(\S)$ and, for every exponent $r\in \left(0,\frac {sd}{d+t+s}\right)$, a constant $C'(r,t,K,s,d,\p)<\infty$ such that 
\begin{equation}
\label{e.minimalbound}
\mathcal{M}_t(\S) = \O_r(C')
\end{equation}
and
\begin{equation} \label{e.minimalboundbis}
m\in \N, \, 3^m\geq \mathcal{M}_t(\S) \implies \Lambda_t(\cu_m,\S) \leq C 
\quad \mbox{and} \quad 
\sup_{x\in \cu_m} \size(\cu_{\S}(x)) \leq 3^{\frac{dm}{d+t}}.
\end{equation}
\end{proposition}
\begin{proof}
The proof is organized as follows. In steps 1 and 2, we prove that there exists a random variable $\mathcal{M}_t^0(\S)$ satisfying~\eqref{e.minimalbound} and 
\begin{equation*}
m\in \N, \, 3^m\geq \mathcal{M}_t^0(\S) \implies \Lambda_t(\cu_m,\S) \leq C.
\end{equation*}
In step 3 we prove that there exists andom variable $\mathcal{M}_t^1(\S)$ satisfying~\eqref{e.minimalbound} and 
\begin{equation*}
m\in \N, \, 3^m\geq \mathcal{M}_t^1(\S) \implies \sup_{x\in \cu_m} \size(\cu_{\S}(x)) \leq 3^{\frac{dm}{d+t}}.
\end{equation*}
We then define $\mathcal{M}_t(\S) := \max( \mathcal{M}_t^0(\S), \mathcal{M}_t^1(\S))$ which, in view of the previous results, clearly satisfies~\eqref{e.minimalbound} and~\eqref{e.minimalboundbis}.

\smallskip

\emph{Step 1.} Fix $t \in [1, \infty)$ and $r\in\left(0,\frac {sd}{d+t+s}\right)$. Also fix $m,n\in\N$ with $m>n$. To shorten the notation, we write $\S_{\mathrm{loc}}(z') := \S_{\mathrm{loc}}(z'+\cu_{n+1})$ for each $z'\in 3^n\Zd$. Using Lemma~\ref{l.localizeLambda}, we have 
\begin{align}
\label{e.Lambdatbounds}
\lefteqn{
\Lambda_t(z+\cu_m,\S) 
} \qquad & \\  \notag
& = \frac{|\cu_n|}{|\cu_m|} \sum_{z'\in 3^n\Zd\cap(z+\cu_m)} \Lambda_t(z'+\cu_n,\S) \\   \notag
& \leq \frac{|\cu_n|}{|\cu_m|} \sum_{z'\in 3^n\Zd\cap(z+\cu_m)} \Lambda_t\left(z'+\cu_{n},\S_{\mathrm{loc}}(z') \right) + \O_{\frac s{t+t'}}\left( C3^{-nt'} \right).
\end{align}
Denote
\begin{equation*} \label{}
Z:= \frac{|\cu_n|}{|\cu_m|} \sum_{z'\in 3^n\Zd\cap(z+\cu_m)} \Lambda_t\left(z'+\cu_{n},\S_{\mathrm{loc}}(z') \right).
\end{equation*}
We estimate $Z$, using independence. The claim is that 
\begin{equation}
\label{e.Zbound}
Z \leq C + \O_1 \left( C3^{nt-d(m-n)} \right). 
\end{equation}
First we note that, by the properties of $\S_{\mathrm{loc}}(z')$, we have 
\begin{equation*} \label{}
\Lambda_t(z'+\cu_n,\S_{\mathrm{loc}}(z')) \leq 3^{nt}, \quad \mbox{$\P$--a.s.}
\end{equation*}
Thus $Z \leq 3^{nt}$, $\P$--a.s. 

\smallskip

We now take an enumeration $\{ z^j : 1 \leq j \leq 3^{d(m-n-1)} \}$ of the elements of the set $3^{n+1}\Zd \cap \cu_m$. Next, for each $1\leq j \leq 3^{md}$, we let $\{ z^{i,j} \,:\, 1 \leq i \leq 3^d\}$ be an enumeration of the elements of the set $3^n \Zd \cap \left( z^j + \cu_{n+1} \right)$, such that, for every $1\leq j,j'\leq 3^{d(m-n-1)}$ and $1\leq i \leq 3^d$, $z^j - z^{j'} = z^{i,j} - z^{i,j'}$.The point of this is that, for every $1 \leq i \leq 3^d$ and $1 \leq j < j' \leq 3^{d(m-n-1)}$, we have $\dist\left(z^{i,j} + \cu_n ,z^{i,j'}+ \cu_n\right) \geq 3^n$ and therefore, $\Lambda_t\left(z^{i,j}+\cu_n,\S_{\mathrm{loc}}\left(z^{i,j}\right)\right)$ and $\Lambda_t\left(z^{i,j'}+\cu_n,\S_{\mathrm{loc}}\left(z^{i,j'}\right)\right)$ are $\P$--independent.
Now fix $h > 0$ and compute, using the H\"older inequality and the independence
\begin{align*}
\lefteqn{
\log \E\left[ \exp \left( h3^{-nt} Z \right)  \right] 
} \qquad & \\
& = \log \E \left[ \prod_{i =1}^{3^d}  \prod_{ j =1 }^{ 3^{d(m-n-1)}} \exp\left( h 3^{-nt-d(m-n)} \Lambda_t\left(z^{i,j} +\cu_n,\S_{\mathrm{loc}}(z^{i,j}) \right) \right) \right] \\
& \leq 3^{-d} \sum_{i =1}^{3^d} \log \E \left[  \prod_{ j =1 }^{ 3^{d(m-n-1)}}   \exp\left( h 3^{-nt-d(m-n-1)}  \Lambda_t\left(z^{i,j}+\cu_n,\S_{\mathrm{loc}}(z^{i,j}) \right) \right) \right] \\
& \leq 3^{-d} \sum_{i =1}^{3^d}   \sum_{ j =1 }^{ 3^{d(m-n-1)}}  \log \E \left[  \exp\left( h 3^{-nt-d(m-n-1)} \Lambda_t\left(z^{i,j}+\cu_n,\S_{\mathrm{loc}}(z^{i,j}) \right) \right) \right].
\end{align*}
This inequality can be rewritten
\begin{multline*}
\log \E\left[ \exp \left( h3^{-nt} Z \right)  \right]  \\
\leq 3^{-d} \sum_{z'\in 3^n\Zd\cap(z+\cu_m)}  \log \E \left[  \exp\left( h 3^{-nt-d(m-n-1)} \Lambda_t\left(z+\cu_n,\S_{\mathrm{loc}}(z') \right) \right) \right].
\end{multline*}
Next we use the elementary inequality
\begin{equation*} \label{}
\forall y \in [0,1], \quad \exp(y) \leq 1+2y
\end{equation*}
to get, for every $h \in [0,3^{d(m-n-1)}]$, 
\begin{multline*} \label{}
\exp\left( h 3^{-nt-d(m-n-1)} \Lambda_t\left(z'+\cu_n,\S_{\mathrm{loc}}(z') \right) \right) \\
\leq 1+2h 3^{-nt-d(m-n-1)} \Lambda_t\left(z'+\cu_n,\S_{\mathrm{loc}}(z') \right).
\end{multline*}
Taking the expectation of this, applying the previous display and using the elementary inequality
\begin{equation*}
\forall y\geq 0, \ \log(1+y) \leq y,
\end{equation*}
we get
\begin{align*} \label{}
\log \E\left[ \exp \left( h3^{-nt} Z \right)  \right] 
& \leq 3^{d(m-n)} \log \left( 1 + 2 h 3^{-nt-d(m-n-1)} \E \left[  \Lambda_t\left(z'+\cu_n,\S_{\mathrm{loc}}(z') \right) \right] \right)  \\
& \leq 2h 3^{-nt+d} \E \left[  \Lambda_t\left(z'+\cu_n,\S_{\mathrm{loc}}(z') \right) \right] \\
& \leq Ch3^{-nt}.
\end{align*}
Taking $h:= 3^{d(m-n-1)}$ yields
\begin{equation*}
\E\left[ \exp \left( 3^{d(m-n-1)-nt} Z \right)  \right] \leq \exp\left( C3^{d(m-n)-nt} \right). 
\end{equation*}
From this and Chebyshev's inequality, we obtain a constant~$C$ such that 
\begin{equation*}
\P \left[ Z \geq C + h \right] \leq \exp\left( -h 3^{d(m-n)-nt} \right)
\end{equation*}
This implies~\eqref{e.Zbound}.

\smallskip

\emph{Step 2.} We complete the proof by applying union bounds. Combining~\eqref{e.Lambdatbounds} and~\eqref{e.Zbound} yields
\begin{equation*} \label{}
\Lambda_t(z+\cu_m,\S)  \leq C_0 + \O_1 \left( C3^{nt-d(m-n)} \right)+ \O_{\frac s{t+t'}}\left( C3^{-nt'} \right).
\end{equation*}
We now choose
\begin{equation*} \label{}
n:= \left\lceil \frac{dm}{d+t+s} \right\rceil 
\end{equation*}
so that the previous line becomes 
\begin{equation*} \label{}
\Lambda_t(z+\cu_m,\S)  \leq C_0 + \O_1 \left( C3^{-\frac{sd}{d+t+s}m} \right)+ \O_{\frac s{t+t'}}\left( C3^{-\frac{dt'}{d+t+s}m}  \right).
\end{equation*}
Define 
\begin{equation*} \label{}
\M_t^0(\S):= \sup\left\{ 3^m \,:\, \Lambda_t(z+\cu_m,\S)  \geq C_0 +2 \right\}
\end{equation*}
and apply Lemma~\ref{l.OtoM} to find that 
\begin{equation*} \label{}
\M_t^0(\S) \leq \O_{\frac{sd}{d+t+s}}(C) + \O_{\frac{st'd}{(t+t')(d+t+s)}} (C) \leq\O_{\frac{st'd}{(t+t')(d+t+s)}} (C). 
\end{equation*}
Taking $t'$ sufficiently large, depending on $(r,d,s,t)$, then 
\begin{equation*} \label{}
\frac{st'd}{(t+t')(d+t+s)} > r
\end{equation*}
and thus we obtain
\begin{equation*} \label{}
\M_t^0(\S) \leq \O_{r} (C).
\end{equation*}

\smallskip

\emph{Step 3.} By Proposition~\ref{p.partitions} (iii), we have, for every $m \in \N$
\begin{align*}
\P \left[  \sup_{x\in \cu_m} \size(\cu_{\S}(x)) > 3^{\frac{dm}{d+t}} \right] & \leq \sum_{x\in \cu_m} \P\left[ \size(\cu_{\S}(x)) > 3^{\frac{dm}{d+t}} \right] 
 \leq 2 \cdot 3^{dm} \exp \left( - C^{-1}3^{\frac{dsm}{d+t}} \right).
\end{align*}
From this we deduce that for every $m \in \N$
\begin{equation*}
\indc_{\left\lbrace  \sup_{x\in \cu_m} \size(\cu_{\S}(x)) > 3^{\frac{dm}{d+t}} \right\rbrace} \leq \O_1 \left( C3^{-\frac{dsm}{d+t}}\right).
\end{equation*}
Applying Lemma~\ref{l.OtoM} with $s=1$, $\beta = \frac{ds}{d+t}$ and 
\begin{equation*}
X_m := 2 \cdot \indc_{\left\lbrace  \sup_{x\in \cu_m} \size(\cu_{\S}(x)) > 3^{\frac{dm}{d+t}} \right\rbrace}
\end{equation*}
shows that the random variable
\begin{equation*}
\M_t^1(\S) := \sup \left\lbrace 3^m \in \N ~:~ X_m \geq 1 \right\rbrace
\end{equation*}
satisfies
\begin{equation*}
\M_t^1(\S) \leq \O_{\frac{dm}{d+t}}(C).
\end{equation*}
Since $r \leq \frac{dm}{d+t+s} \leq \frac{dm}{d+t}$, we also have
\begin{equation*}
\M_t^1(\S) \leq \O_{r}(C)
\end{equation*}
and by definition of $\left( X_m\right)_{m \in \N}$ and $\M_t^1(\S)$, we have
\begin{equation*}
m\in \N, \, 3^m\geq \mathcal{M}_t^1(\S) \implies \sup_{x\in \cu_m} \size(\cu_{\S}(x)) \leq 3^{\frac{dm}{d+t}}.
\end{equation*}
The proof is complete.
\end{proof}

\subsection{The partition $\Pa$ of well-connected cubes}

We apply the construction of the previous subsection to obtain a random partition~$\Pa$ of $\Zd$ which simplifies the geometry of the percolation cluster. This partition plays an important role in the rest of the paper. For bounds on the ``good event" which allows us to construct the partition, we use the important results of Pisztora~\cite{P}, Penrose and Pisztora~\cite{PP} and Antal and Pisztora~\cite{AP}. We first recall some definitions introduced in those works. 

\begin{definition}[Crossability and Crossing cluster]
We say that a cube $\cu$ is \emph{crossable} (with respect to $\a \in\Omega$) if each of the $d$ pairs of opposite $(d-1)$--dimensional faces of $\cu$ are joined by an open path in $\cu$. We say that a cluster~$\C \subseteq \cu$ is a \emph{crossing cluster for $\cu$} if $\C$ intersects each of the $(d-1)$--dimensional faces of $\cu$. 
\end{definition}

\begin{definition}[Good cube] \label{def.goodcube}
We say that a triadic cube $\cu\in\T$ is \emph{well-connected} if there exists a crossing cluster~$\C$ for the cube $\cu$ such that:
\begin{enumerate}

\item[(i)] each cube $\cu'$ with $\size(\cu') \in \left[ \frac{1}{10} \size(\cu), \frac 12 \size(\cu) \right]$ and $\cu' \cap \frac 34 \cu \neq \emptyset$ is crossable; and 

\item[(ii)] for every cube $\cu'$ as in (i) and every path~$\gamma\subseteq \cu'$ with~$\diam(\gamma) \geq \frac{1}{10} \size(\cu)$, we have that~$\gamma$ is connected to~$\C$ within~$\cu'$. That is, there is another path~$\gamma'$ within~$\cu'$ which connects a point of~$\gamma$ to a point of~$\mathscr{C}$. 
\end{enumerate}
We say that~$\cu\in\T$ is a \emph{good cube} if $\size(\cu) \geq 3$, $\cu$ is well-connected and each of the~$3^d$ successors of~$\cu$ are well-connected. We say that $\cu\in\T$ is a \emph{bad cube} if it is not a good cube. 
\end{definition}

The following estimate on the probability of the cube $\cu_n$ being good is a consequence~\cite[Theorem 3.2]{P} and~\cite[Theorem 5]{PP}, as recalled in~\cite[(2.24)]{AP}.

\begin{lemma}[{\cite[(2.24)]{AP}}]
\label{l.AP}
There exist $C(d,\p)<\infty$ and $c(d,\p)\in(0,\frac12]$ such that, for every $m\in\N$, 
\begin{equation} \label{e.pgoodness}
\P \left[ \cu_n \ \mbox{is good} \right] \geq 1 - C \exp\left( - c3^{n} \right).
\end{equation}
\end{lemma}

It follows from Definition~\ref{def.goodcube} that, for every good cube $\cu$, there exists a unique maximal crossing cluster for $\cu$ which is contained in $\cu$. We denote this cluster by $\C_*(\cu)$. In the next lemma, we record the observation that adjacent triadic cubes which have similar sizes and are both good have connected clusters. 

\begin{lemma}
\label{l.connectivity}
Let $n,n'\in\N$ with $|n-n'|\leq 1$ and $z,z'\in 3^n\Zd$ such that 
\begin{equation*} \label{}
\dist\left(\cu_{n}(z), \cu_{n'}(z') \right) \leq1.
\end{equation*}
Suppose also that $\cu_{n}(z)$ and $\cu_{n'}(z')$ are good cubes. Then there exists a cluster~$\C$ such that 
\begin{equation*} \label{}
\C_*( \cu_{n}(z) ) \cup \C_*( \cu_{n'}(z') )\subseteq  \C \subseteq \cu_{n}(z) \cup \cu_{n'}(z').
\end{equation*}
\end{lemma}
\begin{proof}
We may suppose that $n\leq n'$. Let~$x$ be the center point on the face of $\cu_n(z)$ which is adjacent to $\cu_{n'}(z')$. If $n' = n+1$, then we let ${\cu}_{n}'$ be the successor of $\cu_{n'}(z')$ which is adjacent to $\cu_n(z)$ and otherwise, if $n'=n$, we set ${\cu}_{n}':= \cu_{n'}(z')$. Consider the cube $\cu$ of size~$\frac123^n$ centered at~$x$. Since $\cu_n(z)$ is a good cube and therefore well-connected, $\cu$ is crossable. Let $\gamma\subseteq\cu$ be a path which connects the two faces of $\cu$ which are parallel to the face of $\cu_n(z)$ containing~$x$. There are two subpaths of $\gamma \subseteq \cu$ which, respectively, lie inside of $\cu_{n}(z)$ and $\tilde{\cu}_{n}$ and have length at least $\frac14 3^n$. Therefore, since both of the cubes $\cu_n(z)$ and ${\cu}_{n}'$ are well-connected, we conclude that $\gamma$ intersects both $\C_*(\cu_n(z))$ and $\C_*({\cu}_n')\subseteq\C_*(\cu_{n'}(z'))$. Taking $\C$ to be the cluster
\begin{equation*} \label{}
\C:= \gamma \cup \C_*(\cu_n(z)) \cup \C_*(\cu_{n'}(z'))
\end{equation*}
completes the proof. 
\end{proof}

We next define our partition~$\Pa$.

\begin{definition}
\label{d.Pa}
We let $\Pa\subset \T$ be the partition~$\S$ of $\Zd$ obtained by applying Proposition~\ref{p.partitions} to the collection
\begin{equation*} \label{}
\G:= \left\{ \cu\in \T\,:\, \cu \ \mbox{ is good} \right\}.
\end{equation*}
We also let $\Pa_{\mathrm{loc}}(z+\cu_n)$ denote the local partitions $\S_{\mathrm{loc}}(z+\cu_n)$ in the statement of Proposition~\ref{p.partitions}(iv). 
\end{definition}

The (random) partition $\Pa$ plays an important role throughout the rest of the paper. We also denote by $\Pas$ the collection of triadic cubes which contain some element of $\Pa$, that is, 
\begin{equation*} \label{}
\Pas:= \left\{ \cu\,:\,  \mbox{$\cu$ is a triadic cube and $\cu \supseteq \cu'$ for some $\cu'\in \Pa$} \right\}.
\end{equation*}
Notice that every element of $\Pas$ can be written in a unique way as a disjoint union of elements of $\Pa$.
According to Proposition~\eqref{p.partitions}(i), every triadic cube containing an element of $\Pa$ is good. By Propositions~\ref{p.partitions}(iii) and Lemma~\ref{l.AP}, there exists $C(d,\p)<\infty$ such that, for every $x\in\Zd$,
\begin{equation}
\label{e.partitionO1}
\size\left( \cu_{\Pa}(x) \right) \leq \O_1(C).
\end{equation}
By the properties of~$\Pa$ given in Proposition~\ref{p.partitions}(i) and (ii) and Lemma~\ref{l.connectivity}, the maximal crossing cluster $\C_*(\cu)$ of an element $\cu\in \Pas$ must satisfy $\C_*(\cu) \subseteq \C_\infty$, since the union of all crossing clusters of elements of $\Pa$ is unbounded and connected. Indeed, we have the stronger property that, for every $n\in\N$ and $z\in 3^n\Zd$, 
\begin{equation} 
\label{e.goodness}
\cu_{n+1}(z) \in \Pas
 \implies 
\emptyset\neq \C_*(\cu_n(z)) \subseteq \C_*(\cu_{n+1}(z)) \cap \cu_n(z)  = \C_\infty \cap \cu_n(z).
\end{equation}
Given $\cu\in \Pa$ with $\cu = \cu_{n+1}(z)$ for $n\in\N$ and $z\in3^{n+1}\Zd$, we let $\zbar(\cu)$ denote the element of~$\C_*(\cu)\cap \cu_n(z)$ which is closest to~$z$ in the manhattan distance; if this is not unique, then we break ties by the lexicographical order.

\begin{corollary}
\label{c.paths}
For every $x,y\in\Zd$ and path $\gamma_0$ connecting $x$ and $y$, there exists a path $\gamma \subseteq\C_\infty$ connecting $\zbar(\cu_{\Pa}(x))$ and $\zbar(\cu_{\Pa}(y))$ such that 
\begin{equation*} \label{}
\gamma \subseteq \bigcup_{z\in \gamma_0} \cu_{\Pa}(z). 
\end{equation*}
\end{corollary}
\begin{proof}
This is immediate from Lemma~\ref{l.connectivity} and the property of~$\Pa$ from~Proposition~\ref{p.partitions}(ii).
\end{proof}

\section{Elliptic and functional inequalities on clusters}
\label{s.functional}

In this section, we use the partition $\Pa$ constructed in the previous section to tame the large-scale geometry of the percolation cluster. In particular, we give direct arguments leading to a quantitative Sobolev-Poincar\'e inequality. This allows us to develop some basic elliptic estimates we will need later in the paper. Many of the results in this section are similar to, and overlap with, those of Barlow~\cite{Ba}. Our approach however is somewhat different and, we believe, can be pushed beyond the particular case of a Bernoulli bond percolation cluster considered in this paper. Also, some of the estimates we prove (e.g., the Meyers estimate in Proposition~\ref{p.meyers}) are new and needed in the following sections.

\subsection{Functional inequalities on clusters}
As in the previous section, for each $\cu\in \Pas$, we let $\C_*(\cu)$ denote the unique maximal crossing cluster for $\cu$. Recall that, given an arbitrary subset $U \subseteq \Zd$, we define the \emph{closure $\cl_\Pa(U)$ of $U$ with respect to $\Pa$} by 
\begin{equation*} \label{}
\cl_\Pa(U):= \bigcup_{z\in U} \cu_\Pa(z). 
\end{equation*}
We then define $\C_*(U)$ to be the maximal cluster contained in $\cl_\Pa(U)$ which contains each of the clusters $\C_*(\cu_\Pa(z))$ for every $z\in U$. 

\begin{definition} \label{def.coarsenfunction}
Given a function $w: \C_*(U) \to \R$, the \emph{coarsening $\left[ w \right]_{\Pa}$ of $w$ with respect to $\Pa$} is a function $\cl_\Pa(U)\to \R$ defined by 
\begin{equation*} \label{}
\left[ w \right]_{\Pa}\!(x):= w\left( \zbar\left(\cu_\Pa(x) \right) \right),\quad x\in \cl_\Pa(U).
\end{equation*}
where $ \zbar\left(\cu_\Pa(x) \right)$ is the point of $\C_*(\cu_\Pa(x))$ which is the closest to the center of the cube for the infinite norm $|\cdot|_\infty$ (if there is more than one candidat, pick the one that comes first for the lexicographical order). In particular, $\left[ w \right]_{\Pa}$ we note that is defined on the closure $\cl_{\Pa}(U)$ and is constant on the elements of~$\Pa$.
\end{definition}

 The advantage of~$\left[ w \right]_{\Pa}$ is that it allows us to make use of the simpler and more favorable geometric structure of~$\Zd$ compared to the percolation clusters. The price to pay is the difference between~$w$ and~$\left[ w \right]_{\Pa}$, which depends of course on the coarseness of the partition~$\Pa$ and the control one has on~$\nabla u$. Indeed, we show next that the difference~$w-\left[ w \right]_{\Pa}$ can be controlled in $L^s(\C_*(U))$ for $s\in [1,\infty)$ in terms of a weighted~$L^s(\C_*(U))$ norm of~$\nabla w$. The weight function represents the coarseness of the partition~$\Pa$ in~$\cl_\Pa(U)$.
In what follows, $\nabla w \indc_{\{\a\neq0\}}$ denotes the vector field of $\nabla w$ restricted to the open edges:
\begin{equation*}
\left(\nabla w\indc_{\{\a\neq 0\}} \right)(e):= \nabla w(e) \indc_{ \{ \a(e)\neq 0\}}, \quad e\in \Ed(U).
\end{equation*}
Recall that the notation $|F|(x)$ for a vector field $F$ is defined in~\eqref{e.magF}.

\begin{lemma}
\label{l.coarseLs}
For every bounded~$U\subseteq\Zd$, $1\leq s<\infty$ and $w: \C_*(U) \to \R$,
\begin{multline}
\label{e.coarseLs}
\int_{\C_*(U)} \left| w(x) - \left[ w \right]_\Pa\!(x) \right|^s\,dx \\
\leq C^s \sum_{\cu \in \Pa,\, \cu \subseteq \cl_\Pa(U)}  \size(\cu)^{sd} \int_{\cu\cap \C_*(U)}\left| \nabla w\indc_{\{\a\neq 0\}}\right|^s(x)  \,dx.
\end{multline}
\end{lemma}
\begin{proof}
For each $x\in \C_*(U)$, there is a non-self intersecting path connecting $x$ and $\zbar\left(\cu_\Pa(x) \right)$ which belongs to the union of the elements $\cu$ of $\Pa$ which are contained in $\cl_{\Pa}(U)$ and satisfy $\dist(\cu,\cu_\Pa(x)) \leq 1$. It follows that
\begin{align} \label{e.justintegrate}
\left| w(x) - \left[ w \right]_\Pa\!(x) \right|
& = \left| w(x) - w\!\left( \zbar\left(\cu_\Pa(x) \right)\right) \right|   \\
& \leq \sum_{\cu\in \Pa, \ \dist(\cu,\cu_\Pa(x))\leq 1} \int_{\cu \cap \C_*(U)} \left| \nabla w\indc_{\{\a\neq 0\}}\right| (y) \,dy. \notag
\end{align}
Summing over $x\in \cu\cap \C_*(U)$ for a fixed $\cu\in\Pa$ with $\cu\subseteq \cl_\Pa(U)$ and using property~(ii) from Proposition~\ref{p.partitions} for $\Pa$ yields
\begin{align*}
\int_{\cu\cap \C_*(U)} \left| w(x) - \left[ w \right]_\Pa\!(x) \right|^s\,dx
& \leq C^s \left| \cu \right| \left( \int_{\cu\cap \C_*(U)}\left| \nabla w\indc_{\{\a\neq 0\}}\right|(x) \,dx \right)^s \\
& \leq C^s \left| \cu \right|^{s} \int_{\cu\cap \C_*(U)}\left| \nabla w\indc_{\{\a\neq 0\}} \right|^s(x) \,dx.
\end{align*}
Summing over $\cu\in \Pa$ with $\cu\subseteq \cl_{\Pa}(U)$ yields the lemma. 
\end{proof}

We next show that we can control $L^s$ norms of $\left| \nabla \left[ w \right]_{\Pa}\right|$ by those of $\left| \nabla w \indc_{\{\a\neq0\}}\right|$ and the coarseness of the partition~$\Pa$. The proof is very simple and similar to that of the previous lemma.

\begin{lemma}
\label{l.coarsegrads}
For every bounded $U\subseteq \Zd$, $1\leq s < \infty$ and $w:\C_*(U) \to \R$, 
\begin{multline} 
\label{e.coarsegrads}
\int_{\cl_{\Pa}(U)} \left| \nabla  \left[ w \right]_{\Pa} \right|^s(x) \,dx \\
\leq C^s  \sum_{\cu\in\Pa,\, \cu\subseteq \cl_{\Pa}(U)} \size(\cu)^{sd-1} \int_{\cu\cap\C_*(U)} \left|\nabla w \indc_{\{\a\neq 0\}} \right|^s(x)\,dx.
\end{multline}
\end{lemma}
\begin{proof}
The gradient $\nabla\!\left[ w \right]_{\Pa}$ is supported on the edges $\{x,y\}$ such $x\in \cu$ and $y\in \cu'$ for two disjoint, neighboring elements $\cu\sim\cu'$ of $\Pa$. On such edges, we have
\begin{align*} \label{}
\left| \nabla  \left[ w \right]_{\Pa} (\{x,y\}) \right|
& = \left| w(\zbar(\cu(x))) - w(\zbar(\cu'(y)))\right|
\end{align*}
Recalling that there exists a path between $\zbar(\cu(x))$ and $\zbar(\cu'(y))$ which lies entirely in $\cu\cup \cu'$ and summing over the edges along this path, we find that 
\begin{equation*} \label{}
 \left| w(\zbar(\cu(x))) - w(\zbar(\cu'(y)))\right| \leq C \int_{(\cu\cup \cu') \cap \C_*(U)} \left| \nabla w\indc_{\{\a\neq 0\}} \right|(z)\,dz.
\end{equation*}
Since two neighboring elements of $\Pa$ have sizes within a factor of three, a property given by Proposition~\ref{p.partitions}(ii), the number of such edges between~$\cu$ and~$\cu'$ is at least
\begin{equation*} \label{}
c \size(\cu)^{d-1}. 
\end{equation*}
Finally, we note that every $\cu\in\Pa$ has at least $2^d$ neighboring elements of~$\Pa$. 

\smallskip

The above assertions imply that 
\begin{align*} \label{}
\lefteqn{
\int_{\cl_\Pa(U)} \left| \nabla  \left[ w \right]_{\Pa} \right|^s(x) \,dx
}  \ \  & \\
& \leq C^s \sum_{\cu\in\Pa,\, \cu\subseteq \cl_\Pa(U)} \size(\cu)^{d-1} \sum_{\cu' \in\Pa, \, \cu\sim \cu'} \left( \int_{(\cu\cup \cu') \cap \C_*(U)} \left| \nabla w \indc_{\{\a\neq 0\}}\right|(z)\,dz \right)^s \\
& \leq C^s  \sum_{\cu\in\Pa,\, \cu\subseteq \cl_\Pa(U)} \size(\cu)^{d-1} \sum_{\cu' \in\Pa, \, \cu\sim \cu'} \left| \cu \cup \cu' \right|^{s-1} \int_{(\cu\cup \cu') \cap \C_*(U)} \left| \nabla w\indc_{\{\a\neq 0\}} \right|^s(z)\,dz \\
& \leq C^s  \sum_{\cu\in\Pa,\, \cu\subseteq \cl_\Pa(U)} \size(\cu)^{d-1+d(s-1)} \int_{\cu\cap \C_*(U)} \left|\nabla w\indc_{\{\a\neq 0\}} \right|^s(z)\,dz.
\end{align*}
This completes the proof. 
\end{proof}

The previous two lemmas imply a Sobolev-Poincar\'e-type inequality on the clusters, borrowing the result from the classical inequalities on $\Rd$ by comparing $w$ to $\left[ w \right]_{\Pa}$. This is strong evidence of our informal assertion that ``the geometry of $\C_\infty$ is quantitatively like that of~$\Rd$ on scales larger than~$\Pa$." 

\smallskip

Before giving the statement, we recall that if~$s\in [1,\infty)$ then the Sobolev conjugates~$s^*$ and $s_*$ of~$s$ in dimension~$d$ are defined by
\begin{equation*} \label{}
s^* := \left\{ 
\begin{aligned}
& \frac{sd}{d-s} & \mbox{if} & \ s < d, \\
&  \infty & \mbox{if} & \ s \geq d.
\end{aligned} 
\right.
\end{equation*}
If $s \in \left[ \frac{d}{d-1},\infty\right)$, then we also define
\begin{equation*} \label{}
s_*:= \frac{sd}{s+d}
\end{equation*}
so that $(s_*)^* = s$. 

\begin{proposition}[{Sobolev inequality for $\C_*(\cu)$}]
\label{p.sobolev}
Suppose that $s\in [\frac{d}{d-1},\infty)$ and $\cu\in\Pas$. Let $w:\C_*(\cu)\to \R$ satisfy one of the following conditions:  
\begin{equation}
\label{e.sobolevass}
\left\{
\begin{aligned}
& \int_{\C_*(\cu)} w(x) \,dx = 0 
\quad \mbox{or} \quad 
 w(x) = 0 \ \mbox{for every} \ x \in \partial_\a\cu,  \\
& \int_{\cu} \left[ w \right]_{\Pa}(x) \,dx = 0
\quad \mbox{or} \quad 
 \left[w\right]_{\Pa} (x) = 0 \ \mbox{for every} \ x \in \partial_\a\cu.
\end{aligned}
\right.
\end{equation}
Then there exists $C(s,d,\p)<\infty$ such that
\begin{equation} 
\label{e.sobolev}
\int_{\C_*(\cu)} \left| w(x) \right|^s\,dx 
\leq C \left( \sum_{\cu'\in\Pa,\, \cu'\subseteq \cu} \size(\cu')^{sd} \int_{\cu'\cap\C_*(\cu)} \left|\nabla w \indc_{\{\a\neq 0\}}\right|^{s_*}(x)\,dx \right)^{\frac{s}{s_*}}
\end{equation}
\end{proposition}

Before giving the proof of Proposition~\ref{p.sobolev}, let us comment on the form of the right side in~\eqref{e.sobolev}. The classical Sobolev inequality for functions $w \in W_0^{1,s_*}(\cu)$ (or mean-zero functions $w \in W^{1,s_*}(\cu)$) in a cube $\cu \subseteq\Rd$ states that 
\begin{equation*}
\int_\cu \left| w(x) \right|^s\,dx \leq C \left( \int_{\cu} \left| \nabla w (x) \right|^{s_*}\,dx \right)^{\frac s{s_*}}.
\end{equation*}
The term on the right side of this inequality is similar to the first term on the right side of~\eqref{e.sobolev}, except there is the weight $\size(\cu')^{sd}$ representing the size of the local cube in the partition~$\Pa$. The size of the elements of~$\Pa$ are of course not uniformly bounded, however they are typically of unit size, by~\eqref{e.partitionO1}, with exponential stochastic integrability. In particular, one may use H\"older's inequality to separate this weight from the function $\left|\nabla w\right|^{s_*}$ in the integrand at the cost of an arbitrarily small loss of the exponent $s_*$ while the sums over the weights can be controlled above a minimal scale by Proposition~\ref{p.minimalscales}.

\begin{proof}[{Proof of Proposition~\ref{p.sobolev}}]
Rather than~\eqref{e.sobolevass}, we first prove the proposition under the assumption that 
\begin{equation}
\label{e.sobolevass2}
\int_{\cu} \left[ w \right]_{\Pa}(x) \,dx = 0 \quad \mbox{or} \quad  \left[ w \right]_{\Pa}= 0 \ \mbox{on} \ \partial_\a \cu. 
\end{equation}
In this case, the usual Sobolev inequality on $\Zd$ (which follows easily from the one on $\Rd$ by affine interpolation, for example) applied to $\left[ w \right]_{\Pa}$ gives us that 
\begin{equation*} \label{ususobineqZd}
\left( \int_{\cu} \left| \left[ w \right]_{\Pa}\!(x) \right|^s\,dx \right)^{\frac1s} \leq C\left( \int_{\cu}  \left| \nabla \left[ w\right]_{\Pa} \right|^{s_*}\!(x) \,dx \right)^{\frac1{s_*}}.
\end{equation*}
We then apply Lemma~\ref{l.coarsegrads} to estimate the right side, which gives
\begin{align*} \label{}
\left( \int_{\cu} \left| \left[ w \right]_{\Pa}\!(x) \right|^s\,dx \right)^{\frac1s} 
& \leq \left( \sum_{\cu\in\Pa,\, \cu'\subseteq \cu} \size(\cu')^{sd-1} \int_{\cu'\cap\C_*(\cu)} \left|\nabla w \indc_{\{\a\neq 0\}}\right|^{s_*}(x)\,dx \right)^{\frac1{s_*}},
\end{align*}
and use Lemma~\ref{l.coarseLs} and the triangle inequality to estimate the left side and combine this with the previous inequality to get 
\begin{multline*} \label{}
\left( \int_{\C_*(\cu)} \left| w(x) \right|^s\,dx \right)^{\frac1s}
\leq C \left( \sum_{\cu'\in\Pa,\, \cu'\subseteq \cu} \size(\cu')^{sd-1} \int_{\cu'\cap\C_*(\cu)} \left|\nabla w \indc_{\{\a\neq 0\}}\right|^{s_*}(x)\,dx \right)^{\frac1{s_*}} \\
+ C \left(  \sum_{\cu'\in\Pa,\, \cu'\subseteq \cu} \size(\cu')^{sd} \int_{\cu'\cap\C_*(\cu)} \left|\nabla w \indc_{\{\a\neq 0\}}\right|^{s}(x)\,dx \right)^{\frac1{s}}.
\end{multline*}
We may estimate the second term on the right side thanks to the following inequality: since $\frac{s}{s_*} > 1$, we have that, for every $n \in \N$ and finite sequence of nonnegative real numbers $\left\{ a_i\right\}_{1 \leq i \leq n}$,
\begin{equation*}
\sum_{i=1}^n a_i^{\frac{s}{s_*}} \leq \left( \sum_{i=1}^n a_i \right)^{\frac{s}{s_*}}.
\end{equation*}
Applying this to the second term on the right side gives
\begin{multline*}
\left(  \sum_{\cu'\in\Pa,\, \cu'\subseteq \cu} \size(\cu')^{sd} \int_{\cu'\cap\C_*(\cu)} \left|\nabla w \indc_{\{\a\neq 0\}}\right|^{s}(x)\,dx \right)^{\frac1{s}} \\
\leq \left(  \sum_{\cu'\in\Pa,\, \cu'\subseteq \cu} \size(\cu')^{s_*d} \int_{\cu'\cap\C_*(\cu)} \left|\nabla w \indc_{\{\a\neq 0\}}\right|^{s_*}(x)\,dx \right)^{\frac1{s_*}}.
\end{multline*}
Noticing that $s_* d \leq sd$ and $sd -1 \leq sd$ completes the proof of the proposition under the assumption~\eqref{e.sobolevass2}.

\smallskip

To prove the proposition under the assumption
\begin{equation} \label{e.sobolevassbis}
\int_{\C_*(\cu)} w(x) \, dx =0,
\end{equation}
we apply the result to the function $w - \frac{1}{|\C_*(\cu)|} \int_{\C_*(U)}   \left[ w \right]_\Pa\!(x) \,dx$ which satisfies assumption~\eqref{e.sobolevass2}. This yields
\begin{multline*} \label{}
\left( \int_{\C_*(\cu)} \left| w(x) - \frac{1}{|\C_*(\cu)|} \int_{\C_*(U)}   \left[ w \right]_\Pa\!(x) \,dx \right|^s\,dx \right)^{\frac1s} \\
\leq C \left( \sum_{\cu'\in\Pa,\, \cu'\subseteq \cu} \size(\cu')^{sd-1} \int_{\cu'\cap\C_*(\cu)} \left|\nabla w \indc_{\{\a\neq 0\}}\right|^{s_*}(x)\,dx \right)^{\frac1{s_*}} \\
+ C \left(  \sum_{\cu'\in\Pa,\, \cu'\subseteq \cu} \size(\cu')^{sd} \int_{\cu'\cap\C_*(\cu)} \left|\nabla w \indc_{\{\a\neq 0\}}\right|^{s}(x)\,dx \right)^{\frac1{s}}.
\end{multline*}
To complete the proof we use Lemma~\ref{l.coarseLs} to obtain
\begin{align*}
\lefteqn{\left|\frac{1}{|\C_*(\cu)|} \int_{\C_*(U)}   \left[ w \right]_\Pa\!(x) \,dx \right|^s} \qquad & \\
		& \leq \left|\frac{1}{|\C_*(\cu)|} \int_{\C_*(U)}    w (x) \,dx - \frac{1}{|\C_*(\cu)|} \int_{\C_*(U)}   \left[ w \right]_\Pa\!(x) \,dx \right|^s \\
		& \leq \frac{1}{|\C_*(\cu)|} \int_{\C_*(U)} \left|  w(x) -  \left[ w \right]_\Pa\!(x) \right|^s \, dx  \\
		& \leq \frac{C^s}{|\C_*(\cu)|}  \sum_{\cu' \in \Pa,\, \cu' \subseteq \cu}  \size(\cu')^{sd} \int_{\cu' \cap \C_*(\cu)}\left| \nabla w\indc_{\{\a\neq 0\}}\right|^s(x)  \,dx.
\end{align*}
Combining the two previous displays completes the proposition under the assumption~\eqref{e.sobolevassbis}.

\smallskip

We finally prove the proposition under the assumption
\begin{equation*}
w(x) = 0 \mbox{ for every } x \in \partial_\a \cu
\end{equation*}
The main idea is to apply the Sobolev inequality under the assumption $\left[ w \right]_\Pa =0$ on $\partial \cu$. To do so we define the following function $v$ on $\cu$
\begin{equation*}
 v(x) := \left\lbrace
  \begin{array}{lll}
   &  \left[ w \right]_\Pa (x)& \mbox{ if } \cu_\Pa(x) \notin \partial_\Pa \cu, \\
   & 0 & \mbox{ if } \cu_\Pa(x) \in \partial_\Pa \cu. \\
  \end{array}
\right.
\end{equation*}
The function $v$ is almost equal to $\left[ w \right]_\Pa$ with a slight modification on the boundary cubes of the partition where we set $v=0$. Since $u=0$ on $\partial_\a \cu$ and in view of Definition~\ref{def.coarsenfunction}, one can observe that the results of Lemmas~\ref{l.coarseLs} and~\ref{l.coarsegrads} hold with $v$ instead of $\left[ w \right]_\Pa$. We then complete the proof of the Sobolev inequality by adapting the argument as in the case $\left[ w \right]_\Pa =0 $ on $\partial \cu$ with $v$ instead of $\left[ w \right]_\Pa$.
\end{proof}

\subsection{Basic elliptic estimates on clusters}
In this subsection, we record some basic elliptic estimates and show how these allow us to improve some of the estimates from the previous subsection for $\a$-harmonic functions. We remark that the estimates in this section do not use the independence of the ensemble $\{ \a(e) \}_{e\in \Ed}$, merely the independence of the ensemble $\{ \indc_{\{ \a(x)\neq 0\}}\}_{e\in \Ed}$, and so they work for general coefficient fields defined on the percolation clusters.

\smallskip

We begin with Caccioppoli's inequality, following the standard argument.  

\begin{lemma}[{Caccioppoli inequality}]
\label{l.caccioppoli}
Assume $U\subseteq \Zd$ is a cluster and $V \subseteq U$ such that $\dist(V,\partial_\a U) \geq r\geq 1$. Suppose that $u \in \A(U)$. Then there exists $C(\lambda)<\infty$ such that 
\begin{equation}
\label{e.caccioppoli}
\int_{V} \left| \nabla u \indc_{\{\a\neq 0\}}\right|^2(x)\,dx \leq \frac{C}{r^2} \int_{U \setminus \intr(V)} \left| u(x) \right|^2\,dx.
\end{equation}
\end{lemma}
\begin{proof}
Select a function $\eta \in C^1(\Rd)$ satisfying 
\begin{equation}
\label{e.etaprops}
\indc_{V} \leq \eta \leq 1, \quad \eta \equiv 0 \ \mbox{on} \ \partial_\a U, \quad \mbox{and} \quad \left| \nabla \eta \right|^2 \leq \frac{C\eta}{r^2}.
\end{equation}
Testing the equation for $u$ with $\eta u$ (that is, applying~\eqref{e.Avarchar0} with $w=\eta u)$ yields 
\begin{align*}
0  & = \sum_{x,y\in U, x\sim y} \left( \eta(x)u(x) - \eta(y)u(y) \right) \a(\{ x,y\}) \left( u(x) - u(y) \right) \\
& = \sum_{x,y\in U, x\sim y} \eta(x) \left( u(x) - u(y) \right) \a(\{ x,y\}) \left( u(x) - u(y) \right) \\
& \qquad + \sum_{x,y\in U, x\sim y} u(y) \left( \eta(x) - \eta(y) \right) \a(\{ x,y\}) \left( u(x) - u(y) \right).
\end{align*}
Thus we obtain
\begin{align*}
\lefteqn{
\sum_{x,y\in U, x\sim y} \eta(x) \a(\{ x,y\}) \left( u(x) - u(y) \right)^2
} \qquad & \\
& \leq  \sum_{x,y\in U, x\sim y} \left| u(y)\right|  \left| \eta(x) - \eta(y) \right| \a(\{ x,y\}) \left( u(x) - u(y) \right)\\
& \qquad + \sum_{x,y\in U, x\sim y} \zeta((x,y)) \a(\{x,y\}) \eta(x) \left(u(x) - u(y) \right)  \\
& \leq C \sum_{x,y\in U, x\sim y} \frac{\left|\eta(x) - \eta(y)\right|^2}{\eta(x)+ \eta(y)} \left| u(y)\right|^2 \\
& \qquad  + \frac14 \sum_{x,y\in U, x\sim y} \left( \eta(x) + \eta(y) \right)  \a(\{ x,y\})^2 \left( u(x) - u(y) \right)^2 \\
& \leq \frac C{r^2}\sum_{x,y\in U, x\sim y} \indc_{\{\eta(x) \neq \eta(y)\}} (u(y))^2+ \frac12 \sum_{x,y\in U, x\sim y}  \eta(x) \a(\{ x,y\})^2 \left( u(x) - u(y) \right)^2.
\end{align*}
We obtain~\eqref{e.caccioppoli} after absorbing the last term on the right back on the left side and rewriting the expression, using $\a \geq \lambda \indc_{\{\a\neq0\}}$ and~\eqref{e.etaprops}. 
\end{proof}


An important tool for the arguments later in the paper is Meyers' improvement of integrability for the gradients of solutions, adapted to percolation clusters. In the classical setup (for uniformly elliptic equations in $\Rd$), this is a very simple consequence of the Caccioppoli and Sobolev inequalities which imply a reverse H\"older inequality and thus, after an application of the Gehring lemma, the desired estimate. The situation is more complicated in our setting, since the Sobolev inequality is not uniform and depends on the local coarseness of the partition~$\Pa$, as we have seen. The first step is therefore to quantify the probability of a (deterministic) reverse H\"older inequality on large triadic cubes. This will give us another notion of ``good cube" and thus another triadic partition~$\Ra$ which we will use to prove our generalization of Meyers' estimate.

\smallskip

To simplify the statement, we define, for each cube $\cu$ and exponent $s>\frac12(d+2)$ (so the H\"older conjugate $s'$ satisfies $s' < \frac{2}{2_*}$), the quantity
\begin{equation*}
\mathrm{RH}_{s}(\cu) :=
\left\{
\begin{aligned}
& \sup_{u\in \A(\C_\mathrm{max}(3\cu))} 
\frac{\left( \frac{1}{|\cu|} \int_{\C_*(\cu)} \left| \nabla u \indc_{\{\a\neq 0\}}\right|^2(x)\,dx  \right)^{\frac12}}
{\left( \frac1{|3\cu|}\int_{\C_\mathrm{max}(3\cu)} \left|\nabla u \indc_{\{\a\neq 0\}}\right|^{s'2_*}(x)\,dx \right)^{\frac1{s'2_*}}}
   & \ \mbox{if} & \ \cu \ \mbox{is good,} \\
& +\infty & \ \mbox{if} & \ \cu \ \mbox{is bad,}
\end{aligned}
\right.
\end{equation*}
where $\C_\mathrm{max}(3\cu)$ denote the maximal cluster of $3\cu$ containing $\C_*(\cu)$. Notice that if both $\cu$ and $3\cu$ are good cubes (which is in particular the case if $\cu \in \Pa_*$) then $\C_\mathrm{max}(3\cu) = \C_*(3\cu)$ but thanks to this definition the random variable $\mathrm{RH}_{s}(\cu)$ is $\F(3\cu)$--measurable and thus hypothesis~\eqref{e.goodlocal} will be satisfied when we apply Proposition~\ref{p.partitions} is Definition~\ref{def.minscaleR} below. Also we obviously mean the supremum to exclude constant functions. 
In other words, $\mathrm{RH}_s(\cu)$ is the smallest constant $C$ such that every $u\in\A(\C_\mathrm{max}(3\cu))$ satisfies the reverse H\"older inequality 
\begin{multline}
\label{e.reverseHolder}
\left( \frac{1}{|\cu|} \int_{\cu \cap \C_*(\cu)} \left| \nabla u \indc_{\{\a\neq 0\}}\right|^2(x)\,dx  \right)^{\frac12}  \\
\leq C \left( \frac1{|3\cu|}\int_{\C_\mathrm{max}(3\cu)} \left|\nabla u \indc_{\{\a\neq 0\}}\right|^{s'2_*}(x)\,dx \right)^{\frac1{s'2_*}}.
\end{multline}
We next estimate the probability that $\mathrm{RH}_s(\cu_m)$ is larger than a fixed deterministic constant. Recall that the random variable $\mathcal{M}_t(\Pa)$ is given in Proposition~\ref{p.minimalscales}. 

\begin{lemma}[Reverse H\"older inequality]
\label{l.reverse.Holder}
Fix an exponent $s>\frac12(d+2)$. Then there exists a constant $C(s,d,\lambda,\p)<\infty$ such that, for every $m\in\N$,
\begin{equation}
\label{e.reverse.Holder}
3^m \geq \mathcal{M}_{2sd}(\Pa) +C \implies \mathrm{RH}_s(\cu_m) \leq C.
\end{equation}
In particular, for every $m\in\N$ and exponent $t\in \left(0,\frac{d}{d+1 + 2sd}\right)$, there exists a constant $C'(t,s,d,\lambda,\p)<\infty$ such that 
\begin{equation}
\label{e.reverse.Holder2}
\P \left[  \mathrm{RH}_s(\cu_m) > C \right] \leq C' \exp\left( -3^{mt} \right).
\end{equation}
\end{lemma}
\begin{proof}
To setup the argument, fix an exponent 
\begin{equation*} \label{}
s> \frac12(d+2).
\end{equation*}
We also fix an integer $m\in\N$ satisfying
\begin{equation*}
3^m \geq \mathcal{M}_{2sd}(\Pa)
\end{equation*}
and a function $u \in \A(\C_*(\cu_m))$. Note that $m \geq \mathcal{M}_{2sd}(\Pa)$ implies that 
\begin{equation}
\label{e.goodnessm}
\frac{1}{\left| \cu_m \right|} \sum_{x \in \cu_m} \size\left(\cu_\Pa(x) \right)^{2sd} \leq C 
\end{equation}
and, in particular, that~$\cu_m$ is good. The goal is to prove~\eqref{e.reverseHolder} for a deterministic constant $C(s,d,\lambda,\p)<\infty$.

\smallskip
We begin by applying~\eqref{e.caccioppoli} and then~\eqref{e.sobolev}, which give us 
\begin{align}
\label{e.reverseHolderpre}
\lefteqn{
\int_{\C_*(\cu_m)} \left| \nabla u \indc_{\{\a\neq 0\}}\right|^2(x)\,dx 
}  \quad & \\
& \leq C 3^{-2m} \inf_{a \in \R} \int_{\C_*(\cu_{m+1})} \left| u(x)-a \right|^2\,dx \notag \\
& \leq  C3^{-2m} \left( \sum_{\cu\in\Pa,\, \cu\subseteq \cu_{m+1}} \size(\cu)^{2d} \int_{\cu\cap\C_*(\cu_{m+1})} \left|\nabla u \indc_{\{\a\neq 0\}}\right|^{2_*}\!(x)\,dx \right)^{\frac{2}{2_*}} \notag.
\end{align}
We turn our attention to the first term on the right of~\eqref{e.reverseHolderpre}. Using the H\"older inequality with exponents $s$ and its H\"older conjugate $s':= \frac{s}{s-1}<\frac{2}{2_*}$, we get
\begin{align*}
\lefteqn{
 \sum_{\cu\in\Pa,\, \cu\subseteq \cu_{m+1}} \size(\cu')^{2d} \int_{\cu\cap\C_*(\cu_{m+1})} \left|\nabla u \indc_{\{\a\neq 0\}}\right|^{2_*}(x)\,dx 
 } \qquad & \\
 & \leq \left(   \sum_{\cu\in\Pa,\, \cu\subseteq \cu_{m+1}}\size(\cu)^{2sd+d}\right)^{\frac1s} \\
 & \qquad \times \left(   \sum_{\cu\in\Pa,\, \cu\subseteq \cu_{m+1}} \left| \cu \right|  \left( \frac{1}{|\cu|}\int_{\cu\cap\C_*(\cu_{m+1})}\left|\nabla u \indc_{\{\a\neq 0\}}\right|^{2_*}(x)\,dx \right)^{s'}\right)^{\frac1{s'}}\\
 & \leq \left(  \sum_{x\in \cu_{m+1}} \size(\cu_\Pa(x))^{2sd}\right)^{\frac1s} 
 \left( \int_{\C_*(\cu_{m+1})} \left|\nabla u \indc_{\{\a\neq 0\}}\right|^{s'2_*}(x)\,dx \right)^{\frac1{s'}} \\
 & \leq C \left| \cu_{m+1} \right|^{\frac1s}\left( \int_{\C_*(\cu_{m+1})} \left|\nabla u \indc_{\{\a\neq 0\}}\right|^{s'2_*}(x)\,dx \right)^{\frac1{s'}},
\end{align*}
where in the last line we used the first inequality of~\eqref{e.goodnessm}. 
Combining the above displays, we obtain
\begin{align*}
\lefteqn{
\frac{1}{|\cu_m|} \int_{ \C_*(\cu_m)} \left| \nabla u \indc_{\{\a\neq 0\}}\right|^2(x)\,dx 
} \qquad & \\
& \leq C3^{-m\left(d+2-\frac{2d}{s2_*}\right)} \left( \int_{\C_*(\cu_{m+1})} \left|\nabla u \indc_{\{\a\neq 0\}}\right|^{s'2_*}(x)\,dx \right)^{\frac2{s'2_*}} \\
& = C \left( \frac1{|\cu_{m+1}|}\int_{\C_*(\cu_{m+1})} \left|\nabla u \indc_{\{\a\neq 0\}}\right|^{s'2_*}(x)\,dx \right)^{\frac2{s'2_*}}.
\end{align*}
This completes the proof of~\eqref{e.reverseHolder} and therefore of~\eqref{e.reverse.Holder}. The second statement is obtained from the first and~\eqref{e.minimalbound}.
\end{proof}

\begin{definition}[{The partition~$\Ra$ and minimal scale $\mathcal{M}_t(\Ra)$}] \label{def.minscaleR}
We denote by~$\Ra$  the partition obtained by applying Proposition~\ref{p.partitions} to the family of ``good events" $\mathcal{G}:= \left\{\cu\in\T\,:\, \mathrm{RH}_{d+2}(\cu) \leq C\right\}$ in which a deterministic reverse H\"older inequality holds for gradients of elements of $\A(\C_*(\cu))$ with exponent $s'2_* = \frac{2d}{d+1}$ and with $C(d,\lambda,\p)<\infty$ as in the statement of Lemma~\ref{l.reverse.Holder}. Given an exponent $t\geq 1$ and according to Proposition~\ref{p.minimalscales}, we denote the minimal scale for this partition $\mathcal{M}_t(\Ra)$, which we note has integrability
\begin{equation}
\label{e.Mtint}
\mathcal{M}_t(\Ra) = \O_{r}(C'(r,t,d,\lambda,\p)) \quad \mbox{for every} \ r \in \left( 0, \frac{d^2}{(d+t)(2d^2+5d+1)+d} \right). 
\end{equation}
\end{definition}

%
%

We next obtain a version of the Meyers improvement of integrability estimate. For uniformly elliptic equations in Euclidean space, this estimate asserts the existence of an exponent $\ep>0$ (depending only on dimension and ellipticity) and a constant $C$ such that, for every solution $u$ in the ball $B_R$, 
\begin{equation*}
\left( \fint_{B_R} \left| \nabla u(x) \right|^{2+\ep}\,dx \right)^{\frac2{2+\ep}} 
\leq C  \fint_{B_R} \left| \nabla u(x) \right|^2\,dx.
\end{equation*}
This deterministic gain in integrability is an important ingredient in the theory developed in~\cite{AS}. In our setting, the analogue presented below (which holds only above a random scale) plays an even more essential role in the developments in Section~\ref{s.convergence} because it allows us to ``H\"older away" the sizes of random partitions from our estimates without giving up any exponent. 

\smallskip

We define, for each triadic $\cu\in\T$ and exponent $\ep>0$, the random variable
\begin{equation*}
\mathrm{ME}_{\ep}(\cu) :=
\left\{
\begin{aligned}
& \sup_{u\in \A(\C_\mathrm{max}(3\cu))} 
\frac{\left( \frac{1}{|\cu|} \int_{\C_*(\cu)} \left| \nabla u \indc_{\{\a\neq 0\}}\right|^{2+\ep}(x)\,dx  \right)^{\frac1{2+\ep}}}
{\left( \frac1{|3\cu|}\int_{\C_\mathrm{max}(3\cu)} \left|\nabla u \indc_{\{\a\neq 0\}}\right|^{2}(x)\,dx \right)^{\frac1{2}}}
   & \ \mbox{if} & \ 3\cu \ \mbox{is good,} \\
& +\infty & \ \mbox{if} & \ 3\cu \ \mbox{is bad.}
\end{aligned}
\right.
\end{equation*}
We are interesting in showing that, for an exponent $\ep(d,\lambda,\p)>0$, the quantity $\mathrm{ME}_{\ep}(\cu)$ is bounded by a deterministic constant above a random scale. The statement is given in the following proposition.

\begin{proposition}[{Meyers estimate}]
\label{p.meyers}
There exist $\ep(d,\lambda,\p)>0$ and $t(d,\lambda,\p)<\infty$ and a deterministic constant $C(d,\lambda,\p)<\infty$ such that, for every $m\in\N$, 
\begin{equation}
\label{e.Meyersimpl}
3^m \geq \mathcal{M}_{t}(\Ra) +C \implies \mathrm{ME}_{\ep}(\cu_m) \leq C.
\end{equation}
In particular, there exists $\delta(d,\lambda,\p)>0$ and $C(d,\lambda,\p)<\infty$ such that, for every $m\in\N$,  
\begin{equation}
\label{e.Meyerscheby}
\P \left[  \mathrm{ME}_\ep (\cu_m) > C \right] \leq C' \exp\left( -3^{m\delta} \right).
\end{equation}
\end{proposition}
\begin{proof}  
The classical proof of the Meyers estimate (cf.~\cite{GG}) is a combination of the reverse H\"older inequality (given by the Caccioppoli and Sobolev inequalities, as in the previous lemma) and the Gehring lemma (cf.~\cite{Giu}). Here the situation is more complicated, because we only have the reverse H\"older inequality above a random scale, and therefore we must modify the strategy slightly by applying the Gehring lemma to a coarsening of $\left|\nabla u\indc_{\{\a\neq0\}}\right|^2$ with respect to $\Ra$, using the reverse H\"older inequality holds on scales larger than $\Ra$, to obtain an improvement of integrability. Then we use the H\"older's inequality to get rid of the partition. 

\smallskip

We fix $m\in\N$ with $m\geq \mathcal{M}_t(\Ra) +C$ and $t$ is a fixed exponent to be selected at the end of the proof in such a way that it depends only on $(d,\lambda,\p)$. We fix a solution $u\in \A(\C_*(\cu_{m+1}))$ and introduce the coarsened function
\begin{equation*} \label{}
f(x):= \sum_{\cu \in\Ra} \indc_{\left\{ \left\lceil x \right\rceil \in \cu\right\}} \left( \frac{1}{|\cu|} \int_{\cu \cap \C_*(\cu_{m+1})} \left| \nabla u \indc_{\{\a\neq0\}} \right|^{\frac{2d}{d+1}}(y)\,dy \right), \quad x\in \Rd. 
\end{equation*}

\smallskip

\emph{Step 1.} We write the reverse H\"older estimate from the previous lemma in terms of~$f$. The claim is that there exists a constant~$C(d,\lambda,\p)<\infty$ such that, for every cube $Q \subseteq \Rd$, 
\begin{equation}
\label{e.reverseHolder.f}
 \fint_{Q} f(x)^{\frac{d+1}{d}} \,dx  \leq C \left( \fint_{63Q} f(x)\,dx\right)^{\frac{d+1}{d}}.
\end{equation}
First, observe that if the size of $Q$ is smaller than $100$, then since $f$ is constant on all cubes of the form $z+[-\frac12,\frac12]^d$ with $z\in\Zd$, it is easy to see that 
\begin{equation*} \label{}
\sup_{x\in Q} f(x)\leq 2^d \frac{1}{|Q|\wedge 1} \int_{2Q} f(x)\,dx\leq C \fint_{2Q} f(x)\,dx. 
\end{equation*}
Therefore we need only to check~\eqref{e.reverseHolder.f} for cubes of size larger than 100. Note that every cube~$Q$ of size at least 100 contains some cube of the form $\frac19\cu$ for some triadic cube $\cu\in \T$.

\smallskip

Let $\cu$ be the largest triadic cube satisfying $\frac19\cu\subseteq Q$. It follows from simple geometric considerations that $Q \subseteq \cu$. If $\cu\not\in\Ra_*$, then $Q\subseteq \cu'$ for some $\cu'\in \Ra$. In this case, $f$ is constant on $Q$ and therefore the bound~\eqref{e.reverseHolder.f} is obvious. We may therefore assume $\cu\in \Ra_*$. We now compute, using~\eqref{e.reverseHolder},
\begin{align*}
\fint_Q f(x)^{\frac{d+1}{d}}\,dx
& \leq C \fint_{\cu} f(x)^{\frac{d+1}{d}}\,dx \\
& \leq C \fint_{\cu \cap\C_*(\cu_m)} \left| \nabla u \indc_{\{ \a\neq 0\}} \right|^2(y)\,dy \\
& \leq C \left( \frac1{|3\cu|}\int_{\C_*(3\cu)} \left|\nabla u \indc_{\{\a\neq 0\}}\right|^{\frac{2d}{d+1}}(x)\,dx \right)^{\frac{d+1}{d}}.
\end{align*}
Let~$W$ be the union of elements of~$\Ra$ which have nonempty intersection with~$3\cu$. It is easy to check from the fact that $\Ra$ has property~(ii) of Proposition~\ref{p.partitions} that $W \subseteq 7\cu$. Thus
\begin{align*}
\lefteqn{
\left( \frac1{|3\cu|}\int_{\C_*(3\cu)} \left|\nabla u \indc_{\{\a\neq 0\}}\right|^{\frac{2d}{d+1}}(x)\,dx \right)^{\frac{d+1}d}
} \qquad & \\
& \leq C \left( \frac1{|W|}\int_{W\cap \C_*(\cu_m)} \left|\nabla u \indc_{\{\a\neq 0\}}\right|^{\frac{2d}{d+1}}(x)\,dx \right)^{\frac{d+1}d} \\
& =  C \left( \fint_{W} f(x)\,dx \right)^{\frac{d+1}d} \leq  C \left( \fint_{7\cu} f(x)\,dx \right)^{\frac{d+1}d}. 
\end{align*}
Since $7\cu \subseteq 63Q$, this completes the proof of~\eqref{e.reverseHolder.f}.

\smallskip

\emph{Step 2.} We apply the Gehring Lemma to the function $f$ and show that the result implies~\eqref{e.meyers}. By an application of~Theorem 6.6 \& Corollary 6.1 from~\cite{Giu} (we again can get the result in the discrete case from the continuum case by using affine interpolation), we obtain the existence of an exponent~$\ep(d,\lambda,\p)>0$ and a constant $C(d,\lambda,\p)<\infty$ such that
\begin{equation}
\label{e.Gehringapp.f}
\fint_{\cu_m} f(x)^{\frac{d+1}{d}(1+\ep)}\,dx \leq C \left( \fint_{\cu_{m+1}} f(x)^{\frac{d+1}{d}}\,dx \right)^{1+\ep}.
\end{equation}
It is obvious that 
\begin{equation*}
 \int_{\cu_{m+1}} f(x)^{\frac{d+1}{d}}\,dx \leq \int_{\C_*(\cu_{m+1})} \left| \nabla u \indc_{\{ \a\neq 0\}} \right|^2(x)\,dx.
\end{equation*}
To bound the left side of~\eqref{e.Gehringapp.f} from below, observe that for each $\cu\in\Ra$ and $x\in \C_*(\cu_{m+1})$, we have
\begin{equation*}
f(x) \geq \left| \cu_{\Ra}\left(\left\lceil x\right\rceil\right) \right|^{-1} \left| \nabla u \indc_{\{\a\neq0\}} \right|^{\frac{2d}{d+1}}\left(\left\lceil x\right\rceil\right).
\end{equation*}
Therefore, by H\"older's inequality,
\begin{align*}
\lefteqn{ 
\int_{\C_*(\cu_m)}  \left| \nabla u \indc_{\{ \a\neq 0\}} \right|^{2+\ep}(x)\, dx 
} \qquad & \\
& \leq \left( \int_{\C_*(\cu_m)}  \left| \nabla u \indc_{\{ \a\neq 0\}} \right|^{2+2\ep}(x)\,\left| \cu_{\Ra}(x) \right|^{-\frac{d+1}d(1+\ep)} dx \right)^{\frac{2+\ep}{2+2\ep}}\\
& \qquad \times \left( \int_{\C_*(\cu_m)} \left| \cu_{\Ra}(x) \right|^{\frac{(2+\ep)(1+\ep)(d+1)}{2\ep d}} \,dx\right)^{\frac{\ep}{2+2\ep}} \\
& \leq C \left( \int_{\cu_m} f(x)^{\frac{d+1}{d}(1+\ep)}\,dx \right)^{\frac{2+\ep}{2+2\ep}} \left( \int_{\C_*(\cu_m)} \left| \cu_{\Ra}(x) \right|^{\frac{(2+\ep)(1+\ep)(d+1)}{2\ep d}} \,dx\right)^{\frac{\ep}{2+2\ep}} \\
& \leq C\left|\cu_m\right| \left( \frac{1}{ \left|\cu_{m+1}\right|}  \int_{\C_*(\cu_{m+1})} \left| \nabla u \indc_{\{ \a\neq 0\}} \right|^2(x)\,dx \right)^{\frac{2+\ep}{2}} \\
& \qquad \times
\left( \frac{1}{ \left|\cu_m\right|} \int_{\C_*(\cu_m)} \left| \cu_{\Ra}(x) \right|^{\frac{(2+\ep)(1+\ep)(d+1)}{2\ep d}} \,dx\right)^{\frac{\ep}{2+2\ep}}.
\end{align*}
We now choose $t:=\frac{(2+\ep)(1+\ep)(d+1)}{2\ep}$ which, as required, depends only on $(d,\lambda,\p)$. Under the assumption that $m\geq \mathcal{M}_t(\Ra)$, the second factor on the right is then bounded by $C(d,\lambda,\p)$ and we obtain, for every $u\in \A(\C_*(\cu))$,
\begin{multline}
\label{e.meyers}
 \frac{1}{ \left|\cu_m\right|}\int_{\C_*(\cu_m)}  \left| \nabla u \indc_{\{ \a\neq 0\}} \right|^{2+\ep}(x)\, dx \\
  \leq C \left( \frac{1}{ \left|\cu_{m+1}\right|}  \int_{\C_*(\cu_{m+1})} \left| \nabla u \indc_{\{ \a\neq 0\}} \right|^2(x)\,dx \right)^{\frac{2+\ep}{2}}.
\end{multline}
Thus $\mathrm{ME}_\ep(\cu_m) \leq C$, completing the proof of~\eqref{e.Meyersimpl}. The bound~\eqref{e.Meyerscheby} follows from~\eqref{e.Meyersimpl},~\eqref{e.Mtint} and the Chebyshev inequality. 
\end{proof}
We finish this section with the definition of the partition~$\Qa$ quantifying the local scale on which the Meyers estimate holds.
\begin{definition}
[{The partition~$\Qa$ and minimal scale $\mathcal{M}_t(\Qa)$}]
\label{d.partQa}
We denote by~$\Qa$  the partition obtained by applying Proposition~\ref{p.partitions} to the family of ``good events" $\mathcal{G}:= \left\{\cu\in\T\,:\, \mathrm{ME}_\ep(\cu_m) \leq C\right\}$ in which a deterministic Meyers estimates holds for gradients of elements of $\A(\C_*(\cu))$ with exponent $\ep(d,\lambda,\p) > 0 $ and with $C(d,\lambda,\p)<\infty$ as in the statement of Proposition~\ref{p.meyers}. We denote, for each exponent $t\geq 1$, the minimal scale for this partition (given by Proposition~\ref{p.minimalscales}) by $\mathcal{M}_t(\Qa)$. Note that 
\begin{equation}
\label{e.Mtint2}
\mathcal{M}_t(\Qa) = \O_{r}(C'(r,t,d,\lambda,\p)) \quad \mbox{for every} \ r \in \left( 0, \frac{\delta d}{d+t+\delta} \right). 
\end{equation}
where $\delta := \delta(d,\lambda,\p) > 0$ is as in the statement of Proposition~\ref{p.meyers}.
\end{definition}

%
%
%

\smallskip

\section{Subadditive energy quantities and basic properties}
\label{s.quantities}

Here we introduce the subadditive energy quantities, which are modeled on the ones from~\cite{AS}, and record their basic properties. We also prove estimates on their uniform convexity, boundedness, the ordering relation between them and their subadditivity. These properties are mostly trivial in the uniformly elliptic case, but more technical in our setting since we must take into account the geometry of the percolation clusters (using the results from the previous two sections).

\smallskip

\subsection{Definition of the subadditive energy quantities}

\smallskip

We next introduce the subadditive quantities. These are based on similar quantities introduced in the continuum, uniformly elliptic setting in~\cite{AS} and variants of the quantities which have been recently used to obtain optimal estimates and scaling limits in stochastic homogenization of uniformly elliptic equations~\cite{AKM1,AKM2}. 

\smallskip

We first define, for each finite subset $U \subseteq \Zd$, the set 
\begin{equation*} \label{}
\A_*(U) := \left\{ u : \C_*(U) \to \R \,:\, -\nabla \cdot \a \nabla u(x) = 0, \ \forall x \in \C_*(U) \setminus \partial \cl_\Pa(U) \right\}. 
\end{equation*}
Note that $\A( \C_*(U))\subseteq \A_*(U)$ since $\partial_\a \C_*(U) \subseteq \partial \cl_\Pa(U)$, but neither of these inclusions in necessarily an equality. As in~\eqref{e.Avarchar0}, we have,
\begin{equation}
\label{e.Avarchar}
u\in \A_*(U) 
\iff 
\left\langle \nabla w, \a\nabla u \right\rangle_U = 0 \quad \mbox{for every} \ w \in \mathcal{C}_0(U),
\end{equation}
where $\mathcal{C}_0(U)$ denotes the set of functions $w : \C_*(U) \rightarrow \R$ equal to $0$ on $\C_*(U) \cap \cl_\Pa(U)$.
\begin{definition}
For each  $U \subseteq \Zd$ and $p,q\in\Rd$, we define the random variables
\begin{equation*} \label{}
\mu(U,q):= 
\inf_{u \in \A_*(U) } \frac1{|\cl_\Pa(U)|} 
\left(
\frac12 \left\langle \nabla u, \a\nabla u \right\rangle_{\C_*(U)} 
 - \left\langle q, \nabla\! \left[ u \right]_{\Pa} \right\rangle_{\cl_\Pa(U)}
 \right)
\end{equation*}
and
\begin{equation*} \label{}
\nu(U,p) := 
\sup_{v \in \A_*(U)} \frac1{|\cl_\Pa(U)|} 
\left( 
- \frac12 \left\langle  \nabla v, \a\nabla v \right\rangle_{\C_*(U)} 
+ \left\langle p, \a\nabla v \right\rangle_{\C_*(U)}
\right).
\end{equation*}
\end{definition}
The optimization problems in the above definitions of $\mu(U,q)$ and $\nu(U,p)$ are strictly convex and concave, respectively, and therefore they have unique optimizers in~$\A_*(U)$, up to additive constants,  which we denote by
\begin{equation*} \label{}
u(\cdot,U,q):= \mbox{minimizing element of $\A_*(U)$ in the definition of $\mu(U,q)$} 
\end{equation*}
and
\begin{equation*} \label{}
v(\cdot,U,p):= \mbox{maximizing element of $\A_*(U)$ in the definition of $\nu(U,q)$.} 
\end{equation*}
We choose the additive constants for~$u(\cdot,U,q)$ and~$v(\cdot,U,p)$ so that 
\begin{equation} \label{e.additiveconstants}
\fint_{\cl_\Pa(U)} \left[ u(\cdot,U,q) \right]_{\Pa}\!(x)\,dx = 0
\quad \mbox{and} \quad 
\fint_{\cl_\Pa(U)} \left[ v(\cdot,U,p) \right]_{\Pa}\!(x) \,dx = 0.
\end{equation}
Notice that, for each bounded $U\subseteq\Zd$,  
\begin{equation*} \label{}
q \mapsto -\mu(U,q) \quad \mbox{and} \quad p \mapsto \nu(U,p) \quad \mbox{are nonnegative and quadratic.}
\end{equation*}
In particular, these maps are convex. 

\smallskip

In Lemma~\ref{l.nuisdirichlet}, below, we will show that the function $v(\cdot,U,p)$ is the solution of the Dirichlet problem in $\C_*(U)$ with affine data $x\mapsto p\cdot x + c$ on $\C_*(U) \cap \partial \cl_\Pa(U)$, for some $c\in\R$. Therefore $\nu(U,p)$ is just (up to the normalization) the energy of the familiar cell problem solution in $\C_*(U)$. The quantity $\mu$ represents the energy of the ``dual" cell problem introduced in~\cite{AS}. It is important here that the linear term in the definition of $\mu$ is \emph{not} $\left\langle q, \nabla u \right\rangle_{\C_*(U)}$, which is what one might naively guess when attempting to generalize from the uniformly elliptic case. This will not possess the correct convex dual relationship with $\nu$: in particular,~\eqref{e.munucomparison} would be false, rendering attempts at proving Proposition~\ref{p.subadd} hopeless. 
Indeed, if $u$ is close to an affine function with slope $p$ (for example, the function $v(\cdot,U,p)$), there is no reason to expect that $\left\langle q, \nabla u \right\rangle_{\C_*(U)}$ should be close to $q\cdot p$, because we are ``missing'' the contribution of~$\nabla u$ in the closed edges. While the exact form of the linear term is not very important, we need something that will be close to $q\cdot p$ if (as expected on large scales) $u(\cdot,U,q)$ is close to an affine function with slope $p$. Using the spatial average of the gradient $\nabla \left[ u \right]_{\Pa}$ of the coarsened function satisfies this property and turns out to be very convenient. One of the central ideas of~\cite{AS} is that one should focus on the spatial averages of the gradients and energy densities of the solutions. We do the same in the generalization here, except that when it comes to gradient we consistently replace a solution~$u$ with its coarsening~$\left[ u \right]_{\Pa}$. 

\smallskip

In the next section, we quantify the convergence of quantities $\mu(\cu,q)$ and $\nu(\cu,p)$ for $\cu \in \T$ as $\size(\cu) \to \infty$, see Proposition~\ref{p.subadd}. In the rest of this section, we prepare for this analysis by presenting some basic properties of these quantities. The geometry of the percolation cluster forces us to give up some very nice properties possessed by $\mu$ and $\nu$ in the continuum uniformly elliptic setting~\cite{AS}. For example, $-\mu$ and $\nu$ are not (strictly speaking) \emph{uniformly} convex independently of~$U$, because in general the partition can be quite coarse and the geometry of the percolation cluster very complicated. They are not \emph{stationary} with respect to $\Zd$-translations (because the partition $\Pa$ is not stationary), nor are they \emph{local} quantities (since they depend on the coefficient field~$\a(\cdot)$ on the whole of $\Bd$, since $\Pa$ does), nor are they precisely subadditive. Most of this section is therefore consumed by the quite technical task of showing that each of these important properties does in fact hold in an approximate sense which is quantified with sufficiently strong stochastic integrability. 

\smallskip

We conclude this subsection by computing the first and second variations of the optimization problems in the definitions of $\mu$ and $\nu$ and then checking that the function~$v$ is the solution of the Dirichlet problem with affine data as claimed above.

\begin{lemma}[First and second variations]
\label{l.variations}
Fix a bounded $U\subseteq \Zd$ and $p,q\in\Rd$. For every $w\in \A_*(U)$,
\begin{equation}
\label{e.firstvarmu}
\left\langle  \nabla w, \a \nabla u(\cdot ,U,q) \right\rangle_{\C_*(U)}= \left\langle q, \nabla\!\left[ w\right]_{\Pa}\!\right\rangle_{\cl_\Pa(U)},
\end{equation}
\begin{multline}
\label{e.secondvarmu}
\frac1{|\cl_\Pa(U)|} \left( \frac12  \left\langle  \nabla w, \a\nabla w\right\rangle_{\C_* (U)} -  \left\langle q, \nabla\! \left[ w \right]_{\Pa}\! \right\rangle_{\cl_{\Pa}(U)} \right) \\
= \mu(U,q) + \frac1{|\cl_\Pa(U)|} \frac12   \left\langle  \nabla (w-u(\cdot,U,q)) ,\a\nabla (w-u(\cdot,U,q)) \right\rangle_{\C_*(U)},
\end{multline}
\begin{equation} \label{e.firstvarnu}
\left\langle \nabla w, \a\nabla v(\cdot,U,p)\right\rangle_{\C_*(U)}  =\left\langle  p, \a \nabla w \right\rangle_{\C_*(U)},
\end{equation}
and
\begin{multline}
 \label{e.secondvarnu}
\frac1{|\cl_\Pa(U)|} \left(   -\frac12  \left\langle \nabla w, \a\nabla w \right\rangle_{\C_*(U)}  + \left\langle  p, \a \nabla w \right\rangle_{\C_*(U)} \right)  \\
= \nu(U,p) - \frac1{|\cl_\Pa(U)|}  \left\langle  \frac12 \nabla (w-v(\cdot,U,p)), \a\nabla (w-v(\cdot,U,p)) \right\rangle_{\C_*(U)} .
\end{multline}
\end{lemma}
\begin{proof}
Let $w \in \A(\C_*(U))$. For each $h\in [0,1)$, define $u_h := u(\cdot,U,q) + hw$. By comparing $u_h$ to $u:=u_0=u(\cdot,U,q)$ in the definition of $\mu(U,q)$, we obtain, for every $h> 0$, 
\begin{align*} \label{}
0 & \geq  \frac12 \left\langle  \nabla u, \a \nabla  u \right\rangle_{\C_*(U)}  -   \frac 12 \left\langle \nabla u_h, \a \nabla  u_h \right\rangle_{\C_*(U)}
 -  \left\langle q , \nabla \left[ u - u_h \right]_{\Pa}\! \right\rangle_{\cl_\Pa(U)} \\
& = -\frac12 h^2 \left\langle \nabla w, \a \nabla w \right\rangle_{\C_*(U)} - h \left\langle \nabla w, \a \nabla u \right\rangle_{\C_*(U)} 
+ h \left\langle q,  \nabla \left[ w \right]_{\Pa}\! \right\rangle_{\cl_\Pa(U)}.
\end{align*}
Rearranging this and dividing by $h>0$ gives
\begin{equation*} \label{}
-   \left\langle \nabla u, \a \nabla w \right\rangle_{\C_*(U)}
+  \left\langle q, \nabla \left[ w \right]_{\Pa}\!  \right\rangle_{\cl_\Pa(U)} 
\leq \frac12 h\left\langle \nabla w, \a \nabla w \right\rangle_{\C_*(U)}.
\end{equation*}
Sending $h\to 0$ yields, for every $w \in \A(\C_*(U))$,
\begin{equation*} \label{}
\left\langle  q, \nabla \left[ w \right]_{\Pa}\! \right\rangle_{\cl_\Pa(U)}  \leq
\left\langle \nabla u, \a \nabla w  \right\rangle_{\C_*(U)} .
\end{equation*}
The reverse of the previous inequality follows by replacing~$w$ by~$-w$, which completes the proof of~\eqref{e.firstvarmu}. Returning now to the first display and inserting~\eqref{e.firstvarmu}, we obtain~\eqref{e.secondvarmu} for $u_h$ in place of $w$.

\smallskip

The proofs of~\eqref{e.firstvarnu} and~\eqref{e.secondvarnu} are similar and thus omitted.
\end{proof}

For future reference, we record some identities which are consequences of those in Lemma~\ref{l.variations}. By combining~\eqref{e.firstvarnu} and~\eqref{e.secondvarnu} with $w=v(\cdot,U,p)$, we get
\begin{equation}
\label{e.nuform1}
\nu(U,p) = \frac1{|\cl_\Pa(U)|} \frac12  \left\langle \nabla v(\cdot,U,p), \a\nabla v(\cdot,U,p) \right\rangle_{\C_*(U)}.
\end{equation}
Next, inserting this into~\eqref{e.secondvarmu} with $w=v(\cdot,U,p)$, we get
\begin{multline}
\label{e.omegaform0}
\nu(U,p) - \mu(U,q) -  \left\langle q, \nabla\!\left[ v(\cdot,U,p)\right]_{\Pa}\!\right\rangle_{\cl_\Pa(U)} \\
= \frac1{|\cl_\Pa(U)|} \frac12   \left\langle  \nabla (v(\cdot,U,p)-u(\cdot,U,q)) ,\a\nabla (v(\cdot,U,p)-u(\cdot,U,q)) \right\rangle_{\C_*(U)}.
\end{multline}

We next show that $v(\cdot,\cu,p)$ is the solution of the Dirichlet problem with affine boundary data. Recall that the space $\mathcal{C}_0(U)$ is defined between~\eqref{e.defA} and~\eqref{e.Avarchar}, above. 

\begin{lemma}
\label{l.nuisdirichlet}
There exists $c\in \R$ such that 
\begin{equation} \label{e.vaffine}
v(x,U,p) = p\cdot x + c \quad \mbox{for every} \ x \in \C_*(U) \cap \partial \cl_\Pa(U)
\end{equation}
and
\begin{equation*} \label{e.nucellproblem}
\nu(U,p) = \inf_{w \in \mathcal{C}_0(\cl_\Pa(U))} \frac{1}{2|\cl_\Pa(U)|} \left\langle \left( p+\nabla w \right),\a \left( p+\nabla w \right) \right\rangle_{\C_*(U)}.
\end{equation*}
\end{lemma}
\begin{proof}
From Lemma~\ref{l.variations} we have that, for every $w \in \A_*(U)$,
\begin{equation}\label{e.v-p=0}
 \left\langle \nabla w, \a(\left( \nabla v(\cdot,U,p) - p \right) \right\rangle_{\C_*(U)} = 0.
\end{equation}
Pick $w \in \A_*(U)$ such that $w(x) =  v(x,U,p) - p \cdot x$ for every $ x \in \C_*(U) \cap \partial \cl_\Pa(U)$. Then $w - v(\cdot,U,p) - p \cdot x \in \mathcal{C}_0(\cl_\Pa(U))$, since $w\in \A_*(U)$, we have
\begin{equation} \label{ebis.v-p=0}
\left\langle \nabla w, \a\left( \nabla w - v(\cdot,U,p) + p \right)\right\rangle_{\C_*(U)} = 0.
\end{equation}
Summing the equations \eqref{e.v-p=0} and \eqref{ebis.v-p=0} gives
\begin{equation*}
\left\langle \nabla w, \a \nabla w \right\rangle_{\C_*(U)} = 0.
\end{equation*}
Therefore $w$ is constant and so $x\mapsto v(x,U,p) - p \cdot x$ is constant on $\C_*(U) \cap \partial \cl_\Pa(U)$.

\smallskip

Since $v \in \A_*(U)$ and $v(x,U,p) = p\cdot x +c$ on $\C_*(U) \cap \partial \cl_\Pa(U)$, we have, by~\eqref{e.nuform1},
\begin{align*}
\nu(U,p)
& = \left\langle  \nabla v(\cdot,U,p), \a \nabla v(\cdot,U,p) \right\rangle_{\C_*(U)}   \\
& =  \inf_{w \in \mathcal{C}_0(\C_*(U))} \left\langle \left( p+\nabla w \right), \a \left( p+\nabla w \right) \right\rangle_{\C_*(U)}.
\end{align*}
The proof of~\eqref{e.nucellproblem} is complete.
\end{proof}

\subsection{Subadditivity, boundedness, uniform convexity, and ordering}
\smallskip

The purpose of this subsection is to prove 
that $\mu$ and $\nu$ retain most of their essential properties from the uniformly elliptic setting, with errors arising due to the coarseness of the random partition~$\Pa$.

\smallskip

We begin by proving the upper bounds for $-\mu$ and $\nu$. While the one for $\nu$ is obvious, the one for $-\mu$ uses the estimate~\eqref{e.partitionO1}. 

\begin{lemma}
\label{l.upperbounds}
There exists a constant $C(d,\p,\lambda) < \infty$ such that, for each $\cu\in\T$ and $p,q\in\Rd$, 
\begin{equation}
\label{upperbound.mu}
-\mu(\cu,q) \leq C|q|^2 \fint_{\cu} \size\left( \cu_\Pa(x) \right)^{2d-2} \,dx
\end{equation}
and
\begin{equation}
\label{e.nuupbound}
\nu(\cu,p) \leq C|p|^2.
\end{equation}
\end{lemma}
\begin{proof}
We begin by noticing that the bound~\eqref{e.nuupbound} is immediate from Young's inequality. We thus focus on the bound for~$-\mu$.
For convenience, we denote $u:= u(\cdot,\cu,q)$. 
Applying Lemma~\ref{l.coarsegrads} with $s=1$ gives
\begin{equation*}
\int_{\cl_\Pa(\cu)} \left|\nabla\! \left[ u \right]_{\Pa}\right|(x) \, dx \leq C \sum_{\cu' \in \Pa, \cu' \subseteq \cl_\Pa(\cu)} \size(\cu')^{d-1} \int_{\cu' \cap \C_*(\cu)} \left|\nabla u \indc_{\{\a\neq0\}} \right|(x) \, dx.
\end{equation*}
By the H\"older inequality, we deduce that
\begin{align}
\label{e.coarsegradbounds}
\lefteqn{
\int_{\cl_\Pa(\cu)} |\nabla\! \left[ u \right]_{\Pa}|(x) \, dx 
} \qquad & \\ \notag
& \leq C \sum_{\cu' \in \Pa, \cu' \subseteq \cl_\Pa(\cu)} \size(\cu')^{\frac 32d-1} \left( \int_{\cu' \cap \C_*(\cu)} \left|\nabla u\indc_{\{\a\neq0\}}\right|^2(x) \, dx \right)^{\frac 12} \\ \notag
										& \leq C \left( \sum_{\cu' \in \Pa, \cu' \subseteq \cl_\Pa(\cu)} \size(\cu')^{3d-2} \right)^{\frac 12} \left( \int_{\C_*(\cu)} \left|\nabla u\indc_{\{\a\neq0\}}\right|^2(x) \, dx \right)^{\frac 12} \\ \notag
										& \leq C \left(  \int_{\cl_\Pa(\cu)} \size(\cu_\Pa(x))^{2d-2}\, dx \right)^{\frac 12}  \left( \int_{\C_*(\cu)} \left|\nabla u\indc_{\{\a\neq0\}}\right|^2(x) \, dx \right)^{\frac 12}. 
\end{align}
Young's inequality then yields
\begin{equation*}
\left\langle q, \nabla\! \left[ u \right]_{\Pa}\right\rangle_{\cl_\Pa(\cu)} \leq C|q|^2 \int_{\cl_\Pa(\cu)} \size(\cu_\Pa(x))^{2d-2}\, dx + \frac12 \left\langle\nabla u,\a \nabla u \right\rangle_{\C_*(\cu)}.
\end{equation*}
Hence
\begin{align*}
-\mu(\cu,q) 
& = \frac{1}{\left| \cl_\Pa(\cu) \right|}\left( -\frac12 \left\langle\nabla u,\a \nabla u \right\rangle_{\C_*(\cu)} +  \left\langle q, \nabla\! \left[ u \right]_{\Pa}\right\rangle_{\cu} \right)  \\
& \leq C|q|^2 \fint_{\cl_\Pa(\cu)}  \size(\cu_\Pa(x))^{2d-2}\, dx.
\end{align*}
Notice that if $\cl_\Pa(\cu) \neq \cu$ then for each $x \in \cl_\Pa(\cu), \,  \size(\cu_\Pa(x)) = \size(\cl_\Pa(\cu))$ and thus
\begin{equation} \label{meanvaluesareequal}
\fint_{\cl_\Pa(\cu)}  \size(\cu_\Pa(x))^{2d-2}\, dx = \size(\cl_\Pa(\cu))^{2d-2} = \fint_\cu \size(\cu_\Pa(x))^{2d-2}\, dx.
\end{equation}
Combining the two previous displays gives~\eqref{upperbound.mu}. 
\end{proof}

Observe that~\eqref{e.partitionO1} and~\eqref{e.muupbound} imply that 
\begin{equation}
\label{e.muupbound}
-\mu(\cu,q) \leq \O_{\frac1{2d-2}}\left(C|q|^2 \right).
\end{equation}
It is also useful to notice that we can bound the right side slightly differently (so that the random part is scaling better) to obtain
\begin{equation}
\label{e.muupboundwhip}
-\mu(\cu,q) \leq C|q|^2 + \O_{\frac{1}{2d+1}} \left(C|q|^2\size^{-1}(\cu)\right).
\end{equation}
To see this, we combine~\eqref{upperbound.mu} with the bounds for the minimal scale $\mathcal{M}_{2d-2}(\Pa)$ given in Proposition~\ref{p.minimalscales} to obtain
\begin{align}
\label{e.boundwhip}
 \fint_{\cu_n} \size\left( \cu_\Pa(x) \right)^{2d-2} \,dx 
& \leq   C +  \fint_{\cu_n} \size\left( \cu_\Pa(x) \right)^{2d-2} \,dx\, \indc_{\{ \mathcal{M}_{2d-2} (\Pa) > 3^n \} } \\ \notag
& \leq C +  3^{-n} \mathcal{M}_{2d-2} (\Pa)\fint_{\cu_n} \size\left( \cu_\Pa(x) \right)^{2d-2} \,dx \\ \notag
& \leq  C +  \O_{\frac13}\left(C3^{-n}\right) \cdot \O_{\frac{1}{2d-2}} (C)  \\ \notag
& \leq  C + \O_{\frac{1}{2d+1}} \left(C3^{-n} \right). 
\end{align}
For future reference, we also note that~\eqref{e.muupbound} and~\eqref{e.nuupbound} imply upper bounds for the $L^2$ norm of the gradients of $u(\cdot,\cu,q)$, $v(\cdot,\cu,p)$ and their coarsenings. Indeed, by the first variations~\eqref{e.firstvarmu} and~\eqref{e.firstvarnu}, we have
\begin{equation}
\label{e.munufirstvariden}
\left\{
\begin{aligned}
& \mu(U,q) = \frac{1}{\left|\cl_\Pa(U)\right|} \frac12 \left\langle \nabla u(\cdot,U,q) , \a \nabla u(\cdot,U,q) \right\rangle_{\cl_{\Pa}(U)}, \quad \mbox{and} \\
&\nu(U,p) = \frac{1}{\left| \cl_\Pa(U) \right|} \frac12 \left\langle \nabla v(\cdot,U,p) , \a \nabla v(\cdot,U,p) \right\rangle_{\cl_{\Pa}(U)} 
\end{aligned}
\right.
\end{equation}
and thus~\eqref{upperbound.mu} and~\eqref{e.boundwhip} imply
\begin{align}
\label{e.uupbound}
\lefteqn{
\frac{1}{|\cl_\Pa(\cu)|} \int_{\C_*(\cu)} \left| \nabla u(\cdot,U,q) \indc_{\{\a\neq 0\}} \right|^2(x)\,dx 
} \qquad & \\ \notag
& \leq C|q|^2 \fint_{\cu} \size\left(\cu_{\Pa}(x)\right)^{2d-2}\,dx \\ \notag
& \leq \O_{\frac1{2d-2}}(C|q|^2) \wedge \left( C|q|^2 + \O_{\frac{1}{2d+1}} \left(C|q|^2\size^{-1}(\cu) \right) \right)
\end{align}
and~\eqref{e.nuupbound} implies
\begin{equation}
\label{e.vupbound}
 \frac{1}{|\cl_\Pa(\cu)|} \int_{\C_*(\cu)} \left| \nabla v(\cdot,U,p) \indc_{\{\a\neq 0\}} \right|^2(x)\,dx 
\leq C |p|^2.
\end{equation}
Combining these with~\eqref{e.coarsegradbounds},~\eqref{meanvaluesareequal} and the analogous bound for $v(\cdot,\cu,p)$, we get the following bounds for the $L^1$ norm of the coarsened functions:
\begin{multline}
\label{e.coarsegradu}
\fint_{\cl_\Pa(\cu)} \left| \nabla \left[ u(\cdot,\cu,q) \right]_{\Pa} \right|(x)\,dx \\
\leq C|q| \fint_{\cu} \size\left( \cu_{\Pa}(x) \right)^{2d-2}\,dx \leq \O_{\frac{1}{2d-2}}\left(C|q|\right)
\end{multline}
and
\begin{multline}
\label{e.coarsegradv}
\fint_{\cl_\Pa(\cu)} \left| \nabla \left[ v(\cdot,\cu,p) \right]_{\Pa} \right|(x)\,dx \\
\leq C|p| \left(\fint_{\cu} \size\left( \cu_{\Pa}(x) \right)^{2d-2}\,dx\right)^\frac12 \leq \O_{\frac{1}{d-1}}\left(C|p|\right).
\end{multline}
By~\eqref{e.partitionO1}, we have that for every $t > 0$
\begin{align*}
\P\left[ \sup_{x \in \cu} \size\left( \cu_\Pa(x) \right) \geq t \right] & \leq \sum_{x \in \cu} \P \left[ \size\left( \cu_\Pa(x) \right) \geq t \right] \\
													& \leq C \size(\cu)^d \exp \left(- C^{-1} t\right).
\end{align*}
From this we deduce that for every $\delta > 0$, there exists $C := C(d,\p,\delta) < \infty$ such that
\begin{equation*}
\sup_{x \in \cu} \size\left( \cu_\Pa(x) \right) \leq \O_1 \left(C \size(\cu)^\delta \right).
\end{equation*}
Combining this with Lemma~\ref{l.coarsegrads} (with $s = 2$) gives
\begin{align} \label{e.coarsegraduL2}
\fint_{\cl_\Pa(\cu)} \left| \nabla \left[ u(\cdot,\cu,q) \right]_{\Pa} \right|^2(x)\,dx &
\leq C|q|^2 \left( \sup_{x \in \cu} \size\left( \cu_\Pa(x) \right)^{2d-1} \right) \fint_{\cu} \size(\cu)^{2d-2} \, dx \\ \notag
& \leq |q|^2 \O_{\frac1{2d-1}} \left(C \size(\cu)^\delta \right) \O_{\frac1{2d-2}}(C) \\ \notag
& \leq \O_{\frac1{4d-3}} \left(C|q|^2 \size(\cu)^\delta \right). \notag
\end{align}
Similarly we obtain for each $\delta > 0$ and for some $C := C(d,\p,\delta) < \infty$
\begin{align} \label{e.coarsegradvL2}
\fint_{\cl_\Pa(\cu)} \left| \nabla \left[ v(\cdot,\cu,p) \right]_{\Pa} \right|^2(x)\,dx &
\leq C|p|^2 \left( \sup_{x \in \cu} \size\left( \cu_\Pa(x) \right)^{2d-1} \right) \left( \fint_{\cu} \size(\cu)^{2d-2} \, dx\right)^\frac12 \\ \notag
& \leq |p|^2 \O_{\frac1{2d-1}} \left(C \size(\cu)^\delta \right) \O_{\frac1{d-1}}(C) \\ \notag
& \leq  \O_{\frac1{3d-2}} \left(C |p|^2 \size(\cu)^\delta \right). \notag
\end{align}

%

We next prove the ordering relation between $\mu$ and $\nu$. In the uniformly elliptic case, this is proved in one line and comes from testing the definition of $\mu$ with the minimizer of $\nu$ and using an integration by parts to see that the spatial average of $\nabla v$ is exactly $p$. In our situation, the latter computation is not exact but holds up to a coarsening error, as stated in the following lemma.

\begin{lemma} 
\label{l.nablavaffine}
There exists a constant $C(d,\lambda,\p)<\infty$ such that, for every $\cu \in \T$ and $q,p \in \Rd$,
\begin{equation}
\label{e.nablavaffine}
\left| \frac{1}{|\cl_\Pa(\cu)|}\left\langle q,\nabla\! \left[v(\cdot,\cu,p) \right]_{\Pa} \right\rangle_{\cl_\Pa( \cu)} - p\cdot q \right|  \leq  \O_{\frac{2}{2d-1}}\left(C|p||q| \size(\cu)^{-\frac12} \right).
\end{equation}
\end{lemma}
\begin{proof}
Fix $p\in\Rd$ and denote $v:= v(\cdot,\cu,p)$. The estimate~\eqref{e.nablavaffine} is a consequence of~\eqref{e.partitionO1} and the following claimed inequality:
\begin{multline} \label{e.nablavaffine.tech}
\left| \frac{1}{|\cl_\Pa(\cu)|}\left\langle q,\nabla\! \left[v \right]_{\Pa} \right\rangle_{\cl_\Pa (\cu)} - p \cdot q \right|  \\
\leq C |p||q| \left( \frac{|\partial_\Pa \cu|}{|\cu|}
+  \left(\frac{1}{| \cu|}  \sum_{x \in \partial \cu} \size\left(\cu_\Pa (x)\right)^{2 d-1} \right)^\frac12 \right).
\end{multline}
Recall that
\begin{equation*} \label{}
\partial_\Pa \cu:= \bigcup\left\{ \cu' \in\Pa\,:\, \cu' \subseteq \cu, \ \dist\left(\cu',\partial\cu\right) = 0 \right\}.
\end{equation*}
First, we deal with the case $\cl_\Pa(\cu) \neq \cu$. In that case, we know, by definition of the partition $\Pa$, that $\left[v \right]_{\Pa}$ is constant, thus
\begin{equation*}
 \frac{1}{|\cl_\Pa(\cu)|}\left\langle q,\nabla\! \left[v \right]_{\Pa} \right\rangle_{\cl_\Pa (\cu)} = 0.
\end{equation*}
We also have by definition of $\partial_\Pa \cu$
\begin{equation*}
\cu \subset \cl_\Pa(\cu) = \partial_\Pa \cu 
\end{equation*}
Thus
\begin{align*}
\left| \frac{1}{|\cl_\Pa(\cu)|}\left\langle q,\nabla\! \left[v \right]_{\Pa} \right\rangle_{\cl_\Pa (\cu)} - p \cdot q \right| & = \left| p \cdot q \right| \\
																								& \leq  |p||q|  \\
																								& \leq C |p||q|  \frac{|\partial_\Pa \cu|}{|\cu|}.
\end{align*}
which shows~\eqref{e.nablavaffine.tech}. We now assume for the rest of the proof $\cl_\Pa(\cu) = \cu$.

For $x \in \partial \cu$, denote the outer normal vector~$\mathbf{n}(x) \in \Rd$ to~$\partial\cu$ at~$x$ by
\begin{equation*}
\mathbf{n}(x) := \sum_{i =1}^d \left( e_i \indc_{\{x - e_i \in \cu\}} - e_i  \indc_{\{x + e_i \in \cu\}} \right).
\end{equation*}
Note that $|\mathbf{n}(x)| \leq d$. Applying the discrete Stokes formula, we find that 
\begin{align*}
\left|  \frac{1}{|\cu|}\left\langle q, \nabla\! \left[v \right]_{\Pa} \right\rangle_{\cu} - p \cdot q \right|
& = \left| \frac{1}{|\cu|}\int_{\partial \cu} \left( \left[ v\right]_\Pa(x) -p\cdot x\right) q \cdot \mathbf{n}(x) \,dx \right| \\
& \leq \frac{d |q|}{|\cu|} \int_{\partial \cu} \left|  \left[v \right]_{\Pa}(x) - p\cdot x \right| \,dx.
\end{align*}
For each $\cu' \in\Pa$ with $\cu' \subseteq \partial_{\Pa}\cu$, we can find a point $\bar{y}(\cu') \in \partial_\a\C_*(\cu)$ and thus, by Lemma~\ref{l.nuisdirichlet}, for each $x\in \partial \cu$,
\begin{align*} \label{}
\left|\left[ v \right]_{\Pa}(x) - p\cdot x \right| 
& =  \left| v\left(\bar{z}\left(\cu_{\Pa}(x)\right)\right) - v\left(\bar{y}\left(\cu_{\Pa}(x)\right)\right) \right| + \left| p\cdot \left( x-\bar{y}\left(\cu_{\Pa}(x)\right) \right)\right|\\
& \leq \int_{\cu_{\Pa}(x) \cap \C_*(\cu)} \left| \nabla v \indc_{\{\a\neq 0\}} \right|  (x')\,dx' + C|p| \size\left(\cu_{\Pa}(x)\right).
\end{align*}
Integrating this over $x\in \partial \cu$ and using the H\"older inequality gives
\begin{align*}
\lefteqn{
\int_{\partial \cu} \left|  \left[v \right]_{\Pa}\! (x) - p\cdot x \right| \,dx
} \quad & \\
& \leq \int_{\partial \cu} \left( \int_{\C_*(\cu) \cap  \cu_\Pa(x)} \left| \nabla v \indc_{\{\a\neq 0\}} \right| (y) \,dy + C|p| \size\left( \cu_{\Pa}(x) \right) \right)\,dx \\
& \leq C |p| \left|\partial_{\Pa} \cu\right|  +  \int_{\partial \cu} |\cu_\Pa(x) |^\frac12 \left( \int_{\C_*(\cu) \cap \cu_\Pa(x)} \left| \nabla v \indc_{\{\a\neq 0\}} \right|^2 (y) \,dy \right)^\frac 12\,dx \\
& =  C |p| \left|\partial_{\Pa} \cu\right|  + \left( \sum_{\Pa\ni\cu' \subseteq \partial_\Pa\cu} \size(\cu')^{\frac32 d-1}\left( \int_{\C_*(\cu)\cap \cu'} \left| \nabla v \indc_{\{\a\neq 0\}} \right|^2(y) \,dy\right)^\frac 12  \right)\\
& \leq  C |p| \left|\partial_{\Pa} \cu\right|  + \left( \sum_{\Pa\ni\cu' \subseteq \partial_\Pa\cu} \size(\cu')^{3 d-2} \right)^\frac12 \left(\sum_{\Pa\ni\cu' \subseteq \partial_\Pa\cu} \int_{\C_*(\cu)\cap \cu'} \left| \nabla v \indc_{\{\a\neq 0\}} \right|^2(y) \,dy  \right)^\frac12\\
& \leq C |p| \left|\partial_{\Pa} \cu\right| +  \left( \sum_{x \in \partial \cu} \size(\cu_\Pa (x))^{2 d-1} \right)^\frac12 \left(\int_{\C_*(\cu)} \left| \nabla v \indc_{\{\a\neq 0\}} \right|^2(y) \,dy  \right)^\frac12.
\end{align*}
The first variation~\eqref{e.firstvarnu} for $\nu$ combined with~\eqref{e.nuupbound} yields that  
\begin{equation*} \label{}
\frac{1}{\left| \cl_\Pa(\cu) \right|} \frac12 \left\langle \nabla v, \a\nabla v \right\rangle_{\C_*(\cu)} = \nu(\cu,p)\leq C|p|^2. 
\end{equation*}
Thus 
\begin{equation*} \label{}
\left(\int_{\C_*(\cu)} \left| \nabla v \indc_{\{\a\neq 0\}} \right|^2(y) \,dy  \right)^\frac12 \leq C|p|\left| \cl_\Pa(\cu) \right|. 
\end{equation*}
Combining the above inequalities gives the desired bound~\eqref{e.nablavaffine.tech}.

\smallskip

To obtain~\eqref{e.nablavaffine}, we observe that~\eqref{e.partitionO1} implies that
\begin{equation}
\label{e.boundaryclip}
\frac{|\partial_\Pa \cu|}{|\cu|} \leq \O_1\left( \size(\cu)^{-1} \right)
\end{equation}
and
\begin{equation*} \label{}
\left(\frac{1}{| \cu|}  \sum_{x \in \partial \cu} \size\left(\cu_\Pa (x)\right)^{2 d-1} \right)^\frac12 \leq \O_{\frac{2}{2d-1}} \left( C \left( \frac{\left| \partial\cu\right|}{\left|\cu\right|}\right)^{-\frac{1}2} \right) = \O_{\frac{2}{2d-1}} \left( C \size(\cu)^{-\frac{1}2} \right). \qedhere
\end{equation*}
\end{proof}

\begin{corollary}
\label{c.munucomparison}
There exists a constant $C(d,\lambda,\p)<\infty$ such that, for every $\cu\in \T$ and $p,q\in\Rd$,
\begin{multline}
\label{e.munucomparison}
\Bigg| \left( \nu(\cu,p) - \mu(\cu,q) -  p\cdot q\right) \\
- \frac1{|\cl_\Pa(\cu)|} \frac12   \left\langle  \nabla (v(\cdot,\cu,p)-u(\cdot,\cu,q)) ,\a\nabla (v(\cdot,\cu,p)-u(\cdot,\cu,q)) \right\rangle_{\C_*(\cu)} \Bigg|  \\
\leq \O_{\frac{2}{2d-1}}\left(C|p||q| \size(\cu)^{-\frac12} \right).
\end{multline}

\end{corollary}
\begin{proof}
Combine~\eqref{e.omegaform0} and Lemma~\ref{l.nablavaffine}.
\end{proof}

We next show that the combination of the upper bounds in Lemma~\ref{l.upperbounds} and the inequality in Corollary~\ref{c.munucomparison} give us the desired lower bounds.

\begin{lemma}
\label{l.lowerbounds}
There exists a constant $C := C(d,\p, \lambda) < \infty$ such that, for every $\cu\in\T$ and $p,q\in\Rd$,
\begin{equation}
\label{e.mulobound}
-\mu(\cu,q) \geq |q|^2 \left( \frac{1}{C} - \O_{\frac{2}{2d-1}} \left(C\size(\cu)^{-\frac 12}\right) \right).
\end{equation}
and
\begin{equation}
\label{e.nulobound}
\nu(\cu,p) \geq |p|^2 \left(\frac1C - \O_{\frac{2}{2d+1}} \left(C\size(\cu)^{-\frac12} \right) \right).
\end{equation}
\end{lemma}
\begin{proof}
These estimates are consequences of Lemma~\ref{l.upperbounds} and Lemma~\ref{c.munucomparison}.
We first give the lower bound for $-\mu$. Combining~\eqref{e.nuupbound} and~\eqref{e.munucomparison}, we find that, for every $p,q\in\Rd$,  
\begin{align*}
\mu(\cu,q) 
& \leq \nu(\cu,p) - p\cdot q + \O_{\frac{2}{2d-1}}\left( C|p||q| \size(\cu)^{-\frac12}\right) \\
& \leq C|p|^2 - p\cdot q + \O_{\frac{2}{2d-1}}\left( C|p||q| \size(\cu)^{-\frac12}\right).
\end{align*}
Choosing $p=q/2C$ to minimize the first two terms on the right, we get
\begin{equation*}
\mu(\cu,q) \leq -\frac1{C} |q|^2 + \O_{\frac{2}{2d-1}}\left( C|q|^2 \size(\cu)^{-\frac12}\right),
\end{equation*}
which is~\eqref{e.mulobound}.

\smallskip

To prove the lower bound for $\nu$, we argue similarly: by~\eqref{e.muupbound} and~\eqref{e.munucomparison}, for every $p,q\in\Rd$, we have
\begin{align*}
\nu(\cu,p) & \geq \mu(\cu,q) + p\cdot q - \O_{\frac{2}{2d-1}}\left( C|p||q| \size(\cu)^{-\frac12}\right) \\
& \geq -C|q|^2 \left( \fint_{\cu} \size\left( \cu_\Pa(x) \right)^{2d-2}\,dx \right)+ p\cdot q - \O_{\frac{2}{2d-1}}\left( C|p||q| \size(\cu)^{-\frac12}\right).
\end{align*}
Optimize by taking $q: = C^{-1}  \left( \fint_{\cu} \size\left( \cu_\Pa(x) \right)^{2d-2}\,dx \right)^{-1}p$ in order to maximize the first two terms on the right side, we get
\begin{equation*}
\nu(\cu,p)
\geq |p|^2 \bigg( \frac1C \left( \fint_{\cu} \size\left( \cu_\Pa(x) \right)^{2d-2}\,dx \right)^{-1} - \O_{\frac{2}{2d-1}}\left( C \size(\cu)^{-\frac12} \right)  \bigg).
\end{equation*}
Estimating the first term on the right side of the previous line by~\eqref{e.boundwhip}, we obtain
 \begin{equation*}
\nu(\cu,p)
\geq |p|^2 \bigg( \frac1C - \O_{\frac{1}{2d+1}} \left(C\size(\cu)^{-1} \right) - \O_{\frac{2}{2d-1}}\left( C \size(\cu)^{-\frac12} \right)  \bigg).
\end{equation*}
Since $\nu(\cu,p)$ is nonnegative (and thus bounded below almost surely), the previous line and~\eqref{e.improves} implies
\begin{equation*} \label{}
\nu(\cu,p)
\geq |p|^2 \bigg( \frac1C - \O_{\frac{2}{2d+1}} \left(C\size(\cu)^{-\frac12} \right) \bigg). \qedhere
\end{equation*}
\end{proof}

The final result of this subsection concerns the approximate subadditivity of~$-\mu$ and~$\nu$.

\begin{lemma}
\label{l.subadd}
For every $p,q\in\Rd$ and $m,n\in\N$ with $n\leq m$ and $\cu\in \T_m$,
\begin{equation}
\label{e.musubadd}
\mu(\cu,q) \geq 3^{-d(m-n)} \sum_{z\in3^n\Zd\cap \cu} \mu(\cu_n(z),q) - |q|^2 \O_{\frac{1}{2d-1}} \left( C|q|^2 3^{-\frac n4} \right)
\end{equation}
and
\begin{equation}
\label{e.nusubadd}
\nu(\cu,p) \leq 3^{-d(m-n)} \sum_{z\in3^n\Zd\cap \cu} \nu(\cu_n(z),p) + |p|^2 \O_{\frac{2}{2d-1}}\left(C 3^{-n}\right).
\end{equation}
\end{lemma}
\begin{proof}
We first give the proof of~\eqref{e.musubadd}. Denote~$u:= u(\cdot,\cu,q)$. Testing the definition of~$\mu(z+\cu_n,q)$ for $z\in 3^n\Zd\cap \cu$ with~$u$ gives
\begin{multline}
\label{e.dumbplug}
\mu(z+\cu_n,q) \indc_{\{z+\cu_n\in\Pas\}} \\
\leq  \frac{1}{|\cu_n|}\left( \frac12  \left\langle  \nabla u, \a\nabla u\right\rangle_{\C_* (z+\cu_n)} - \left\langle q, \nabla\! \left[ u \right]_{\Pa}\! \right\rangle_{z+\cu_n} \right) \indc_{\{z+\cu_n\in\Pas\}}.
\end{multline}
If we sum the right side over $z\in 3^n\Zd\cap \cu$, the result is close to $\mu(\cu,q)$. There are two sources of error: (i) $\C_*(\cu)$ contains edges which do not belong to any of the clusters $\C_*(z+\cu_n)$; (ii) when we sum the second term in parentheses we miss the edges between two adjacent subcubes and edges deleted because $z+\cu_n\not\in\Pas$. We treat each of these errors in turn. The claim is that
\begin{multline}
\label{e.testingerrormu}
\frac1{|\cu|}\sum_{z\in3^n\Zd\cap \cu}  \left( \frac12  \left\langle  \nabla u, \a\nabla u\right\rangle_{\C_* (z+\cu_n)} - \left\langle q, \nabla\! \left[ u \right]_{\Pa}\! \right\rangle_{z+\cu_n} \right) \indc_{\{z+\cu_n\in\Pas\}} \\
 \leq \mu(\cu,q) + \O_{\frac{1}{2d-1}} \left( C|q|^2 3^{-\frac n4} \right).
\end{multline}

\smallskip

First, it is clear that while $\C_*(\cu)\cap (z+\cu_n)$ and $\C_*(z+\cu_n)$ may be different, every open edges in the latter cluster belongs to the former. Therefore, since the quadratic form is nonnegative on each edge, we have, for every $z\in 3^n\Zd\cap\cu$,
\begin{equation*}
\sum_{z\in 3^n\Zd \cap \cu} \frac12\left\langle  \nabla u, \a\nabla u\right\rangle_{\C_* (z+\cu_n)} \indc_{\{z+\cu_n\in\Pas\}}  
\leq  \frac12 \left\langle  \nabla u, \a\nabla u\right\rangle_{\C_* (\cu)}.
\end{equation*}
Next, let $V$ be the set of vertices $x \in\cu$ with an adjacent edge $\{x,y\}$ such that $y\not \in \cu_n(x)$ or such that $\cu_n(x) \not \in \Pas$. It is clear that
\begin{equation*}
|V| \leq   C| \cu|3^{-n} + C \sum_{x\in \cu} \indc_{\left\{ \size(\cu_{\Pa}(x))>3^n\right\}} \leq | \cu| \left( C3^{-n} + \O_1\left(3^{-n} \right) \right) \leq \O_1\left( C| \cu|3^{-n} \right). 
\end{equation*}
From this, the H\"older inequality and~\eqref{e.coarsegraduL2} with $\delta = \frac14$, we get
\begin{align*}
\lefteqn{
 \frac{1}{\left|\cu\right|} \left| \left\langle q, \nabla\! \left[ u \right]_{\Pa}\! \right\rangle_{\cu}  -  \sum_{z\in 3^n\Zd \cap \cu} \left\langle q, \nabla\! \left[ u \right]_{\Pa}\! \right\rangle_{z+\cu_n} \indc_{\{z+\cu_n\in\Pas\}}  \right| 
 } \qquad \qquad\qquad& \\
 & \leq \frac{C|q|}{|\cu|} \int_V \left|  \nabla\! \left[ u \right]_{\Pa} \right|(x)\,dx \hspace{2in}\\
&  \leq C|q| \left( \frac{\left| V \right|}{|\cu|} \right)^{\frac12}  \left( \fint_{\cu} \left| \nabla \left[ u \right]_{\Pa} \right|^2(x)\,dx \right)^{\frac12} \\
& \leq C|q|\cdot \O_2\left( C 3^{-\frac n2} \right)\cdot \O_{\frac2{4d-3}}\left(C|q| 3^{\frac n4}\right) \\
& \leq \O_{\frac{1}{2d-1}} \left( C|q|^2 3^{-\frac n4} \right). 
\end{align*}
Combining the above yields~\eqref{e.testingerrormu}. To obtain~\eqref{e.musubadd} from~\eqref{e.dumbplug} and~\eqref{e.testingerrormu}, we just recall that $\mu(z+\cu_n,q) \indc_{\{ z+\cu_n\not\in\Pas\}} = 0$. 

\smallskip

We turn to the proof of~\eqref{e.nusubadd}, which is only slightly different. Testing the definition of $\nu(z+\cu_n,p)$ with $v:=v(\cdot,\cu,p)$ gives
\begin{multline}
\label{e.dumbplugnu}
\nu(z+\cu_n,p) \indc_{\{ z+ \cu_n\in\Pas\}}  \\
\geq  \frac{1}{|\cu_n|}\left( - \frac12  \left\langle  \nabla v, \a\nabla v\right\rangle_{\C_* (z+\cu_n)} 
+ \left\langle p, \a\nabla v \right\rangle_{\C_*(z+\cu_n)} \right) \indc_{\{z+\cu_n\in\Pas\}}.
\end{multline}
As above, we have
\begin{equation*}
\sum_{z\in 3^n\Zd \cap \cu} \frac12\left\langle  \nabla v, \a\nabla v\right\rangle_{\C_* (z+\cu_n)} \indc_{\{z+\cu_n\in\Pas\}}  
\leq  \frac12 \left\langle  \nabla v, \a\nabla v\right\rangle_{\C_* (\cu)}.
\end{equation*}
Let $W$ denote the set of vertices $x\in\cu$ with an edge $\{x,y\}$ belonging to the cluster $\C_*(\cu)$ but not any of the clusters $\C_*(z+\cu_n)$ for $z\in 3^n\Zd$ satisfying $z+\cu_n\in\Pas$. It is clear that $W$ must be contained in the union of elements of $\Pa$ which touch the boundaries of one of the cubes $z+\cu_n$ and those cubes $z+\cu_n$ which do not belong to $\Pas$. Therefore, in view of~\eqref{e.boundaryclip}, 
\begin{align}
\label{e.stupidedges}
|W| 
& \leq \left| \bigcup_{z\in 3^n\Zd \cap\cu} \partial_\Pa (z+\cu_n) \right| + \left|  \bigcup_{z\in 3^n\Zd \cap\cu, \ z+\cu_n\not\in\Pas} (z+\cu_n) \right| \\ \notag
& \leq C \sum_{z\in 3^n\Zd \cap\cu}\left| \partial_\Pa (z+\cu_n) \right| + C \sum_{x\in \cu} \indc_{\left\{ \size(\cu_{\Pa}(x))>3^n\right\}} \\ \notag
& \leq \O_1\left( C|\cu|3^{-n} \right).
\end{align}
By the previous inequality, the H\"older inequality and~\eqref{e.vupbound}, we get
\begin{align*}
\lefteqn{
 \frac{1}{\left|\cu\right|} \left| \left\langle p, \a\nabla v\right\rangle_{\C_*(\cu)}  -  \sum_{z\in 3^n\Zd \cap \cu} \left\langle p, \a\nabla v \right\rangle_{\C_*(z+\cu_n)} \indc_{\{z+\cu_n\in\Pas\}}  \right| 
 } \qquad \qquad\qquad& \\
 & \leq \frac{C|p|}{|\cu|} \int_{W\cap \C_*(\cu)} \left|  \nabla v \indc_{\{\a\neq0\}}\right|(x)\,dx \hspace{2in}\\
&  \leq C|p| \left( \frac{\left| W \right|}{|\cu|} \right)^{\frac12}  \left( \frac{1}{|\cu|}\int_{\C_*(\cu)} \left| \nabla v \indc_{\{\a\neq0\}}\right|^2(x)\,dx \right)^{\frac12} \\
& \leq C|p|^2\cdot \O_2\left( C 3^{-\frac n2} \right)\\
& \leq \O_{2} \left( C|p|^2 3^{-\frac n2} \right). 
\end{align*}
Combining these yields
\begin{multline}
\label{e.testingerrornu}
\frac1{|\cu|}\sum_{z\in3^n\Zd\cap \cu}  \left( - \frac12  \left\langle  \nabla v, \a\nabla v\right\rangle_{\C_* (z+\cu_n)} 
+ \left\langle p, \a\nabla v \right\rangle_{\C_*(z+\cu_n)} \right) \indc_{\{z+\cu_n\in\Pas\}} \\
 \geq \nu(\cu,q) - \O_{2} \left( C|p|^2 3^{-\frac n2} \right).
\end{multline}
Combined with~\eqref{e.dumbplugnu}, this yields
\begin{equation*}
\nu(\cu,p) \leq 3^{-d(m-n)} \sum_{z\in 3^n\Zd\cap \cu} \nu(z+\cu_n,p) \indc_{\{ z+\cu_n\in\Pas\}} + \O_{2}\left(C|p|^2 3^{-\frac n2}  \right).
\end{equation*}
This implies~\eqref{e.nusubadd} since $\nu$ is nonnegative.
\end{proof}

\subsection{Localization and approximate stationarity}

We next define local and stationary versions of $\mu$ and $\nu$ and show that they are the same as the original quantities, up to a small error. For each $m,n \in \N$ satisfying $m>n$, we denote 
\begin{equation*}
\T_m^{(n)} := \left\{ z + \cu_m \, : \, z \in 3^n \Zd \right\}
\end{equation*}
and define the following random family of good cubes
\begin{equation*}
\G^{(n)} := \G \cup \left( \cup \left\{ \cu' \in \T \, : \, \size(\cu') \geq 3^{n} \right\} \right),
\end{equation*}
where $\G$ is the set of cube cubes from Definition~\ref{d.Pa}. We then apply Proposition~\ref{p.partitions} which gives two partitions $\Pa^{(n)}$ and $\Pa^{(n)}_\mathrm{loc}(\cu)$. Notice that thanks to this construction we obtain a partition $\Pa^{(n)}$ which is stationary for translations of vectors within $3^n \Zd$ (in particular, this will ensure that~\eqref{e.munustat} holds) and will be important in Section 6. Before introducing the local and stationary versions of $\mu$ and $\nu$, we prove a quantitative result, showing that $\Pa$ and $\Pa^{(n)}$ are equals on $\cu$ on a set of large probability.
\begin{proposition} \label{Pa.stationaryPa}
For each $m,n \in \N$ satisfying $m>n$, $3^n > C m$ and $\cu\in \T_m^{(n)}$, the following estimates holds
\begin{equation*}
\P \left[ \forall x \in \cu, \, \cu_\Pa(x) = \cu_{\Pa^{(n)}} (x) \right] \geq 1 - C  \exp \left( - C^{-1} 3^{n} \right).
\end{equation*}
\end{proposition}
\begin{proof}
For each $x \in \cu$, we have
\begin{equation*}
\cu_\Pa(x) \neq \cu_{\Pa^{(n)}}(x) \iff \exists \cu' \in \mathcal{K} \left( \cu_{\Pa^{(n)}} (x)\right) \setminus \G
\end{equation*}
Thus, with a similar argument as in the proof of (iv) of Proposition~\ref{p.partitions}
\begin{equation*}
\exists x \in \cu, \, \cu_\Pa(x) \neq \cu_{\Pa^{(n)}}(x) \implies \exists \cu' \notin \G, \, \size(\cu') \geq 3^{n} \mbox{ and } \dist( \cu', \cu) \leq C \size(\cu')
\end{equation*}
and thus
\begin{align*}
\P \left[ \forall x \in \cu, \, \cu_\Pa(x) \neq \cu_{\Pa^{(n)}} (x) \right] & \leq C 3^{m-n} \exp \left( - C^{-1} 3^{n} \right) \\
														& \leq C  \exp \left( - C^{-1} 3^{n} \right).
\end{align*}
The proof is complete.
\end{proof}

\smallskip

Now, for $\cu \in \T_m^{(n)}$, consider the event
\begin{equation*}
A_{n}(\cu) := \left\{ \sup_{\cu' \in \Pa_{\mathrm{loc}}^{(n)} (\cu) } \size( \cu') \leq 3^{n-1}\right\}.
\end{equation*}
Observe that $A_{n}(\cu) \in \F(\cu)$.
As in Section~\ref{s.partition}, we define, for $n \in \N$,
\begin{equation*}
\cu^{(n)} := \left\{ x\in \cu\,:\, \dist(x,\partial \cu)\geq 3^n \right\}.
\end{equation*}
Note that on the event $A_{n}(\cu)$, by definition of the local partition~$\Pa_{\mathrm{loc}}^{(n)}$ in Proposition~\ref{p.partitions}, every cube belonging to $\Pa_{\mathrm{loc}}^{(n)}(\cu)$ contained in $\cu^{\left( n\right)}$ is good (in the sense of Definition~\ref{def.goodcube}),
so that the union of the clusters $\left( \C_*\left(\cu_{\Pa_{\mathrm{loc}}^{(n)}(\cu)} (x)\right)\right)_{x \in \cu^{\left( n \right)}}$ over $x\in \cu^{(n)}$ is connected. We then define the local cluster $\C_*^{(n)}(\cu)$ as the maximal cluster which is a subset of $\cu$ and which contains every cluster of the form $\left( \C_*\left(\cu_{\Pa_{\mathrm{loc}}^{(n)}(\cu)} (x)\right)\right)_{x \in \cu^{\left( n \right)}}$ with $x\in \cu^{(n)}$. If $A_n(\cu)$ fails to hold, we define $\C_*^{(n)}(\cu)$ to be empty. 

\smallskip

From the local cluster we define the local versions of the energy quantities by
\begin{equation*} \label{}
\mu_{\mathrm{loc}}^{(n)}(\cu,q)
 := 
 \indc_{A_{n}(\cu)}
\inf_{u \in \A\left(\C_*^{(n)}(\cu)\right)} \frac1{|\cu|} 
\left(
\frac12 \left\langle \nabla u, \a\nabla u \right\rangle_{\C_*^{(n)}(\cu)} 
 - \left\langle q, \nabla\! \left[ u \right]_{\Pa_{\mathrm{loc}}^{(n)}(\cu)} \right\rangle_{\cu^{\left(n\right)}}
 \right)
\end{equation*}
and
\begin{equation*} \label{}
\nu_{\mathrm{loc}}^{(n)}(\cu,q) 
:= 
\indc_{A_n(\cu)} 
\sup_{v \in \A\left(\C_*^{(n)}(\cu)\right)} \frac1{|\cu|} 
\left( 
- \frac12 \left\langle  \nabla v, \a\nabla v \right\rangle_{\C_*^{(n)}(\cu)} 
+ \left\langle p, \a\nabla v \right\rangle_{\C_*^{(n)}(\cu)}
\right).
\end{equation*}
In other words, $\mu_{\mathrm{loc}}^{(n)}(\cu,q)$ is the same as~$\mu(\cu,q)$ except that we use the local partition~$\Pa_{\mathrm{loc}}^{(n)}(\cu)$ instead of~$\Pa$ and that we integrate the second term only on~$\cu^{\left(n\right)}$; meanwhile~$\nu_{\mathrm{loc}}^{(n)}(\cu,q)$ is the same as~$\nu(\cu,q)$, except in the event (which is unlikely for~$\cu$ large and $n$ large) that $A_n(\cu)$ does not hold.

\smallskip

It is immediate from Proposition~\ref{p.partitions}(iv), which gives the locality of $\Pa_{\mathrm{loc}}^{(n)}(\cu)$, that for every $\cu\in \T_m^{(n)}$ and $p,q\in\Rd$, 
\begin{equation} 
\label{e.munuloc}
\mu_{\mathrm{loc}}^{(n)}(\cu,q) 
\ \mbox{and} \ 
\nu_{\mathrm{loc}}^{(n)}(\cu,p)
\quad 
\mbox{are $\F(\cu)$--measurable.}
\end{equation}
It is clear that the construction above yields that $\mu_{\mathrm{loc}}^{(n)}(\cu_m,q)$ and $\nu_{\mathrm{loc}}^{(n)}(\cu_m,p)$ are $3^m\Zd$--stationary and thus
\begin{equation} 
\label{e.munustat}
\mbox{the laws of $\mu_{\mathrm{loc}}^{(n)}(z+\cu_m,q)$ and $\nu_{\mathrm{loc}}^{(n)}(z+\cu_m,p)$ are independent of $z\in 3^n\Zd$.} 
\end{equation}
In particular,~$\left\{ \mu_{\mathrm{loc}}^{(n)}(z+\cu_m,q) \right\}_{z\in 3^m\Zd}$ and~$\left\{ \nu_{\mathrm{loc}}^{(n)}(z+\cu_m,p) \right\}_{z\in 3^m\Zd}$ are~i.i.d.

\smallskip

We denote by $u_{\mathrm{loc}}^{(n)}(\cdot,\cu,q)$ and $v_{\mathrm{loc}}^{(n)}(\cdot,\cu,q)$ the optimizers in the definitions of $\mu_{\mathrm{loc}}^{(n)}(\cu,q)$ and $\nu_{\mathrm{loc}}^{(n)}(\cu,p)$, respectively. We choose the additive constant in the same way as above, so that~\eqref{e.additiveconstants} holds. In the event that $A_n(\cu)$ does not hold, we define $u_{\mathrm{loc}}^{(n)}(\cdot,\cu,q) = v_{\mathrm{loc}}^{(n)}(\cdot,\cu,p) = 0$. 

\smallskip

Most of the estimates that we proved in the previous subsection continue to hold for the localized quantities. In particular, since $\Pa_{\mathrm{loc}}^{(n)}$ is finer than $\Pa^{(n)}$ which is finer than $\Pa$, we have, by the same proof as the one in Lemma~\ref{l.upperbounds}, the bounds
\begin{equation} 
\label{e.munulocbounds}
\left\{
\begin{aligned}
& -\mu_{\mathrm{loc}}^{(n)}(\cu,q) \leq C|q|^2 \fint_\cu \size\left( \cu_{\Pa}(x) \right)^{2d-2}\,dx, \\
& \nu_{\mathrm{loc}}^{(n)}(\cu,p) \leq C|p|^2.
\end{aligned}
\right.
\end{equation}
We also record the fact that, by the same argument as the one leading to~\eqref{e.uupbound},~\eqref{e.vupbound} and~\eqref{e.coarsegraduL2},  we have the estimates
\begin{align}
\label{e.coarsegraduloc}
 & \frac{1}{|\cu|}\int_{\C_*(\cu)} \left| \nabla u_{\mathrm{loc}}^{(n)}(\cdot,\cu,q) \indc_{\left\lbrace \a \neq 0 \right\rbrace} \right|^2(x)\,dx
\leq C|q|^2 \fint_{\cu} \size\left( \cu_{\Pa}(x) \right)^{2d-2}\,dx,
\\ \notag
&\frac{1}{|\cu|} \int_{\C_*^{(n)}(\cu)} \left| \nabla v_{\mathrm{loc}}^{(n)}(\cdot,\cu,p)\right|^2(x)\,dx 
 \leq C|p|^2, \\ \notag
\end{align}

We next estimate the difference between $\mu$ and $\mu_{\mathrm{loc}}^{(n)}$ as well as $\nu$ and $\nu_{\mathrm{loc}}^{(n)}$ and their minimizers. 

\begin{proposition}
\label{p.localization}
There exists $C(d, \lambda, \p)<\infty$ 
such that, 
for every $m,n \in \N$ with $m>n$ and $3^n\geq Cm$, every $\cu\in\T_m^{(n)}$ and $p,q\in\Rd$, we have
\begin{equation} 
\label{e.localization}
\left\{
\begin{aligned}
& \left| \mu(\cu,q) - \mu_{\mathrm{loc}}^{(n)}(\cu,q) \right| \leq \O_{\frac1{2d-1}}\left(C|q|^2\left( 3^{-\frac{m-n}2} + 3^{-n}  \right) \right), \\
& \left| \nu(\cu,p) - \nu_{\mathrm{loc}}^{(n)}(\cu,p) \right| \leq \O_1\left( C|p|^23^{-n}\right). 
\end{aligned}
\right.
\end{equation}
and 
\begin{multline} 
\label{e.ulocloc}
\frac{\indc_{\left\{ \C_*(\cu) = \C^{(n)}_*(\cu) \right\}}}{|\cu|}\int_{\C_*(\cu)} \left| \left(\nabla u(\cdot,\cu,q)- \nabla u_{\mathrm{loc}}^{(n)}(\cdot,\cu,q)\right)\indc_{\{\a\neq 0\}} \right|^2(x)\,dx 
\\
\leq \O_{\frac1{2d-1}}\left(C|q|^2 \left( 3^{-\left(\frac{m-n}{2}\right)}+ 3^{-n}  \right) \right).
\end{multline}
\end{proposition}
\begin{proof}
First note that on the event $\left\lbrace  \sup_{ x \in \cu } \size( \cu_\Pa(x) ) \leq 3^{n} \right\rbrace$, $\C_*^{(n)}(\cu) = \C_*(\cu)$. Moreover, since $\Pa_{\mathrm{loc}}^{(n)}(\cu)$ is finer than $\Pa$, we have
\begin{equation*}
 \left\lbrace  \sup_{ x \in \cu } \size( \cu_\Pa(x) ) \leq 3^{n} \right\rbrace \subseteq A_n(\cu).
\end{equation*}
By Proposition~\ref{p.partitions} (iii) and the assumption $3^n\geq Cm$, we can estimate the probability of this event by
\begin{equation} \label{e.probabilitysizesmall}
\P \left[ \sup_{ x \in \cu } \size( \cu_\Pa(x) ) > 3^{n} \right]  \leq 3^{dm} \exp( -C^{-1} 3^{n}) \leq C \exp \left( -C^{-1} 3^n \right). 
\end{equation}
It is clear that $\nu_{\mathrm{loc}}^{(n)}(\cu,p) = \nu(\cu,p)$ if $\C_*^{(n)}(\cu) = \C_*(\cu)$ and $A_n(\cu)$ holds. Thus on the event $\left\lbrace  \sup_{ x \in \cu } \size( \cu_\Pa(x) ) \leq 3^{n} \right\rbrace$, $\nu_{\mathrm{loc}}^{(n)}(\cu,p) = \nu(\cu,p)$. Using the bounds~\eqref{e.nuupbound} and~\eqref{e.munulocbounds}, we therefore obtain 
\begin{align*} \label{}
\left| \nu(\cu,p) - \nu_{\mathrm{loc}}^{(n)}(\cu,p) \right|
& \leq \left( \left| \nu(\cu,p) \right| + \left|\nu_{\mathrm{loc}}^{(n)}(\cu,p) \right| \right) \indc_{\left\lbrace  \sup_{ x \in \cu } \size( \cu_\Pa(x) ) > 3^{n} \right\rbrace} \\
& \leq C|p|^2 \indc_{\left\lbrace  \sup_{ x \in \cu } \size( \cu_\Pa(x) ) > 3^{n} \right\rbrace} \\
& \leq \O_1\left( C|p|^23^{-n} \right). 
\end{align*}
We turn to the bound for $\mu$. Denote by $B_{n}(\cu)$ the event
\begin{equation*}
B_{n}(\cu) := \left\lbrace  \sup_{ x \in \cu } \size( \cu_\Pa(x) ) \leq 3^{n} \right\rbrace \bigcap \left\lbrace \forall x \in \cu^{\left( n \right)},~   \cu_\Pa(x) = \cu_{\Pa_{\mathrm{loc}}^{(n)}}(x)  \right\rbrace.
\end{equation*}
By (iv) of Proposition~\ref{p.partitions}, Proposition~\ref{Pa.stationaryPa}, and~\eqref{e.probabilitysizesmall}, we can estimate
\begin{align*}
\P \left[ B_{n}(\cu) \right] & \geq 1 - 3^{dm} \exp \left( -C^{-1} 3^{n}\right) -  C 3^{m-n} \exp \left( -C^{-1} 3^{n}\right) \\
				& \geq 1 - C \exp \left( -C^{-1} 3^{n}\right).
\end{align*}
On this event, we can use the function $u( \cdot, \cu, q)$ as a minimizer candidate for $\mu_{\mathrm{loc}}^{(n)}(\cu,q)$. This yields
\begin{align*}
\mu_{\mathrm{loc}}(\cu, q) & \geq \frac1{|\cu|}  \left( \frac12 \left\langle \nabla u, \a\nabla u \right\rangle_{\C_*(\cu)}  - \left\langle q, \nabla\! \left[ u \right]_{\Pa_{\mathrm{loc}}^{(n)}(\cu)} \right\rangle_{\cu^{\left( n \right)}} \right) \\
						& \geq \frac1{|\cu|}  \left( \frac12 \left\langle \nabla u, \a\nabla u \right\rangle_{\C_*(\cu)}  - \left\langle q, \nabla\! \left[ u \right]_{\Pa} \right\rangle_{\cu^{\left( n \right)}} \right) \\
						& \geq \mu(\cu,q) - \left|\frac1{|\cu|} \left\langle q, \nabla\! \left[ u \right]_{\Pa} \right\rangle_{\cu \setminus \cu^{\left( n \right)}} \right| \\
 & \geq \mu(\cu,q) - |q| \frac{1}{|\cu|}\int_{\cu \setminus \cu^{\left( n \right)}} \left| \nabla \left[ u(\cdot,\cu,q) \right]_{\Pa} \right|(x)\,dx.
\end{align*}
To estimate the last term on the right hand side, we can extract from the proof of Lemma~\ref{l.coarsegrads} with $s=1$ the following inequality 
\begin{align*}
\lefteqn{\int_{\cu \setminus \cu^{\left( n \right)}} \left| \nabla \left[ u_{\mathrm{loc}}^{(n)}(\cdot,\cu,q) \right]_{\Pa} \right|(x)\,dx} \qquad & \\ & \leq C \sum_{\cu' \in \Pa, \cu' \subseteq \cu \setminus \cu^{\left( n-1 \right)}} \size(\cu')^{d-1} \int_{\cu' \cap \C_*(\cu)} \left| \nabla u_{\mathrm{loc}}^{(n)} \indc_{\left\lbrace \a \neq 0 \right\rbrace} \right| (x) \, dx \\
& \leq C  \left( \sum_{\cu' \in \Pa, \cu' \subseteq \cu \setminus \cu^{\left(n-1 \right)}} \size(\cu)^{3d-2} \right)^\frac12 \left( \int_{\left( \cu \setminus \cu^{\left( n-1\right)}\right) \cap \C_*(\cu)} \left| \nabla u_{\mathrm{loc}}^{(n)} \indc_{\left\lbrace \a \neq 0 \right\rbrace} \right|^2 (x) \, dx \right)^\frac12 \\
& \leq  C  \left( \int_{ \cu \setminus \cu^{\left( n-1 \right)}} \size(\cu_\Pa(x))^{2d-2} \, dx\right)^\frac12 \left( \int_{ \C_*(\cu)} \left| \nabla u_{\mathrm{loc}}^{(n)} \indc_{\left\lbrace \a \neq 0 \right\rbrace} \right|^2 (x) \, dx \right)^\frac12
\end{align*}
This gives
\begin{multline*}
\frac1{|\cu|} \int_{\cu \setminus \cu^{\left( n \right)}} \left| \nabla \left[ u_{\mathrm{loc}}^{(n)}(\cdot,\cu,q) \right]_{\Pa} \right|(x)\,dx \\  
\leq  C 3^{\frac{n-m}{2}} \left( \frac1{\left|\cu \setminus \cu^{\left( n-1 \right)}\right|}\int_{\cu \setminus \cu^{\left( n-1 \right)}} \size(\cu_\Pa(x))^{2d-2} \, dx \right)^\frac12 \\
\times \left( \frac 1{|\cu|} \int_{ \C_*(\cu)} \left| \nabla u_{\mathrm{loc}}^{(n)} \indc_{\left\lbrace \a \neq 0 \right\rbrace} \right|^2 (x) \, dx \right)^\frac12.
\end{multline*}
This yields by~\eqref{e.Oavgs} and~\eqref{e.coarsegraduloc}
\begin{align} \label{comparisonmumuloc1}
\mu_{\mathrm{loc}}(\cu, q) \indc_{B_{n}(\cu)} & \geq \mu(\cu,q) \indc_{B_{n}(\cu)} - C 3^{\frac{n-m}{2}}  \O_{\frac1{d-1}}(C)  \O_{\frac1{d-1}}(C)  \\
										& \geq  \mu(\cu,q) \indc_{B_{n}(\cu)} - \O_{\frac1{2d-2}}\left(C|q|^2 3^{\frac{n-m}2} \right). \notag
\end{align}
Similarly and still on the event $B_{n}(\cu)$, using $u_{\mathrm{loc}}^{(n)}$ as a minimizer candidate for $\mu(\cu,q)$ shows
\begin{align*}
\mu(\cu, q) & \geq \frac1{|\cu|}  \left( \frac12 \left\langle \nabla u_{\mathrm{loc}}, \a\nabla u_{\mathrm{loc}} \right\rangle_{\C_*(\cu)}  - \left\langle q, \nabla\! \left[ u_{\mathrm{loc}} \right]_{\Pa} \right\rangle_{\cu} \right) \\
						& \geq \mu_{\mathrm{loc}}(\cu,q) - \left|\frac1{|\cu|} \left\langle q, \nabla\! \left[ u_{\mathrm{loc}} \right]_{\Pa} \right\rangle_{\cu \setminus \cu^{\left( n \right)}} \right| \\
						& \geq  \mu_{\mathrm{loc}}(\cu,q) -  \frac{|q|}{|\cu|} \int_{\cu \setminus \cu^{\left( n\right)}} \left| \nabla \left[ u_{\mathrm{loc}}(\cdot,\cu,q) \right]_{\Pa} \right|(x)\,dx .
\end{align*}
Thus
\begin{equation} 
\label{comparisonmumuloc2}
\mu(\cu, q) \indc_{B_{n}(\cu)}  \geq \mu_{\mathrm{loc}}^{(n)}(\cu,q) \indc_{B_{n}(\cu)} - \O_{\frac1{2d-2}}(C|q|^2 3^{\frac{n-m}{2}}). 
\end{equation}
Combining~\eqref{comparisonmumuloc1} and~\eqref{comparisonmumuloc2} gives
\begin{equation*}
\left|\mu(\cu, q) - \mu_{\mathrm{loc}}^{(n)}(\cu,q)  \right| \indc_{B_{n}(\cu)}  \leq \O_{\frac1{2d-2}}\left(C|q|^2 3^{\frac{n-m}{2}}\right) 
\end{equation*}
On the event $\Omega \setminus B_n(\cu)$, we have by~\eqref{e.munulocbounds},
\begin{align*} \label{}
\left| \mu(\cu,q) - \mu_{\mathrm{loc}}^{(n)}(\cu,q) \right| \indc_{\left\lbrace  \Omega \setminus B_{n}(\cu) \right\rbrace} 
&  \leq \left( \left| \mu(\cu,q) \right| + \left|\mu_{\mathrm{loc}}^{(n)}(\cu,q) \right| \right) \indc_{\left\lbrace  \Omega \setminus B_{n}(\cu) \right\rbrace} \\
& \leq \O_{\frac{1}{2d-2}}(C|q|^2) \O_1\left(C|q|^23^{-n} \right) \\
& \leq \O_{\frac{1}{2d-1}}\left( C|q|^23^{-n} \right).
\end{align*}
Summing the two previous displays completes the proof of~\eqref{e.localization}.

\smallskip

We now turn to the proof of~\eqref{e.ulocloc}. We apply the second variations~\eqref{e.secondvarmu} to $w = u_{\mathrm{loc}}^{(n)}(\cdot, \cu , q)$, on the event $B_{n}(\cu)$, which yields
\begin{multline*}
\frac1{|\cu|} \left( \frac12 \left\langle \nabla u_{\mathrm{loc}}, \a \nabla u_{\mathrm{loc}}^{(n)} \right\rangle_{\C_*(\cu)} - \left\langle q, \nabla\left[ u_{\mathrm{loc}}^{(n)}\right]_\Pa \right\rangle_{\cu}  \right) \\
		= \mu(\cu,q) + \frac1{|\cu|} \frac12\left\langle \nabla u_{\mathrm{loc}}^{(n)} - \nabla u, \a \left( \nabla u_{\mathrm{loc}}^{(n)} - \nabla u  \right)\right\rangle_{\C_*(\cu)}.
\end{multline*}
On the other hand, we have the estimate
\begin{multline*}
\frac1{|\cu|} \left( \frac12 \left\langle \nabla u_{\mathrm{loc}}^{(n)}, \a\nabla u_{\mathrm{loc}}^{(n)} \right\rangle_{\C_*(\cu)} - \left\langle q, \nabla\left[ u_{\mathrm{loc}}^{(n)}\right]_\Pa \right\rangle_{\cu}  \right) \\
\leq \mu_{\mathrm{loc}}^{(n)}(\cu,q) + \frac{|q|}{|\cu|} \int_{\cu \setminus \cu^{\left( n \right)}} \left| \nabla \left[ u_{\mathrm{loc}}^{(n)}(\cdot,\cu,q) \right]_{\Pa} \right|(x)\,dx .
\end{multline*}
Combining the two previous displays with the same computation as the one leading to~\eqref{comparisonmumuloc2}, we obtain
\begin{multline*}
\frac1{|\cu|}  \left\langle \nabla u_{\mathrm{loc}}^{(n)} - \nabla u, \a \left( \nabla u_{\mathrm{loc}}^{(n)} - \nabla u \right)\right\rangle_{\C_*(\cu)} \indc_{B_{n}(\cu)}
\\
\leq 2 \left| \mu_{\mathrm{loc}}^{(n)}(\cu,q) - \mu(\cu,q) \right| + \O_{\frac1{2d-2}}\left(C|q|^2 3^{\frac{n-m}{2}}\right)
 \leq \O_{\frac1{2d-2}}\left(C|q|^2 3^{\frac{n-m}{2}}\right).
\end{multline*}
On the event $\left\{ \C_*(\cu) = \C^{(n)}_*(\cu) \right\} \setminus B_n(\cu)$, we have, by~\eqref{e.uupbound} and~\eqref{e.coarsegraduloc},
\begin{multline*}
\frac{\indc_{\left\{ \left\{ \C_*(\cu) = \C^{(n)}_*(\cu) \right\}\setminus B_{n}(\cu)\right\}}}{|\cu|} \left\langle \nabla u_{\mathrm{loc}}^{(n)} - \nabla u, \a \left( \nabla u_{\mathrm{loc}}^{(n)} - \nabla u \right) \right\rangle_{\C_*(\cu)} 
 \\
  \leq \O_{\frac1{2d-2}}(C |q|^2) \O_1\left(  C3^{-n} \right) 
 \leq \O_{\frac{1}{2d-1}}\left(C|q|^23^{-n} \right).
\end{multline*}
Combining the two previous displays and recalling that $\a \geq \lambda \indc_{\left\lbrace \a \neq 0 \right\rbrace}$ completes the proof of~\eqref{e.ulocloc}.
\end{proof}

We conclude this section by recording some consequences of Proposition~\ref{p.localization}, for our reference. According to~\eqref{e.munustat} and~\eqref{e.localization} with $n=\left\lceil \frac m2\right\rceil$, we have, for every $m\in\N$,
\begin{equation*}
\sup_{z\in3^m\Zd} \left| \E \left[ \mu(z+\cu_m,q) \right] - \E \left[ \mu(\cu_m,q) \right] \right| \leq C|q|^2 3^{-\frac m4}. 
\end{equation*}
Combining this with~\eqref{e.musubadd}, we obtain, for each $m\in\N$,
\begin{equation*}
\E \left[ \mu(\cu_{m+1},q) \right] \geq \E \left[ \mu(\cu_{m},q) \right] - C|q|^23^{-\frac m4}. 
\end{equation*}
Summing this from $n$ to $m-1$ yields that, for every $m,n\in\N$ with $n<m$, 
\begin{equation}
\label{e.almostmonotone}
\E \left[ \mu(\cu_{m},q) \right] \geq \E \left[ \mu(\cu_{n},q) \right] - C|q|^23^{-\frac n4}. 
\end{equation}
By a similar argument, we have the bound
\begin{equation}
\label{e.almostmonotonenu}
\E \left[ \nu(\cu_{m},p) \right] \leq \E \left[ \nu(\cu_{n},p) \right] + C|p|^23^{- n}. 
\end{equation}
As a consequence of~\eqref{e.ulocloc}, we also get a localization estimate for the coarsened functions, summarized in the following lemma. 
\begin{lemma}
\label{l.ulocloccoarse}
There exist $C(d, \lambda, \p)<\infty$ and $s(d)>0$ such that,  
for every $m,n \in \N$ with $3^n\geq Cm$, every $\cu\in\T_m^{(n)}$ and $p,q\in\Rd$, we have
\begin{multline}
\label{e.ulocloccoarse}
\fint_{\cu^{(n)}} \left| \nabla \left[ u(\cdot,\cu,q) \right]_{\Pa} - \nabla \left[ u_{\mathrm{loc}}^{(n)}(\cdot,\cu,q)\right]_{\Pa_{\mathrm{loc}}^{(n)}(\cu)} \right|(x)\,dx
\\
\leq \O_{s}\left(C|q| \left( 3^{-\left(\frac{m-n}{4}\right)}+ 3^{-\frac n4}  \right) \right).
\end{multline}
\end{lemma}
\begin{proof}
Write $u:=u(\cdot,\cu,q)$ and $u_{\mathrm{loc}}^{(n)}:=u_{\mathrm{loc}}^{(n)}(\cdot,\cu,q)$ for short. We recall the definition of the event $B_n(\cu)$:
\begin{equation*}
B_{n}(\cu) := \left\lbrace  \sup_{ x \in \cu } \size( \cu_\Pa(x) ) \leq 3^{n} \right\rbrace \bigcap \left\lbrace \forall x \in \cu^{\left( n \right)},~   \cu_\Pa(x) = \cu_{\Pa_{\mathrm{loc}}^{(n)}}(x)  \right\rbrace.
\end{equation*}
We also record that for some $C := C(d,\p,\lambda) < + \infty$,
\begin{equation*} \label{}
\indc_{\Omega\setminus B_{n}(\cu)} \leq \O_{1}\left(C 3^{- n}\right).
\end{equation*}
We split the left side of~\eqref{e.ulocloccoarse}:
\begin{multline}
\label{e.techsplitting}
\fint_{\cu^{(n)}} \left| \nabla \left[ u \right]_{\Pa} - \nabla \left[ u_{\mathrm{loc}}^{(n)}\right]_{\Pa_{\mathrm{loc}}^{(n)}(\cu)} \right|(x)\,dx \\
=   \left( \fint_{\cu^{(n)}} \left| \nabla \left[ u \right]_{\Pa} - \nabla \left[ u_{\mathrm{loc}}^{(n)}\right]_{\Pa} \right|(x)\,dx \right) \cdot \indc_{B_n(\cu)} \\
+ \left(\fint_{\cu^{(n)}} \left| \nabla \left[ u\right]_{\Pa}  -   \nabla \left[ u_{\mathrm{loc}}^{(n)}\right]_{\Pa_{\mathrm{loc}}^{(n)}(\cu)} \right|(x)\,dx \right) \cdot \indc_{\Omega \setminus B_n(\cu)}.
\end{multline}
We estimate the first term on the right side of~\eqref{e.techsplitting} using~\eqref{e.coarsegrads} and~\eqref{e.ulocloc}: for some $s(d)>0$, we have,
\begin{align*}
\lefteqn{
\fint_{\cu^{(n)}} \left| \nabla \left[ u\right]_{\Pa} - \nabla \left[ u_{\mathrm{loc}}^{(n)}\right]_{\Pa} \right|(x)\,dx \cdot \indc_{B_n(\cu)}
} \quad & \\ \notag
& \leq  \frac{C \indc_{B_n(\cu)}}{\left|\cu^{(n)}\right|}\sum_{\Pa\ni \cu' \subseteq\cu^{(n)}} \size(\cu')^{d-1}  \int_{\cu'\cap \C_*\left(\cu^{(n)}\right)} \left| \left(\nabla u - \nabla u_{\mathrm{loc}}^{(n)} \right) \indc_{\{\a\neq 0\}} \right|(x)\,dx \\ \notag 
& \leq \frac{C \indc_{B_n(\cu)}}{\left|\cu^{(n)}\right|} \left(\sum_{\Pa\ni \cu' \subseteq\cu^{(n)}} \size(\cu')^{3d-2}  \right)^{\frac12}  \left( \sum_{\Pa\ni \cu' \subseteq\cu^{(n)}} \int_{\cu'\cap \C_*\left(\cu^{(n)}\right)} \left|\left(\nabla u - \nabla u_{\mathrm{loc}}^{(n)} \right) \indc_{\{\a\neq 0\}} \right|^2(x)\,dx \right)^{\frac12}  
\\ \notag
& = C \indc_{B_n(\cu)} \left(\frac{1}{\left|\cu^{(n)}\right|} \sum_{x\in\cu^{(n)}} \size(\cu_\Pa(x))^{2d-2}  \right)^{\frac12}  \left( \frac{1}{\left|\cu^{(n)}\right|} \int_{\C_*\left(\cu^{(n)}\right)} \left| \left(\nabla u - \nabla u_{\mathrm{loc}}^{(n)} \right) \indc_{\{\a\neq 0\}} \right|^2(x)\,dx \right)^{\frac12}   
\\ \notag
& \leq C \indc_{B_n(\cu)} \left(\frac{1}{\left|\cu^{(n)}\right|} \sum_{x\in\cu^{(n)}} \size(\cu_\Pa(x))^{2d-2}  \right)^{\frac12}  \left( \frac{1}{|\cu|} \int_{\C_*\left(\cu\right)} \left| \left(\nabla u - \nabla u_{\mathrm{loc}}^{(n)} \right) \indc_{\{\a\neq 0\}} \right|^2(x)\,dx \right)^{\frac12}   
\\ \notag
& \leq C \O_{\frac{2}{2d-1}}\left(C \right) \O_{\frac2{2d-1}}\left(C|q| \left( 3^{-\left(\frac{m-n}{4}\right)}+ 3^{-\frac n2}  \right) \right)
\\ \notag
& \leq \O_{s}\left(C|q| \left( 3^{-\left(\frac{m-n}{4}\right)}+ 3^{-\frac n2}  \right) \right),
\end{align*}
where we used that, for $m> n$, we have $|\cu| \leq 3^d \left|\cu^{(n)}\right|$, $\indc_{B_n(\cu)} \leq \indc_{\left\{ \C_*(\cu) = \C^{(n)}_*(\cu) \right\}}$ and $\C_*\left(\cu^{(n)} \right) \subseteq \C_*(\cu)$.

\smallskip

To estimate the second term on the right side of~\eqref{e.techsplitting}, notice that if $\cl_\Pa(\cu) \neq \cu$ then $\nabla \left[ u\right]_{\Pa}$ is constant. This remark and~\eqref{e.coarsegradu} yields
\begin{multline*}
\fint_{\cu^{(n)}} \left| \nabla \left[ u\right]_{\Pa} \right| (x) \, dx \leq C \fint_\cu \left| \nabla \left[ u\right]_{\Pa} \right| (x) \, dx \\ 
\leq C  |q| \fint_\cu \size(\cu_\Pa(x))^{2d-2} \, dx \leq \O_{\frac1{2d-2}}(C|q|).
\end{multline*}
A similar computation gives the same result with $\left[ u_{\mathrm{loc}}^{(n)}\right]_{\Pa_{\mathrm{loc}}^{(n)}}$ instead of $\left[ u\right]_{\Pa}$. Thus we have, for some (possibly smaller) exponent $s(d)>0$,
\begin{align*}
\lefteqn{
 \fint_{\cu^{(n)}} \left| \nabla \left[ u \right]_{\Pa}  -   \nabla \left[ u_{\mathrm{loc}}^{(n)}\right]_{\Pa_{\mathrm{loc}}^{(n)}(\cu)} \right|(x)\,dx \cdot \indc_{\Omega \setminus B_n(\cu)}
} \quad & \\ \notag
& \leq  \left( \fint_{\cu^{(n)}} \left| \nabla \left[ u \right]_{\Pa} \right|(x) +  \left| \nabla \left[ u_{\mathrm{loc}}^{(n)}\right]_{\Pa_{\mathrm{loc}}^{(n)}(\cu)} \right|(x)\,dx \right) \cdot \indc_{\Omega \setminus B_n(\cu)} \\
& \leq \O_{\frac1{2d-2}} (C |q|) \O_1(C3^{-n}) \\
& \leq \O_{s} \left( C|q|3^{-n} \right). 
\end{align*}
This completes the proof of~\eqref{e.ulocloccoarse}. 
\end{proof}

\section{Convergence of the subadditive quantities}
\label{s.convergence}

An immediate consequence of the approximate subadditivity~\eqref{e.nusubadd} and stationarity~\eqref{e.munustat},~\eqref{e.localization} is the approximate monotonicity of $n\mapsto \E\left[ \nu(\cu_n,p) \right]$: we have that, for every $p\in\Rd$ and $m,n\in\N$ with $n \leq m$,
\begin{equation*} \label{}
\E \left[ \nu (\cu_m,p) \right] \leq \E \left[ \nu(\cu_n,p) \right] + C|p|^2 3^{-n}.
\end{equation*}
It follows that, for each $p\in\Rd$, 
\begin{equation*} \label{}
\lim_{n\to \infty} \E \left[ \nu (\cu_m,p) \right] \quad \mbox{exists.}
\end{equation*}
Since $p\mapsto \E \left[ \nu (\cu_m,p) \right]$ is a quadratic form which, for sufficiently large $n$, is bounded above and below by multiples of $|p|^2$ by~\eqref{e.nuupbound} and~\eqref{e.nulobound}, the same is true of $\lim_{n\to \infty} \E \left[ \nu (\cu_m,p) \right]$. This allows us to make the following definition.

\begin{definition}[{Homogenized diffusion matrix $\ahom$}]
We define $\ahom$ to be the unique (deterministic) positive matrix~$\ahom$
\begin{equation*} \label{}
\frac12 p\cdot \ahom p = \lim_{n\to \infty} \E \left[ \nu (\cu_m,p) \right].
\end{equation*} 
\end{definition}

By the bounds~\eqref{e.nuupbound} and~\eqref{e.nulobound} on $\nu$, there exist $0<c(d,\lambda,\p) \leq C(d,\lambda)<\infty$ such that 
\begin{equation}
\label{e.ahombounds}
cI_d \leq \ahom \leq CI_d. 
\end{equation}
Arguing in a similar manner, we can show that $\E\left[-\mu(\cu_n,q)\right]$ also has a limit as $n\to \infty$ to a quadratic function in~$q$ which is bounded above and below by~$|q|^2$. As we will prove in this section, that quadratic form turns out to be $q\mapsto \frac12 q\cdot \ahom^{-1}q$, the convex dual of the quadratic form $p\mapsto \frac12 p\cdot \ahom p$. Moreover, by the approximate localization property~\eqref{e.localization}, one can argue as in the proof of the subadditive ergodic theorem that these quantities converge~$\P$--a.s.~to these deterministic constants.

\smallskip

The main result of this section is a quantitative rate of convergence for the subadditive quantities to their limits, which is summarized in the following proposition.

\begin{proposition}
\label{p.subadd}
There exist~$s(d)>0$, $\alpha(d,\p,\lambda)\in \left(0,\tfrac14\right]$ and $C(d,\p,\lambda)<\infty$ such that, for every $\cu\in\T$ and $p,q\in\Rd$, 
\begin{equation}
\label{e.convmu}
\left|\frac12 q\cdot \ahom^{-1}q + \mu(\cu,q) \right| 
\leq \O_{s} \left( C|q|^2\size(\cu)^{-\alpha} \right)
\end{equation}
and
\begin{equation} 
\label{e.convnu}
\left|\frac12 p\cdot \ahom p - \nu(\cu,p) \right| 
\leq \O_{s} \left( C|p|^2\size(\cu)^{-\alpha} \right). 
\end{equation}
\end{proposition}

The proof of Proposition~\ref{p.subadd} is an adaptation of arguments in~\cite{AS}. The main step is to control the \emph{expectations} of the quantities under the absolute value signs in~\eqref{e.convmu} and~\eqref{e.convnu}. That is, we want to show that there exists an exponent $\alpha(d,\p,\lambda)>0$ such that, for every $n\in\N$ and $p,q\in\Rd$,
\begin{equation} 
\label{e.convexp}
\left| \E\left[ \mu(\cu_n,q)  \right] + \frac12 q\cdot \ahom^{-1}q \right| 
+ \left| \E\left[  \nu(\cu_n,p)  \right]  - \frac12 p\cdot \ahom p \right| 
\leq C \left(|p|^2+ |q|^2\right) 3^{-n\alpha}.
\end{equation}
Once this is accomplished, we obtain the conclusion of Proposition~\ref{p.subadd} by gaining stochastic integrability via a straightforward use of subadditivity and independence. 

\smallskip

It may appear from~\eqref{e.convexp} that we have two estimates to prove, but one of the insights from~\cite{AS} is that it is really just one estimate. Indeed, let us consider the quantity $\omega(\cu,q)$, defined for each $\cu\in\T$ and $q\in\Rd$ by
\begin{equation*} \label{}
\omega(\cu,q):=  \nu\left( \cu ,\ahom^{-1}q\right) -  \mu(\cu,q) - q\cdot \ahom^{-1}q .
\end{equation*}
In order to prove~\eqref{e.convexp}, it is enough to show that, for every $n\in\N$ and $q\in\Rd$, 
\begin{equation} 
\label{e.omegagoal}
\E \left[ \omega(\cu_n,q)_+ \right] \leq C|q|^2 3^{-n\alpha}. 
\end{equation}
Since this fact motivates the rest of the analysis in this section, we pause now to prove it. 

\begin{lemma}
\label{l.minimalset}
There exists $C(d,\lambda,\p)<\infty$ such that
\begin{multline}
\label{e.minimalset}
\sup_{q\in \Rd} \frac{1}{|q|^2} \left| \mu(\cu,q) + \frac12q\cdot \ahom^{-1} q \right| + \sup_{p\in\Rd} \frac{1}{|p|^2} \left| \nu(\cu,p) - \frac12p\cdot \ahom p\right| \\
 \leq C \sup_{e \in \partial B_1} \omega(\cu,e)_+^{\frac12} + \O_{\frac2{2d+1}}\left( C\size^{-\frac14}(\cu) \right). 
\end{multline}
\end{lemma}
\begin{proof}
According to Lemma~\ref{l.lowerbounds}, there exists $C<\infty$ such that 
\begin{equation} 
\label{e.kcuest}
k(\cu) := \sup_{q\in \Rd} \left(  \frac{1}{C} + \frac1{|q|^2} \mu(\cu,q) \right) 
\leq  \O_{\frac2{2d-1}}\left( C\size^{-\frac12}(\cu) \right).
\end{equation}
Fix $p\in\Rd$. Define the function
\begin{equation*} \label{}
f(q):= \nu(\cu,p) - \mu(\cu,q) - q\cdot p + k(\cu)|q|^2,\quad q\in\Rd. 
\end{equation*}
Observe that $f$ is a quadratic function and 
\begin{equation*} \label{}
f(q) \geq \nu(\cu,p) + \frac1{C}|q|^2 -q\cdot p.
\end{equation*}
It follows that $f$ is uniformly convex. Thus there exists a unique point $q_0\in\Rd$ at which $f$ attains its minimum. From the inequality $f(q_0) \leq f(0)$ we see that $|q_0| \leq C|p|$, and we have
\begin{equation*} \label{}
p = - \nabla \mu(\cu,\cdot)(q_0) + 2k(\cu)q_0.
\end{equation*}
In particular, 
\begin{equation} \label{e.snapghtrgt}
\left| p +  \nabla \mu(\cu,\cdot)(q_0) \right| \leq C|q_0|k(\cu) \leq  \O_{\frac2{2d-1}}\left( C|p|\size^{-\frac12}(\cu) \right).
\end{equation}
By the uniform convexity of $f$ and the fact it achieves its minimum at~$q$,
\begin{equation*} \label{}
f(\ahom p) \geq f(q_0) + \frac1C\left| \ahom p - q_0 \right|^2.
\end{equation*}
By~\eqref{e.munucomparison} and~\eqref{e.kcuest}, 
\begin{equation*} \label{}
f(q_0) \geq 
k(\cu)|q_0|^2 - \O_{\frac{2}{2d-1}}\left(C|p||q_0| \size(\cu)^{-\frac12} \right) 
\geq - \O_{\frac2{2d-1}}\left( C|p|^2\size^{-\frac12}(\cu) \right).
\end{equation*}
Thus we obtain
\begin{align*}
\left| \ahom p - q_0 \right|^2
& \leq C \left( f(\ahom p) - f(q_0) \right) \\
&  \leq Cf(\ahom p) + \O_{\frac2{2d-1}}\left( C|p|^2\size^{-\frac12}(\cu) \right) \\
& \leq  C \omega(\cu,\ahom p) + \O_{\frac2{2d-1}}\left( C|p|^2\size^{-\frac12}(\cu) \right).
\end{align*}
Since $-\mu(\cu,\cdot)$ is quadratic, for every $q\in\Rd$, 
\begin{equation*} \label{}
-\mu(\cu,q) = -\frac12 q \cdot \nabla \mu(\cu,\cdot)(q).
\end{equation*}
The bound~\eqref{e.muupboundwhip} gives us that, for every $q_1,q_2\in\Rd$, 
\begin{equation*} \label{}
\left| \nabla \mu(\cu,\cdot)(q_1) - \nabla \mu(\cu,\cdot)(q_2) \right| \\
\leq  C |q_1-q_2| \left( 1 +  \O_{\frac{2}{2d+1}} \left(C\size^{-\frac12}(\cu)\right) \right).
\end{equation*}
We deduce from the previous three displays, $|\ahom p - q_0| \leq C|p|$ and~\eqref{e.snapghtrgt} that
\begin{align*} \label{}
\lefteqn{
\left| -\mu(\cu,\ahom p) -\frac12 p\cdot \ahom p  \right| 
} \qquad & \\
& = \left| -\frac12 \ahom p \cdot \nabla \mu(\cu,\cdot)(\ahom p) -\frac12 p\cdot \ahom p \right| \\
& \leq \left| \ahom p \right| \left|  \nabla \mu(\cu,\cdot)(\ahom p) + p \right| \\
& \leq C |p| \left( \left|  \nabla \mu(\cu,\cdot)(\ahom p) - \nabla \mu(\cu,\cdot)(q_0)  \right| + \O_{\frac2{2d-1}}\left( C|p|\size^{-\frac12}(\cu) \right) \right)\\
& \leq C |p| \left( \left| \ahom p - q_0  \right| + \O_{\frac2{2d+1}}\left( C|p|\size^{-\frac12}(\cu) \right) \right)\\
& \leq C |p| \left( \omega(\cu,\ahom p) + \O_{\frac2{2d-1}}\left( C|p|^2\size^{-\frac12}(\cu) \right) \right)^{\frac12} + \O_{\frac2{2d+1}}\left( C|p|^2\size^{-\frac12}(\cu) \right).
\end{align*}
From this and~\eqref{e.ahombounds}, we deduce that 
\begin{equation*} \label{}
\left| -\mu(\cu, q) -\frac12 q\cdot \ahom^{-1}q  \right| 
\leq C|q|^2 \sup_{e\in\partial B_1} \omega(\cu,e)_+^{\frac12}  +  \O_{\frac2{2d+1}}\left( C|q|^2\size^{-\frac14}(\cu) \right).
\end{equation*}
Using the previous inequality and           
\begin{equation*}
\left| f(\ahom p) - \omega( \cu,\ahom p) \right| = k(\cu)|\ahom p|^2 \leq \O_{\frac2{2d-1}}\left( C|p|^2\size^{-\frac12} (\cu) \right)
\end{equation*}
we also obtain
\begin{equation*}
\left| \nu(\cu,p) -\frac12 p\cdot \ahom p \right| \leq C|p|^2 \sup_{e\in\partial B_1} \omega(\cu,e)_+^{\frac12}  +  \O_{\frac2{2d+1}}\left( C|p|^2\size^{-\frac14}(\cu) \right).
\end{equation*}
This completes the proof. 
\end{proof}

\smallskip

In view of the previous lemma, we are motivated to prove the bound~\eqref{e.omegagoal}, which would follow if we can show that the expectation of~$\omega(\cu_n,q)_+$ contracts by a factor~$\theta(d,\p,\lambda)<1$ as we pass from scale~$n$ to scale~$n+1$ (so that an iteration produces the desired estimate). It is therefore natural to work with the change in the expectation of $\omega$ between triadic scales~$n$ and~$n+1$. In fact, it is convenient to use the slightly different quantity
\begin{multline*} \label{}
\tau_n:= 
 \sum_{i=1}^d \left(  \E \left[ \mu(\cu_{n+1},\e_i)\right] - \E\left[ \mu(\cu_{n},\e_i)\right] \right)_+ 
 \\
 + \sum_{i=1}^d \left( \E \left[ \nu(\cu_{n},\ahom^{-1}\e_i ) \right] - \E\left[ \nu(\cu_{n+1},\ahom^{-1}\e_i ) \right] \right)_+. 
\end{multline*}
Recall that, by~\eqref{e.almostmonotone} and~\eqref{e.almostmonotonenu}, 
\begin{equation*} \label{}
\left\{ 
\begin{aligned}
&
\E \left[ \mu(\cu_{n+1},q)\right] - \E\left[ \mu(\cu_{n},q)\right] + C|q|^23^{-\frac n4} \geq 0, \ \mbox{and} \\
& \E \left[ \nu(\cu_{n},\ahom^{-1}p ) \right] - \E\left[ \nu(\cu_{n+1},\ahom^{-1}p ) \right] + C|p|^23^{-n} \geq 0. 
\end{aligned}
\right.
\end{equation*}
Since these are quadratic functions of $q$ and $p$, respectively, it follows that they are convex and therefore sums control supremums:
\begin{multline*} \label{}
\sup_{q\in\Rd} \frac{1}{|q|^2} \left( \E \left[ \mu(\cu_{n+1},q)\right] - \E\left[ \mu(\cu_{n},q)\right] + C|q|^23^{-\frac n4} \right) \\
\leq \sum_{i=1}^d \left(\E \left[ \mu(\cu_{n+1},\e_i)\right] - \E\left[ \mu(\cu_{n},\e_i)\right] + C3^{-\frac n4}\right)
\end{multline*}
with a similar inequality for $\nu$. 
Using this observation, we deduce that, for every $n\in\N$ and $p,q\in\Rd$,
\begin{multline} \label{e.suptosum}
\left( \E \left[ \mu(\cu_{n+1},q)\right] - \E \left[ \mu(\cu_{n},q) \right] \right)_+ + \left( \E \left[ \nu(\cu_{n},p )\right] -\E\left[ \nu(\cu_{n+1},p ) \right]\right)_+ \\
\leq C(|p|^2+|q|^2)\left( \tau_n + C3^{-\frac n4} \right).
\end{multline}
Since $\omega(\cu,\cdot)$ is almost nonnegative and the quantities $\omega(\cdot,q)$, $-\mu(\cdot,q)$ and $\nu(\cdot,\ahom^{-1}q)$ are almost subadditive, and therefore their expectations are (almost) monotone, $\tau_n$ is essentially the same (up to negligible errors) as
\begin{equation*} \label{}
\sum_{i=1}^d \left(\E \left[  \omega(\cu_{n},\e_i) \right] - \E\left[ \omega(\cu_{n+1},\e_i) \right] \right). 
\end{equation*}
Thus to prove the inequality $\E \left[ \omega(\cu_{n+1},e)\right] \leq \theta  \E\left[ \omega(\cu_n,e) \right]$ for $\theta<1$, it suffices to show that $\E \left[ \omega(\cu_n,e) \right] \leq C \tau_n$ for some $C<\infty$. We do not prove exactly this, but something close enough (see the statement of Lemma~\ref{l.omegadecay}, below) which can still be iterated to obtain~\eqref{e.omegagoal}; see Lemma~\ref{l.omegaiteration}, below. 

\smallskip

\smallskip

The proof of Proposition~\ref{p.subadd} begins with the simple observation that, by quadratic response, the expected difference in the gradients of $u(\cdot,\cu,q)$ at two successive triadic scales is controlled by $\tau_n$. This will aid us by localizing the functions $u(\cdot,\cu,q)$.  In the uniform elliptic setting, this argument is two lines (cf.~\cite[(2.25)]{AS}). In our situation, the idea is the same but the statement is necessarily weaker and the proof is more technical due to the discreteness and the non-uniformity of the geometry of the clusters. 


\begin{lemma}
\label{l.restrictquadmicro}
There exists $C(d,\lambda,\p)<\infty$  such that, for every $m,n\in\N$ with $n\in \left[ \frac12 m  , m \right)$, $\cu\in\T_m$ and $p,q\in\Rd$, 
\begin{multline}
\label{e.restrictquadmicro}
\E \left[  \frac{1}{|\cu|} \sum_{z\in3^n\Zd\cap \cu} \int_{\C_*(z+\cu_n)} \left| \nabla\left(  u(\cdot,\cu,q) - u(\cdot,z+\cu_n,q) \right) \indc_{\{\a\neq 0\}} \right|^2(x)\,dx \right]  \\
+ \E \left[  \frac{1}{|\cu|} \sum_{z\in3^n\Zd\cap \cu} \int_{\C_*(z+\cu_n)} \left| \nabla\left(  v(\cdot,\cu,p) - v(\cdot,z+\cu_n,p) \right) \indc_{\{\a\neq 0\}} \right|^2(x)\,dx \right] \\
\leq C\left(|p|^2+|q|^2 \right) \left(  \sum_{k=n}^{m-1} \tau_k + 3^{-\frac n4} \right).
\end{multline}
\end{lemma}
\begin{proof}
For convenience, we write $u:= u(\cdot,\cu,q)$ and $u_z := u(\cdot,z+\cu_n,q)$ for $z\in 3^n\Zd \cap \cu$. The second variation~\eqref{e.secondvarmu} gives, for every $z\in3^n\Zd\cap \cu$, 
\begin{multline*} \label{}
\frac1{|\cu_n|} \int_{\C_*(z+\cu_n)} \left|\nabla (u-u_z) \indc_{\{\a\neq 0\}} \right|^2(x)\,dx \, \indc_{\{ z +\cu_n \in \Pas \}}  \\
\leq C \left( \frac1{|\cu_n|} \left( \frac12  \left\langle  \nabla u, \a\nabla u\right\rangle_{\C_* (z+\cu_n)} - \left\langle q, \nabla\! \left[ u \right]_{\Pa}\! \right\rangle_{z+\cu_n} \right) - \mu(z+\cu_n,q)\right) \indc_{\{ z +\cu_n \in \Pas \}}.
\end{multline*}
Summing this inequality over $z\in 3^n\Zd \cap \cu$  yields
\begin{align*}
\lefteqn{
\frac{\left| \cu_n \right|}{\left|\cu\right|} \sum_{z\in3^n\Zd\cap \cu} \frac{1}{|\cu_{n}|} \int_{\C_*(z+\cu_n)} \left|\nabla (u-u_z) \indc_{\{\a\neq 0\}} \right|^2(x)\,dx \,\indc_{\{ z +\cu_n \in \Pas \}}
} \qquad & \\ \notag
& \leq C \;\Bigg( \frac{1}{\left|\cu\right|}\sum_{z\in 3^n\Zd \cap \cu}  
\left( \frac12  \left\langle  \nabla u, \a\nabla u\right\rangle_{\C_* (z+\cu_n)} - \left\langle q, \nabla\! \left[ u \right]_{\Pa} \right\rangle_{z+\cu_n} \right)
\indc_{\{ z +\cu_n \in \Pas \}} \\ \notag
& \qquad \qquad\qquad\qquad - \frac{1}{\left| 3^n\Zd \cap \cu\right|} \sum_{z\in3^n\Zd \cap \cu}\mu(z+ \cu_{n},q)\,\indc_{\{ z +\cu_n \in \Pas \}} \Bigg).
\end{align*}
Next, we notice that if $z+\cu_n\not\in\Pas$, then $\cl_{\Pa}(z+\cu_n)$ is an element of~$\Pa$, thus all coarsened functions are constant on~$\cl_{\Pa}(z+\cu_n)$ and so $u_z\equiv 0$ and $\mu(z+\cu_n,q) =0$. Thus we may remove the indicator function in the last line of the previous display. 
Combining the above and using~\eqref{e.testingerrormu} gives
\begin{multline}
\label{e.almosdone}
\frac{1}{\left|\cu\right|} \sum_{z\in3^n\Zd\cap \cu} \int_{\C_*(z+\cu_n)} \left|\nabla (u-u_z) \indc_{\{\a\neq 0\}} \right|^2(x)\,dx\,\indc_{\{ z +\cu_n \in \Pas \}} \\
\leq C\left( \mu(\cu,q) - 3^{-d(m-n)} \sum_{z\in3^n\Zd \cap \cu} \mu(z+ \cu_{n},q) \right)+ \O_{\frac{2}{4d-3}}\left( C|q|^23^{-\frac n2} \right). 
\end{multline}
Let $\Gamma$ denote the event 
\begin{equation*} \label{}
\Gamma:= \left\{ \exists z\in 3^n\Zd\cap \cu \ \mbox{such that} \ z+\cu_n \not\in \Pas \right\}. 
\end{equation*}
Observe that 
\begin{equation*} \label{}
\P \left[ \Gamma \right] \leq \sum_{z\in 3^n\Zd\cap\cu} \P \left[ z+\cu_n\not\in\Pas \right] \leq C 3^{d(m-n)} \exp\left( -c3^n \right) \leq C \exp\left( -c3^{\frac m2} \right).
\end{equation*}
Thus $\indc_{\Gamma} \leq O_1\left( C3^{-\frac m2} \right)$. Using this bound,~\eqref{e.Oprods} and~\eqref{e.muupbound} (and again the fact that $u_z = 0$ if $z+\cu_n\not \in \Pas$), we obtain
\begin{align*}
\lefteqn{
\frac1{|\cu|} \sum_{z\in 3^n\Zd} \int_{\C_*(z+\cu_n)} \left|\nabla (u-u_z) \indc_{\{\a\neq 0\}} \right|^2(x)\,dx \, \indc_{\{ z +\cu_n \not\in \Pas \}}  
} \qquad & \\
& = \frac1{|\cu|} \sum_{z\in 3^n\Zd} \int_{\C_*(z+\cu_n)} \left|\nabla u \indc_{\{\a\neq 0\}} \right|^2(x)\,dx \, \indc_{\{ z +\cu_n \not\in \Pas \}}  \\
& \leq \frac1{|\cu|} \int_{\C_*(z+\cu_n)} \left|\nabla u \indc_{\{\a\neq 0\}} \right|^2(x)\,dx \, \indc_{\Gamma}  \\
& \leq \O_{\frac1{2d-2}}\left( C|q|^2 \right)\cdot \O_1\left( C3^{-\frac m2} \right) \\
& \leq \O_{\frac1{2d-1}}\left( C3^{-\frac m2} \right).
\end{align*}
Combining this with~\eqref{e.almosdone}, taking the expectation of the result and applying~\eqref{e.almostmonotone} yields the estimate for the first term on the left side of~\eqref{e.restrictquadmicro}. The estimate of the second term is similar and we omit the details, except for the remark that~\eqref{e.testingerrornu} should be used in place of~\eqref{e.testingerrormu}. 
\end{proof}

We next obtain a version of the previous lemma for the spatial averages of the gradients of the coarsened functions $[u]_\Pa$. This is tricky and somewhat technical, because in passing from $u$ to $[u]_\Pa$, using~\eqref{e.coarsegrads}, we make errors depending on the coarseness of the partition~$\Pa$. If the energy density $|\nabla u|^2$ happens to be concentrated in the very largest cubes of $\Pa$, then this does not give us a good enough estimate. We deal with this issue by using the Meyers estimate, Proposition~\ref{p.meyers}, which allows us to ``H\"older away" the factors representing the coarseness of~$\Pa$ on the right of~\eqref{e.coarsegrads}.

\begin{lemma}
\label{l.restrictquadcoarsen}
There exists $C(d,\lambda,\p)<\infty$  such that, for every $m,n\in\N$ with $n\in \left[ \left( \frac{4d+1}{4d+2}\right) m  , m \right)$, $\cu\in\T_m$ and $p,q\in\Rd$,
\begin{multline}
\label{e.restrictquadE}
\E \left[ \frac{3^{-d(m-n)}}{|\cu_n|^2}\sum_{z\in 3^n\Zd\cap \cu_m }
\left|
\left\langle \left[ \nabla u(\cdot,\cu,q)\right]_{\Pa}\right\rangle_{z+\cu_{n}} -  \left\langle \left[ \nabla u(\cdot,z+\cu_{n},q)\right]_{\Pa} \right\rangle_{z+\cu_{n}} 
\right|^2
 \right]   \\
+ \E\left[ \frac{3^{-d(m-n)}}{|\cu_n|^2} \sum_{z\in 3^n\Zd\cap \cu} 
\left| \left\langle \left[ \nabla v(\cdot,\cu,p)\right]_{\Pa}\right\rangle_{z+\cu_{n}} -  \left\langle \left[ \nabla v(\cdot,z+\cu_{n},p)\right]_{\Pa} \right\rangle_{z+\cu_{n}} \right|^2 
 \right] \\
 \leq 
 C \left(|p|^2+|q|^2 \right) \left(  \sum_{k=n}^{m+1} \tau_k  + 3^{-\frac n4} \right).
\end{multline}
\end{lemma}
\begin{proof}
The proof of the estimate for the second term on the left of~\eqref{e.restrictquadE} will be omitted, since it follows from a very similar argument as for the estimate of the first term. For convenience and simplicity, we give the proof only in the case that $\cu=\cu_m$. Throughout, we work with the partition~$\Qa$ defined in Definition~\ref{d.partQa} which gives us good cubes for the Meyers estimate. 
We estimate the left side of~\eqref{e.restrictquadE} by the left side of~\eqref{e.restrictquadmicro}. For this we use~\eqref{e.coarsegrads}, the H\"older inequality and the Meyers estimate.

\smallskip

We fix $q\in\Rd$ and denote, for $z\in 3^n\Zd\cap \cu$, the functions $u:= u(\cdot,\cu_{m},q)$ and $u_z := u(\cdot,z+\cu_n,q)$.

\smallskip

\emph{Step 1.} We reduce to a ``good" event in which every element of the partition~$\Qa$ in $\cu_{2n} \supseteq\cu_{m}$ is not too large. Denote this event by
\begin{equation*} \label{}
\Gamma:= \left\{ 
\max_{x\in \cu_{2n}} \size( \cu_\Qa(x)) \leq 3^{\frac n2-1}
\right\}.
\end{equation*}
By~\eqref{e.Mtint2}, there exists an exponent $s(d,\lambda,\p)<\infty$ such that 
\begin{equation*} \label{}
\indc_{\Omega \setminus \Gamma} \leq 3^{-\frac{n}2+1}  \sup_{x\in \cu_{2n}} \size(\cu_\Qa(x)) \leq 3^{-\frac n2+1}  \O_s\left( C3^{\frac n8} \right) \leq  \O_s\left( C3^{-\frac{3n}{8}} \right).
\end{equation*}
Using~\eqref{e.coarsegraduL2} with $\delta = \frac18$, the fact that $u_z=0$ (resp. $u=0$) if $z + \cu_n \notin \Pa_*$ (resp. $\cu_m \notin \Pa_*$) and the H\"older inequality, we obtain
\begin{align*}
\lefteqn{
\frac{3^{-d(m-n)}}{|\cu_n|^2} \sum_{z\in 3^n\Zd\cap \cu_m}\left|
\left\langle \nabla \left[ u \right]_{\Pa}\right\rangle_{z+\cu_{n}} -  \left\langle \nabla \left[u_z\right]_{\Pa} \right\rangle_{z+\cu_{n}} 
\right|^2 \indc_{\Omega\setminus \Gamma} 
} \quad & \\ \notag &
\leq C 3^{-d(m-n)}\sum_{z\in 3^n\Zd\cap \cu_m} \left( \fint_{z+\cu_n} \left( \left| \nabla \left[ u \right]_{\Pa}\right|^2\!(x) + 
 \left|  \nabla \left[u_z\right]_{\Pa} \right|^2\!(x) \right) \,dx \right) \, \indc_{\Omega\setminus \Gamma}
\\ \notag &
\leq C 
\left( 
\fint_{\cu_{m}} \left| \nabla \left[ u\right]_{\Pa}\right|^2\!(x)\,dx 
+
3^{-d(m-n)} \sum_{z\in 3^n\Zd\cap \cu_m}\fint_{z+\cu_n}
 \left|  \nabla \left[u_z\right]_{\Pa} \right|^2\!(x) \,dx 
 \right)
  \, \indc_{\Omega\setminus \Gamma} 
 \\ \notag &
\leq \O_{\frac{1}{4d-3}}\left( C|q|^23^{\frac n8} \right) \cdot \O_s\left( C3^{-\frac{3n}{8}} \right).
\end{align*}
We deduce that  
\begin{equation}
\label{e.estoffGamma}
\E \left[ \frac{3^{-d(m-n)}}{|\cu_n|^2} \sum_{z\in 3^n\Zd\cap \cu_m}
\left|
\left\langle \nabla \left[ u \right]_{\Pa}\right\rangle_{z+\cu_{n}} -  \left\langle \nabla \left[u_z\right]_{\Pa} \right\rangle_{z+\cu_{n}} 
\right|^2\, \indc_{\Omega\setminus \Gamma}
\right] 
\leq 
C|q|^23^{-\frac n4}.
\end{equation}

\smallskip

\emph{Step 2.} We further prepare for the use of the Meyers estimate application by removing a boundary layer around each of the subcubes~$z+\cu_n$. This is necessary because Proposition~\ref{p.meyers} is only an interior estimate and we do not have a good boundary condition anyway for minimizers of $\mu$. Let $K_z$ denote the union of the elements of $\T_{\left\lceil n/2\right\rceil}$ which are subsets of~$z+\cu_n$ and intersect the boundary of $z+\cu_n$:
\begin{equation*} \label{}
K_z:= 
\bigcup \,
\left\{ \cu'\in \T\,:\, 
\size(\cu') = 3^{\left\lceil n/2\right\rceil},\ \cu' \subset z+\cu_n, \ \dist(\cu',\partial (z+\cu_n) ) = 0
\right\}.
\end{equation*}
Let $E_z$ denote the edges $(x,y)$ such that $x$ or $y$ belongs to $K_z$. Then by the H\"older inequality, the triangle inequality, and the fact that $|K_z| \leq C3^{-\frac n2} |\cu_n|$, we find that, for each $z\in 3^n\Zd\cap \cu_m$, 
\begin{align*}
\E \left[ \frac{1}{\left| \cu_n \right|^2} \left|\sum_{e\in E_z} 
\nabla\! \left[ u - u_z \right]_{\mathcal{P}}\!(e) 
 \right|^2\right]
& \leq C \E \left[ \frac{|K_z|}{\left| \cu_n \right|^2}\int_{K_z} \left( \left|\nabla\! \left[ u  \right]_{\mathcal{P}} \right|^2\!(x) +  \left|\nabla\! \left[ u_z \right]_{\mathcal{P}} \right|^2\!(x) \right)\,dx \right] 
  \\ \notag &
 \leq C3^{-\frac n2} \E \left[ \fint_{z+\cu_n} \left( \left|\nabla\! \left[ u  \right]_{\mathcal{P}} \right|^2\!(x) +  \left|\nabla\! \left[ u_z \right]_{\mathcal{P}} \right|^2\!(x) \right)\,dx \right].
\end{align*}
Summing over $z\in 3^n\Zd\cap \cu_m$ and using~\eqref{e.coarsegraduL2} with $\delta = \frac14$ gives 
\begin{equation}
\label{e.removeKz}
3^{-d(m-n)} \sum_{z\in3^n\Zd\cap \cu_m} \E \left[ \frac{1}{\left| \cu_n \right|^2} \left|\sum_{e\in E_z} 
\nabla\! \left[ u - u_z \right]_{\mathcal{P}}\!(e) 
 \right|^2\right]
 \leq
 C|q|^23^{-\frac n4}.
\end{equation}

\smallskip

\emph{Step 3.} We estimate the expected difference between $\nabla \left[ u \right]_\Pa$ and $\nabla \left[ u_z \right]_{\Pa}$ in the strong $L^2$ norm, after removing the $K_z$'s and in the case that the good event $\Gamma$ holds. The precise objective of this step is to prove~\eqref{e.step3goal}, below. 

\smallskip

We begin by using~\eqref{e.coarsegrads} to estimate, for each $z\in 3^n\Zd\cap \cu_m$, 
\begin{align*} \label{}
\lefteqn{
\int_{(z+\cu_n)\setminus K_z } \left| \nabla  \left[ u - u_z \right]_{\Pa} \right|^2(x) \,dx \, \indc_{\Gamma}
} \qquad &  \\
& \leq C  \sum_{\Pa\ni\cu'\subseteq (z+\cu_n)\setminus K_z} \size(\cu')^{2d-1} \int_{\C_*(\cu')} \left|\nabla (u-u_z) \indc_{\{\a\neq 0\}} \right|^2(x)\,dx \,  \indc_{\Gamma} \\
& \leq C  \sum_{\Qa\ni\cu'\subseteq (z+\cu_n)\setminus K_z} \size(\cu')^{2d-1} \int_{\C_*(\cu')} \left|\nabla (u-u_z) \indc_{\{\a\neq 0\}} \right|^2(x)\,dx \, \indc_{\Gamma}.
\end{align*}
Here we used that $\Qa$ is coarser than $\Pa$ and that no element of $\Qa$ in $z+\cu_n$ is larger than $3^{\left\lceil n/2\right\rceil}$ on the event~$\Gamma$. Applying the H\"older inequality to the previous sum yields, for every $s\in (2,\infty)$, 
\begin{align*}
\lefteqn{ 
\frac{1}{|\cu_n|} \int_{(z+\cu_n)\setminus K_z  } \left| \nabla  \left[ u - u_z \right]_{\Pa} \right|^2(x) \,dx\, \indc_{\Gamma}
} \qquad & \\
& \leq C \left(\frac1{|\cu_n|} \sum_{\Qa\ni\cu'\subseteq (z+\cu_n)\setminus K_z } \size(\cu')^{\frac{s(2d-1)}{s-2}}  \right)^{\frac{s-2}s}  \\
& \qquad \times
\left( \frac1{|\cu_n|}\sum_{\Qa\ni\cu'\subseteq (z+\cu_n)\setminus K_z } 
 \int_{\C_*(\cu')} \left|\nabla (u-u_z) \indc_{\{\a\neq 0\}} \right|^s(x)\,dx \right)^{\frac2s} \, \indc_{\Gamma}.
\end{align*}
Take $s = 2+\ep$ with $\ep(d,\lambda,\p)>0$ as in Proposition~\ref{p.meyers} and apply the proposition in each $\cu'\in\Qa$, $\cu' \subseteq (z+\cu_n)\setminus K_z$, which we note that on the event~$\Gamma$ implies $3\cu' \subseteq z+\cu_n$, to get an estimate of the second factor on the right side:
\begin{align*}
\lefteqn{
\left(\frac1{|\cu_n|} \sum_{\Qa\ni\cu'\subseteq (z+\cu_n)\setminus K_z} 
 \int_{\C_*(\cu')} \left|\nabla (u-u_z) \indc_{\{\a\neq 0\}} \right|^s(x)\,dx \right)^{\frac2s}\indc_{\Gamma}
} \quad & \\
& \leq \left(\frac1{|\cu_n|} \sum_{\Qa\ni\cu'\subseteq (z+\cu_n)\setminus K_z} \left| \cu' \right|
\left(  \frac1{\left| \cu' \right|} \int_{\C_*(3\cu')} \left|\nabla (u-u_z) \indc_{\{\a\neq 0\}} \right|^2(x)\,dx \right)^{\frac s2} \right)^{\frac2s} \, \indc_{\Gamma} \\
& \leq \frac1{|\cu_n|}\sum_{\Qa\ni\cu'\subseteq (z+\cu_n)\setminus K_z} \left| \cu' \right|^{\frac2s-1} \int_{\C_*(3\cu')} \left|\nabla (u-u_z) \indc_{\{\a\neq 0\}} \right|^2(x)\,dx \, \indc_{\Gamma} \\
& \leq \frac1{|\cu_n|}\int_{\C_*(z+\cu_n)} \left|\nabla (u-u_z) \indc_{\{\a\neq 0\}} \right|^2(x)\,dx 
\end{align*}
To get the last line, we used that every point of $\C_*(z+\cu_n)$ belongs to $\C_*(3\cu')$ for at most~$C$ elements $\cu'$ of the sum, since $\Qa$ satisfies property~(ii) of Proposition~\ref{p.partitions}, and $|\cu'|^{\frac2s-1} \leq 1$. Putting the above inequalities together, we get
\begin{multline}
\label{e.stoppiz}
\frac{1}{|\cu_n|} \int_{(z+\cu_n)\setminus K_z  } \left| \nabla  \left[ u - u_z \right]_{\Pa} \right|^2(x) \,dx\, \indc_{\Gamma}
  \\
\leq C\left(\frac 1{|\cu_n|} \sum_{\Qa\ni\cu'\subseteq z+\cu_n} \size(\cu')^{\frac{s(2d-1)}{s-2}}  \right)^{\frac{s-2}s} \indc_{\Gamma}
\cdot \frac1{|\cu_n|}\int_{\C_*(z+\cu_n)} \left|\nabla (u-u_z) \indc_{\{\a\neq 0\}} \right|^2(x).
\end{multline}
It follows that, for $r:= \frac{s(2d-1)}{s-2} = \frac{s(2d-1)}{\ep}$, which we note depends only on $(d,\lambda,\p)$,
\begin{equation*}
C\left(\frac 1{|\cu_n|} \sum_{\Qa\ni\cu'\subseteq z+\cu_n} \size(\cu')^{\frac{s(2d-1)}{s-2}}  \right)^{\frac{s-2}s} \indc_{\Gamma}\, \indc_{\{n\geq \mathcal{N}_r(z)\}} \leq C. 
\end{equation*}
and therefore 
\begin{multline}
\label{e.smackchta}
\frac{1}{|\cu_n|} \int_{(z+\cu_n)\setminus K_z  } \left| \nabla  \left[ u - u_z \right]_{\Pa} \right|^2(x) \,dx\, \indc_{\Gamma} \, \indc_{\{n\geq \mathcal{N}_r(z)\}}\\
\leq \frac C{|\cu_n|} \int_{\C_*(z+\cu_n)} \left|\nabla (u-u_z) \indc_{\{\a\neq 0\}} \right|^2(x).
\end{multline}
To complete the proof of~\eqref{e.restrictquadE}, we split the expectation of the left side of~\eqref{e.stoppiz}, using the minimal scale for the partition~$\Qa$:
\begin{multline}
\label{e.EsplitNr}
\E \left[ \fint_{(z+\cu_n)\setminus K_z} \left| \nabla  \left[ u - u_z \right]_{\Pa} \right|^2(x) \,dx \, \indc_{\Gamma} \right]  \\
\leq 
\E \left[ \fint_{(z+\cu_n)\setminus K_z} \left| \nabla  \left[ u - u_z \right]_{\Pa} \right|^2(x) \,dx \, \indc_{\Gamma} \, \indc_{\{n\geq \mathcal{N}_r(z)\}} \right]  \\
+ \E \left[ \fint_{(z+\cu_n)\setminus K_z} \left| \nabla  \left[ u - u_z \right]_{\Pa} \right|^2(x) \,dx\, \indc_{\Gamma}\, \indc_{\{n< \mathcal{N}_r(z)\}} \right].
\end{multline}
Taking the expectation of~\eqref{e.smackchta} gives an estimate for the first term on the right side:  
\begin{multline}
\label{e.ENrdn}
\E \left[ \fint_{(z+\cu_n)\setminus K_z} \left| \nabla  \left[ u - u_z \right]_{\Pa} \right|^2(x) \,dx \, \indc_{\Gamma} \, \indc_{\{n\geq \mathcal{N}_r(z)\}} \right]  \\
\leq C \E \left[ \frac{1}{\left| \cu_n \right|}\int_{\C_*(z+\cu_n)}  \left|\nabla (u-u_z) \indc_{\{\a\neq 0\}} \right|^2(x)\,dx \right].
\end{multline}
We  estimate the second term on the right of~\eqref{e.stoppiz} rather crudely: we use~\eqref{e.Oprods} and combine this with~\eqref{e.coarsegraduL2} (with $\delta = \frac14$) to obtain, for some small exponent~$s_0(d,\lambda,\p)>0$,
\begin{align*}
\lefteqn{
\fint_{z+\cu_n } \left| \nabla  \left[ u - u_z \right]_{\Pa} \right|^2(x) \,dx\, \indc_{\{n< \mathcal{N}_r(z)\}} 
} \qquad & \\
& \leq 2\left( 3^{d(m-n)} \fint_{\cu_{m} } \left| \nabla  \left[ u \right]_{\Pa} \right|^2(x)  \,dx + \fint_{z+\cu_n } \left| \nabla  \left[ u_z \right]_{\Pa} \right|^2(x)\,dx  \right) \indc_{\{n< \mathcal{N}_r(z)\}} \\
& \leq \O_{s_0}\left(C|q|^2 3^{d(m-n)+\frac{m}{4} -n}\right).
\end{align*}
Taking the expectation of this yields 
\begin{equation}
\label{e.ENrup}
\E\left[ \fint_{z+\cu_n } \left| \nabla  \left[ u - u_z \right]_{\Pa} \right|^2(x) \,dx\, \indc_{\{n< \mathcal{N}_r\}} \right] \leq C|q|^2 3^{d(m-n)+\frac{m}{4} -n}.
\end{equation}
Notice that the assumptions $n\geq \left( \frac{4d+1}{4d+2}\right)m$ implies that $d(m-n)+\frac{m}{4}-n \leq -\frac n2$.
Combining this observation with~\eqref{e.EsplitNr},~\eqref{e.ENrdn} and~\eqref{e.ENrup} yields
\begin{multline*}
\E \left[ \frac{1}{|\cu_n|} \int_{(z+\cu_n)\setminus K_z} \left| \nabla  \left[ u - u_z \right]_{\Pa} \right|^2(x) \,dx \, \indc_{\Gamma} \right] \\
\leq C|q|^23^{-\frac n2} + C \E \left[ \frac{1}{\left| \cu_n \right|}\int_{\C_*(z+\cu_n)}  \left|\nabla (u-u_z) \indc_{\{\a\neq 0\}} \right|^2(x)\,dx \right].
\end{multline*}
Summing over~$z\in 3^n\Zd\cap \cu_m$ and using Lemma~\ref{l.restrictquadmicro}, we get
\begin{multline}
\label{e.step3goal}
3^{-d(m-n)}\sum_{z\in3^n\Zd\cap \cu_m}
\E \left[ \frac{1}{|\cu_n|}\int_{(z+\cu_n)\setminus K_z} \left| \nabla  \left[ u - u_z \right]_{\Pa} \right|^2(x) \,dx \, \indc_{\Gamma} \right]
\\
\leq C|q|^2 \left(\sum_{k=n}^{m-1} \tau_k + 3^{-\frac n4} \right).
\end{multline}

\smallskip

\emph{Step 4.} The conclusion. Combining~\eqref{e.estoffGamma},~\eqref{e.removeKz} and~\eqref{e.step3goal}, we obtain
\begin{equation*}
\frac{3^{-d(m-n)}}{|\cu_n|^2}\sum_{z\in 3^n\Zd\cap \cu_m }\E \left[ 
\left|
\left\langle \left[ \nabla u\right]_{\Pa}\right\rangle_{z+\cu_{n}} -  \left\langle \left[ \nabla u_z\right]_{\Pa} \right\rangle_{z+\cu_{n}} 
\right|^2
 \right] 
 \leq C|q|^2 \left( \sum_{k=n}^{m-1} \tau_k + 3^{-\frac n4} \right).
\end{equation*}
This completes the proof of the lemma. 
\end{proof}

\smallskip

\begin{definition}[The matrix $\ahom_{\cu_n}$]
We define $\ahom^{-1}_{\cu_n}$ to be the matrix satisfying, for every $q\in\Rd$,
\begin{equation}
\label{e.abarn}
\frac12 q\cdot \ahom^{-1}_{\cu_n} q = \E \left[ -\mu(\cu_n,q) \right].
\end{equation}
Since the right side of~\eqref{e.abarn} is a nonnegative and quadratic function of~$q$, it can be written in terms of a matrix and thus~$\ahom^{-1}_{\cu_n}$ is well-defined.
\end{definition}

We continue with some observations regarding the matrices $\ahom_{\cu_n}$. Notice that~\eqref{e.muupbound} and~\eqref{e.mulobound} imply the existence of~$C(d,\lambda,\p)<\infty$ such that, for every $n\in\N$ with $n\geq C$, 
\begin{equation}
\label{e.ahominvbound}
\frac1C I_d \leq \ahom^{-1}_{\cu_n} \leq C I_d. 
\end{equation}
Since our model is invariant under permutations and reflections of the coordinate axes, the matrices $\ahom^{-1}_{\cu_n}$ (as well as $\ahom$) are actually a multiple of the identity~$I_d$. However, since we do not use this anywhere and our arguments which actually can handle more general models, we ignore this fact. 

\smallskip

According to the first variation~\eqref{e.firstvarmu}, we can also write~$\ahom^{-1}_{\cu_n}$ in terms of the expected spatial average of the minimizers of $\mu(\cu_n,q)$: for every $q,q' \in\Rd$,
\begin{equation}
\label{e.ahominvspatavg}
q'\cdot \ahom^{-1}_{\cu_n} q = \E \left[\frac{1}{\left| \cl_\Pa(\cu_n) \right|} \left\langle q', \nabla \left[ u(\cdot,\cu_n,q) \right]_{\Pa} \right\rangle_{\cl_{\Pa}(\cu_n)}  \right].
\end{equation}
In other words,
\begin{equation*} \label{}
\ahom^{-1}_{\cu_n} q = \E \left[ \frac{1}{\left| \cl_\Pa(\cu_n) \right|} \left\langle \nabla\! \left[ u(\cdot,\cu_{n},q) \right]_{\mathcal{P}} \right\rangle_{\cl_{\Pa}(\cu_n)}  \right].
\end{equation*}
For future reference, we record, for every $m,n\in\N$ with $m>n$, the estimate
\begin{equation}
\label{e.ahominvclose}
\left| \ahom^{-1}_{\cu_m} - \ahom^{-1}_{\cu_{n}} \right| \leq C \left( 3^{-\frac n4}+ \sum_{k=n}^{m-1} \tau_k \right).
\end{equation}
To prove~\eqref{e.ahominvclose}, we use~\eqref{e.almostmonotone} and~\eqref{e.abarn} to see that, for every $q\in\Rd$, 
\begin{align*}
\left| \frac12 q\cdot \ahom^{-1}_{\cu_m} q -  \frac12 q\cdot \ahom^{-1}_{\cu_n} q \right|
& = \left| \E\left[ \mu(\cu_m,q)\right] - \E\left[ \mu(\cu_n,q)\right] \right| \\
& \leq  \sum_{k=n}^{m-1}\left| \E \left[ \mu(\cu_{k+1},q)\right] - \E\left[  \mu(\cu_k,q)\right] \right| \\
& \leq \sum_{k=n}^{m-1}\left( \left( \E \left[ \mu(\cu_{k+1},q)\right] - \E\left[  \mu(\cu_k,q)\right] \right)_+ + C3^{-\frac k4} \right) \\
& \leq C|q|^2 \left( 3^{-\frac n4} +  \sum_{k=n}^{m-1} \sum_{i=1}^d  \left( \E \left[  \mu(\cu_{k+1},\e_i)\right] - \E\left[ \mu(\cu_k,\e_i)\right]\right)_+  \right) \\
& \leq C|q|^2 \left( 3^{-\frac n4} +  \sum_{k=n}^{m-1} \tau_k \right). 
\end{align*}
Taking the supremum of the previous inequality over $q\in \Rd\setminus \{0\}$, after dividing by $|q|^2$, yields~\eqref{e.ahominvclose}.

\smallskip

We show in the next lemma that the variance of $\left\langle \nabla\! \left[ u(\cdot,\cu_{n},q) \right]_{\mathcal{P}} \right\rangle_{\cl_{\Pa}(\cu_n)}$ is controlled by $\tau_n$. This is perhaps the main step in the proof of Proposition~\ref{p.subadd}. It is a variation of~\cite[Lemma 3.2]{AS}.

\begin{lemma}
\label{l.weakconv}
There exists $C(d,\p,\lambda)<\infty$ such that, for every $n\in\N$ and $q \in \Rd$,
\begin{equation} 
\label{e.weakconv}
\E \left[ \left| \frac{1}{|\cu_n|}\left\langle \nabla\! \left[ u(\cdot,\cu_{n},q) \right]_{\mathcal{P}} \right\rangle_{\cu_n}  - \ahom^{-1}_{\cu_n} q \right|^2  \right]
\leq C|q|^2 \left( \tau_n + 3^{-\frac n4} \right).
\end{equation}
\end{lemma}
\begin{proof}

Fix a unit direction $e\in \partial B_1$. 

\smallskip

\emph{Step 1.} We construct a (deterministic) compactly supported, bounded and solenoidal vector field $\g$ on $\cu_{n+1}$ which is constant and equal to $e$ on $\cu_n$. Precisely, the claim is that $\g$ is a vector field on $\cu_n$ satisfying
\begin{equation}
\label{e.gbounded}
\sup_{(x,y) \in \Ed(\cu_{n+1})} \left| \g(x,y) \right| \leq C,
\end{equation}
is constant in the middle third subcube, i.e., 
\begin{equation}
\label{e.gconstant}
\g(x,y) = e \cdot (x-y) \quad \mbox{for all} \ x,y\in \cu_{n} \ \mbox{with}  \ x\sim y,
\end{equation}
and is almost (discretely) divergence-free in the sense that, for every $w:\cu_{n+1}\to \R$, we have 
\begin{equation}
\label{e.gdivfree}
\left| \left \langle \nabla w, \g \right\rangle_{\cu_{n+1}} \right| \leq C3^{-n} \sum_{x\in \cu_{n+1}} \left| \nabla w \right|(x).
\end{equation}
According to the proof of~\cite[Lemma 3.2]{AS}, there exists a smooth, (continuum) vector field $\mathbf{g} \in C^\infty(\Rd;\Rd)$ satisfying, for every $k\in\N$,
\begin{equation*}
\supp \mathbf{g} \subseteq [-3,3]^d,  \quad 
\mathbf{g}(x) = e \ \mbox{in} \ [-1,1]^d, \quad
\nabla \cdot \mathbf{g} = 0 \ \mbox{in} \ \Rd, \quad
\left| \nabla^k \mathbf{g} \right| \leq C(k,d). 
\end{equation*}
We define the (discrete) vector field $\g$ by
\begin{equation*}
\g(x,y) = \frac12\left(\mathbf{g}\left(3^{-n}x\right)+\mathbf{g}\left(3^{-n}y\right) \right) \cdot (x-y), \quad x,y\in \cu_{n+1}. 
\end{equation*}
To check~\eqref{e.gdivfree}, we fix $w:\cu_{n+1}\to \R$, assume without loss of generality that $\left( w \right)_{\cu_{n+1}}=0$, and compute
\begin{align*}
\left \langle \nabla w, \g \right\rangle_{\cu_{n+1}} 
& = \frac12\sum_{x,y\in \cu_{n+1},\, x\sim y} \left( w(x) - w(y) \right)\left(\mathbf{g}\left(3^{-n}x\right)+\mathbf{g}\left(3^{-n}y\right) \right) \cdot (x-y) \\
& = \sum_{x\in \cu_{n+1}} w(x) \sum_{y\sim x} \left(\mathbf{g}\left(3^{-n}x\right)+\mathbf{g}\left(3^{-n}y\right) \right) \cdot (x-y).
\end{align*}
For each $x$, we have
\begin{equation*}
\sum_{y\sim x} \mathbf{g}\left(3^{-n}x\right) \cdot (x-y) 
=  \mathbf{g}\left(3^{-n}x\right)  \sum_{y\sim x}\cdot (x-y) = 0,
\end{equation*}
and therefore 
\begin{align*}
 \sum_{y\sim x} \left(\mathbf{g}\left(3^{-n}x\right)+\mathbf{g}\left(3^{-n}y\right) \right) \cdot (x-y) 
 =
\sum_{y\sim x} \left(\mathbf{g}\left(3^{-n}y\right)-\mathbf{g}\left(3^{-n}x\right) \right) \cdot (x-y).
\end{align*}
The divergence-free condition and $\left| \nabla^2 \mathbf{g} \right| \leq C$ yield, for each fixed $x$,
\begin{align*}
\left|
\sum_{y\sim x} \left(\mathbf{g}\left(3^{-n}y\right)-\mathbf{g}\left(3^{-n}x\right) \right) \cdot (x-y) \right| 
& \leq
\left|
\sum_{y\sim x} (x-y) \cdot \nabla \mathbf{g} \left(3^{-n}x\right) (y-x) \right|  + C3^{-2n} \\
& =2 \left|  \nabla \cdot \mathbf{g}\left(3^{-n}x\right) \right| + C3^{-2n} 
 = C3^{-2n}.
\end{align*}
Combining the above yields
\begin{align*}
\left| \left \langle \nabla w, \g \right\rangle_{\cu_{n+1}} \right|
& \leq  \sum_{x\in \cu_{n+1}}\left| w(x) \right| \left| \sum_{y\sim x} \left(\mathbf{g}\left(3^{-n}x\right)+\mathbf{g}\left(3^{-n}y\right) \right) \cdot (x-y) \right| \\
& \leq C3^{-2n} \sum_{x\in \cu_{n+1}} \left|w(x) \right|. 
\end{align*}
The Poincar\'e inequality and $\left( w \right)_{\cu_{n+1}}=0$ now yields~\eqref{e.gdivfree}. 

\smallskip

\emph{Step 2.} We next show that 
\begin{multline}
\label{e.Holderaway}
\left|  
 \left\langle \nabla \left[ u(\cdot,\cu_{n+1},q)  \right]_{\mathcal{P}}, \g \right\rangle_{\cu_{n+1}}
-  \sum_{z\in 3^n\Zd\cap\cu_{n+1}} 
 \left\langle \nabla \left[ u(\cdot,\cu_{n+1},q)  \right]_{\mathcal{P}}, \g \right\rangle_{z+\cu_{n}}
\right|
\\
\leq \O_{\frac{2}{4d-3}} \left( C|\cu_{n+1}||q|3^{-\frac n4} \right). 
\end{multline}
For any vector field $F$, 
\begin{equation*} \label{}
\left\langle F, \g \right\rangle_{\cu_{n+1}} - \sum_{z\in 3^n\Zd\cap\cu_{n+1}} 
 \left\langle F,\g \right\rangle_{z+\cu_{n}}
= \sum_{(x,y) \in D} \left\langle F,\g \right\rangle_{z+\cu_{n}}
\end{equation*}
where $D$ is the set of edges $(x,y)$ such that $x,y\in\cu_{n+1}$ and $x$ and $y$ belong to different subcubes of the form $z+\cu_n$, $z\in \{ -3^n,0,3^n\}^d$. Notice that $|D| \leq C3^{-n}|\cu_{n+1}|$. By the H\"older inequality and the bound $| \g | \leq C$, we therefore obtain
\begin{align*} \label{}
\left| \left\langle F, \g \right\rangle_{\cu_{n+1}} - \sum_{z\in 3^n\Zd\cap\cu_{n+1}} 
 \left\langle F,\g \right\rangle_{z+\cu_{n}} \right|
&
\leq \sum_{(x,y) \in D} \left\langle F,\g \right\rangle_{z+\cu_{n}}
\\ \notag &
\leq C\int_D \left| F\right|(x)\,dx 
\\ \notag &
\leq C |D|^{\frac12} \left( \int_{D} \left| F\right|^2(x)\,dx \right)^{\frac12}
\\ \notag &
\leq C3^{-\frac n2} |\cu_{n+1}|^\frac12 \left( \int_{\cu_{n+1}} \left| F\right|^2(x)\,dx \right)^{\frac12}.
\end{align*}
Applying this inequality with $F=\nabla \left[ u(\cdot,\cu_{n+1},q)  \right]_{\mathcal{P}}$ and then using~\eqref{e.coarsegraduL2} with $\delta = \frac 12$ to bound the right side of the result, we obtain~\eqref{e.Holderaway}.

\smallskip

\emph{Step 3.} The conclusion. We apply~\eqref{e.gdivfree} with  $w = \left[ u(\cdot,\cu_{n+1},q)  \right]_{\mathcal{P}}$ and then use~\eqref{e.uupbound},~\eqref{e.restrictquadE} and~\eqref{e.Holderaway} to get
\begin{equation} \label{e.onevarest}
\E\left[ \left| \sum_{z\in 3^n\Zd\cap\cu_{n+1}} 
 \frac 1{|\cu_n|}\left\langle \nabla \left[ u(\cdot,z+\cu_{n},q)  \right]_{\mathcal{P}}, \g \right\rangle_{z+\cu_{n}} \right|^2 \right]
\leq  C|q|^2\left( \tau_n  + 3^{-\frac n4} \right). 
\end{equation}
We conclude by noticing that, since the elements in the sum on the left side are almost $\P$--independent, the variance of each term in the sum should be controlled by the variance of the entire sum. To make this precise, we prove the following inequality:
\begin{multline}
\label{e.switchloc}
\E \Bigg[  \bigg| \frac{1}{|\cu_n|} \left\langle \nabla\left[ u(\cdot,z+\cu_{n},q)  \right]_{\mathcal{P}} ,  \g(x,y) \right\rangle_{ z+\cu_n}  
\\
- 
 \frac{1}{|\cu_n|} \left\langle \nabla\left[ u^{(\lceil n/2\rceil)}_\mathrm{loc}(\cdot,z+\cu_{n},q)  \right]_{\mathcal{P}_{\mathrm{loc}}(z+\cu_n)} ,  \g(x,y) \right\rangle_{ z+\cu_n}  
\bigg|^2 
\Bigg] 
\leq C3^{-\frac n4}.
\end{multline}
To prove this inequality, we first appeal to~\eqref{e.ulocloccoarse}, which gives 
\begin{multline*}
\E \Bigg[  \bigg| \frac{1}{|\cu_n|}  \left\langle \nabla\left[ u(\cdot,z+\cu_{n},q)  \right]_{\mathcal{P}} ,  \g(x,y) \right\rangle_{\left( z+\cu_n\right)^{(\lceil n/2\rceil)}}  
\\
- 
 \frac{1}{|\cu_n|} \left\langle \nabla\left[ u^{(\lceil n/2\rceil)}_\mathrm{loc}(\cdot,z+\cu_{n},q)  \right]_{\mathcal{P}_{\mathrm{loc}}(z+\cu_n)} ,  \g(x,y) \right\rangle_{\left( z+\cu_n\right)^{(\lceil n/2\rceil)}}  
\bigg|^2 
\Bigg] 
\leq C3^{-\frac n4}.
\end{multline*}
and combine it with the result of the following computation, which is an application of~\eqref{e.coarsegraduL2} (with $\delta = \frac14$):
\begin{align*}
\lefteqn{\E \Bigg[  \left|  \frac{1}{|\cu_n|} \left\langle \nabla\left[ u(\cdot,z+\cu_{n},q)  \right]_{\mathcal{P}} ,  \g(x,y) \right\rangle_{\left( z+\cu_n\right) \setminus \left( z+\cu_n\right)^{(\lceil n/2\rceil)}}  \right|^2
\Bigg] } \qquad & \\
& \leq C \E \Bigg[  \left|  \frac{1}{|\cu_n|} \int_{\left( z+\cu_n\right) \setminus \left( z+\cu_n\right)^{(\lceil n/2\rceil)}} \left| \nabla\left[ u(\cdot,z+\cu_{n},q)  \right]_{\mathcal{P}} \right|(x) \, dx  \right|^2
\Bigg]\\
& \leq  C \E \Bigg[  \frac{\left| \left( z+\cu_n\right) \setminus \left( z+\cu_n\right)^{(\lceil n/2\rceil)} \right|}{| \cu_n|}  \fint_{ z+\cu_n} \left| \nabla\left[ u(\cdot,z+\cu_{n},q)  \right]_{\mathcal{P}}  \right|^2(x) \, dx
\Bigg]\\
& \leq  C3^{-\frac n4 }.
\end{align*}

Since 
\begin{equation*} \label{}
 \left\langle \nabla\left[ u^{(\lceil n/2\rceil)}_\mathrm{loc}(\cdot,z+\cu_{n},q)  \right]_{\mathcal{P}_{\mathrm{loc}}(z+\cu_n)} ,  \g(x,y) \right\rangle_{\left( z+\cu_n\right)^{(\lceil n/2\rceil)}}  
 \quad \mbox{is $\F(z+\cu_n)$--measurable,}
\end{equation*}
and $|\g|\leq C$ by~\eqref{e.gbounded}, we may use independence (in the second line in the display below), ~\eqref{e.switchloc} and the triangle inequality (twice, in the first and third lines) and~\eqref{e.onevarest} (in the fourth line) to obtain
\begin{align*}
\lefteqn{ 
\var\left[ \frac{1}{|\cu_n|} \left\langle \nabla\left[ u(\cdot,\cu_{n},q)  \right]_{\mathcal{P}} ,  \g(x,y) \right\rangle_{\cu_n} \right]
} \quad & \\
& \leq \sum_{z\in 3^n\Zd\cap\cu_{n+1}} \var\left[  \frac{1}{|\cu_n|} \left\langle \nabla\left[ u^{(\lceil n/2\rceil)}_\mathrm{loc}(\cdot,z+\cu_{n},q)  \right]_{\mathcal{P}_{\mathrm{loc}}(z+\cu_n)},  \g(x,y) \right\rangle_{\left( z+\cu_n\right)^{(\lceil n/2\rceil)}}  \right] + C3^{-\frac n4} \\
& =\var\left[   \sum_{z\in 3^n\Zd\cap\cu_{n+1}} \frac{1}{|\cu_n|} \left\langle \nabla\left[ u^{(\lceil n/2\rceil)}_\mathrm{loc}(\cdot,z+\cu_{n},q)  \right]_{\mathcal{P}_{\mathrm{loc}}(z+\cu_n)},  \g(x,y) \right\rangle_{\left( z+\cu_n\right)^{(\lceil n/2\rceil)}}  \right] + C3^{-\frac n4} \\
& \leq \var\left[ \sum_{z\in 3^n\Zd\cap\cu_{n+1}}\frac{1}{|\cu_n|}  \left\langle \nabla\left[ u(\cdot,z+\cu_{n},q)  \right]_{\mathcal{P}} ,  \g(x,y) \right\rangle_{ z+\cu_n} \right]  + C3^{-\frac n4}  \\
& \leq C|q|^2\left( \tau_n  + 3^{-\frac n4} \right). 
\end{align*}
Using~\eqref{e.gconstant}, we may rewrite this as 
\begin{equation*} \label{}
\var\left[ \frac{1}{|\cu_n|} \left\langle \nabla\left[ u(\cdot,\cu_{n},q)  \right]_{\mathcal{P}} , e \right\rangle_{\cu_n} \right]
\leq  C\left( \tau_n  + 3^{-\frac n4} \right). 
\end{equation*}
Summing this over $e\in \{\e_1,\ldots,\e_d\}$ and recalling~\eqref{e.ahominvspatavg} completes the proof.
\end{proof}

We next pass from control of the spatial averages of $\nabla \left[u\right]_{\Pa}$ given in the previous lemma to control over the function~$\left[ u \right]_{\Pa}$ itself. The ingredients for this are~Lemmas~\ref{l.restrictquadcoarsen} and~\ref{l.weakconv} and the (discrete) multiscale Poincar\'e inequality (Proposition~\ref{p.mspoincare}).

\begin{lemma}
\label{l.flatness}
There exists $C(d,\p,\lambda)<\infty$ such that, for $n\in\N$ and $p,q \in \Rd$,
\begin{multline} 
\label{e.flatness}
\E \left[ \fint_{\cu_{n+1}} \left| \left[ u(\cdot,\cu_{n+1},q) \right]_{\Pa}\!(x)  - x\cdot \ahom^{-1}_{\cu_n} q \right|^2\,dx   \right]   
\\
\leq C|q|^2 3^{2n} \left(  3^{-\left(\frac{1}{4d+2}\right)n} + 3^{-n} \sum_{k=0}^n 3^k \tau_k\right).
\end{multline}
and
\begin{multline} 
\label{e.flatnessv}
\E \left[ \fint_{\cu_{n+1}} \left|  \left[ v(\cdot,\cu_{n+1},p) \right]_{\Pa}\!(x)- x\cdot p \right|^2\,dx   \right] 
\\
\leq C|p|^2 3^{2n} \left(  3^{-\left(\frac{2}{2d+1}\right)n}  + 3^{-n} \sum_{k=0}^n 3^k \tau_k\right).
\end{multline}
\end{lemma}
\begin{proof}
We first give the proof of~\eqref{e.flatness}. The main tool for passing from spatial averages of gradients to strong norms for the function itself is the multiscale Poincar\'e inequality (Proposition~\ref{p.mspoincare}), which we apply to the function 
$$
x\mapsto \left[ u (\cdot,\cu_{n+1},q)\right]_{\Pa}\! (x) - x\cdot \ahom^{-1}_{\cu_n} q.
$$
We note that this function has zero mean on~$\cu_{n+1}$ by the chosen normalization~\eqref{e.additiveconstants}. Proposition~\ref{p.mspoincare} yields, for $n_0:=\left\lceil \left(\frac{4d+1}{4d+2} \right) (n+1)\right\rceil$, 
\begin{multline}
\label{e.mspoincareappl}
\fint_{\cu_{n+1}} \left| \left[ u (\cdot,\cu_{n+1},q)\right]_{\Pa}\! (x) - x\cdot \ahom^{-1}_{\cu_n} q \right|^2,dx   \\ 
 \leq C 3^{2n_0} \fint_{\cu_{n+1}} \left| \nabla \left[ u (\cdot,\cu_{n+1},q) \right]_{\Pa}\!(x) - \ahom^{-1}_{\cu_n} q \right|^2\,dx \\ 
+ C \left( \sum_{k=n_0}^n 3^k \left( 3^{-d(n-k)} \sum_{z\in 3^k\Zd \cap \cu_{n+1}} \left| \frac{1}{|\cu_k|} \left\langle \nabla \left[ u(\cdot,\cu_{n+1},q) \right]_{\Pa}\!  \right\rangle_{z+\cu_k} - \ahom^{-1}_{\cu_n} q \right|^2 \right)^{\frac12}  \right)^{2}.
\end{multline}
The first term on the right side is controlled by the triangle inequality,~\eqref{e.coarsegraduL2} (with $\delta = \frac{1}{4d+2}$) and~\eqref{e.ahominvbound}, which give
\begin{align}
\label{e.firsttermmsp}
\lefteqn{
\fint_{\cu_{n+1}} \left| \nabla \left[ u(\cdot,\cu_{n+1},q) \right]_{\Pa}\!(x)  - \ahom^{-1}_{\cu_n} q \right|^2\,dx  
} \qquad & \\ \notag
& \leq \fint_{\cu_{n+1}} \left| \nabla \left[ u \right]_{\Pa}\! (\cdot,\cu_{n+1},q)\right|^2\,dx  + \left| \ahom^{-1}_{\cu_n} q \right|^2 \\ \notag
& \leq \O_{\frac1{4d-3}}\left( C|q|^2 3^{\frac{n}{4d+2}} \right) + C|q|^2  \leq  \O_{\frac1{4d-3}}\left( C|q|^2 3^{\frac{n}{4d+2}} \right).
\end{align}
To bound the second term on the right side of~\eqref{e.mspoincareappl}, we have to estimate the expectation of the (square of the) difference between~$\frac{1}{|\cu_k|} \left\langle \nabla \left[ u(\cdot,\cu_{n+1},q) \right]_{\Pa}\!  \right\rangle_{z+\cu_k}$ and~$\ahom^{-1}_{\cu_n} q$ in all successors $z+\cu_k$ of~$\cu_{n+1}$, down to the mesoscale of size $3^{n_0}$. The claim is that, for every $k \geq n_0$,
\begin{multline}
\label{e.mspoinbigshow}
3^{-d(n-k)} \sum_{z\in 3^k\Zd\cap \cu_{n+1}} \E \left[ \left| \frac{1}{|\cu_k|}\left\langle \nabla \left[ u (\cdot,\cu_{n+1},q)  \right]_{\Pa}\!\right\rangle_{z+\cu_k} - \ahom^{-1}_{\cu_n} q \right|^2 \right] \\
\leq C|q|^2 \left( \sum_{m=k}^n \tau_m + 3^{-\frac k4} \right).
\end{multline}
This inequality is a  consequence of what we have shown above. Indeed, by the inequality $(a+b+c)^2\leq 3(a^2+b^2+c^2)$,~\eqref{e.restrictquadE},~\eqref{e.ahominvbound},~\eqref{e.ahominvclose} and~\eqref{e.weakconv}, we have
\begin{align*}
\lefteqn{
3^{-d(n-k)} \sum_{z\in 3^k\Zd\cap \cu_{n+1}} \E \left[ \left| \frac{1}{|\cu_k|}\left\langle \nabla \left[ u  (\cdot,\cu_{n+1},q)\right]_{\Pa}\! \right\rangle_{z+\cu_k} - \ahom^{-1}_{\cu_n} q \right|^2 \right]
} \qquad & \\
& \leq \frac{3^{1-d(n-k)}}{|\cu_k|^2}\sum_{z\in 3^k\Zd\cap \cu_{n+1}}
 \E \left[ \left| \left\langle \nabla \left[ u  (\cdot,\cu_{n+1},q)\right]_{\Pa}\! \right\rangle_{z+\cu_k} 
 -   \left\langle \nabla \left[ u  (\cdot,z+\cu_{k},q)\right]_{\Pa}\! \right\rangle_{z+\cu_k}
  \right|^2 \right]
\\
& \qquad + 3^{1-d(n-k)}\sum_{z\in 3^k\Zd\cap \cu_{n+1}}  \E \left[ \left| \frac{1}{|\cu_k|} \left\langle \nabla \left[ u (\cdot,z+\cu_{k},q) \right]_{\Pa}\! \right\rangle_{z+\cu_k} - \ahom^{-1}_{\cu_k} q \right|^2 \right] \\
& \qquad + 3|q|^2\left| \ahom^{-1}_{\cu_k} - \ahom^{-1}_{\cu_{n}} \right|^2 \\
& \leq C|q|^2 \left( \sum_{m=k}^n \tau_m + 3^{-\frac k4} \right). 
\end{align*}
This is~\eqref{e.mspoinbigshow}.

\smallskip

We next combine~\eqref{e.mspoinbigshow} with~\eqref{e.mspoincareappl} and~\eqref{e.firsttermmsp}, after taking the expectation of the latter two inequalities, to obtain the estimate
\begin{equation*}
\E \left[ \fint_{\cu_{n+1}} \left| \left[ u (\cdot,\cu_{n+1},q)\right]_{\Pa}\! (x) - x\cdot \ahom^{-1}_{\cu_n} q \right|^2\,dx \right] 
\leq C \left( 3^{2n_0+ \frac{n}{4d+2}} |q|^2 + \E \left[ \left(\sum_{k=n_0}^n 3^k X_k^{\frac12} \right)^2 \right] \right) 
\end{equation*}
where the random variable 
\begin{equation*}
X_k:= 3^{-d(n-k)} \sum_{z\in 3^k\Zd\cap \cu_{n+1}} \left| \frac{1}{|\cu_k|}\left\langle \nabla \left[ u(\cdot,\cu_{n+1},q) \right]_{\Pa}\!  \right\rangle_{z+\cu_k} - \ahom^{-1}_{\cu_n} q \right|^2
\end{equation*}
satisfies
\begin{equation*}
\E \left[ X_k \right] \leq C|q|^2 \left( \sum_{m=k}^n \tau_m + 3^{-\frac k4} \right).
\end{equation*}
Using the fact that 
\begin{equation*}
\left( \sum_{k=n_0}^n 3^k X_k^{\frac12} \right)^2 \leq C3^n \sum_{k=n_0}^n 3^k X_k
\end{equation*}
and taking expectations, we obtain 
\begin{align*}
\lefteqn{
\E \left[ \fint_{\cu_{n+1}} \left| \left[ u (\cdot,\cu_{n+1},q)  \right]_{\Pa}\!(x)- x\cdot \ahom^{-1}_{\cu_n} q \right|^2,dx \right]  
} \qquad & \\ \notag &
\leq C|q|^2 \left( 3^{2n_0+ \frac{n}{4d+2}} + 3^n\sum_{k=n_0}^n3^k  \left( \sum_{m=k}^n\tau_m+3^{-\frac k4} \right) \right) 
\\ \notag &
\leq C|q|^2  \left( 3^{2n_0+ \frac{n}{4d+2}} + 3^{n+\frac{3n_0}4} + 3^n\sum_{k=0}^n 3^k \tau_k\right).
\end{align*}
Observe from the definition of~$n_0$ that
\begin{equation*}
3^{n+\frac{3n_0}4} \leq C3^{2n_0+ \frac{n}{4d+2}} \leq C 3^{2n - \frac{n}{4d+2}}.
\end{equation*}
We thus obtain
\begin{equation*}
\E \left[ \fint_{\cu_{n+1}} \left| \left[ u (\cdot,\cu_{n+1},q)  \right]_{\Pa}\!(x)- x\cdot \ahom^{-1}_{\cu_n} q \right|^2,dx \right]
\leq C|q|^2 3^{2n} \left( 3^{-\left(\frac{1}{4d+2}\right)n} + 3^{-n} \sum_{k=0}^n 3^k \tau_k\right).
\end{equation*}
This completes the proof of~\eqref{e.flatness}. The proof of~\eqref{e.flatnessv} is so similar to that of~\eqref{e.flatness} that we omit the details. The only difference is that we do not need Lemma~\ref{l.weakconv} and, in place of~\eqref{e.weakconv}, use the estimate~\eqref{e.nablavaffine}, which implies
\begin{equation*}
\E \left[ \left| \frac{1}{|\cu_k|} \left\langle \nabla \left[ v(\cdot,z+\cu_{k},p) \right]_{\Pa} \right\rangle_{z+\cu_k} - p\right|^2 \right] \leq C|p|^2 3^{-k}.
\end{equation*}
This completes the proof of the lemma. 
\end{proof}

We now combine the previous lemma, the Caccioppoli inequality (Lemma~\ref{l.caccioppoli}) and quadratic response~\eqref{e.secondvarnu} to obtain an estimate on the expectation of a quantity very close to~$\omega(\cu_n,q)$ in terms of a weighted average of $\{\tau_k\}_{k\leq n}$.

\begin{lemma}
\label{l.omegadecay}
There exists a constant $C(d,\p,\lambda)<\infty$ such that for every $n\in\N$ and $q \in \Rd$,
\begin{equation}
\label{e.omegadecay}
\E \left[ \left| \nu\left(\cu_n,\ahom^{-1}_{\cu_n}q\right) - \mu(\cu_n,q) - q\cdot \ahom_{\cu_n}^{-1}q \right| \right] \\
\leq C |q|^2 \left(  3^{-\left(\frac{1}{4d+2}\right)n} + \sum_{k=0}^n 3^{k-n}\tau_k \right).
\end{equation}
\end{lemma}
\begin{proof}
Denote $\overline{p}_n:=\ahom_{\cu_n}^{-1}q$. 
According to Corollary~\ref{c.munucomparison} and $|p_n| \leq C|q|$,
\begin{multline} \label{e.quadrespuv}
\left| \nu\left( \cu_n ,\overline{p}_n\right) -  \mu(\cu_n,q) - q\cdot \overline{p}_n \right| \\
\leq \frac{1}{|\cu_n|} \int_{\C_*(\cu_n)} \left| \left( \nabla u(\cdot,\cu_n,q) - \nabla v(\cdot,\cu_n,\overline{p}_n) \right)\indc_{\{\a\neq0\}} \right|^2(x)\,dx + \O_{\frac{2}{2d-1}}\left( C|q|^23^{-\frac n2}\right).
\end{multline}
We focus the rest of the argument on estimating the expectation of the first term on the right side of~\eqref{e.quadrespuv}. Using~\eqref{e.restrictquadmicro} and the Caccioppoli inequality (Lemma~\ref{l.caccioppoli}), we have that 
\begin{align}
\label{e.cacciosetup} 
\lefteqn{
 \E \left[ \frac{1}{|\cu_n|} \int_{\C_*(\cu_n)} \left| \left( \nabla u(\cdot,\cu_n,q) - \nabla v(\cdot,\cu_n,\overline{p}_n) \right) \indc_{\{\a\neq0\}}\right|^2(x)\,dx  \right]
 } \quad & \\ \notag
& \leq \E \left[ \frac{1}{|\cu_n|} \int_{\C_*(\cu_n)} \left| \left( \nabla u(\cdot,\cu_{n+1},q) - \nabla v(\cdot,\cu_{n+1},\overline{p}_n)\right)\indc_{\{\a\neq0\}} \right|^2(x)\,dx \right] \\ \notag
& \qquad + C|q|^2 \left(  \sum_{k=n}^{m-1} \tau_k + 3^{-\frac n2} \right) \\ \notag
& \leq \E \left[ \frac{C3^{-2n}}{|\cu_n|} \int_{\C_*(\cu_{n+1})} \left| u(x,\cu_{n+1},q) - v(x,\cu_{n+1},\overline{p}_n) \right|^2\,dx \right] 
\\ \notag
& \qquad + C|q|^2 \left(  \sum_{k=n}^{m-1} \tau_k + 3^{-\frac n2} \right).
\end{align}
This reduces the lemma to an appropriate estimate of the last expectation in the previous display. 

\smallskip

We continue by writing $u:= u(x,\cu_{n+1},q)$ and $v:=v(x,\cu_{n+1},\overline{p}_n)$ for short. We also allow $s$ to be a positive exponent depending on~$d$ which may vary in each occurrence. 
According to Lemma~\ref{l.coarsegrads},  
\begin{multline} \label{e.coarsendiff}
 \int_{\C_*(\cu_{n+1})} \left| (u - v)(x) - \left[ u - v \right]_{\Pa}\!(x) \right|^2\,dx 
  \\ 
 \leq C \sup_{x\in \C_*(\cu_{n+1})} \left| \cu_{\Pa} (x) \right|^{2} \int_{\C_*(\cu_{n+1})} \left| \nabla (u-v) \indc_{\{\a\neq0\}} \right|^2(x)\,dx.
\end{multline}
Let $A_n$ be the event
\begin{equation*}
A_n := \left\{ \sup_{x\in \C_*(\cu_{n+1})} \left| \cu_{\Pa} (x) \right| > 3^{\frac n2}  \right\}. 
\end{equation*}
Then according to Proposition~\ref{p.minimalscales}, there exists $s(d)>0$ such that 
\begin{equation*}
\indc_{A_n} \sup_{x\in \C_*(\cu_{n+1})} \left| \cu_{\Pa} (x) \right|^{2} \leq 3^{-\frac{n}2} \sup_{x\in \C_*(\cu_{n+1})} \left| \cu_{\Pa} (x) \right|^{3} \leq \O_{s} \left( C \right). 
\end{equation*}
Hence, by~\eqref{e.coarsendiff} and the bounds~\eqref{e.uupbound} and~\eqref{e.vupbound}, we get
\begin{align*}
\lefteqn{
 \int_{\C_*(\cu_{n+1})} \left| (u - v)(x) - \left[ u - v \right]_{\Pa}\!(x) \right|^2\,dx 
 } \quad & 
 \\  \notag &
  \leq C \sup_{x\in \C_*(\cu_{n+1})} \left| \cu_{\Pa} (x) \right|^{2} \int_{\C_*(\cu_{n+1})}\left(  \left| \nabla u \indc_{\{\a\neq0\}} \right|^2(x) +  \left| \nabla v \indc_{\{\a\neq0\}} \right|^2(x) \right)\,dx 
  \\ \notag  &
  \leq 
  C 3^n \int_{\C_*(\cu_{n+1})}\left(  \left| \nabla u \indc_{\{\a\neq0\}} \right|^2(x) +  \left| \nabla v \indc_{\{\a\neq0\}} \right|^2(x) \right)\,dx \\
  & \quad + C \indc_{A_n} \sup_{x\in \C_*(\cu_{n+1})} \left| \cu_{\Pa} (x) \right|^{2} \int_{\C_*(\cu_{n+1})}\left(  \left| \nabla u \indc_{\{\a\neq0\}} \right|^2\!(x) +  \left| \nabla v \indc_{\{\a\neq0\}} \right|^2\!(x) \right)\,dx 
  \\ \notag  &
 \leq \left( C3^n + \O_{s} \left( C \right) \right) \O_{\frac{1}{2d+1}} \left( C\left| \cu_n \right| |q|^2 \right) 
 \\ \notag & 
 \leq \O_{s}\left( C\left| \cu_n \right||q|^23^n \right). 
 \end{align*}
Thus
\begin{equation*}
\E \left[ \frac{1}{|\cu_n|} \int_{\C_*(\cu_{n+1})} \left| (u - v)(x) - \left[ u - v \right]_{\Pa}\!(x) \right|^2\,dx \right] \leq C|q|^23^n. 
\end{equation*}
Taking expectations and using the triangle inequality, we get
\begin{multline}
\label{e.handhand}
\E \left[ \frac{C3^{-2n}}{|\cu_n|} \int_{\C_*(\cu_{n+1})} \left| u(x) - v(x) \right|^2\,dx \right] \\
\leq C3^{-2n}\E \left[ \frac{1}{|\cu_n|} \int_{\C_*(\cu_{n+1})} \big| \left[u\right]_\Pa\!(x) - \left[v\right]_\Pa\!(x) \big|^2\,dx\right]  
+ C |q|^23^{-n}.
\end{multline}
Combining~\eqref{e.cacciosetup} and~\eqref{e.handhand}, we get
\begin{multline*}
\E \left[ \frac{1}{|\cu_n|} \int_{\C_*(\cu_{n})} \left| (\nabla u- \nabla v)\indc_{\{\a\neq 0\}} \right|^2(x)\,dx \right] \\
 \leq  C3^{-2n}\E \left[ \frac{1}{|\cu_n|} \int_{\C_*(\cu_{n+1})} \big| \left[u\right]_\Pa\!(x) - \left[v\right]_\Pa\!(x) \big|^2\,dx\right] 
 + C|q|^2 \left(  \sum_{k=n}^{m-1} \tau_k + 3^{-\frac n2} \right).
\end{multline*}
An application of Lemma~\ref{l.flatness} and the triangle inequality yields 
\begin{equation*} 
\E \left[ \int_{\C_*(\cu_{n+1})} \left| \left[ u \right]_{\Pa}\!(x) - \left[v \right]_{\Pa}\!(x) \right|^2\,dx \right] \leq C|q|^2 3^{2n} \left( 3^{-\left(\frac{2}{2d+1}\right)n} + \sum_{k=0}^n 3^{k-n} \tau_k\right). 
\end{equation*}
The previous two displays and~\eqref{e.quadrespuv} imply~\eqref{e.omegadecay} and complete the proof. 
\end{proof}

We next show that an iteration of the result of the previous lemma yields a rate of decay for the expectation of~$\omega$. 

\begin{lemma}
\label{l.omegaiteration} 
There exist an exponent $\alpha(d,\p,\lambda)>0$ and $C(d,\p,\lambda)<\infty$ such that for every $n\in\N$,
\begin{equation}
\label{e.omegadecayrate}
\E\left[\sup_{q\in\Rd}  \frac1{|q|^2}  \left( \omega(\cu_{n},q)\right)_+ \right] \leq C3^{-n\alpha}. 
\end{equation}
\end{lemma}
\begin{proof}
In view of the left side of~\eqref{e.omegadecay}, it is natural to consider the quantity $D_n$ defined for each $n\in\N$ by
\begin{equation*} \label{}
D_n:= \sum_{i=1}^d  \E \left[ \nu\left(\cu_n,\e_i\right) - \mu(\cu_n,\ahom_{\cu_n} \e_i) - \e_i \cdot \ahom_{\cu_n}\e_i   \right]. 
\end{equation*}
Notice that~\eqref{e.nuupbound} and~\eqref{e.muupbound} imply that $D_n$ is bounded and Corollary~\ref{c.munucomparison} implies that $D_n$ has a small negative part: for every $n\in\N$,
\begin{equation}
\label{e.Dnnonneg}
-C3^{-\frac n2} \leq D_n \leq C.
\end{equation}
As we will see in Step~4, below,~$D_n$ controls $\E\left[\sup_{q\in\Rd} |q|^{-2}\left( \omega(\cu_n,q) \right)_+\right]$ in the sense that~\eqref{e.Dnwts} implies~\eqref{e.omegadecayrate}. Therefore our goal is to show that, for $\alpha$ and $C$ as in the statement of the lemma, 
\begin{equation} 
\label{e.Dnwts}
D_n \leq C3^{-n\alpha}.
\end{equation}

\smallskip

\emph{Step 1.} We show that the existence of $c(d,\lambda,\p)>0$, $C(d,\lambda,\p)<\infty$ and $N_0(d,\lambda,\p)\in\N$ such that, for every $n\in\N$ with $n\geq N_0$, 
\begin{equation}
\label{e.Dincrement}
D_{n+1} \leq D_n - c\tau _n + C3^{-\frac{n}4}. 
\end{equation}
By the definition of the matrix $\ahom_{\cu_{n+1}}$, we have that, for each $p\in\Rd$, the map
\begin{equation*} \label{}
q\mapsto -\E \left[ \mu(\cu_{n+1},q) \right] - p\cdot q = \frac12 q \cdot \ahom^{-1}_{\cu_{n+1}} q - p\cdot q
\end{equation*}
achieves its minimum at the point $q=\ahom_{\cu_{n+1}} p$. By quadratic response and the bounds on $\E\left[ \mu(\cu_{n+1},q) \right]$ implied by~\eqref{e.muupbound}, this implies the existence of $C(d,\lambda,\p)<\infty$ such that, for every $q\in\Rd$,
\begin{align} 
\label{e.minimuminq}
\lefteqn{
 -\E \left[ \mu(\cu_{n+1},\ahom_{\cu_{n+1}} \e_i) \right] - \e_i \cdot \ahom_{\cu_{n+1}} \e_i
 } \qquad & \\ \notag
& \leq -\E \left[ \mu(\cu_{n+1},q) \right] - \e_i \cdot q  \\ \notag
& \leq  -\E \left[ \mu(\cu_{n+1},\ahom_{\cu_{n+1}} \e_i) \right] - \e_i \cdot \ahom_{\cu_{n+1}} \e_i + C \left|q-\ahom_{\cu_{n+1}} \e_i\right|^2. 
\end{align}
Using the first line of~\eqref{e.minimuminq} with $q = \ahom_{\cu_{n}} \e_i$ yields
\begin{align*} \label{}
D_{n+1} 
& = \sum_{i=1}^d \left( \E \left[\left( \nu\left(\cu_{n+1},\e_i\right) - \mu(\cu_{n+1},\ahom_{\cu_{n+1}} \e_i) - \e_i \cdot \ahom_{\cu_{n+1}}\e_i \right)  \right] \right) \\
& \leq \sum_{i=1}^d \left( \E \left[\left( \nu\left(\cu_{n+1},\e_i\right) - \mu(\cu_{n+1},\ahom_{\cu_{n}} \e_i) - \e_i \cdot \ahom_{\cu_{n}}\e_i \right)  \right] \right) \\
& = D_n + \sum_{i=1}^d \left( \E \left[\nu\left(\cu_{n+1},\e_i\right) \right] - \E \left[\nu\left(\cu_{n},\e_i\right) \right] \right) \\
& \qquad\qquad -  \sum_{i=1}^d\left( \E \left[\mu(\cu_{n+1},\ahom_{\cu_{n}} \e_i) \right] - \E \left[\mu(\cu_{n},\ahom_{\cu_{n}} \e_i) \right] \right) \\
& \leq D_n - c\tau_n + C 3^{-\frac{n}4},
\end{align*}
where in the last line we used the bounds~\eqref{e.almostmonotone},~\eqref{e.almostmonotonenu} and~\eqref{e.ahominvbound}, which hold for all sufficiently large $n$ depending only on $(d,\lambda,\p)$. This completes the proof of~\eqref{e.Dincrement}. 

\smallskip

In view of the form of the right side of~\eqref{e.omegadecay} as well as the result of the previous step, it is natural to modify $D_n$ slightly by defining, for every $n\geq N_0$, 
\begin{equation*} \label{}
\tilde{D}_n:= 3^{-\frac n2} \sum_{k=N_0}^n 3^{\frac k2} D_k. 
\end{equation*}
Notice that $\tilde{D}_n$ is, up to a constant, a weighted average of $D_{N_0},\ldots,D_n$ and, in particular, by~\eqref{e.Dnnonneg}, we have that 
\begin{equation}
\label{e.DntildeDn}
\tilde{D}_n = D_n + 3^{-\frac n2} \sum_{k=N_0}^{n-1} 3^{\frac k2} D_k \geq D_n - C n 3^{-\frac n2} \geq D_n - C 3^{-\frac n4}.
\end{equation}
Therefore, rather than~\eqref{e.Dnwts}, we may prove the stronger bound 
\begin{equation} 
\label{e.tildeDnwts}
\tilde{D}_n \leq C3^{-n\alpha}.
\end{equation}

\smallskip

\emph{Step 2.} We show that there exists $\theta(d,\lambda,\p) \in \left[\frac12, 1\right)$ and $C(d,\lambda,\p)<\infty$ such that, for every $n\in\N$ with $n \geq N_0$,
\begin{equation}
\label{e.tildeDngettingiteron}
\tilde{D}_{n+1} \leq \theta \tilde{D}_n + C3^{-\left(\frac{2}{2d+1}\right)n}.
\end{equation}
Using~\eqref{e.Dnnonneg} and~\eqref{e.Dincrement}, we find that 
\begin{align}
\label{e.Dtildeincrem}
\tilde{D}_n - \tilde{D}_{n+1} 
& =  3^{-\frac n2} \sum_{k=N_0}^n 3^{\frac k2} \left(D_k - D_{k-1} \right) - 3^{\frac12(N_0-(n+1))} D_{N_0} \\ \notag
& \geq c3^{-\frac n2} \sum_{k=N_0}^n 3^{\frac k2} \left( \tau_k - C 3^{-\frac n4} \right)- C 3^{-\frac n2} \geq c3^{-\frac n2} \sum_{k=N_0}^n 3^{\frac k2} \tau_k - C 3^{-\frac n4}.
\end{align}
Next, we apply Lemma~\ref{l.omegadecay}, which tells us that 
\begin{equation*} \label{}
D_k \leq C\left( 3^{-\left(\frac{2}{2d+1}\right)k} + \sum_{l=0}^k 3^{l-k}\tau_l \right) 
\leq 
C\left( 3^{-\left(\frac{2}{2d+1}\right)k} + \sum_{l=N_0}^n 3^{l-k}\tau_l \right).
\end{equation*}
Summing this over $k\in \{N_0,\ldots,n\}$ gives
\begin{align*}
\tilde{D}_n 
& \leq C3^{-\frac n2} \sum_{k=N_0}^n 3^{\frac k2} \left( 3^{-\left(\frac{2}{2d+1}\right)k} + \sum_{l=N_0}^k 3^{l-k}\tau_l \right) 
\\
& = C3^{-\frac n2}  \sum_{k=N_0}^n \sum_{l=N_0}^k  3^{-\frac k2} 3^{l}\tau_l + C3^{-\left(\frac{2}{2d+1}\right)n} \\
& =  C3^{-\frac n2} \sum_{l=N_0}^n \sum_{k=l}^n   3^{-\frac k2} 3^{l}\tau_l + C3^{-\left(\frac{2}{2d+1}\right)n} \\
& \leq C 3^{-\frac n2} \sum_{l=N_0}^n 3^{\frac l2}\tau_l + C 3^{-\left(\frac{2}{2d+1}\right)n}.
\end{align*}
Combining the previous displays with~\eqref{e.Dtildeincrem} gives 
\begin{equation*}
\tilde{D}_{n+1} \leq (1-c)\tilde{D}_n + C3^{-\left(\frac{2}{2d+1}\right)n}.
\end{equation*}
This completes the proof of~\eqref{e.tildeDngettingiteron}. 

\smallskip

\emph{Step 3.} We complete the proof of~\eqref{e.tildeDnwts}. By an interation of~\eqref{e.tildeDngettingiteron} we get, for every $n\geq N_0$,
\begin{equation*} \label{}
\tilde{D}_n \leq \left( \sum_{k=N_0}^n \theta^k 3^{(\frac{2}{2d-1})(k-n)} \right) \tilde{D}_{N_0}. 
\end{equation*}
Taking $\theta$ closer to~$1$, if necessary, so that $\theta > \left(\frac13\right)^{\frac{2}{2d-1}}$, we get that each term in the sum is at most $\theta^n$. Using also $\tilde{D}_{N_0} \leq C$, we therefore obtain
\begin{equation*} \label{}
\tilde{D}_n \leq Cn\theta^n \leq C\theta^{\frac n2}. 
\end{equation*}
Taking $\alpha:= \log 3 / 2\left| \log \theta \right|$ so that $\theta^{\frac12}=3^{-\alpha}$ yields the desired bound~\eqref{e.tildeDnwts}. 

\smallskip

\emph{Step 4.} We complete the proof of~\eqref{e.omegadecayrate}. First, we observe that, due to Corollary~\ref{e.munucomparison}, we have the function 
\begin{equation*}
q\mapsto \omega(\cu_n,q)  + C3^{-\frac n2} |q|^2. 
\end{equation*}
is nonnegative and quadratic and hence convex. It follows that 
\begin{align*}
\sup_{q\in\Rd} \frac1{|q|^2} \left(\omega(\cu_n,q)\right)_+
 \leq \sup_{q\in\Rd} \frac1{|q|^2} \left(\omega(\cu_n,q)  + C3^{-\frac n2} |q|^2 \right) 
\leq \sum_{i=1}^d \omega(\cu_n,e_i)  + C3^{-\frac n2}.
\end{align*}
Next we observe that~\eqref{e.munucomparison} and~\eqref{e.Dnwts} (which we recall is a consequence of~\eqref{e.DntildeDn} and~\eqref{e.tildeDnwts}) imply that
\begin{equation*}
\left| D_n \right| \leq C3^{-n\alpha}. 
\end{equation*}
Using this and~\eqref{e.Dincrement}, we find that 
\begin{equation*}
\tau_n \leq C3^{-\frac n2} + C (D_n - D_{n+1}) \leq C 3^{-n\alpha}. 
\end{equation*}
According to the previous line and~\eqref{e.ahominvclose}, we deduce that 
\begin{equation*} \label{}
\left| \ahom_{\cu_n} - \ahom \right| \leq C3^{-n\alpha}.
\end{equation*}
Therefore, by the previous line,~\eqref{e.munucomparison} and~\eqref{e.minimuminq} applied with $q=\ahom \e_i$, we find 
\begin{align*}
\lefteqn{
 \left|\E\left[  \omega(\cu_n,\e_i)_+ \right] -\E\left[   \nu\left(\cu_n,\e_i\right) - \mu(\cu_n,\ahom_{\cu_n} \e_i) - \e_i \cdot \ahom_{\cu_n}\e_i  \right]   \right| 
} \qquad & \\
& \leq  \left|\E\left[  \omega(\cu_n,\e_i) \right] - \E\left[   \nu\left(\cu_n,\e_i\right) - \mu(\cu_n,\ahom_{\cu_n} \e_i) - \e_i \cdot \ahom_{\cu_n}\e_i \right]     \right| 
 + C3^{-\frac n2} \\
& \leq C \left| \ahom_{\cu_n} - \ahom \right|^2 + C3^{-\frac n2} \\
& \leq C 3^{-n\alpha}. 
\end{align*}
Combining the above and using~\eqref{e.Dnwts} again, we get 
\begin{align*}
\E \left[ \sup_{q\in\Rd} \frac1{|q|^2} (\omega(\cu_n,q))_+ \right] 
&  \leq  \sum_{i=1}^d \E \left[ (\omega(\cu_n,\e_i) )_+ \right] 
+C3^{-\frac n2} \\
& \leq \sum_{i=1}^d  \E \left[ \nu\left(\cu_n,\e_i\right) - \mu(\cu_n,\ahom_{\cu_n} \e_i) - \e_i \cdot \ahom_{\cu_n}\e_i   \right]  + C 3^{-n\alpha} \\
& = D_n + C3^{-n\alpha} \\
& \leq C 3^{-n\alpha}.
\end{align*}
This completes the argument. 
\end{proof}

To complete the proof of Proposition~\ref{p.subadd}, we need to show that the control over the expectation of $\omega$ given in the previous lemma can be enhanced, using independence, to control over exponential moments of~$\omega$. This is a consequence of the following lemma. 

\begin{lemma}
\label{l.SIupgrade}
There exist exponents $s(d)>0$ and $\alpha(s,\lambda,\p)>0$ and a constant $C(s,d,\lambda,\p)<\infty$ such that, for every $\cu\in \T$,
\begin{equation}
 \sup_{q\in \Rd} \frac{1}{|q|^2} \left| \omega(\cu,q) \right| = \O_{s}\left( C\left(\size(\cu)\right)^{-\alpha} \right). 
\end{equation}
\end{lemma}
\begin{proof}
The argument is an application of the exponential moment method and subadditivity, modified to take care of the fact that $\omega$ is not a bounded random variable. It is enough to prove the result for cubes of the form $\cu=\cu_m$ for $m\in\N$ by approximate stationarity, see~\eqref{e.munustat} and~\eqref{e.localization}. We fix $m\in\N$ and set $n:= \left\lceil \frac12 m \right\rceil$ and $k:=\left\lceil \frac14 m \right\rceil$. Throughout the argument, we let $s$ denote a positive exponent depending only on~$d$ which may vary in each occurrence. Likewise $\alpha$ is a positive exponent depending only on $(d,\lambda,\p)$ which may vary. 

\smallskip

We denote, for each $\cu\in\T$,
\begin{equation*} \label{}
\rho(\cu):= \sup_{q\in\Rd} \frac1{|q|^2} \left(\omega(\cu,q)\right)_+.
\end{equation*}
To prepare for the use of independence, we also let $\rho_{\mathrm{loc}}^{(k)}$ and $\omega^{(k)}_{\mathrm{loc}}$ denote the same quantities as $\rho$ and $\omega$ except with the local quantities $\mu^{(k)}_{\mathrm{loc}}$ and $\nu^{(k)}_{\mathrm{loc}}$ in place of $\mu$ and $\nu$ in their definitions. That is, 
\begin{equation*} \label{}
\omega^{(k)}_{\mathrm{loc}}(\cu,q)
:= 
\nu^{(k)}_{\mathrm{loc}}\left( \cu ,\ahom^{-1}q\right) -  \mu^{(k)}_{\mathrm{loc}}(\cu,q) - q\cdot \ahom^{-1}q
\end{equation*}
and
\begin{equation*} \label{}
\rho_{\mathrm{loc}}^{(k)}(\cu):= \sup_{q\in\Rd} \frac1{|q|^2} \left(\omega^{(k)}_{\mathrm{loc}}(\cu,q)\right)_+.
\end{equation*}
By~\eqref{e.munuloc} and Proposition~\ref{p.localization}, we see that 
\begin{equation} \label{e.rhomeas}
\mbox{$\rho_{\mathrm{loc}}^{(k)}(\cu)$ and $\omega^{(k)}_{\mathrm{loc}}(\cu,q)$ is $\F(\cu)$--measurable}
\end{equation}
and
\begin{equation} 
\label{e.omegalocal}
\left| \omega(\cu,q) - \omega^{(k)}_{\mathrm{loc}} (\cu,q)\right| 
+ \left| \rho(\cu) - \rho_{\mathrm{loc}}^{(k)}(\cu) \right| 
\leq \O_s\left( C3^{-\frac k4} \right). 
\end{equation}

\smallskip

\emph{Step 1.} By removing a bad event of small probability, we essentially reduce to the case that $\rho(z+\cu_n)$ is bounded on the subcubes $z+\cu_n \subseteq \cu_m$. 

\smallskip

According to Lemma~\ref{l.upperbounds}, its consequence~\eqref{e.muupboundwhip} and the Markov inequality, there exists $C_1(d,\p,\lambda)<\infty$ such that, for every $z\in 3^n\Zd$,
\begin{equation*} \label{}
\P \left[ \rho(z+\cu_n) > C_1 \right] \leq C \exp\left( - \frac1C 3^{\frac n{2d+1}} \right). 
\end{equation*}
By~\eqref{e.omegalocal}, this implies that 
\begin{align*} \label{}
\P \left[ \rho^{(k)}_{\mathrm{loc}}(z+\cu_n) > C_1 \right] 
& \leq C \exp\left( - \frac1C 3^{\frac n{2d+1}} \right) + C\exp\left( -\frac1C 3^{\frac{ks}{4}} \right) \\
& \leq C\exp\left( -\frac1C 3^{m\alpha} \right). 
\end{align*}
Thus
\begin{equation*} \label{}
\P \left[ \sup_{z\in 3^n\Zd \cap \cu_m} \rho^{(k)}_{\mathrm{loc}}(z+\cu_n) > C_1 \right] \leq C 3^{d(m-n)}  \exp\left( - \frac1C 3^{m\alpha} \right)  \leq C\exp\left( -\frac1C 3^{m\alpha} \right). 
\end{equation*}
For each $z\in 3^n\Zd$, denote the events
\begin{equation*} \label{}
G_z:= \left\{ \rho^{(k)}_{\mathrm{loc}}(z+\cu_n) \leq C_1 \right\}
\end{equation*}
and 
\begin{equation*} \label{}
H:=  \left\{ \sup_{z\in 3^n\Zd \cap \cu_m} \rho^{(k)}_{\mathrm{loc}}(z+\cu_n) > C_1 \right\} = \Omega \setminus \left(\bigcap_{z\in 3^n\Zd\cap \cu_m} G_z \right).
\end{equation*}
Note that $G_z\in \F(z+\cu_n)$. 
Note that the estimate above gives a bound for $\indc_H$: 
\begin{equation*} \label{}
\indc_H 
\leq \O_{1}\left( C3^{-m\alpha} \right),
\end{equation*}
and thus we can bound $\omega$ from above on the ``bad" event $H$:
\begin{equation}
\label{e.boundonH}
\rho (\cu_m)\, \indc_{H} \leq \O_{s}\left( C \right) \cdot \O_{1}\left( C3^{-m\alpha} \right) \leq \O_{s}\left( C3^{-m\alpha} \right).
\end{equation}

\smallskip

\emph{Step 2.} The concentration argument. The claimed estimate is\begin{equation}
\label{e.bigestrho}
\rho(\cu_m) \indc_{\Omega\setminus H} \leq \O_{s}\left( C3^{-m\alpha} \right).
\end{equation}
We begin by noticing that~\eqref{e.musubadd},~\eqref{e.nusubadd} and~\eqref{e.omegalocal} give us the approximate subadditivity bound
\begin{align}
\label{e.rhosubadd}
\rho(\cu_m) \indc_{\Omega\setminus H} 
& \leq 3^{-d(m-n)} \sum_{z\in 3^n\Zd\cap \cu_m}
\rho^{(k)}_{\mathrm{loc}}(z+\cu_n) \indc_{G_z} + \O_{s}\left( C3^{-\frac k4} \right) \\ \notag 
& \leq 3^{-d(m-n)} \sum_{z\in 3^n\Zd\cap \cu_m}
\rho^{(k)}_{\mathrm{loc}}(z+\cu_n) \wedge C_1 + \O_{s}\left( C3^{-\frac k4} \right).
\end{align}
Note that for $z,z' \in 3^n\Zd\cap \cu_m$ with $z \neq z'$, 
\begin{equation}
\label{e.expdist}
\rho^{(k)}_{\mathrm{loc}}\left(z + \cu_{n}\right)
\quad \mbox{and} \quad  \rho^{(k)}_{\mathrm{loc}}\left(z' + \cu_{n}\right)
\quad \mbox{are $\P$--independent.}
\end{equation}
We now fix $t>0$ and compute
\begin{align*}
\lefteqn{
 \log \E \left[ \exp\left( t \sum_{z\in 3^n\Zd\cap \cu_m} \rho^{(k)}_{\mathrm{loc}}(z+\cu_n) \wedge C_1\right)\right] 
 } \qquad & \\
& = \log \E \left[ \prod_{z\in 3^n\Zd\cap \cu_m}  \exp\left( t \rho^{(k)}_{\mathrm{loc}}\left(z +\cu_n\right)\wedge C_1\right)\right] & \\
& \leq  \sum_{z\in 3^n\Zd\cap \cu_m} \log  \E\left[ \ \exp\left( t\rho^{(k)}_{\mathrm{loc}}\left(z +\cu_n\right)\wedge C_1\right) \right] & \mbox{(by \eqref{e.expdist})} \\
& = 3^{d(m-n)} \log \E \left[  \exp\left( t\rho^{(k)}_{\mathrm{loc}}(\cu_n)\wedge C_1\right)\right] & \mbox{(by~\eqref{e.munustat}).}
\end{align*}   
We take $t:= 1/K$ and estimate the last term using the elementary inequalities
\begin{equation*}\label{}
\left\{ \begin{aligned}
& \exp(s) \leq 1+Cs && \mbox{for all} \ 0\leq s \leq C_1, \\
& \log(1+s) \leq s && \mbox{for all} \ s\geq 0,
\end{aligned} \right.
\end{equation*}
to get
\begin{align*}  
3^{-d(m-n)} \log \E \left[ \exp\left(  C^{-1}\sum_{z\in 3^n\Zd\cap \cu_m} \rho^{(k)}_{\mathrm{loc}}(z+\cu_n)\wedge C_1\right) \right] \leq  C \E\left[ \rho^{(k)}_{\mathrm{loc}}(\cu_n) \right].
\end{align*}
Applying~\eqref{e.omegadecayrate} and~\eqref{e.omegalocal}, we find that 
\begin{equation*} \label{}
\E\left[ \rho^{(k)}_{\mathrm{loc}}(\cu_n) \right] \leq C3^{-n\alpha} \leq C3^{-\frac{m\alpha}2}. 
\end{equation*}
The previous two lines and Chebyshev's inequality imply that 
\begin{equation*}
\P \left[ 3^{-d(m-n)}\sum_{z\in 3^n\Zd\cap \cu_m} \rho^{(k)}_{\mathrm{loc}}(z+\cu_n) \wedge C_1  > t \right] \leq \exp\left( -c3^{d(m-n)}\left( t - C3^{-\frac{m\alpha}2} \right) \right).
\end{equation*}
This implies that 
\begin{equation*}
 3^{-d(m-n)}\sum_{z\in 3^n\Zd\cap \cu_m} \rho^{(k)}_{\mathrm{loc}}(z+\cu_n) \wedge C_1 \leq \O_1\left( C3^{-\frac{m\alpha}2}\right).
\end{equation*}
Combined with~\eqref{e.rhosubadd}, we get 
\begin{equation*}
\rho(\cu_m) \indc_{\Omega\setminus H}  \leq  \O_1\left( C3^{-\frac{m\alpha}2}\right) + \O_{s}\left( C3^{-\frac k2} \right) \leq \O_{s}\left( C3^{-m\alpha} \right),
\end{equation*}
which is~\eqref{e.bigestrho}.

\smallskip

\emph{Step 3.} We complete the argument. By combining the previous steps, we get
\begin{align*}
\rho(\cu_m) 
 = \rho(\cu_m) \indc_{\Omega\setminus H} +\rho(\cu_m) \indc_{H}  \leq \O_s\left( C3^{-\frac{m\alpha}2} \right). 
\end{align*}
We also recall that, by~\eqref{e.munucomparison}, we have 
\begin{equation*}
\sup_{q\in\Rd} \frac{1}{|q|^2} \omega(\cu_m,q)  \geq - \O_{\frac{2}{2d-1}}\left( C3^{-\frac m2} \right). 
\end{equation*}
The previous two inequalities yield the lemma after we shrink~$\alpha$. 
\end{proof}

\begin{proof}[{Proof of Proposition~\ref{p.subadd}}]
Applying Lemma~\ref{l.minimalset} and then Lemma~\ref{l.SIupgrade}, we obtain
\begin{multline*}
\sup_{q\in \Rd} \frac{1}{|q|^4} \left| \mu(\cu,q) - \frac12q\cdot \ahom^{-1} q \right|^2 + \sup_{p\in\Rd} \frac{1}{|p|^4} \left| \nu(\cu,p) - \frac12p\cdot \ahom p\right|^2 
\\ \leq C \left( \sup_{e \in \partial B_1} \omega(\cu,e)_+ + \O_{\frac{1}{2d+1}}\left( C\size(\cu)^{-\frac12} \right) \right) 
 \leq  \O_{s}\left( C\size(\cu)^{-\alpha} \right).
\end{multline*}
Taking square roots and shrinking~$\alpha$ gives the desired bound for some~$s(d,\lambda,\p)>0$. Using~\eqref{e.improves} to interpolate this results with the bounds~\eqref{e.muupbound} and~\eqref{e.nuupbound}, we can allow the exponent~$s$ to depend only on~$d$ by further shrinking~$\alpha$.
\end{proof}

We conclude this section by introducing $\mathcal{N}$ which will be an important tool in the proofs of Section~6.
\begin{definition} \label{definition.rvN}
For $m \in \N$, we define the random variable $\mathcal{N}$ by
\begin{multline*}
\mathcal{N} := \sup \Bigg\{ 3^m  \,:\, m\in\N, \, n= \left\lceil \frac m4 \right\rceil, \ k = \left\lceil \frac m8 \right\rceil,  \\   \sup_{z \in 3^k \Zd \cap \cu_m} \left( \sup_{q\in \Rd} \frac{1}{|q|^2} \left| \mu(z + \cu_{n},q) - \frac12q\cdot \ahom^{-1} q \right| + \sup_{p\in\Rd} \frac{1}{|p|^2} \left| \nu(z + \cu_{n} ,p) - \frac12p\cdot \ahom p\right| \right) \\
\geq C3^{-m\alpha}  \Bigg\}.
\end{multline*}
Here the exponent $\alpha=\alpha(d,\p, \lambda) > 0$ is defined in the proof of the following proposition, and may be smaller than the one in~Proposition~\ref{p.subadd}. 
\end{definition}

\begin{proposition}
\label{p.boundN}
There exist~$s(d,\p,\lambda)>0$ and~$C(d,\p, \lambda) < \infty$ such that 
\begin{equation} 
\label{Nisbounded}
\mathcal{N} \leq \O_s(C).
\end{equation}
\end{proposition}
\begin{proof}
First we prove, for each fixed $m\in\N$ and $z\in 3^k\Zd$, an estimate of the form
\begin{equation*}
\P \left[ \sup_{q\in \Rd} \frac{1}{|q|^2} \left| \mu(z + \cu_{n},q) - \frac12q\cdot \ahom^{-1} q \right|  \geq C3^{-m\alpha} \right] \leq C \exp\left( - C^{-1} 3^{-\alpha m} \right),
\end{equation*}
where $n:=\left\lceil \frac m4\right\rceil$. 
To do this we apply Proposition~\ref{p.localization} and use the stationarity property~\eqref{e.munustat} and~\eqref{e.localization} to obtain
\begin{align*}
\lefteqn{ 
\P \left[ \sup_{q\in \Rd} \frac{1}{|q|^2} \left| \mu(z + \cu_{n},q) - \frac12q\cdot \ahom^{-1} q \right|  \geq C3^{-m\alpha} \right] 
} \qquad & 
\\
& \leq  \P \left[ \sup_{q\in \Rd} \frac{1}{|q|^2} \left| \mu^{(k)}_\mathrm{loc}(z + \cu_{n},q) - \frac12q\cdot \ahom^{-1} q \right|  \geq \frac C2 3^{-m\alpha} \right] + C \exp\left( - C^{-1} 3^{-\alpha m} \right) 
\\
& \leq \P \left[ \sup_{q\in \Rd} \frac{1}{|q|^2} \left| \mu^{(k)}_\mathrm{loc}( \cu_{n},q) - \frac12q\cdot \ahom^{-1} q \right|  \geq \frac C2 3^{-m\alpha} \right] + C \exp\left( - C^{-1} 3^{-\alpha m} \right) 
\\
& \leq \P \left[ \sup_{q\in \Rd} \frac{1}{|q|^2} \left| \mu_\mathrm{loc}( \cu_{n},q) - \frac12q\cdot \ahom^{-1} q \right|  \geq \frac C4 3^{-m\alpha} \right] + C \exp\left( - C^{-1} 3^{-\alpha m} \right) 
\\
& \leq C \exp\left( - C^{-1} 3^{-\alpha m} \right).
\end{align*}
This results and the fact that $3^m \geq Cn$ gives to
\begin{align*}
\lefteqn{ \P \left[ \sup_{z \in 3^k \Zd \cap \cu_m } \sup_{q\in \Rd} \frac{1}{|q|^2} \left| \mu(z + \cu_{n},q) - \frac12q\cdot \ahom^{-1} q \right|  \geq C3^{-m\alpha} \right] } \qquad & \\
				& \leq \sum_{z \in 3^k \Zd \cap \cu_m } \P \left[ \sup_{q\in \Rd} \frac{1}{|q|^2} \left| \mu^{(k)}_\mathrm{loc}(z + \cu_{n},q) - \frac12q\cdot \ahom^{-1} q \right|  \geq \frac C2 3^{-m\alpha} \right] \\
				& \leq 3^{d(m-k)} C \exp\left( - C^{-1} 3^{-\alpha m} \right) \\
				& \leq  C \exp\left( - C^{-1} 3^{-\alpha m} \right).
\end{align*}
Similarly, we obtain
\begin{equation*}
\P \left[ \sup_{z \in 3^k \Zd \cap \cu_m } \sup_{p\in\Rd} \frac{1}{|p|^2} \left| \nu(z + \cu_{n} ,p) - \frac12p\cdot \ahom p\right| \geq C3^{-m\alpha} \right] \leq C \exp\left( - C^{-1} 3^{-\alpha m} \right).
\end{equation*}
The estimate~\eqref{Nisbounded} is now a consequence of Lemma~\ref{l.OtoM} with  $X_m$ defined as the indicator function of the event
\begin{multline*}
\left\{ \sup_{z \in 3^k \Zd \cap \cu_m} \left( \sup_{q\in \Rd} \frac{1}{|q|^2} \left| \mu(z + \cu_{n},q) - \frac12q\cdot \ahom^{-1} q \right| \right.  \right. \\ \left. \left. + \sup_{p\in\Rd} \frac{1}{|p|^2} \left| \nu(z + \cu_{n} ,p) - \frac12p\cdot \ahom p\right| \right) \geq C3^{-m\alpha} \right\}. \qedhere
\end{multline*}
\end{proof}

\section{Homogenization error estimates for the Dirichlet problem}
\label{s.dirichlet}

In this section, we pass from control on the subadditive quantities~$-\mu$ and~$\nu$ to quenched control on the error in homogenization for the Dirichlet problem. Combined with Proposition~\ref{p.subadd}, this allows us to complete the proof of Theorem~\ref{t.homog}. The arguments here are entirely deterministic in the sense that they produce an estimate for the homogenization error in terms of the coarseness of the partition~$\Pa$ and the convergence of~$-\mu$ and~$\nu$ in mesoscopic cubes. In particular, we are not using theory developed in the previous section. The arguments here are a variation of similar ones in~\cite{AS}. 

\smallskip

In this section, we abuse notation by using the symbol $\cu_m$ to also denote the \emph{continuum} cube 
\begin{equation*} \label{}
\left[ -\frac12 (3^m-1) ,\frac12 (3^m-1)\right]^d \subseteq \Rd. 
\end{equation*}
It will be made clear from the context whether $\cu_m$ refers to the continuum cube or the discrete one. Moreover in this section we will use
\begin{equation*}
|\cu_m| = \card (\cu_m) = 3^{dm}
\end{equation*}
which is slightly different from $\mathrm{Leb}(\cu_m) = \left( 3^m - 1 \right)^d$. We will also write $\fint_{\cu_m} = \frac{1}{|\cu_m|} \int_{\cu_m} = 3^{-dm} \int_{\cu_m}$.
We further abuse notation by extending the coarsened function $\left[ u \right]_\Pa$ to be defined on a continuum domain by taking it to be constant on each unit cube of the form $z + \left( - \frac 12 , \frac 12 \right]^d$ with $z\in\Zd$. To avoid confusion, here we will use the symbols $\int$ and $\fint$ only to denote integration with respect to Lebesgue measure on~$\Rd$ and write sums with $\sum$. 

\smallskip

We fix, once and for all, a positive integer $m\in\N$ and a function $u \in \A_*(\cu_m)$. We also fix an exponent $p>2$ and set 
\begin{equation*}
M := \left( \frac{1}{|\cu_m|}\sum_{x \in \C_*(\cu_m)} |\nabla u \indc_{\a \neq 0}|^p (x) \right)^{\frac 1p}.
\end{equation*}
To define $u_{\mathrm{hom}}$ on the continuum cube $\cu_m$, we first define
\begin{equation*} \label{}
\tilde{w}(x) := \left\{ 
\begin{aligned}
& u(x)  & \mbox{if} \ x\in \C_*(\cu_m) \cap \partial \cu_m, \\
& \left[ u \right]_{\Pa} (x) & \mbox{otherwise.} 
\end{aligned}
\right.
\end{equation*}
We then extend $\tilde w$ to the continuum by taking it to be constant on unit cubes of the form $z + \left(-\frac 12, \frac 12\right]^d$ with $z\in\Zd$. We then make it smooth by convolving it with a smooth bump function $\rho \in C_c^\infty(\Rd,\R)$ which is supported in $\left(-\frac 12, \frac 12\right)^d$ and has unit mass. Call the function obtained in this way $w\in C^\infty(\cu_m)$ and notice that, 
\begin{equation*}
\forall x \in \Zd \cap \cu_m, \, w(x) = \tilde w(x).
\end{equation*}
We take $u_{\mathrm{hom}} \in H^1(\cu_m)$ to be the solution of the homogenized Dirichlet problem
\begin{equation*}
\left\{
  \begin{aligned}
 &  - \nabla \cdot \left(\ahom \nabla \uhom \right) = 0 & \mbox{in}  & \ \cu_m, \\
  &  \uhom =  w & \mbox{on} & \ \partial \cu_m.\\
  \end{aligned}
\right.
\end{equation*}
The purpose of this section is to prove the following proposition which, in view of our setup in this section, implies Theorem~\ref{t.homog}. 

\smallskip

(Recall that the random scales $\M_t$ and $\mathcal{N}$ are given in Proposition~\ref{p.minimalscales} and Definition~\ref{definition.rvN}, respectively, and the partition $\Qa$ is the one for the Meyers estimate (see Definition~\ref{d.partQa}) which is coarser than $\Pa$.)

\begin{proposition} \label{prop.dirichlet}
There exist $t := t(d,\p,\lambda,p) < + \infty$, $\alpha := \alpha(d,\p,\lambda,p)> 0$ and $C := C(d,\p,\lambda,p) < + \infty$ such that  $3^m\geq \mathcal{N} \vee \M_t(\Qa)$ implies that 
\begin{equation*}
\frac{1}{|\cu_m|} \sum_{x \in \C_*(\cu_m)} |u(x) - \uhom(x)|^2 \leq C M^2 \size(\cu_m)^{(2- \alpha)}.
\end{equation*}
\end{proposition}

The main idea in the proof of the proposition is to construct a function $\tilde{u} \in w  + H^1_0(\cu_m)$ which satisfies the property that its continuous homogenized energy is smaller that the discrete heterogeneous energy of $u$. This is explicited in Lemma~\ref{l.lemmaA1}. In the other direction we will construct a function $\tuhom : \C_*(\cu_m) \rightarrow \R$ which statisfies the following three properties:
\begin{enumerate}
\item $\tuhom$ is equal to $w $ on $\C_*(\cu_m) \cap \partial \cu_m$,
\item The discrete heterogenous energy of $\tuhom$ is smaller than the continuous homogenized energy of $\uhom$,
\item The discrete $L^2$-norm of $\tuhom - \uhom$ is small.
\end{enumerate}
This is specified in Lemma~\ref{l.lemmaA2}. We will then deduce Proposition~\ref{prop.dirichlet} from these results.

\begin{lemma} \label{l.lemmaA1}
There exists $t := t(d,\p,\lambda,p) < + \infty$, $\alpha := \alpha(d,\p,\lambda,p)> 0$ and $C := C(d,\p,\lambda,p) < + \infty$ such that, in the case that $3^m \geq  \mathcal{N} \vee \M_t ( \mathcal{Q})$, there exist a function $\tilde{u} \in w  + H^1_0(\cu_m)$ satisfying the energy bound
\begin{equation} \label{e.lemmaA1}
\fint_{\cu_m} \nabla \tilde{u}(x) \cdot \ahom \nabla \tilde{u}(x) \, dx\leq \frac{1}{|\cu_m|}\left\langle \nabla u, \a \nabla u \right\rangle_{\C_*(\cu_m)} +CM^2 3^{-m\alpha }.
\end{equation}
\end{lemma}

\begin{lemma}\label{l.lemmaA2}
There exists $t := t(d,\p,\lambda,p) < + \infty$, $\alpha := \alpha(d,\p,\lambda,p)> 0$ and $C := C(d,\p,\lambda,p) < + \infty$ such that, in the case that $3^m \geq  \mathcal{N} \vee \M_t ( \mathcal{Q})$, there exists a function $\tuhom : \C_*(\cu_m) \rightarrow \R$ satisfying the following three properties:
\begin{enumerate}
\item[(i)] Boundary condition:
\begin{equation*}
\forall x \in \C_*(\cu_m) \cap \partial \cu_m, \quad \tuhom(x) = u(x).
\end{equation*}
\item[(ii)] Energy estimate:
\begin{multline*} \label{e.lemmaA2}
	\frac{1}{|\cu_m|} \left\langle \nabla \tuhom, \a(x) \nabla \tuhom\right\rangle_{\C_*(\cu_m)} \\ \leq \fint_{\cu_m} \left( \nabla \uhom(x) \cdot \ahom \nabla \uhom(x) \right)\, dx +C M^2  3^{-m\alpha},
\end{multline*}
\item[(iii)] $L^2$ estimate:
\begin{equation*}
  \frac{3^{-2m}}{|\cu_m|} \sum_{x \in \C_*(\cu_m)} (\tuhom(x)- \uhom(x))^2 \leq  C  M^2 3^{-m\alpha}.
\end{equation*}
\end{enumerate}
\end{lemma}

We now prove Proposition~\ref{prop.dirichlet} and postpone the proof of these Lemmas.

\begin{proof}[Proof of Proposition~\ref{prop.dirichlet}]
Using that $u \in \A_*(\cu_m)$, the definition of $\uhom$, Lemma~\ref{l.lemmaA1} and Property (ii) of Lemma~\ref{l.lemmaA2}, we have the following sequence of inequalities
\begin{align} \label{uhomtutuhom.est}
 \fint_{\cu_m} \nabla \uhom(x) \cdot \ahom \nabla \uhom(x)  \, dx & \leq  \fint_{\cu_m} (\nabla \tilde{u}  \cdot \ahom \nabla \tilde{u}) \, dx \notag \\
												& \leq  \frac{1}{|\cu_m|} \left\langle  \nabla u, \a \nabla u \right\rangle_{\cu_m}   + C 3^{-m\alpha}  \\
												& \leq \frac{1}{|\cu_m|} \left\langle  \nabla  \tuhom, \a \nabla  \tuhom \right\rangle_{\cu_m}  + C 3^{-m\alpha}  \notag \\
												& \leq \fint_{\cu_m} \nabla \uhom(x) \cdot \ahom \nabla \uhom(x) \, dx + C 3^{-m\alpha} . \notag
\end{align} 
Using that for every $e \in \Bd(\cu_m),~\a(e) \in \left\{ 0 \right\} \times \left[ \lambda, 1 \right]$, we obtain, for every $e \in  \Bd(\C_*(\cu_m))$ such that $\a(e) \neq 0$
\begin{align*}
\lambda \left( \nabla \tuhom(e) - \nabla u(e)\right)^2 & \leq \a(e) \left(\nabla \tuhom(x) - \nabla u(e)\right)^2  \\
						 &\leq 2  \a(e) \left( \nabla \tuhom(e) \right)^2 + 2  \a(e) \left( \nabla u(e) \right)^2\\
						 & - 4 \a(e) \left( \frac{\nabla \tuhom(e) + \nabla u(e)}{2} \right) .
\end{align*}
 Summing over $e \in \Bd(\C_*(\cu_m))$ and using \eqref{uhomtutuhom.est} yields
\begin{multline*}
\frac{1}{|\cu_m|} \sum_{x \in \C_*(\cu_m )}  \left|\nabla \left( \tuhom - u\right)  \indc_{\left\{ \a \neq 0 \right\}}\right|^2 (x) \leq C  3^{-m\alpha} \\ + C \left( 4\frac{1}{|\cu_m|}  \left\langle \nabla u, \a \nabla u \right\rangle_{\cu_m}  \right. 
 \left. - 4 \frac{1}{|\cu_m|}  \left\langle \frac{\nabla \tuhom + \nabla u}{2}, \a \frac{\nabla \tuhom + \nabla u}{2} \right\rangle_{\cu_m} \right).
\end{multline*}
By Property (i) of Lemma~\ref{l.lemmaA2}, $\tuhom = u$ on $\C_*(\cu_m) \cap \partial \cu_m$. Thus, since $u \in \A_*(\cu_m)$,
\begin{equation*}
\left\langle \frac{\nabla \tuhom + \nabla u}{2}, \a \frac{\nabla \tuhom+ \nabla u}{2} \right\rangle_{\cu_m}  \leq \left\langle \nabla u, \a \nabla u \right\rangle_{\cu_m} .
\end{equation*}
This finally gives
\begin{equation*} 
\frac{1}{|\cu_m|} \sum_{x \in \C_*(\cu_m)}  \left|\nabla \left( \tuhom - u\right)  \indc_{\left\{ \a \neq 0 \right\}} \right|^2 (x) \leq C M^2 3^{-m\alpha}.
\end{equation*}
The Poincar\'e inequality (Proposition~\ref{p.sobolev} with $s= 2$) then gives
\begin{multline*}
\frac{3^{-2m}}{|\cu_m|}\sum_{x \in \C_*(\cu_m)}  \left( \tuhom(x) - u(x)  \right)^2 
\\ 
\leq \left( \sup_{x \in \cu_m } \size\left( \cu_\Pa(x) \right)^{2d}\right) \frac{1}{|\cu_m|} \sum_{x \in \C_*(\cu_m)}  \left|\nabla \left( \tuhom - u\right)  \indc_{\left\{ \a \neq 0 \right\}} \right|^2 (x)
\end{multline*}
Thus since $\Pa$ is finer than $\mathcal{Q}$, and $3^m \geq \M_t(\mathcal{Q})$, we have, for $t$ large enough,
\begin{equation} \label{blackboxfinal.est1}
\frac{3^{-2m}}{|\cu_m|}\sum_{x \in \C_*(\cu_m)}  \left( \tuhom(x) - u(x)  \right)^2 \leq C M^2 3^{-m\alpha}.
\end{equation}
Combining this with Property (iii) of Lemma~\ref{l.lemmaA2} shows
\begin{equation*}
\frac{3^{-2m}}{|\cu_m|}\sum_{x \in \C_*(\cu_m)}  \left( \uhom(x) - u(x)  \right)^2 \leq C M^2 3^{-m\alpha}
\end{equation*}
and the proof is complete.
\end{proof}

\smallskip

Before starting the proof of Lemmas~\ref{l.lemmaA1} and~\ref{l.lemmaA2}, we need to introduce some definitions and vocabulary which will be useful for the proof of both Lemmas. We keep the same notations as in Definition~\ref{definition.rvN} by setting
\begin{equation*}
n := \left\lceil \frac{m}{4} \right\rceil \mbox{ and } k := \left\lceil \frac{m}{8} \right\rceil, 
\end{equation*}
so that by definition of $\mathcal{N}$, for every $p,q \in \Rd$ and $z \in 3^k \Zd \cap \cu_m$,
\begin{equation} \label{muhasconverged}
\left| \frac 12 q \cdot \ahom^{-1} q + \mu(z + \cu_n,q) \right| \leq C |q|^2 3^{-m\alpha}
\end{equation}
and
\begin{equation} \label{nuhasconverged}
\left| \frac 12 p \cdot \ahom p + \nu(z + \cu_n,p) \right| \leq C |p|^2 3^{-m\alpha}.
\end{equation}
We also pick $t$ large enough such that
\begin{equation*}
3^m \geq \M_t(\mathcal{Q}) \implies \sup_{x \in \cu_m} \size\left( \cu_\mathcal{Q}(x) \right) \leq 3^k,
\end{equation*}
in particular since $\Pa$ is finer than $\mathcal{Q}$,
\begin{equation*}
3^m \geq \M_t(\mathcal{Q}) \implies \sup_{x \in \cu_m} \size\left( \cu_\Pa(x) \right) \leq 3^k.
\end{equation*}

\smallskip

We define $l$ the size of a boundary layer we need to remove in our argument:
\begin{equation*}
l := \left\lceil \frac{7 m}{8} \right\rceil,
\end{equation*}
so that $3^l \simeq \size(\cu_m)^{\frac 78}$ is the size of the boundary layer. Notice that $k \leq m \leq l$. From this we define the two interior cubes
\begin{equation*}
\cu^\circ := \left[ -\frac{3^m - 3^l}{2} , \frac{3^m - 3^l}{2}\right]^d  
\end{equation*}
and
\begin{equation*}
\cu^\circ_{\mathrm{int}} := \left[ -\frac{3^m -2\cdot 3^l}{2} , \frac{3^m - 2 \cdot 3^l}{2}\right]^d.
\end{equation*}
They correspond to the cube $\cu_m$ minus a boundary layer of size respectively $3^l$ and $2 \cdot 3^l$. From this we define $\eta \in C_c^\infty(\cu_m)$ which satisfies the following properties
\begin{equation} \label{def.eta.cutoff}
\eta = 1 \mbox{ on } \cu^\circ_{\mathrm{int}},~ \eta = 0 \mbox{ on } \cu_m \setminus \cu^\circ \mbox{ and } \left| \nabla \eta \right| \leq C 3^{-l}.
\end{equation}

In what follows, it is convenient to work with three different scales: the microscopic scale (of size $1$), the macroscopic scale (of size $\size(\cu_m) = 3^m$) and the mesoscopic scale (of size $\size(\cu_m)^\frac 14 = 3^{\frac m4}$). The last step before starting the proofs of Lemmas~\ref{l.lemmaA1} and~\ref{l.lemmaA2} involves proving some estimates related to $w $.
\begin{lemma} \label{l.propwws}
The following properties hold
\begin{enumerate}
\item[(i)] For every $x \in \Zd \cap \cu_m$ such that $\dist(x, \partial \cu_m) > 1$ and for every $y \in x + \left[-\frac 12, \frac 12\right]^d$
\begin{equation*}
|w (y) - \left[ u\right]_\Pa(x)|  \leq  \sum_{z \in \Zd : |z-x|_\infty \leq 1} |\nabla \left[ u\right]_\Pa|(z)
\end{equation*}
and for every $x \in \Zd \cap \cu_m$ such that $\dist(x, \partial \cu_m) \leq 1$, every $y \in x + \left[-\frac 12, \frac 12\right]^d \cap \cu_m$,
\begin{multline*}
|w (y) - \left[ u\right]_\Pa(x)|  \leq \sum_{z \in \Zd : |z-x|_\infty \leq 1} |\nabla \left[ u\right]_\Pa|(z) \\ +\sum_{z\in \C_*(\cu_m) \cap \cu': \cu' \in \Pa \dist(\cu',\cu_\Pa(x)) \leq 1} |\nabla  u\indc_{\left\{ \a \neq 0 \right\}}|(z).
\end{multline*}
\item[(ii)] There exists a constant $C := \sup_\Rd |\nabla \rho| < + \infty$ such tha, for every $x \in \Zd \cap \cu_m$ satisfying $\dist(x, \partial \cu_m) > 1$ and for every $y \in x + \left[-\frac 12, \frac 12\right]^d$
\begin{equation*}
|\nabla w (y)|  \leq C \sum_{z \in \Zd : |z-x|_\infty \leq 1} |\nabla \left[ u\right]_\Pa|(z)
\end{equation*}
and for every $x \in \Zd \cap \cu_m$ such that $\dist(x, \partial \cu_m) \leq 1$, every $y \in x + \left[-\frac 12, \frac 12\right]^d \cap \cu_m$,
\begin{equation*}
|\nabla w (y)|  \leq C \sum_{z \in \Zd : |z-x|_\infty \leq 1} |\nabla \left[ u\right]_\Pa|(z) + C \sum_{z\in \C_*(\cu_m) \cap \cu': \cu' \in \Pa \dist(\cu',\cu_\Pa(x)) \leq 1} |\nabla  u \indc_{\left\{ \a \neq 0 \right\}}|(z).
\end{equation*}
\item[(iii)] For every $p' \in \left[ 2, \frac{p+2}{2}\right]$, there exists a constant $C := C(s,d,\p,p) < + \infty$ such that
\begin{equation*}
\left( \fint_{\cu_m}  |\nabla w  (x)|^{p'} \, dx \right)^{\frac1{p'}} \leq C M.
\end{equation*}
\item[(iv)] For every $p' \in \left[ 2, \frac{p+2}{2}\right]$, there exists $C := C(s,d,\p,p) < + \infty$ such that
\begin{equation*}
\left( \fint_{\cu_m} \left| w(x) \right|^{p'} \right)^{\frac1{p'}} \leq CM3^m.
\end{equation*}
\end{enumerate}
\end{lemma}

\begin{proof} We prove (i). For every $x \in \Zd \cap \cu_m$ and every $y \in  \left( x + \left[-\frac 12, \frac 12\right]^d \right) \cap \cu_m$,
\begin{equation*}
|w (y) - \left[ u\right]_\Pa(x)| \leq \int_\Rd \left|\tilde w(z) - \left[ u\right]_\Pa(x)\right| \rho(y-z) \, dz.
\end{equation*}
Since $\supp \rho \subseteq (-\frac 12, \frac 12)^d$ we have
\begin{align*}
|w (y) - \left[ u\right]_\Pa(x)| & \leq \sup_{ z \in \cu_m \cap \left(  x + \left[- 1, 1 \right]^d \right)}  \left|\tilde w(z) - \left[ u\right]_\Pa(x)\right| 
\\ &
 \leq  \sup_{z \in \Zd\cap \cu_m : |z-x|_\infty \leq 1} \left|\tilde{w}(z) - \left[ u\right]_\Pa(x)\right|.
\end{align*}
If $\dist(x,\partial \cu_m) > 1$, then for every $z\in \Zd \cap \cu_m$ such that $|z-x|_\infty \leq 1,~ z \in \intr \cu_m$ and thus $\tilde w(z) = \left[ u\right]_\Pa(z)$. Hence
\begin{align*}
|w (y) - \left[ u\right]_\Pa(x)| & \leq  \sup_{z \in \Zd : |z-x|_\infty \leq 1} \left|\tilde w(z) - \left[ u\right]_\Pa(x)\right| \\
							& \leq \sup_{z \in \Zd : |z-x|_\infty \leq 1} \left|\left[ u\right]_\Pa (z) - \left[ u\right]_\Pa(x)\right| \\
							& \leq \sum_{z \in \Zd : |z-x|_\infty \leq 1} |\nabla \left[ u\right]_\Pa|(z) .
\end{align*}
If $\dist(x,\partial \cu_m) \leq 1$, then for every $z\in \Zd \cap \cu_m$ such that $|z-x|_\infty \leq 1$ we have either $\tilde w(z) = u(z)$ or $\tilde w(z) =  \left[ u\right]_\Pa(z)$. Thus
\begin{align*}
|w (y) - \left[ u\right]_\Pa(x)| & \leq  \sup_{z \in \Zd \cap \cu_m: |z-x|_\infty \leq 1} \left|\tilde w(z) - \left[ u\right]_\Pa(x)\right| \\
							& \leq \sup_{z \in \Zd \cap \cu_m: |z-x|_\infty \leq 1} \left( \left|\left[ u\right]_\Pa (z) - \left[ u\right]_\Pa(x)\right|  + \left|\tilde w (z) - \left[ u\right]_\Pa(z)\right|  \right) \\
							& \leq \sum_{z \in \Zd \cap \cu_m : |z-x|_\infty \leq 1} |\nabla \left[ u\right]_\Pa|(z)  + \sup_{z \in \Zd \cap \cu_m : |z-x|_\infty \leq 1} \left|\tilde w(z) - \left[ u\right]_\Pa(z)\right|.
\end{align*}
Using the first inequality~\eqref{e.justintegrate} in the proof of Lemma~\ref{l.coarseLs} yields
\begin{equation*}
\sup_{z \in \Zd \cap \cu_m : |z-x|_\infty \leq 1} \left|\tilde w(z) - \left[ u\right]_\Pa(z)\right|  \leq   \sum_{z\in \C_*(\cu_m) \cap \cu': \cu' \in \Pa \dist(\cu',\cu_\Pa(x)) \leq 1} |\nabla u \indc_{\left\{ \a \neq 0 \right\}}|(z).
\end{equation*}
The proof of (i) is complete. To prove (ii), notice that for every $x \in \Zd \cap \cu_m$ and $y \in \left( x + \left[-\frac 12, \frac 12\right]^d \right) \cap \cu_m,$
\begin{align*}
\left| \nabla w  (y) \right| 
 =  \left| \nabla \left( w  - \left[ u\right]_\Pa (x) \right) (y) \right| 
			& \leq \left| \left( \left( w - \left[ u\right]_\Pa (x) \right) * \nabla \rho\right) (y) \right| \\
			& \leq \int_\Rd \left| \tilde w (z) - \left[ u\right]_\Pa (x) \right| \left|\nabla \rho (y-z) \right| \, dz \\
			& \leq \left( \sup_\Rd |\nabla \rho| \right) \sup_{ z \in \cu_m \cap \left(  x + \left[- 1, 1 \right]^d \right)}  \left|\tilde w(z) - \left[ u\right]_\Pa(x)\right| \\
			 & \leq  C \sup_{z \in \Zd \cap \cu_m : |z-x|_\infty \leq 1} \left|\tilde w(z) - \left[ u\right]_\Pa(x)\right|.
\end{align*} 
The end of the proof of (ii) is similar to the proof of (i) and thus omitted.

To prove (iii), we split the integral
\begin{equation*}
\fint_{\cu_m}  |\nabla w  (x)|^{p'} \, dx =\frac{1}{|\cu_m|} \int_{\cu_m \setminus \partial_\Pa \cu_m}  |\nabla w  (x)|^{p'} \, dx +  \frac{1}{|\cu_m|} \int_{\partial_\Pa \cu_m}  |\nabla w  (x)|^{p'} \, dx.
\end{equation*}
The first term of the right-hand side can be estimated by using the first inequality of (ii)
\begin{align*}
 \int_{\cu_m \setminus \partial_\Pa \cu_m}  |\nabla w  (x)|^{p'} \, dx & \leq C \sum_{x \in \cu_m \setminus \partial_\Pa \cu_m}  \left| \sum_{z \in \Zd : |z-x|_\infty \leq 1} |\nabla \left[ u\right]_\Pa|(z) \right|^{p'} \\
													& \leq C \sum_{x \in \cu_m \setminus \partial_\Pa \cu_m}  \sum_{z \in \Zd : |z-x|_\infty \leq 1} |\nabla \left[ u\right]_\Pa|^{p'}(z) \\
													& \leq C \sum_{x \in \cu_m} |\nabla \left[ u\right]_\Pa|^{p'}(x).
\end{align*}
Applying Lemma~\ref{l.coarsegrads} yields
\begin{align*}
\lefteqn{ \int_{\cu_m \setminus \partial_\Pa \cu_m}  |\nabla w  (x)|^{p'} \, dx} \qquad & \\
 & \leq C  \sum_{\cu \in\Pa,\, \cu\subseteq \cl_{\Pa}(\cu_m)} \size(\cu)^{p'd-1} \sum_{\cu\cap\C_*(\cu_m)} \left|\nabla u \indc_{\{\a\neq 0\}} \right|^{p'}(x) \\
& \leq C  \left( \sum_{\cu\in\Pa,\, \cu\subseteq \cl_{\Pa}(\cu_m)} \size(\cu)^{(p'd-1)\frac{p}{p-p'}} \right)^{1 - \frac{p'}{p} }  \left( \sum_{x \in \C_*(\cu_m)} \left|\nabla u \indc_{\{\a\neq 0\}} \right|^{p}(x) \right)^{\frac{p'}{p}}.
\end{align*}
Taking the exponent $t$ large enough and using the assumption that  $3^m \geq \mathcal{M}_t(\Qa)$, we obtain
\begin{equation} \label{borne.nablaws.1}
\frac{1}{|\cu_m|} \int_{\cu_m \setminus \partial_\Pa \cu_m}  |\nabla w  (x)|^{p'} \, dx \leq C M^{p'}.
\end{equation}
The second term of the right-hand side can be estimated by using the second inequality of (ii):
\begin{align*}
\lefteqn{\int_{\partial_\Pa \cu_m}  |\nabla w  (x)|^{p'} \, dx} \qquad & \\
& \leq C \sum_{x \in \partial_\Pa \cu_m} \left(  \sum_{z \in \Zd \cap \cu_m : |z-x|_\infty \leq 1} |\nabla \left[ u\right]_\Pa|(z)  +  \sum_{z\in \C_*(\cu_m) \cap \cu': \cu' \in \Pa \dist(\cu',\cu_\Pa(x)) \leq 1} |\nabla  u \indc_{\left\{ \a \neq 0 \right\}}|(z) \right)^{p'} \\
& \leq C \sum_{x \in \cu_m}  |\nabla \left[ u\right]_\Pa|^{p'}(x) +C  \sum_{x \in \partial_\Pa \cu_m} \size(\cu_\Pa (x))^{d\left( 1-\frac{1}{p'} \right)} 
\\ & \qquad 
\times \sum_{z\in \C_*(\cu_m) \cap \cu',\, \cu' \in \Pa,\, \dist(\cu',\cu_\Pa(x)) \leq 1} |\nabla u \indc_{\left\{ \a \neq 0 \right\}}|^{p'}(z)  \\
& \leq  C \sum_{x \in \cu_m} |\nabla \left[ u\right]_\Pa|^{p'}(x) + C  \sum_{x \in  \C_*( \cu_m) }  \size(\cu_\Pa(x))^{d\left( 2 -\frac{1}{p'} \right)} |\nabla u \indc_{\left\{ \a \neq 0 \right\}}|^{p'}(x) \\
& \leq C \sum_{x \in \cu_m} |\nabla \left[ u\right]_\Pa|^{p'}(x)  + C  \left( \sum_{x \in  \cu_m }  \size(\cu_\Pa(x))^{d\left( 2 -\frac{1}{p'} \right)\frac{p}{p-p'} } \right)^{1 - \frac{p'}{p} } 
\\ & \qquad 
\times \left( \sum_{x \in  \C_*( \cu_m) }  |\nabla u \indc_{\left\{ \a \neq 0 \right\}}|^p (x) \right)^{\frac{p'}{p}}.
\end{align*}
Taking the exponent $t$ large enough and using $\size(\cu_m) = 3^m \geq \mathcal{M}_t(\Qa)$ again, we get\begin{equation} \label{borne.nablaws.2}
\frac{1}{|\cu_m|}\int_{\partial_\Pa \cu_m}  |\nabla w  (x)|^{p'} \, dx \leq C M^{p'}.
\end{equation}
Combining~\eqref{borne.nablaws.1} and~\eqref{borne.nablaws.2} yields (iii).

The proof of (iv) is a consequence the usual Sobolev inequality on $\Zd$~\eqref{ususobineqZd}, the assumption~\eqref{e.additiveconstants} and the estimate which follows easily from Lemma~\ref{l.propwws}[(i)] and the assumption $3^m \geq \mathcal{M}_t(\Qa)$,
\begin{equation*}
\left( \fint_{ \cu_m}  | w  (x)|^{p'} \, dx \right)^{\frac{1}{p'}} \leq C \left( \frac1{|\cu_m|}\sum_{x \in \cu_m} | \left[ u\right]_\Pa|^{p'}(x) \right)^{\frac{1}{p'}} + \left( \frac1{|\cu_m|}\sum_{x \in \cu_m} | \nabla \left[ u\right]_\Pa|^{p}(x) \right)^{\frac{1}{p}}.
\end{equation*}
\end{proof}

We can now prove Lemma~\ref{l.lemmaA1}.

\begin{proof}[Proof of Lemma \ref{l.lemmaA1}]
We now construct the function $\tilde{u} \in C^\infty(\cu_m)$ by removing the microscopic oscillations from $u$. More precisely we are removing the microscopic oscillation of $w $, which is close to $u$, by considering spatial averages on a mesoscopic scale of $w $. We thus define, for every $x \in \cu_m$ such that $x + \cu_n \subseteq \cu_m$,
\[ \xi(x) := \fint_{x + \cu_n} w  (z) \, dz. \]
We next modify $\xi$ in order to get an element of $C^{\infty} (\cu_m)$, equal to $w $ on $\partial \cu_m$, by setting
\[ \tilde{u}(x) := \eta(x) \xi(x) + (1-\eta(x)) w (x). \]
It is clear that $\tilde{u} \in w  + H^1_0(\cu_m)$, so we now focus on proving~\eqref{e.lemmaA1}. We break the argument into several step.

\smallskip

\textit{Step 1.} 
Denote for each $y \in \cu^\circ$,
\begin{equation*}
 p(y) := \nabla \xi (y) =  \frac{1}{|\cu_n|}\int_{y+\cu_n} \nabla w (z) \, dz \in \Rd \mbox{ and } q(y) := \ahom p(y) \in \R^d.
\end{equation*}
For $y \in 3^k \Zd \cap \cu^\circ$, testing $u$ as a minimizer candidate in the definition of $\mu(y + \cu_n, q(y))$, gives
\begin{equation} \label{u.minimizer.candidate.k}
\mu(y + \cu_n, q(y)) \leq \frac{1}{2|\cu_n|} \left\langle \nabla u, \a \nabla u \right\rangle_{\C_*(y+\cu_n)} - \frac{1}{|\cu_n|} \left\langle q(y), \nabla \left[ u\right]_\Pa \right\rangle_{y + \cu_n}.
\end{equation}
Combining this result with~\eqref{muhasconverged}
\begin{align*}
-\frac 12 q(y)\cdot \ahom^{-1} q(y) & \leq \frac{1}{2|\cu_n|} \left\langle \nabla u,\a \nabla u \right\rangle_{\C_*(y+\cu_n)} - \frac{1}{|\cu_n|} \left\langle q(y), \nabla \left[ u\right]_\Pa \right\rangle_{y + \cu_n}+ C |q(y)|^2  3^{-m\alpha} \\
							& \leq \frac{1}{2|\cu_n|} \left\langle \nabla u,\a \nabla u \right\rangle_{\C_*(y+\cu_n)} - q(y) \cdot p(y) + C  |q(y)|^2  3^{-m\alpha} \\
								& + |q(y)| \left| \frac{1}{|\cu_n|} \left\langle  \nabla \left[ u\right]_\Pa \right\rangle_{y + \cu_n}  - p(y) \right| .
\end{align*}
Thus by definition of $q(y)$
\begin{multline*}
p(y)\cdot \ahom p(y) \leq \frac{1}{|\cu_n|}\left\langle \nabla u, \a \nabla u\right\rangle_{\C_*(y+\cu_n)}+ C |q(y)|^2  3^{-m\alpha} \\ + 2 |q(y)| \left| \frac{1}{|\cu_n|} \left\langle  \nabla \left[ u\right]_\Pa \right\rangle_{y + \cu_n}  - p(y) \right| .
\end{multline*}
Summing over $y \in  3^{k} \Zd \cap \cu^\circ_{\mathrm{int}}$ and multiplying by $\frac{3^{k}}{|\cu^\circ_{\mathrm{int}}|} $ yields
\begin{align} \label{almost.result}
\lefteqn{\frac{3^{dk}}{|\cu^\circ_{\mathrm{int}}|}  \sum_{y \in  3^k \Zd \cap \cu^\circ_{\mathrm{int}} } p(y)\cdot \ahom p(y)} \qquad & \\
 & \leq  \frac{3^k}{|\cu^\circ_{\mathrm{int}}|}  \sum_{y \in  3^k \Zd \cap \cu^\circ_{\mathrm{int}} } \frac{1}{|\cu_n|}\left\langle \nabla u, \a \nabla u\right\rangle_{\C_*(y+\cu_n)} +C M^2 3^{-m\alpha } \notag \\
 & + \frac{3^{dk}}{|\cu^\circ_{\mathrm{int}}|} \sum_{y \in  3^k \Zd \cap \cu^\circ_{\mathrm{int}} }  2 |q(y)| \left| \frac{1}{|\cu_n|} \left\langle  \nabla \left[ u\right]_\Pa \right\rangle_{y + \cu_n}  - p(y) \right|  . \notag
\end{align}
Since by (iii) of Lemma~\ref{l.propwws} with $p' = 2 $
\begin{align*}
\frac{3^{dk}}{|\cu^\circ_{\mathrm{int}}|} \sum_{y \in  3^{k} \Zd \cap \cu^\circ_{\mathrm{int}} }  |q(y)|^2 & \leq C  \frac{3^{dk}}{|\cu^\circ_{\mathrm{int}}|} \sum_{y \in  3^{k} \Zd \cap \cu^\circ_{\mathrm{int}} }  |p(y)|^2 \\
																	& \leq C \fint_{\cu^\circ_{\mathrm{int}}} |\nabla w  (z)|^2 \, dz \\
																	& \leq C \fint_{\cu} |\nabla w  (z)|^2 \, dz \\
																	& \leq C M^2.
\end{align*}
The rest of the proof is organized as follow
\begin{itemize}
\item Step 2: We show 
\begin{equation*}
\frac{3^{dk}}{|\cu^\circ_{\mathrm{int}}|} \sum_{y \in  3^k \Zd \cap \cu^\circ_{\mathrm{int}} }  |q(y)| \left| \frac{1}{|\cu_n|} \left\langle  \nabla \left[ u\right]_\Pa \right\rangle_{y + \cu_n}  - p(y) \right|  \leq C M^2 3^{-m \alpha}.
\end{equation*}
\item  Step 3: We show
\begin{multline*}
\frac{3^{dk}}{|\cu^\circ_{\mathrm{int}}|} \sum_{y \in  3^{k} \Zd \cap \cu^\circ_{\mathrm{int}} } \frac{1}{|\cu_n|} \left\langle \nabla u, \a \nabla u \right\rangle_{\C_*(y+\cu_n)} \, dy 
\\
\leq \frac{1}{|\cu_m|} \left\langle \nabla u, \a \nabla u \right\rangle_{\C_*(\cu_m)} + C M^2 3^{- m \alpha }.
\end{multline*}
\item  Step 4 and Step 5 : We show
\begin{equation*}
\fint_{\cu_m} \nabla \tilde{u}(y) \cdot \ahom \nabla \tilde{u}(y)\, dy\leq \frac{3^{dk}}{|\cu^\circ_{\mathrm{int}}|}  \sum_{y \in  3^k \Zd \cap \cu^\circ_{\mathrm{int}} } p(y)\cdot \ahom p(y) + CM^2 3^{-m\alpha } .
\end{equation*}
\end{itemize}
Combining these three results with~\eqref{almost.result} completes the proof of Lemma~\ref{l.lemmaA1}.

\smallskip

\textit{Step 2.} We want to show
\begin{equation} \label{mainresultstep2}
\frac{3^{dk}}{|\cu^\circ_{\mathrm{int}}|} \sum_{y \in  3^k \Zd \cap \cu^\circ_{\mathrm{int}} }  |q(y)| \left| \frac{1}{|\cu_n|} \left\langle  \nabla \left[ u\right]_\Pa \right\rangle_{y + \cu_n}  - p(y) \right|  \leq C M^2 3^{-m \alpha}.
\end{equation}
We already saw that
\begin{equation*}
\frac{3^{dk}}{|\cu^\circ_{\mathrm{int}}|} \sum_{y \in  3^{k} \Zd \cap \cu^\circ_{\mathrm{int}} }  |q(y)|^2  \leq C M^2.
\end{equation*}
By the Cauchy Schwartz inequality it is enough to obtain~\eqref{mainresultstep2} to prove
\begin{equation} \label{mainresultstep2bis}
\frac{3^{dk}}{|\cu^\circ_{\mathrm{int}}|} \sum_{y \in  3^k \Zd \cap \cu^\circ_{\mathrm{int}} }  \left| \frac{1}{|\cu_n|} \left\langle  \nabla \left[ u\right]_\Pa \right\rangle_{y + \cu_n}  - p(y) \right|^2  \leq C M^2 3^{- m \alpha}.
\end{equation}
To prove this, we will prove, for every $y \in 3^k \Zd \cap \cu^\circ_{\mathrm{int}}$,
\begin{equation} \label{nablacoarsenu.ws.close}
\left|\left\langle  \nabla \left[ u\right]_\Pa  \right\rangle_{y + \cu_n} -  \int_{y+\cu_n} \nabla w (z) \, dz \right| \leq C \sum_{x\in \partial \cu_n} |\nabla \left[ u\right]_\Pa|(x).
\end{equation}
For the sake of simplicity we assume $y=0$ to prove~\eqref{nablacoarsenu.ws.close}. By the discrete Stokes formula,
\begin{equation*}
\left\langle  \nabla \left[ u\right]_\Pa  \right\rangle_{\cu_n} =   \sum_{x \in \partial \cu_n} \left[ u\right]_\Pa(x) \mathbf{n}(x).
\end{equation*}
For $i= \{ -d,\ldots ,-1,1,\ldots, d\}$, denote by $ \partial_{i} \cu_n $ the $i$th face of $\cu_n$ given by
\begin{equation*} \label{}
\partial_{i} \cu_n  := \left\{ x \in  \Zd \cap \cu_n~:~ x_{|i|}= \sign(i) \frac12 (3^n-1)  \right\}.
\end{equation*}
Denote also by $\mathbf{n}_i$ the associated outer normal vector, i.e, $\mathbf{n}_i = \e_i$ for $i$ positive and $\mathbf{n}_i = - \e_i$ for $i$ negative.
The previous identity can be rewritten
\begin{equation*}
\left\langle \nabla \left[ u\right]_\Pa \right\rangle_{ \cu_n}=  \sum_{i=\pm 1, \cdots, \pm d} \sum_{x \in \partial_i \cu_n} \left[ u\right]_\Pa(x) \mathbf{n}_i.
\end{equation*}
Thus
\begin{equation*}
\left\langle  \nabla \left[ u\right]_\Pa  \right\rangle_{\cu_n} -  \int_{\cu_n} \nabla w (z) \, dz  = \sum_{i=\pm 1, \cdots, \pm d} \left( \sum_{x \in \partial_i \cu_n}  \left[ u\right]_\Pa(x)  - \int_{\partial_i  \cu_n} w (z) \, dz \right) \mathbf{n}_i.
\end{equation*}
Without loss of generality, it is sufficient to prove~\eqref{nablacoarsenu.ws.close} to show
\begin{equation} \label{particularcasei=1}
\left| \sum_{x \in \partial_1 \cu_n \cup \partial_{-1} \cu_n}  \left[ u\right]_\Pa(x)  - \int_{\partial_1 \cu_n \cup \partial_{-1} \cu_n} w (z) \, dz \right| \leq C \sum_{x\in \partial \cu_n} |\nabla \left[ u\right]_\Pa|(x).
\end{equation}
With a few modifications of the proof of (i) of Lemma~\ref{l.propwws}, we can show: for every $x \in \partial_1 \cu_n \cup \partial_{-1} \cu_n$ and every $z \in \left( x + \lbrace 0 \rbrace \times \left[-\frac 12, \frac 12\right]^{d-1} \right) \bigcap \partial \cu_n$
\begin{equation} \label{wscoarsenu.close}
|w (z) - \left[ u\right]_\Pa(x)|  \leq \sum_{y \in \Zd \cap \partial \cu_n : |y-x|_\infty \leq 1} |\nabla \left[ u\right]_\Pa|(y).
\end{equation}
The idea of the proof of~\eqref{particularcasei=1} is to apply~\eqref{wscoarsenu.close} with $x \in \partial_1 \cu_n$ $ \left(\mbox{resp. } \partial_{-1} \cu_n\right)$ and $z\in \left( x + \left\lbrace 0 \right\rbrace \times \left[-\frac 12, \frac 12\right]^{d-1} \right) \bigcap \partial \cu_n$. Unfortunately, a technical difficulty appears if $x$ lies on the boundary of $\partial_1 \cu_n$ $ \left(\mbox{resp. } \partial_{-1} \cu_n\right)$  which we denote
\begin{equation*}
 \partial \partial_1 \cu_n := \left\lbrace x \in \partial_1 \cu_n ~:~ \exists i \in \lbrace 2,\cdots,d \rbrace,~ x_i \mbox{ is maximal or minimal} \right\rbrace.
\end{equation*}
 We define similarly $\partial \partial_{-1} \cu_n$. Thus we distinguish two cases, whether $x \in \partial \partial_1 \cu_n$ $ \left(\mbox{resp. } \partial \partial_{-1} \cu_n\right)$ or not.

\emph{Case 1:} $x \in \partial_1 \cu_n \setminus \partial \partial_1 \cu_n$, then by~\eqref{wscoarsenu.close}
\begin{equation*}
\left| \left[ u\right]_\Pa(x) - \int_{x + \left\lbrace 0 \right\rbrace \times \left[-\frac 12, \frac 12\right]^{d-1}} w (z) \, dz \right| \leq \sum_{z \in \partial \cu_n : |z-x|_\infty \leq 1} |\nabla \left[ u\right]_\Pa|(z).
\end{equation*}
Symmetrically for $x \in \partial_{-1} \cu_n \setminus \partial \partial_{-1} \cu_n$
\begin{equation*}
\left| \left[ u\right]_\Pa(x) - \int_{x + \left\lbrace0 \right\rbrace \times \left[-\frac 12, \frac 12\right]^{d-1}} w (z) \, dz \right| \leq \sum_{z \in \partial \cu_n : |z-x|_\infty \leq 1} |\nabla \left[ u\right]_\Pa|(z).
\end{equation*}

\emph{Case 2:} $x \in \partial \partial_1 \cu_n$ then denote by $\hat{x} := x - (3^n - 1 ) \e_1 \in \partial \partial_{-1} \cu_n$. We have
\begin{equation*}
 \left| \left[ u\right]_\Pa(x) - \left[ u\right]_\Pa(\hat{x}) \right| \leq \sum_{i = 0}^{3^n - 1} \left| \nabla \left[ u\right]_\Pa\right| (x - i \e_1) .
\end{equation*}
Summing over $x \in \partial \partial_1 \cu_n$ yields
\begin{align*}
\sum_{x \in \partial \partial_1 \cu_n} \left| \left[ u\right]_\Pa(x) - \left[ u\right]_\Pa(\hat{x}) \right| & \leq \sum_{x \in \partial \partial_1 \cu_n} \sum_{i = 0}^{3^n-1} \left| \nabla \left[ u\right]_\Pa\right| (x - i \e_1) \\
																				& \leq \sum_{x \in \partial \cu_n} \left| \nabla \left[ u\right]_\Pa\right|(x),
\end{align*}
since for every $x \in  \partial \partial_1 \cu_n$ and every $i \in \lbrace 0, \cdots , 3^n -1 \rbrace$, $x - i \e_1 \in \partial \cu_n$ (and for every $y \in \partial \cu_n$ there  exists at most one $x \in \partial \partial_1  \cu_n$ and one $i \in \lbrace 0, \cdots , 3^n -1 \rbrace$ such that $y = x - i \e_1$).

Moreover, for every $x \in \partial \partial_1 \cu_n$ and every $z \in \left( x + \lbrace 0 \rbrace \times \left[-\frac 12, \frac 12\right]^{d-1} \right) \bigcap \partial \cu_n$, we define $\hat{z} := z - (3^n -1 ) \e_1$. By~\eqref{wscoarsenu.close} and the previous computations
\begin{align*}
|w (z) - w (\hat{z})|  & \leq |w (z) - \left[ u\right]_\Pa(x)| + |\left[ u\right]_\Pa(x) - \left[ u\right]_\Pa(\hat{x})| + |\left[ u\right]_\Pa(\hat{x}) - w (\hat{z})| \\
				& \leq  \sum_{y \in \cu_m : |y-x|_\infty \leq 1} |\nabla \left[ u\right]_\Pa|(y) + \sum_{i = 1}^{3^m} \left| \nabla \left[ u\right]_\Pa\right| (x - i \e_1) 
\\ &
\qquad + \sum_{y \in \cu_m : |y-\hat{x}|_\infty \leq 1} |\nabla \left[ u\right]_\Pa|(y).
\end{align*}
Thus, integrating over $\left( x + \left\lbrace0 \right\rbrace \times \left[-\frac 12, \frac 12\right]^{d-1} \right) \bigcap \partial \cu_n$
\begin{multline*}
\left| \int_{\left( x + \left\lbrace0 \right\rbrace \times \left[-\frac 12, \frac 12\right]^{d-1} \right) \bigcap \partial \cu_n} w (z) \, dz -  \int_{\left( \hat{x} + \left\lbrace0 \right\rbrace \times \left[-\frac 12, \frac 12\right]^{d-1}\right) \bigcap \partial \cu_n} w (z) \, dz \right| \\  \leq \sum_{y \in \cu_n : |y-x|_\infty \leq 1} |\nabla \left[ u\right]_\Pa|(y) + \sum_{i = 0}^{3^n-1} \left| \nabla \left[ u\right]_\Pa\right| (x - i \e_1) + \sum_{y \in \cu_n : |y-\hat{x}|_\infty \leq 1} |\nabla \left[ u\right]_\Pa|(y).
 \end{multline*}
Summing over $x \in \partial \partial_1 \cu_n$ yields
\begin{multline*}
\sum_{x \in  \partial \partial_1 \cu_n} \left|  \int_{\left( x + \left\lbrace0 \right\rbrace \times \left[-\frac 12, \frac 12\right]^{d-1} \right) \bigcap \partial \cu_n} w (z) \, dz -  \int_{\left( \hat{x} + \left\lbrace0 \right\rbrace \times \left[-\frac 12, \frac 12\right]^{d-1}\right) \bigcap \partial \cu_n} w (z) \, dz \right| \\   \leq C \sum_{y \in \partial \cu_n} \left| \nabla \left[ u\right]_\Pa\right| (y). 
 \end{multline*}

\smallskip

Combining the displays of cases 1 and 2 and using the triangle inequality shows
\begin{equation*}
\left| \sum_{x \in \partial_1 \cu_n \cup \partial_{-1} \cu_n}  \left[ u\right]_\Pa(x)  - \int_{\partial_1 \cu_n \cup \partial_{-1} \cu_n} w (z) \, dz \right| \\ \leq C \sum_{x\in \partial \cu_n} |\nabla \left[ u\right]_\Pa|(x),
\end{equation*}
which is~\eqref{nablacoarsenu.ws.close}. We now turn to the proof of~\eqref{mainresultstep2bis}.

Applying the Cauchy-Schwartz inequality yields
\begin{align*}
\left|\frac{1}{|\cu_n|}\left\langle  \nabla \left[ u\right]_\Pa  \right\rangle_{y + \cu_n}  - \frac{1}{|\cu_n|} \int_{y+\cu_n} \nabla w (z) \, dz \right|^2 & \leq C \left( \frac{|\partial \cu_n |}{|\cu_n|}\right)  \left( \frac{1}{|\cu_n|} \sum_{z \in (y+\cu_n)} \left| \nabla \left[ u\right]_\Pa \right|^2 (z) \right) \\
																			& \leq C 3^{- n }  \left( \frac{1}{|\cu_n|} \sum_{z \in (y+\cu_n)} \left| \nabla \left[ u\right]_\Pa \right|^2 (z) \right).
\end{align*}
Summing this over $3^k \Zd \cap \cu^\circ_{\mathrm{int}}$ and applying (iii) of Lemma~\ref{l.propwws} with $p' = 2$ yields~\eqref{mainresultstep2bis} and consequently the main result of this step~\eqref{mainresultstep2}.

\smallskip

\textit{Step 3.} We want to show 
\begin{equation}\label{step2.est}
\frac{3^{dk}}{\left|\cu^\circ_{\mathrm{int}}\right|} \sum_{y \in  3^{k} \Zd \cap \cu^\circ_{\mathrm{int}} } \frac{1}{|\cu_n|} \left\langle \nabla u, \a \nabla u \right\rangle_{\C_*(y+\cu_n)} \leq \frac{1}{|\cu_m|} \left\langle \nabla u, \a \nabla u \right\rangle_{\C_*(\cu_m)} + C M^2 3^{-m \alpha}.
\end{equation}
Notice that, while $\C_*(\cu_m) \cap \left( z + \cu_n\right)$ and $\C_*(z + \cu_n)$ may be different, every open edges in the latter cluster belongs to the former. This remark shows the following inequality, for each $y \in  3^{k} \Zd \cap \cu^\circ_{\mathrm{int}}$,
\begin{equation*}
\left\langle \nabla u, \a \nabla u \right\rangle_{\C_*(y+\cu_n)} \leq \left\langle \nabla u, \a \nabla u \right\rangle_{\C_*(\cu_m) \cap \left( y + \cu_n \right)}.
\end{equation*}
This allows us to bound
\begin{multline*}
\frac{3^{dk}}{|\cu^\circ_{\mathrm{int}}|} \sum_{y \in  3^{k} \Zd \cap \cu^\circ_{\mathrm{int}} }  \frac{1}{|\cu_n|} \left\langle \nabla u, \a \nabla u \right\rangle_{\C_*(y+\cu_n)} 
\\
 \leq \frac{3^{dk}}{|\cu^\circ_{\mathrm{int}}|} \sum_{y \in  3^{k} \Zd \cap \cu^\circ_{\mathrm{int}} }  \frac{1}{|\cu_n|} \left\langle \nabla u, \a \nabla u \right\rangle_{\C_*(\cu_m) \cap \left( y + \cu_n \right)} 
\leq \frac{1}{|\cu^\circ_{\mathrm{int}}|}  \left\langle \nabla u, \a \nabla u \right\rangle_{\C_*\left(\cu_m \right)} .
\end{multline*}
To complete the proof, we need to show the following estimate
\begin{equation*}
\frac{1}{|\cu^\circ_{\mathrm{int}}|} \left\langle \nabla u, \a \nabla u \right\rangle_{\C_*(\cu_m )} \leq \frac{1}{|\cu_m|} \left\langle \nabla u, \a \nabla u \right\rangle_{\C_*(\cu_m)} + CM^2 3^{-m \alpha }.
\end{equation*}
The previous estimate follows from the following computation
\begin{align*}
\lefteqn{ \frac{1}{|\cu^\circ_{\mathrm{int}}|} \left\langle \nabla u, \a \nabla u \right\rangle_{\C_*\left(\cu_m\right)} - \frac{1}{|\cu_m|} \left\langle \nabla u, \a \nabla u \right\rangle_{\C_*(\cu_m)}}\qquad & \\
							&\leq  \frac{|\cu_m \setminus \cu^\circ_{\mathrm{int}} |}{|\cu_m|^2}  \left( \sum_{x \in \C_*(\cu_m)} |\nabla u \indc_{\left\{ \a \neq 0 \right\}} |^2 (x) \right) + \frac{1}{|\cu_m|}  \left( \sum_{x \in \C_*(\cu_m)\setminus \cu^\circ_{\mathrm{int}}} |\nabla u\indc_{\left\{ \a \neq 0 \right\}} |^2(x) \right).
\end{align*}
Recall the definition of $l$, the size of the boundary layer. Using this definition, we have
\begin{equation*}
\frac{|\cu_m \setminus \cu^\circ_{\mathrm{int}} |}{|\cu_m|} \leq C 3^{l-m} \leq C 3^{- m\alpha}.
\end{equation*}
This allows to bound the first term on the right hand side
\begin{equation*}
\frac{|\cu_m \setminus \cu^\circ_{\mathrm{int}} |}{|\cu_m|^2}  \left( \sum_{x \in \C_*(\cu_m)} |\nabla u \indc_{\left\{ \a \neq 0 \right\}} |^2 (x) \right) \leq C  3^{-m \alpha} M^2.
\end{equation*}
The second term can be bounded by applying the H\"older inequality,
\begin{align*}  
\lefteqn{
\frac {1}{ |\cu_m|}  \sum_{x \in \C_*(\cu_m)\setminus \cu^\circ_{\mathrm{int}}} |\nabla u \indc_{\left\{ \a \neq 0 \right\}}|^2(x) 
} \qquad & \\
												&\leq  C\left(  \frac{|\C_*(\cu_m) \setminus \cu^\circ_{\mathrm{int}} |}{|\cu_m|} \right)^{\frac{p-2}{p}} \left( \frac{1}{ |\cu_m|}  \sum_{x \in \C_*(\cu_m)\setminus \cu^\circ_{\mathrm{int}}} |\nabla u \indc_{\left\{ \a \neq 0 \right\}}|^{p} (x) \right)^{\frac{2}{p}} \\
												&\leq C 3^{(l-m)\frac{p-2}{p}} \left( \frac{1}{|\cu_m|}  \sum_{x \in \C_*(\cu_m)} |\nabla u \indc_{\left\{ \a \neq 0 \right\}}|^p(x) \right)^{\frac 2p}\notag \\ 
												& \leq C 3^{- m\alpha} M^2. 
\end{align*}
The proof of~\eqref{step2.est} is complete.

\smallskip
\textit{Step 4.} We want to show
\begin{equation} \label{est.lA.1.step2}
 \fint_{\cu^\circ_{\mathrm{int}}} p(x)\cdot \ahom p(x) \, dx \leq \frac{3^{dk}}{|\cu^\circ_{\mathrm{int}}|} \sum_{x \in  3^{k} \Zd \cap \cu^\circ_{\mathrm{int}} }  p(y)\cdot \ahom p(y) + C M^2 3^{-m \alpha }.
\end{equation}
For every $x \in \cu^\circ_{\mathrm{int}}$ and $y \in x + \left[ -\frac{3^k}{2}, \frac{3^k}{2} \right]^d$
\begin{align*}
|p(x) - p(y)| &\leq \frac{1}{|\cu_n|} \int_{(x + \cu_n) \Delta (y + \cu_n)} \left| \nabla w (z) \right| \, dz \\
			& \leq \left( \frac{|(x + \cu_n) \Delta (y + \cu_n)|}{|\cu_n|} \right)^\frac12 \left(  \frac{1}{|\cu_n|}  \int_{(x + \cu_n) \Delta (y + \cu_n)} \left| \nabla w (z) \right|^2 \, dz\right)^\frac12 \\
			& \leq C 3^{\frac{k-n}{2}} \left( \fint_{x +\cu_n+ \cu_k} \left| \nabla w (z) \right|^2 \, dz\right)^\frac12 \\
			& \leq C 3^{-m \alpha } \left( \fint_{x +\cu_n+ \cu_k} \left| \nabla w (z) \right|^2 \, dz\right)^\frac12.
\end{align*}
Thus
\begin{align*}
\lefteqn{
\left| p(x)\cdot \ahom p(x) - \fint_{x + \left[ -\frac{3^k}{2}, \frac{3^k}{2} \right]^d} p(y) \cdot \ahom p(y) \, dy \right| 
} \quad & \\
& \leq \sup_{y \in x + \left[ -\frac{3^k}{2}, \frac{3^k}{2} \right]^d} |p(x) - p(y)| \left( |p(x)| + |p(y)| \right) 
& \leq C 3^{-m \alpha } \left( \fint_{y +\cu_m+ \cu_k} \left| \nabla w (z) \right|^2 \, dz\right).
\end{align*}
Summing over $x \in 3^{k} \Zd \cap \cu^\circ_{\mathrm{int}}$ and using (iii) of Lemma~\ref{l.propwws}
\begin{align*}
\lefteqn{ \left| \frac{3^{dk}}{|\cu^\circ_{\mathrm{int}}|} \sum_{x \in  3^{k} \Zd \cap \cu^\circ_{\mathrm{int}}}  p(x)\cdot \ahom p(x) - \fint_{\cu^\circ_{\mathrm{int}}} p(x)\cdot \ahom p(x) \, dx \right| } \qquad & \\ & \leq 3^{-m \alpha } \frac{3^{dk}}{|\cu^\circ_{\mathrm{int}}|} \sum_{x \in  3^{k} \Zd \cap \cu^\circ_{\mathrm{int}} } \left( \fint_{x+\cu_n+ \cu_k} \left| \nabla w (z)\right|^2 \, dz\right) \\
																		& \leq 3^{-m \alpha } \left( \frac{1}{|\cu^\circ_{\mathrm{int}}|} \int_{\cu_m} \left| \nabla w (z)\right|^2 \, dz\right) \\
																		&  \leq C  M^2  3^{-m \alpha }.
\end{align*}
The proof of~\eqref{est.lA.1.step2} is complete.

\smallskip
\textit{Step 5.} We want to show
\begin{equation} \label{step3.est}
 \fint_{\cu_m} \nabla \tilde{u}(y) \cdot \ahom \nabla \tilde{u}(y)\, dy \leq \fint_{\cu^\circ_{\mathrm{int}}}  p(y) \cdot \ahom p(y) \, dx + C M^2 3^{-m \alpha }.
\end{equation}
First we need to prove the following estimate: there exists $C := C(s,d,\p, \lambda,p) < + \infty$ such that  for each $p' \in \left[2, \frac{p+2}{2}\right]$,
\begin{equation} \label{uu.est}
\fint_{\cu_m} |\nabla \tilde{u}(x)|^{p'} \, dx \leq C M^{p'}
\end{equation}
Differentiating the expression for $\tilde{u}$, we get
\begin{equation*}
\nabla \tilde{u}(x) =  \nabla \eta(x) (\xi(x) - w (x)) + \eta(x) (\nabla \xi(x) - \nabla w (x)) + \nabla w (x).
\end{equation*}
By Lemma~\ref{l.propwws}(iii), to prove~\eqref{uu.est} it suffices to prove the following estimate:
\begin{equation} \label{xiu.est}
\fint_{\cu^\circ} \left( 3^{-l}\left| \xi(x) - w  (x)\right|+ \left|  \nabla \xi(x)\right| \right)^{p'} \,dx  \leq C  M^{p'}.
\end{equation}
We first estimate the second term in the integrand:
\begin{equation*}
\begin{aligned}
\fint_{\cu^\circ} \left| \nabla \xi(x) \right|^{p'} \,dx  & = \fint_{\cu^\circ} \left| \fint_{x+\cu_n}  \nabla w (y)\,dy\right|^{p'} \,dx \\
							 & \leq \fint_{\cu^\circ} \fint_{x+\cu_n} \left| \nabla w (y)\right|^{p'} \,dy\,dx \\
 							 & \leq \frac{|\cu_m|}{|\cu^\circ|} \fint_{\cu_m}  \left| \nabla w (x)\right|^{p'} \,dx \\
 							 & \leq CM^{p'}.
\end{aligned}
\end{equation*}
To estimate the first term in the integrand, we use that for every $y \in \cu^\circ$,
\begin{align*}
\lefteqn{\frac{1}{|\cu_n|}\int_{y+\cu_n} \left| \xi(x) - w (x)\right|^{p'} \,dx} \qquad & \\  &  = \frac{1}{|\cu_n|} \int_{y+\cu_n} \left| w (x) - \frac{1}{|\cu_n|} \int_{x+\cu_n} w (z)\,dz \right|^{p'} \,dx \\
& \leq C\frac{1}{|\cu_n|} \int_{y+\cu_n} \left| w (x) - \frac{1}{\mathrm{Leb}(\cu_n)} \int_{y+\cu_n} w (z)\,dz \right|^{p'} \,dx \\
 & \qquad + C\frac{1}{|\cu_n|} \int_{y+\cu_n} \left| \frac{1}{|\cu_n|} \int_{x+\cu_n} w (z)\,dz - \frac{1}{|\cu_n|}\int_{y+\cu_n} w (z)\,dz \right|^{p'} \,dx \\
& \qquad +C  \left( \left( 3^n - 1 \right)^{-d} - 3^{-dn} \right)^{p'} \left| \int_{y+\cu_n}  w (z)\,dz \right|^{p'} \,dx.
\end{align*}
Thanks to the Poincar\'e inequality, we can bound the first term on the right side:
\begin{equation*} \label{}
\fint_{y+\cu_n} \left| w (x) -\frac{1}{\mathrm{Leb}(\cu_n)} \int_{y+\cu_n} w (z)\,dz  \right|^{p'} \,dx \leq C3^{p'n} \fint_{y+\cu_n} \left| \nabla w (x) \right|^{p'} \,dx.
\end{equation*}
To compute the second term, we observe that for every $y\in \cu^\circ$ and $x\in y+\cu_n$,
\begin{align*}
\left| \fint_{x+\cu_n} w (z)\,dz - \fint_{y+\cu_n} w (z)\,dz \right| & = \left| \fint_{\cu_n} \int_0^1 (x-y) \cdot \nabla w (tx+(1-t)y+z)\,dt\,dz \right| \\
& \leq C3^{n} \fint_{y+\cu_{n+1}} \left| \nabla w (z) \right| \, dz.
\end{align*}
Assembling these yields
\begin{equation*} \label{}
\fint_{y+\cu_n} \left| \xi(x) - w (x)\right|^{p'} \,dx \leq C3^{p'n} \fint_{y+\cu_{n+1}} \left| \nabla w (x) \right|^{p'} \,dx + 3^{-p'n}\fint_{y+ \cu_n} \left| w(x) \right|^{p'} \, dx
\end{equation*}
and then integrating over $y\in \cu^\circ$ and applying Lemma~\ref{l.propwws}[(iii) and (iv)] yields
\begin{equation*}
\fint_{\cu^\circ} \left| \xi(x) - w (x)\right|^{p'} \,dx \leq C \left( 3^{p'n} + 3^{-p'n + p'm} \right).
\end{equation*}
Inequality~\eqref{uu.est} is then a consequence of the two estimates $l \geq n$ and $l \geq m-n$.

\smallskip

To prove \eqref{step3.est}, notice that for $y \in \cu^\circ,~ \nabla \tilde{u}(y) = p(y)$. Hence we can write
\begin{align*}
 \lefteqn{\fint_{\cu_m} \nabla \tilde{u}(y) \cdot \ahom \nabla \tilde{u}(y)\, dy - \fint_{\cu^\circ} p(y) \cdot \ahom p(y) \, dy} \qquad & \\
& \leq  \frac{|\cu_m \setminus \cu^\circ|}{|\cu_m|} \fint_{\cu^\circ} |\nabla \tilde{u}(x)|^2 \, dx + \frac{1}{|\cu_m|} \int_{\cu_m \setminus \cu^\circ} |\nabla \tilde{u}(x)|^2 \, dx  \\
& \leq C 3^{-m\alpha} M^2 + \left( \frac{|\cu_m \setminus \cu^\circ|}{|\cu|} \right)^{\frac{p-2}{p+2}}\left( \frac{1}{|\cu_m|} \int_{\cu_m \setminus \cu^\circ} |\nabla \tilde{u}(x)|^{\frac{p+2}{2}}\, dx \right)^{\frac{4}{p+2}}\\
& \leq C 3^{-m\alpha} M^2 + C 3^{-m\alpha} \left(  \fint_{\cu_m} |\nabla \tilde{u}(x)|^{\frac{p+2}{2}}\, dx \right)^{\frac{4}{p+2}} \\
& \leq C 3^{-m\alpha} M^2.
\end{align*}
This completes the proof of~\eqref{step3.est} and thus the proof of Lemma~\ref{l.lemmaA1}.
\end{proof}

Before starting the proof of Lemma~\ref{l.lemmaA2} we need to record two estimates from the regularity theory
\begin{proposition}[Meyer and $H^2$ estimates~\cite{Giu}] 
\label{MeyerplusH2}
Suppose $t \in (2, \infty)$, $f \in W^{1,t}(\cu_m)$ and $v \in H^1(\cu_m)$ satisfy
\begin{equation*}
\left\{ 
	\begin{aligned}
		\nabla \cdot \ahom \nabla v = 0 \mbox{ in } \cu_m \\
		v= f  \mbox{ in } \partial \cu_m
	\end{aligned}
\right.
\end{equation*}
Then there exist $r := r(d, \lambda, t) \in (2,t)$ and a constant $C := C (d, \lambda, t) < + \infty$ such that $v \in W^{1,r}(\cu_m)$ and
\begin{equation*}
\left( \fint_{\cu_m} |\nabla v(x)|^r \, dx \right)^{\frac 1r} \leq C \left( \fint_{\cu_m} |\nabla v(x)|^t \, dx \right)^{\frac 1t}.
\end{equation*}
Moreover for every cube $\cu' \subseteq \cu_m$ with $\dist(\cu', \partial \cu_m) > 0, \, v \in H^2(\cu_m)$ and
\begin{equation*}
\dist(\cu', \partial \cu_m) \left( \fint_{\cu'} |\nabla\nabla \tilde{w} (x)|^2 \, dx \right)^{\frac 12} \leq C \left( \fint_{\cu_m} |\nabla v(x)|^t \, dx \right)^{\frac 1t}.
\end{equation*}
\end{proposition}

Applying this result with $t = \frac{p+2}{2}$, $f = w $, $\cu' = \cu^\circ$ and using (iii) of Lemma~\ref{l.propwws} shows: there exists $r := (d, \lambda, p) \in \left(2, \frac{p+2}{2}\right)$ such that
\begin{equation}\label{M.estimateuhom}
\left( \fint_{\cu_m} |\nabla \uhom(x)|^r \, dx \right)^{\frac 1r} \leq C M
\end{equation}
and
\begin{equation} \label{int.H2estimateuhom}
3^l \left( \fint_{\cu^\circ} |\nabla\nabla \uhom (x)|^2 \, dx \right)^{\frac 12} \leq C M.
\end{equation}
These two estimates will be useful in the proof of Lemma~\ref{l.lemmaA2}.

\smallskip

\begin{proof}[Proof of Lemma \ref{l.lemmaA2}]
We construct $\tuhom$ from $\uhom$ by a stitching together mesoscopic minimizers.

For every $y$ in $\cu^\circ$, we set
\begin{multline*}
\zeta(y) := \frac{1}{|\cu_n|}\int_{y+\cu_{n}} \uhom(x) \, dx,~ p(y) := \nabla \zeta(y) =  \frac{1}{|\cu_n|} \int_{y+\cu_{n}} \nabla \uhom (x) \, dx  \\ \mbox{ and } q(y) = \ahom p(y).
\end{multline*}
We begin the construction by defining an affine approximation to $\uhom$ in the mesoscopic cube $y + \cu_{n}$ by setting, for each $y \in \cu^\circ \cap 3^n\Z^d$
\begin{equation*}
l_y := p(y) \cdot (x-y) + \zeta(y).
\end{equation*}
For each $y \in \cu^\circ \cap 3^n\Z^d$, denote by $v_y$ the unique element of $$\A(y + \cu_n) \cap \left(  l_y + \mathcal{C}_0\left(\C_*(y+ \cu_{n})\right) \right),$$  that is, $v_y$ is the unique element of $ l_y + \mathcal{C}_0\left(\C_*(y+ \cu_{n})\right) $ which satisfies:
\begin{equation*}
\left\langle \nabla v_y, \a \nabla v_y \right\rangle_{y + \cu_n} \leq \left\langle \nabla w, \a \nabla w \right\rangle_{y + \cu_n} \mbox{ for every }  w  \in l_y + \mathcal{C}_0\left(\C_*(y+ \cu_{n})\right).
\end{equation*} 
The objective is then to patch these functions together to obtain a function defined on $\C_*(\cu_m)$ but in general we don't have $\C_*(\cu_m) \cap \left( y + \cu_n\right) = \C_*(y+ \cu_{n})$.
To deal with this technical point, we extend $v_y$ to $\C_*(\cu_m) \cap \left( y + \cu_n\right)$ by setting
\begin{equation*}
v_y (x) := l_y(x) \mbox{ for every } x \in  \left( \C_*(\cu_m) \cap \left( y + \cu_n\right) \right) \setminus \C_*(y+ \cu_{n}).
\end{equation*}
In other words, $v_y$ is the maximizer in the definition of $\nu(y+\cu_n,p(y))$ except that we have added a constant to it and extended its definition to the slightly larger set $\C_*(\cu_m) \cap \left( y + \cu_n\right)$. 
We patch these functions together by setting for each $x \in \C_*(\cu_m)$
\begin{equation} \label{tv.def}
\tilde{v}(x):= \sum_{y \in \cu^\circ \cap 3^n\Z^d}  v_y(x) \indc_{\left\{ x \in y + \cu_n \right\}}.
\end{equation}
Finally, we modify $\tilde{v}$ to match the boundary condition. Take $\eta \in C^\infty_0(\Rd)$ to be the cutoff function satisfying~\eqref{def.eta.cutoff} and define, for each $x \in \C_*(\cu_m)$
\begin{equation*} \label{}
\tuhom (x) := \eta(x) \tilde{v}(x) + \left(1 - \eta(x) \right) \uhom (x).
\end{equation*}
With this definition, it is clear that $\tuhom $ satisfies the boundary condition (i):
\begin{equation*}
\forall x \in \C_*(\cu_m) \cap \partial \cu_m,~ \tuhom(x) = u (x).
\end{equation*}
We now prove the energy estimate (ii). We split the proof into three steps
\begin{itemize}
\item In step 1 and 2, we prove the interior estimate
\begin{equation*}
\frac{1}{|\cu^\circ|} \left\langle \nabla \tilde{v}, \a \nabla \tilde{v} \right\rangle_{\C_*(\cu_m) \cap \cu^\circ}  -\fint_{\cu^\circ}  \nabla \uhom(x) \cdot \ahom \nabla \uhom(x) \, dx  \leq  C  M^2 3^{-m\alpha}.
\end{equation*}
\item In step 3 we prove the two boundary estimates
\begin{equation*}
\frac{1}{|\cu_m|} \left\langle \nabla \tuhom, \a \nabla \tuhom \right\rangle_{\C_*(\cu_m)}  \leq \frac{1}{|\cu^\circ|} \left\langle \nabla \tilde{v}, \a \nabla \tilde{v} \right\rangle_{\C_*(\cu_m) \cap \cu^\circ} + C  M^2 3^{-m\alpha}
\end{equation*}
and
\begin{equation*}
\fint_{\cu^\circ}  \nabla \uhom(x) \cdot \ahom \nabla \uhom(x) \, dx \leq \fint_{\cu_m}  \nabla \uhom(x) \cdot \ahom \nabla \uhom(x) \, dx +  C  M^2 3^{-m\alpha}.
\end{equation*}
\end{itemize}
Combining these three results gives the energy estimate (ii).

\smallskip

\textit{Step 1.} In this step, we show the following interior estimate: for each $y \in 3^n \Zd \cap \cu^\circ$
\begin{equation}\label{intermediate.estimate}
 \frac{1}{|\cu_n|}\left\langle \nabla v_y, \a \nabla v_y \right\rangle_{\C_*(\cu_m) \cap (y + \cu_n)} 
 - \frac{1}{|\cu_n|} \int_{y+\cu_{n}} \nabla\uhom(x)\cdot \ahom \nabla\uhom(x) \, dx \leq |p(y)|^2 3^{-m\alpha}.
\end{equation}
For $y \in 3^n \Zd \cap \cu^\circ$, using $v_y$ as a test function for $\mu(y + \cu_n, q(y))$ yields
\begin{equation*} \label{tv.almost.minimizer}
\mu(y + \cu_n, q(y)) \leq \frac{1}{|\cu_n|} \left\langle \nabla v_y, \a \nabla v_y \right\rangle_{\C_*(y + \cu_n)} - \frac{1}{|\cu_n|} \left\langle q(y), \nabla \left[ v_y \right]_\Pa\right\rangle_{y + \cu_n}.
\end{equation*}
As in the proof of Lemma~\ref{l.nablavaffine} we have
\begin{multline*}
\left| \frac{1}{|\cu_n|} \left\langle q(y) , \nabla \left[ v_y \right]_\Pa \right\rangle_{y + \cu_n} - q(y) \cdot p(y) \right|
\\
 \leq C |p(y)||q(y)| \frac{|\partial_\Pa (y + \cu_n)|}{|\cu_n|} 
 +C |p(y)||q(y)| \left( \frac{1}{|\cu_n|} \sum_{x \in \partial (y + \cu_n)} \size(\cu_\Pa (x))^{2d-1} \right)^\frac12.
\end{multline*}
Taking $t$ large enough and using that $ 3^m \geq \mathcal{M}_t(\Qa)$, we have 
\begin{equation*} \label{}
\max_{x \in \cu_m} \size(\cu_\Pa(x)) \leq C 3^{\frac{dm}{d+t}}
\end{equation*}
and thus
\begin{align*}
\frac{1}{|\cu_n|} \sum_{x \in \partial (y + \cu_n)} \size(\cu_\Pa (x))^{2d-1}  \leq C  \frac{3^{\frac{d(2d-1)m}{d+t}}}{3^n}  
															 \leq C 3^{-m \alpha}
\end{align*}
and
\begin{align*}
\frac{|\partial_\Pa (y + \cu_n)|}{|\cu_n|} &= \frac{1}{|\cu_n|}  \sum_{x \in \partial (y + \cu_n)} \size(\cu_\Pa (x)) 
									 \leq  C  \frac{3^{\frac{dm}{d+t}}}{3^m} 
									 \leq C 3^{-m \alpha}.
\end{align*}
Combining the two previous displays yields
\begin{align*}
  \frac{1}{|\cu_n|} \left\langle \nabla v_y, \a \nabla v_y \right\rangle_{\C_*(y+\cu_n)}  - q(y) \cdot p(y) \leq\mu(y + \cu_n, q(y)) + C |p(y)|^2  3^{-m\alpha}.
\end{align*}
Then, by~\eqref{muhasconverged},
\begin{align*}
 | p(y)\cdot \ahom p(y) - q(y) \cdot p(y)  - \mu(y + \cu_n, q(y)) | & = |p(y) \cdot \ahom p(y) +  \mu(y + \cu_n, q(y))|   \\
														& \leq |p(y)|^2 3^{-m\alpha}.
\end{align*}
This shows
\begin{equation*}
 \frac{1}{|\cu_n|} \left\langle \nabla v_y, \a \nabla v_y \right\rangle_{\C_*(y+\cu_n)}  -  p(y)\cdot \ahom p(y)  \leq C |p(y)|^2 3^{-m\alpha}.
\end{equation*}
To complete the proof of~\eqref{intermediate.estimate}, it is sufficient to show the two following inequalities
\begin{equation} \label{intermediate.estimate1}
\fint_{y+\cu_{n}} \nabla\uhom(x)\cdot \ahom \nabla\uhom(x) \, dx \geq p(y)\cdot \ahom p(y) - C |p(y)|^2 3^{-m\alpha}
\end{equation}
and
\begin{equation} \label{intermediate.estimate2}
\frac{1}{|\cu_n|} \left\langle \nabla v_y, \a \nabla v_y \right\rangle_{\C_*(\cu_m) \cap \left( y + \cu_n \right)} \leq  \frac{1}{|\cu_n|} \left\langle \nabla v_y, \a \nabla v_y \right\rangle_{\C_*(y+\cu_n)} + C |p(y)|^2  3^{-m\alpha}.
\end{equation}
The proof of~\eqref{intermediate.estimate1} relies on a convexity argument: we have, for every $p \in \Rd$,
\begin{equation*}
p\cdot \ahom p \geq p(y) \cdot \ahom p(y) + 2 p(y) \cdot \ahom (p-p(y))
\end{equation*}
and rewriting this inequality with $p = \nabla\uhom(x) $ and integrating over $y + \cu_n$ gives
\begin{align*}
\lefteqn{\fint_{y+\cu_{n}} \nabla\uhom(x)\cdot \ahom \nabla\uhom(x) \, dx  } \qquad &
\\
 & \geq \frac{(3^n-1)^d}{|\cu_n|} p(y) \cdot \ahom p(y) + 2 \fint_{y+\cu_{n}} p(y) \cdot \ahom (\nabla\uhom(x)-p(y)) \, dx  
\\ & \geq \frac{(3^n-1)^d}{|\cu_n|} p(y) \cdot \ahom p(y) \geq p(y)\cdot \ahom p(y) - C |p(y)|^2 3^{-m\alpha},
\end{align*}
by definition of $p(y)$. This gives~\eqref{intermediate.estimate1} and we turn to the proof of~\eqref{intermediate.estimate2}. By the construction of the partition $\Pa$, it is clear that $\left( \C_*(\cu_m) \cap \left( y + \cu_n\right) \right) \setminus  \C_*\left( y + \cu_n\right)$ must be contained in the union of elements of $\Pa$ which touch the boundary of the cube $y + \cu_n$. Therefore,
\begin{equation*}
|\left( \C_*(\cu_m) \cap \left( y + \cu_n\right) \right) \setminus  \C_*\left( y + \cu_n\right)| \leq \left| \partial_\Pa (y + \cu_n)  \right|
\end{equation*}
By definition of $v_y$,
\begin{equation*}
v_y (x) := l_y(x) \mbox{ for every } x \in  \left( \C_*(\cu_m) \cap \left( y + \cu_n\right) \right) \setminus \C_*(y+ \cu_{n}).
\end{equation*}
Combining the two previous displays yields
\begin{align*}
\frac{1}{|\cu_n|} \left\langle \nabla v_y, \a \nabla v_y \right\rangle_{\C_*(\cu_m) \cap (y + \cu_n)} & \leq \frac{1}{|\cu_n|} \left\langle \nabla v_y, \a \nabla v_y \right\rangle_{\C_*(y+\cu_n)} + C |p(y)|^2 \frac{|\partial_\Pa (y + \cu_n)|}{|\cu_n|} \\
																& \leq \frac{1}{|\cu_n|} \left\langle \nabla v_y, \a \nabla v_y \right\rangle_{\C_*(y+\cu_n)}  + C  |p(y)|^2 3^{- m \alpha}.
\end{align*}
This completes the proof of~\eqref{intermediate.estimate2} and consequently the proof of~\eqref{intermediate.estimate}.

\smallskip

\textit{Step 2.}  In this step we prove the following interior estimate
\begin{equation} \label{step42.est}
\frac{1}{\left| \cu^\circ_{\mathrm{int}}\right|} \left\langle \nabla \tilde{v}, \a \nabla \tilde{v} \right\rangle_{\cu^\circ_{\mathrm{int}}}  -\fint_{\cu^\circ_{\mathrm{int}}}  \nabla \uhom(x) \cdot \ahom \nabla \uhom(x) \, dx  \leq  C  M^2 3^{-m\alpha}. 
\end{equation}
Summing~\eqref{intermediate.estimate} over all $y \in \cu^\circ_{\mathrm{int}} \cap 3^m\Zd$ and noticing that
\begin{equation*}
 \frac{3^{dn}}{\left|\cu^\circ_{\mathrm{int}}\right|}  \sum_{y \in \cu^\circ_{\mathrm{int}} \cap 3^n\Zd} |p(y)|^2  \leq C  \fint_{\cu_m} |\nabla \uhom(x)|^2 \, dx \leq CM^2
\end{equation*}
gives 
\begin{equation*}
 \frac{1}{\left|\cu^\circ_{\mathrm{int}}\right|} \sum_{y \in \cu^\circ_{\mathrm{int}} \cap 3^n\Zd} \left\langle \nabla \tilde{v}, \a \nabla \tilde{v} \right\rangle_{\C_*(\cu_m) \cap \left( y + \cu_n \right) } - \fint_{\cu^\circ} \nabla \uhom(x) \cdot \ahom \nabla \uhom(x) \, dx \leq  C   M^2  3^{-m \alpha}. 
\end{equation*}
The technical point relies on the fact that $\C_*(\cu_m)$ contains edges which do not belong to any of the clusters $\left( \C_*(\cu_m) \cap \left( y + \cu_m \right) \right)_{y \in 3^n \Zd \cap \cu^\circ_{\mathrm{int}}} $. These are contained in the set $V$ of edges linking two vertices in different triadic cubes of size $3^n$, i.e,
\begin{equation*}
V := \left\{ \lbrace x,z \rbrace~:~x,z \in \Zd,~ \exists y,y' \in 3^n \Zd \cap \cu^\circ_{\mathrm{int}},~y \neq y',~ x \in y + \cu_n \mbox{ and } z \in y' + \cu_n \right\}.
\end{equation*}
Let $e =  \lbrace x,z \rbrace \in \mathcal{B}_d\left(\C_*(\cu_m)\right) \cap V$ be an edge linking the two triadic cubes $y + \cu_n$ and $y' + \cu_n$. We have
\begin{align*}
|\nabla \tilde{v} (e)| & =  |l_y(x) - l_{y'}(z)| \\
				& = |p(y) \cdot (x-y) + \zeta(y) - p(y') \cdot (z-y') - \zeta(y')| \\
				& \leq |(p(y) - p(y')) \cdot (x-y)| + |p(y)| |(x - z)| + |p(y)\cdot (y' - y) + \zeta(y) - \zeta(y')| \\
				&  \leq 3^n |p(y) - p(y')| + |p(y)| + |p(y) (y' - y) + \zeta(y) - \zeta(y')|.
\end{align*}
We estimate the first term on the right hand side
\begin{align*}
|p(y) - p(y')| & = \left|  \fint_{y + \cu_n} \left( \nabla \uhom(x) -  \nabla \uhom(x + y' - y) \right) \, dx \right| \\
			& \leq  \left|  \fint_{y + \cu_n} \int_0^1 \nabla\nabla \uhom(x +(1- t)(y'-y)) \cdot (y-y')\, dt \, dx \right| \\
			& \leq |y-y'|   \int_0^1  \fint_{y + \cu_n}\left| \nabla\nabla \uhom(x +(1- t)(y'-y)) \right| \, dx \, dt \\
			& \leq C 3^n \int_0^1 \fint_{y + \cu_{n+1}} \left| \nabla\nabla \uhom(x) \right| \, dx \, dt \\
			& \leq C 3^n \fint_{y + \cu_{n+1}} \left| \nabla\nabla \uhom(x) \right| \, dx.
\end{align*}
A similar computation yields
\begin{equation*}
|p(y) \cdot(y' - y) + \zeta(y) - \zeta(y')| \leq C 3^{2n} \fint_{y + \cu_{n+1}} \left| \nabla\nabla \uhom(x) \right| \, dx.
\end{equation*}
For $y \in 3^n \Zd \cap \cu^\circ$, denote by $V_y$ the set of edges of $V$ which link $y + \cu_n$ to another cube, i.e,
\begin{equation*}
V_y := \left\{ \lbrace x,z \rbrace \in V~:~ x \in y + \cu_n \mbox{ or } z \in  y + \cu_n \right\}.
\end{equation*}
The previous displays yields, for $y \in 3^n \Zd \cap \cu^\circ_{\mathrm{int}}$,
\begin{align*}
\frac{1}{|\cu_m|} \sum_{e \in V_y \cap \C_*(\cu_m)} |\nabla \tilde{v}(e)|^2 & \leq C \frac{|\partial \cu_n|}{|\cu_n|} \left( |p(y)|^2 + (3^{2n} + 3^{4n}) \fint_{y + \cu_{n+1}} \left| \nabla\nabla \uhom(x) \right|^2 \, dx \right)\\
															& \leq C 3^{-n} \left( |p(y)|^2 +  3^{4n} \fint_{y + \cu_{n+1}} \left| \nabla\nabla \uhom(x) \right| \, dx \right).
\end{align*}
Thus
\begin{align*}
\lefteqn{ \left| \frac{1}{|\cu^\circ|} \sum_{y \in \cu^\circ \cap 3^n\Zd} \left\langle \nabla \tilde{v}, \a \nabla \tilde{v} \right\rangle_{\C_*(\cu_m) \cap \left( y + \cu_n \right) } -  \frac{1}{|\cu^\circ|} \left\langle \nabla \tilde{v}, \a \nabla \tilde{v} \right\rangle_{\C_*(\cu_m) \cap \cu^\circ } \right| } \qquad & \\
& \leq  \frac{1}{|\cu^\circ|} \sum_{y \in \cu^\circ \cap 3^n\Zd} \sum_{e \in V_y \cap \C_*(\cu_m)} |\nabla \tilde{v}(e)|^2 \\
& \leq C \frac{|\cu_m|}{|\cu^\circ|} \sum_{y \in \cu^\circ \cap 3^n\Zd}  \left( 3^{-n} |p(y)|^2 +  3^{3n} \fint_{y + \cu_{n+1}} \left| \nabla\nabla \uhom(x) \right|^2 \, dx \right) \\
& \leq C 3^{-n} M^2 + C 3^{3n} \left(  \fint_{\cu^\circ} \left| \nabla\nabla \uhom(x) \right|^2 \, dx \right) \\
& \leq C 3^{-n} M^2 + C 3^{3n - 2l} M^2.
\end{align*}
Here we used~\eqref{int.H2estimateuhom} to get the last line. 
Recall that we defined $m = \left\lceil \frac{ m}{4}\right\rceil$ and $l = \left\lceil \frac{ 3m}{4}\right\rceil$ and so, for some $\alpha > 0$, 
\begin{equation*}
3^{-n} \leq C 3^{-m\alpha} \mbox{ and }  3^{3n - 2l} \leq 3^{-m\alpha}.
\end{equation*}
This completes the proof of~\eqref{step42.est}.

\smallskip

\textit{Step 3.} In this final step, we estimate the contribution of $\uhom$ and $\tuhom$ in the boundary strip $ \cu_m \setminus \cu^\circ_{\mathrm{int}}$. The claim is that
\begin{equation}\label{2step32}
\fint_{\cu^\circ_{\mathrm{int}}}  \nabla \uhom(x) \cdot \ahom \nabla \uhom(x) \, dx \leq \fint_{\cu_m}  \nabla \uhom(x) \cdot \ahom \nabla \uhom(x) \, dx +  C  M^2 3^{-m\alpha}
\end{equation}
and
\begin{equation} \label{2step31}
\frac{1}{|\cu_m|} \left\langle \nabla \tuhom, \a \nabla \tuhom \right\rangle_{\C_*(\cu_m)}  \leq \frac{1}{|\cu^\circ_{\mathrm{int}}|} \left\langle \nabla \tilde{v}, \a \nabla \tilde{v} \right\rangle_{\C_*(\cu_m) \cap \cu^\circ_{\mathrm{int}}} + C  M^2 3^{-m\alpha}.
\end{equation}
To prove~\eqref{2step32}, we first recall the Meyer estimate~\eqref{M.estimateuhom} which gives us, for some $r := (d, \lambda, p) \in \left(2, \frac{p+2}{2}\right)$,
\begin{equation*}
\left( \fint_{\cu_m} |\nabla \uhom(x)|^r \, dx \right)^{\frac 1r} \leq C M.
\end{equation*}
This allows to compute
\begin{align*}
\lefteqn{\fint_{\cu^\circ_{\mathrm{int}}}  \nabla \uhom(x) \cdot \ahom \nabla \uhom(x) \, dx - \fint_{\cu_m}  \nabla \uhom(x) \cdot \ahom \nabla \uhom(x) \, dx} \qquad & \\ &  \leq \frac{|\cu_m \setminus \cu^\circ_{\mathrm{int}}|}{|\cu|} \fint_{\cu_m}  |\nabla \uhom(x)|^2 \, dx + \frac{1}{|\cu_m|} \int_{\cu_m \setminus \cu^\circ_{\mathrm{int}}} |\nabla \uhom(x)|^2 \, dx \\
& \leq C 3^{l-n} M^2 + \left( \frac{|\cu_m \setminus \cu^\circ_{\mathrm{int}}|}{|\cu_m|} \right)^{1 - \frac 2r} \left( \frac{1}{|\cu_m|}  \int_{\cu_m \setminus \cu^\circ_{\mathrm{int}}} |\nabla \uhom(x)|^r \, dx \right)^\frac 2r \\
& \leq C M^2 3^{-m \alpha} + 3^{(l-n) \left( 1 - \frac 2r \right)}\left(  \fint_{\cu_m} |\nabla \uhom(x)|^r \, dx \right)^\frac 2r \\
& \leq  C M^2 3^{-m \alpha} .
\end{align*}
This completes the proof of~\eqref{2step32}. 

\smallskip

The proof of~\eqref{2step31} follows from a similar computation but we need to prove the following discrete estimates.
\begin{equation} \label{step52.est}
\frac{1}{|\cu_m|} \sum_{x \in \C_*(\cu_m) \cap \left( \cu_m \setminus \cu^\circ_{\mathrm{int}}\right)} |\nabla \tuhom \indc_{\lbrace \a \neq 0 \rbrace}|^2(x)  \leq  C M^2 3^{-m \alpha}.
\end{equation}
and
\begin{equation}\label{step522.est}
\frac{1}{|\cu_m|} \sum_{x \in \C_*(\cu_m)} |\nabla \tuhom \indc_{\lbrace \a \neq 0 \rbrace}|^2(x)  \leq  C M^2 .
\end{equation}
We will only prove~\eqref{step52.est}. The proof of~\eqref{step522.est} is almost the same and can be easily extracted from the proof of~\eqref{step52.est}.

\smallskip

For $x,y \in \C_*(\cu_m)$ such that $x \sim y$ and $\a(\{ x , y \}) \neq 0$, we compute
\begin{multline*}
\nabla \tuhom((x,y)) \\
 = \eta(y) \nabla \tilde{v}((x,y)) + (1- \eta(y)) \nabla \uhom ((x,y)) + \nabla \eta((x,y)) ( \tilde{v}(x) - \uhom (x)).
\end{multline*}
Thus to prove~\eqref{step52.est} it is sufficient to prove the following three estimates
\begin{enumerate}
\item An estimate on $\nabla \tilde{v}$:
\begin{equation*}
\frac{1}{|\cu_m|} \sum_{x \in \C_*(\cu_m) \cap \left( \cu^\circ \setminus \cu^\circ_{\mathrm{int}}\right) } |\nabla \tilde{v} \indc_{\lbrace \a \neq 0 \rbrace}|^2(x) \leq C M^2  3^{-m \alpha}.
\end{equation*}
\item An estimate on $\nabla \uhom$:
\begin{equation*}
\frac{1}{|\cu_m|} \sum_{x \in \C_*(\cu_m) \cap \left( \cu_m \setminus \cu^\circ_{\mathrm{int}}\right) } |\nabla \uhom \indc_{\lbrace \a \neq 0 \rbrace}|^2(x) \leq C M^2  3^{-m \alpha}.
\end{equation*}
\item An estimate on $ \tilde{v}- \uhom $:
\begin{equation*}
\frac{3^{-2l}}{|\cu_m|} \sum_{x \in \C_*(\cu_m) \cap \cu^\circ } ( \tilde{v}(x) - \uhom (x))^2 \leq  C M^2  3^{-m \alpha}.
\end{equation*}
\end{enumerate}
We prove the first estimate (1). For $y \in 3^n \Zd \cap \cu^\circ$, we have
\begin{align*}
\lefteqn{
\frac{1}{|\cu_n|} \sum_{x \in \C_*(\cu_m) \cap (y + \cu_n)} |\nabla v_y \indc_{\lbrace \a \neq 0 \rbrace} |^2(x)
} 
\ \ & \\ & \leq \frac{1}{|\cu_n|} \sum_{x \in \C_*(y + \cu_n)} |\nabla v_y \indc_{\lbrace \a \neq 0 \rbrace} |^2(x) + \frac{1}{|\cu_n|} \sum_{x \in  \C_*(\cu_m) \cap (y + \cu_n) \setminus \C_*(y + \cu_n)} |\nabla l_y \indc_{\lbrace \a \neq 0 \rbrace}|^2(x)\\
& \leq  \frac{1}{|\cu_n|} \left\langle \nabla v_y, \a \nabla v_y \right\rangle_{ \C_*(y + \cu_n)} + \frac{C}{|\cu_n|} \sum_{x \in  \C_*(\cu_m) \cap (y + \cu_n) \setminus \C_*(y + \cu_n)} |p(y)|^2 \\
& \leq \frac{1}{|\cu_n|} \left\langle \nabla l_y, \a \nabla l_y \right\rangle_{ \C_*(y + \cu_n)}  +C |p(y)|^2 \\
& \leq C |p(y)|^2.
\end{align*}
As in the second step, we prove that for every edges $e$ belonging to the cluster $\C_*(\cu_m)$ and connecting the triadic cube $y + \cu_n$ to another triadic cube of the form $y' + \cu_n$
\begin{equation*}
 |\nabla \tilde{v}(e)|^2 \leq |p(y)|^2 +  3^{4n} \fint_{y + \cu_{n+1}} \left| \nabla\nabla \uhom(x) \right|^2 \, dx.
\end{equation*}
Notice that there are at most $C 3^{(d-1)n}$ such edges since they must lie on the boundary of $y + \cu_n$. Combining the two previous displays yields
\begin{align*}
\lefteqn{
\frac{1}{|\cu_n|} \sum_{x \in \C_*(\cu_m) \cap (y + \cu_n)}  |\nabla \tilde{v}|^2 (x) 
} \qquad & \\
&\leq C |p(y)|^2 + C 3^{-n} |p(y)|^2 + C   3^{3n} \fint_{y + \cu_{n+1}} \left| \nabla\nabla \uhom(x) \right|^2 \, dx \\
																& \leq C |p(y)|^2 + C   3^{3n} \fint_{y + \cu_{n+1}} \left| \nabla\nabla \uhom(x) \right|^2 \, dx \\
																& \leq C\fint_{y + \cu_{n}} \left| \nabla \uhom(x) \right|^2 \, dx + C 3^{3n} \fint_{y + \cu_{n+1}} \left| \nabla\nabla \uhom(x) \right|^2 \, dx.
\end{align*}
Summing over $y \in 3^n \Zd \cap (\cu^\circ \setminus \cu^\circ_{\mathrm{int}})$ gives
\begin{multline*}
\frac{1}{|\cu_m|} \sum_{x \in \C_*(\cu_m) \cap \left( \cu^\circ \setminus \cu^\circ_{\mathrm{int}} \right)}  |\nabla \tilde{v}|^2 (x)  \leq \frac{C}{|\cu_m|} \int_{\cu^\circ \setminus \cu^\circ_{\mathrm{int}}} \left| \nabla \uhom(x) \right|^2 \, dx \\ + C \frac{3^{3n}}{|\cu_m|} \int_{ \cu^\circ \setminus \cu^\circ_{\mathrm{int}} + \cu_n} \left| \nabla\nabla \uhom(x) \right|^2 \, dx.
\end{multline*}
The first term of the right hand side can be estimated using the Meyer's estimate as in the proof of~\eqref{2step32}. We obtain
\begin{equation*}
\frac{1}{|\cu_m|} \int_{\cu^\circ \setminus \cu^\circ_{\mathrm{int}}} \left| \nabla \uhom(x) \right|^2 \, dx \leq C M^2 3^{-m \alpha}.
\end{equation*}
The second term of the right hand side can be estimated using the interior $H^2$ estimate in Proposition~\ref{MeyerplusH2}
\begin{align*}
 \frac{3^{3n}}{|\cu_m|}  \fint_{ \cu^\circ \setminus \cu^\circ_{\mathrm{int}} + \cu_n} \left| \nabla\nabla \uhom(x) \right|^2 \, dx & \leq  \frac{3^{3n}}{|\cu_m|} \fint_{ \cu^\circ \setminus \cu^\circ_{\mathrm{int}} + \cu_n} \left| \nabla\nabla \uhom(x) \right|^2 \, dx \\
																			& \leq  \frac{3^{3n}}{|\cu_m|}  \fint_{ \cu^\circ + \cu_n} \left| \nabla\nabla \uhom(x) \right|^2 \, dx \\
																			& \leq C 3^{3n - 2l} M^2 \\
																			& \leq C M^2 3^{-m\alpha}.
\end{align*}
The proof of (1) is complete.

To prove (2) we will prove the stronger result
\begin{equation*}
\frac{1}{|\cu_m|} \sum_{x \in \cu_m} |\nabla \uhom|^{r \wedge \frac{p+2}{2}}(x) \leq C M^{r \wedge \frac{p+2}{2}}  3^{-m\alpha}.
\end{equation*}
where $r > 2$ is the exponent which appears in the Meyer's estimate~\eqref{int.H2estimateuhom}. This implies (2) by the same argument as in the proof of~\eqref{2step32}. 
Let $x,z \in \Zd \cap \cu_m$ with $x \sim z$. We need to distinguish three cases.

\textit{Case 1: $x,z \in \intr(\cu_m)$.} If  $x,z \in \intr(\cu_m)$
\begin{equation*}
\nabla \uhom\left(\lbrace x,z \rbrace \right) = \uhom(z) - \uhom(x) = \int_0^1 \nabla \uhom(tx + (1-t)z) \cdot (x-z) \, dt.
\end{equation*}
Since $\uhom$ is $\ahom$-harmonic, $\nabla u$ is also $\ahom$-harmonic so it satisfies the mean value principle: for every $x \in \cu_m$ and every $R > 0$ such that $B(x,R) \subseteq \cu_m$,
\begin{equation*}
\nabla u = \det(\ahom)^{-1} \fint_{\ahom^{-1} B(x,R)} \nabla u (y) \, dy.
\end{equation*}
Denote for $x,z \in \intr \cu_m$,
\begin{equation*}
W_{x,z} := \left( x + \left(-\frac 12, \frac 12\right)^d \right) \bigcup \left( z + \left(-\frac 12, \frac 12\right)^d \right).
\end{equation*}
From this we deduce that, for every $x,z \in \intr \cu_m$ and every $t \in [0,1]$
\begin{equation*}
\left| \nabla \uhom(tx + (1-t)z)\right| \leq C \int_{W_{x,z}}  |\nabla u (y)| \, dy,
\end{equation*}
thus
\begin{equation*}
\left| \nabla \uhom\left(\lbrace x,z \rbrace \right) \right| \leq C \int_{W_{x,z}}   |\nabla u (y)|  \, dy.
\end{equation*}
By the Jensen inequality, we obtain
\begin{equation*}
\left| \nabla \uhom\left(\lbrace x,z \rbrace \right) \right|^{r \wedge \frac{p+2}{2}} \leq C \int_{W'_{x,z}} \left|  \nabla \uhom(y ) \right|^{r \wedge \frac{p+2}{2}} \, dy.
\end{equation*}

\textit{Case 2: $x,z \in \partial \cu_m$.}
\begin{equation*}
\nabla \uhom\left(\lbrace x,z \rbrace \right) = \nabla w \left(\lbrace x,z \rbrace \right).
\end{equation*}

\textit{Case 3:  $x \in \intr(\cu_m)$ and $z \in \partial \cu_m$.}
Without loss of generality, we assume that $x-z = \e_1$. Denote by $S_0$ the surface
\begin{equation*}
S_0 := z + \left\{ 0 \right\} \times \left(-\frac 14, \frac 14 \right)^{d-1}
\end{equation*}
and $S_1$ its translation of vector $\e_1$
\begin{equation*}
S_1 := S_0 + \e_1 = z + \left\{ 1 \right\} \times \left(-\frac 14, \frac 14 \right)^{d-1}.
\end{equation*}
By definition of $w $ we have for each $y \in S_0$,
\begin{equation*}
w (y) = w (z).
\end{equation*}
Since $\uhom = w $ on the boundary of $\partial \cu_m$, for each $y \in S_0$,
\begin{equation*}
\uhom(z) = \uhom(y).
\end{equation*} 
With this in mind we have
\begin{align*}
\lefteqn{
\left| \uhom(z) - \uhom(x)\right| 
} \qquad & \\
& = \left| \fint_{S_0} \uhom (y) \, dy - \uhom(x)\right| \\
								& \leq C \left| \int_0^1 \fint_{S_0} \nabla \uhom(y+ t\e_1) \, dy \, dt\right| 
								 + C   \left| \fint_{S_1} \uhom(y) \, dy - \uhom(x)\right| \\
								& \leq  C \int_0^1   \fint_{S_0}\left| \nabla \uhom(y + t\e_1) \right| \, dy \, dt +  C   \fint_{S_1} \left|  \uhom(y)- \uhom(x)\right|  \, dy ,
\end{align*}
or for each $y \in S_1$,
\begin{align*}
 \left|  \uhom(y)- \uhom(x)\right|& \leq  \left| \int_0^1 \nabla \uhom(ty + (1-t)x ) \, dt \right| \\
								& \leq \int_0^1  \left|  \nabla \uhom(ty + (1-t)x ) \right|  \, dt
\intertext{by the mean value property}
								& \leq C  \int_0^1 \int_{ x+(-\frac 12, \frac 12)^d}   \left|  \nabla \uhom(y ) \right| \, dy  \, dt \\
								& \leq  C \int_{ x+(-\frac 12, \frac 12)^d}   \left|  \nabla \uhom(y ) \right| \, dy.
\end{align*}
Denote by
\begin{equation*}
W'_{x,z} := \left( x + \left(-\frac 12, \frac 12\right)^d \right) \bigcup \left( \left( z + \left(-\frac 12, \frac 12\right)^d \right) \cap \cu_m \right).
\end{equation*}
The previous computation yields
\begin{equation*}
\left| \nabla \uhom\left(\lbrace x,z \rbrace \right) \right| \leq C \int_{W'_{x,z}} \left|  \nabla \uhom(y ) \right| \, dy.
\end{equation*}
We then apply the Jensen inequality to obtain
\begin{equation*}
\left| \nabla \uhom\left(\lbrace x,z \rbrace \right) \right|^{r \wedge \frac{p+2}{2}} \leq C \int_{W'_{x,z}} \left|  \nabla \uhom(y ) \right|^{r \wedge \frac{p+2}{2}} \, dy.
\end{equation*}
Summing over all the edges of $\cu_m$ gives
\begin{equation*}
\frac{1}{|\cu_m|} \sum_{x \in \cu_m} |\nabla \uhom|^{r \wedge \frac{p+2}{2}} (x)  \leq \fint_{\cu_m} \left|  \nabla \uhom(z ) \right|^{r \wedge \frac{p+2}{2}} \, dz + \frac{1}{|\cu_m|} \sum_{x \in \partial \cu_m} \left| \nabla w  \right|^{r \wedge \frac{p+2}{2}}  (x).
\end{equation*}
By Lemma~\ref{l.propwws} (ii) and a similar computation as in the proof of Lemma~\ref{l.propwws} (iii), we obtain
\begin{equation*}
\frac{1}{|\cu_m|} \sum_{x \in \partial \cu_m} \left| \nabla w  \right|^{r \wedge \frac{p+2}{2}}  (x) \leq C M^{r \wedge \frac{p+2}{2}}.
\end{equation*}
Combining the two previous displays yields
\begin{equation*}
\frac{1}{|\cu_m|} \sum_{x \in \cu_m} |\nabla \uhom|^{r \wedge \frac{p+2}{2}} (x)\leq C M^{r \wedge \frac{p+2}{2}} .
\end{equation*}

We now prove (3). Since $3^m \geq \M_t(\mathcal{Q})$, picking $t$ large enough, we obtain by applying the Poincar\'e inequality (which is a consequence of Proposition~\ref{p.sobolev} with $s=2$) combined with the bound on the $L^2$ norm of $v_y$ given by~\eqref{e.vupbound}: for every $y \in 3^n \Zd \cap \cu$,
\begin{align*}
\lefteqn{
\frac{3^{-2n}}{|\cu_n|} \sum_{x \in \C_*(\cu_m) \cap \left( y + \cu_n\right) } ( v_y(x) - l_y(x))^2 
} \qquad & \\
& \leq \frac{C}{|\cu_n|}  \sum_{x \in \C_*(\cu_m) \cap \left( y + \cu_n\right) } \left| \nabla (v_y - l_y)\indc_{\{ \a \neq 0 \}} \right|^2 (x) \\
																				& \leq  \frac{C}{|\cu_n|}  \sum_{x \in \C_*(\cu_m) \cap \left( y + \cu_n\right) } \left| \nabla v_y\indc_{\{ \a \neq 0 \}} \right|^2 (x) + C |p(y)|^2\\
																				& \leq C  |p(y)|^2.
\end{align*}
Since $l_y$ is also $\ahom$-harmonic we can apply the mean value principle to the function $l_y(x) -  \uhom(x)$ as in the proof of (2),
\begin{equation*} 
\frac{3^{-2l}}{|\cu_m|} \sum_{x \in \C_*(\cu_m) \cap \left( y + \cu_n\right) } ( l_y(x) -  \uhom(x))^2 \leq C 3^{-2l} \fint_{y+\cu_{n}} ( l_y(x) -  \uhom(x))^2 \, dx.
\end{equation*}
 Applying the Poincar\'e's inequality  twice and taking into account that $\fint_{y+  \cu_n} = 3^{-dn} \int_{y+  \cu_n} \neq \frac{1}{\mathrm{Leb}(\cu_n)}\int_{y+  \cu_n}$ by the conventions established at the beginning of this section then gives
\begin{align*}
\lefteqn{\frac{3^{-2l}}{|\cu_n|} \sum_{x \in \C_*(\cu_m) \cap \left( y + \cu_n\right) } ( l_y(x) -  \uhom(x))^2} \qquad & \\ & \leq C 3^{-2l} \fint_{y+\cu_{n}} ( l_y(x) -  \uhom(x))^2 \, dx  \\ 
																					& \leq C3^{2n-2l} \fint_{y+\cu_{n}} |p(y) - \nabla\uhom(x)|^2 \, dx + C  3^{-2n-2l} \fint_{y+\cu_{n}} |\uhom(x)|^2 \, dx   \\
																					&\leq C 3^{4n-2l} \fint_{y+\cu_{n}} | \nabla\nabla\uhom(x)|^2 \, dx + C 3^{-2l} |p(y)|^2 + C 3^{-2n-2l} \fint_{y+\cu_{n}} |\uhom(x)|^2 \, dx . 
\end{align*}
Combining the two previous displays and using that $l - n \geq \alpha m $ yields, for each $y \in 3^n \Zd \cap \cu$
\begin{multline*}
\frac{3^{-2l}}{|\cu_n|} \sum_{x \in \C_*(\cu_m) \cap \left( y + \cu_n\right) } ( v_y(x)-  \uhom(x))^2 \\
 \leq C 3^{2n} \fint_{y+\cu_{n}} | \nabla\nabla\uhom(x)|^2 \, dx + C 
|p(y)|^23^{-m\alpha} 
+ C 3^{-2n-2l} \fint_{y+\cu_{n}} |\uhom(x)|^2 \, dx.
\end{multline*}
Summing over $y \in 3^m \Zd \cap \cu^\circ$ gives
\begin{multline*} 
\frac{3^{-2l}}{|\cu_m|} \sum_{x \in \C_*(\cu_m) \cap \cu^\circ } ( v_y(x)-  \uhom(x))^2  \leq C \frac{3^{2n}}{|\cu_m|} \int_{\cu^\circ} | \nabla\nabla\uhom(x)|^2 \, dx  \\ + \frac{C3^{-m\alpha}}{|\cu_m|} \int_{\cu^\circ} | \nabla \uhom(x)|^2 \, dx + C \frac{3^{-2n-2l}}{|\cu_m|} \int_{\cu^\circ} |\uhom(x)|^2 \, dx
\end{multline*}
To estimate the last term on the right-hand side, we recall that $\uhom \in w + H^1_0(\cu_m)$ by applying the Poincar\'e inequality
\begin{align*}
\frac{1}{|\cu_m|} \int_{\cu^\circ} |\uhom(x)|^2 \, dx & \leq \fint_{\cu_m} |\uhom(x)|^2 \, dx \\
										& \leq C \fint_{\cu_m}  |w(x)|^2 \, dx + C 3^{2m}  \fint_{\cu_m}  |\nabla \uhom (x) - \nabla w(x)|^2 \, dx \\
										& \leq  C \fint_{\cu_m}  |w(x)|^2 \, dx + C 3^{2m}  \fint_{\cu_m}  |\nabla w(x)|^2 \, dx \\
										& \leq C M^23^{2m},
\end{align*}
by Lemma~\ref{l.propwws}[(iii) and (iv)]. The estimate stated in (3) is then a consequence of the interior $H^2$ estimate of Proposition~\ref{MeyerplusH2} and the fact that by definition of $l$ and $n$ we have $l + n - m \geq \alpha m$. This completes the proof of (3) and consequently the proof of Lemma~\ref{l.lemmaA2}(ii) is complete.

\smallskip

We now prove the $L^2$ estimate in (iii). 
By definition of $\uhom$, we have
\begin{equation*}
\tuhom(x)- \uhom(x) = \eta(x) (\tilde{v}(x)- \uhom(x)).
\end{equation*}
Using
the estimates above and $m- l \geq \alpha m$ gives
\begin{align*}
 \frac{3^{-2m}}{|\cu_m|} \sum_{x \in \C_*(\cu_m)} (\tuhom(x)- \uhom(x))^2 & \leq 3^{2(l-m)} \frac{3^{-2l}}{|\cu_m|} \sum_{x \in \C_*(\cu_m) \cap \cu^\circ } (\tilde{v}(x)- \uhom(x))^2\\
														& \leq  C  M^2 3^{-m\alpha}
\end{align*}
and the proof of (iii) is complete.
\end{proof}

\section{Regularity theory}
\label{s.regularity}

With Theorem~\ref{t.homog} now proved, the second main result of the paper, Theorem~\ref{t.reg}, essentially follows from the arguments introduced in the uniformly elliptic case in~\cite{AS} and elaborated in~\cite{AKM1,AKM2}. The main idea is that an appropriate quantitative homogenization result, like Theorem~\ref{t.homog}, can be thought of as a result about harmonic approximation: it implies that an arbitrary solution of the heterogeneous equation can be well-approximated by an $\ahom$-harmonic function. This allows us to transfer the regularity possessed by $\ahom$--harmonic functions to $\a$-harmonic functions, following the classical ideas from elliptic regularity theory (as in, for instance, the proofs of the Schauder estimates). Of course, the regularity we obtain will only be valid on length scales on which the harmonic approximation is valid, which in our situation is all scales larger than a fixed (random) length scale of size $\O_s(C)$. 

\smallskip

In this section, we abuse notation by letting~$\cu_m$ denote the \emph{continuum} cube 
\begin{equation*} \label{}
\left( -\frac12 3^m ,\frac123^m\right)^d \subseteq \Rd. 
\end{equation*}
It will be made clear from the context whether $\cu_m$ refers to the continuum cube or the discrete one. We further abuse notation by extending the coarsened function $\left[ u \right]_\Pa$ to be defined on a continuum domain by taking it to be constant on each unit cube of the form $z + \cu_0$ with $z\in\Zd$. To avoid confusion, here we will use the symbols $\int$ and $\fint$ only to denote integration with respect to Lebesgue measure on~$\Rd$ and write sums with $\sum$. 

\smallskip

The first step in the proof of Theorem~\ref{t.reg} is to post-process the error estimate proved in Theorem~\ref{t.homog} by writing it in a form that is more convenient for the analysis in this section. We put it in terms of the coarsened functions $\left[ u \right]_{\Pa}$ and emphasize harmonic approximation. The coarsening causes some technicalities to appear in the statement, so we emphasize that the second and third terms of the right side of~\eqref{e.harmapproxcoarse} can be considered to be ``small." 

\begin{lemma}
\label{l.harmapproxcoarse}
There exist~$s(d,\lambda,\p)>0$,~$\alpha(d,\lambda,\p)>0$,~$C(d,\lambda,\p)<\infty$ and a random variable $\X$ satisfying
\begin{equation} 
\label{e.Xharmapprox}
\X \leq \O_s\left( C \right) 
\end{equation}
such that, for every $m\in\N$ with  $3^m\geq \X$ and every $u \in \A(\C_*(\cu_{m+3}))$,
\begin{multline}
\label{e.harmapproxcoarse}
\inf_{w\in \overline{\A}(\cu_m)} \left\| \left[ u \right]_{\Pa} - w \right\|_{\underline{L}^2(\cu_m)} 
\\
\leq 
C3^{-m\alpha} \inf_{a\in\R} \left\| \left[ u \right]_\Pa - a \right\|_{ \underline{L}^2(\cu_{m})} 
+ C3^{-m\alpha} \sum_{k=1}^{n-1} 3^{-k\left(m+\frac32(k+1)\right)} \inf_{a\in\R} \left\| \left[ u \right]_\Pa - a \right\|_{ \underline{L}^2(\cu_{m+3k})} 	
\\ 
  + C 3^{-m\alpha-n\left(m+\frac32(n+1)\right)} \inf_{a\in\R} \left| \cu_{m+3n} \right|^{-\frac 1{2}} \left\| u - a \right\|_{ {L}^2(\C_*(\cu_{m+3n}))}.
\end{multline}
\end{lemma}
\begin{proof}
We take $p:= 2+\ep$, where $\ep(d,\lambda,\p)>0$ is as in the statement of Proposition~\ref{p.meyers}. We also take~$\alpha(d,\lambda,\p)>0$ to be the exponent given in Theorem~\ref{t.homog} with respect to the exponent~$p$ above (and which may be made smaller, if desired, in the course of the argument) and $\X$ to be the maximum of the random variable $\X$ appearing in the statement of Theorem~\ref{t.homog} and $\M_t(\Qa)+C'$ appearing in~\eqref{e.Meyersimpl} with $C'$ and $t$ are large constants depending on $(d,\lambda,\p)$ to be selected in the course of the argument. It is clear then that $\X$ satisfies~\eqref{e.Xharmapprox} for an exponent $s(d,\lambda,\p)>0$ and constant $C(d,\lambda,\p)>0$. 

\smallskip

\emph{Step 1.} We slightly tweak the statement of the error estimate. The claim is that there exists an exponent $p'(d,\lambda,\p)>2$ such that, 
for every $m\in\N$ with $3^m\geq \X$ and every $u \in \A(\C_*(\cu_{m+3}))$,
\begin{equation}
\label{e.tweaked}
\inf_{w\in \overline{\A}(\cu_m)} \left| \cu_m\right|^{-\frac1{p'}}\left\| u - w \right\|_{{L}^{p'}(\C_*(\cu_m))} 
\leq
C3^{-m\alpha} \inf_{a\in\R}  \left| \cu_m\right|^{-\frac1{2}}\left\| u - a \right\|_{{L}^2(\C_*(\cu_{m+3}))}.
\end{equation}
Fix $u\in\A(\cu_{m+2})$ and take $w = u_{\mathrm{hom}} \in \overline{\A}(\cu_{m+1})$ to be the $\ahom$-harmonic function given in the statement of Theorem~\ref{t.homog} for the domain $\cu=\cu_{m+1}$. Then the conclusion of Theorem~\ref{t.homog} gives the estimate
\begin{equation*} 
\left| \cu_m \right|^{-\frac12} \left\| u - w \right\|_{{L}^2(\C_*(\cu_{m+1}))} 
\leq C3^{m(1-\alpha)} 
\left| \cu_m \right|^{-\frac1p} \left\| \nabla u \indc_{\{\a\neq0\}} \right\|_{{L}^p(\cu_{m+1})}.
\end{equation*}
The Meyers estimate (Proposition~\ref{p.meyers}) and the Caccioppoli inequality (Lemma~\ref{l.caccioppoli}) yield
\begin{align} 
\label{e.meyercacc}
\left| \cu_m \right|^{-\frac1p} \left\| \nabla u \indc_{\{\a\neq0\}} \right\|_{{L}^p(\cu_{m+1})}
& \leq C\left| \cu_m \right|^{-\frac12} \left\| \nabla u \indc_{\{\a\neq0\}} \right\|_{{L}^2(\cu_{m+2})} 
\\ \notag
& \leq C3^{-m}\inf_{a\in\R}  \left| \cu_m\right|^{-\frac1{2}}\left\| u - a \right\|_{{L}^2(\C_*(\cu_{m+3}))}. 
\end{align}
The last two displays give us
\begin{equation}
\label{e.L2est}
\left\| u - w \right\|_{{L}^2(\C_*(\cu_{m+1}))} 
\leq C3^{-m\alpha}\inf_{a\in\R}  \left| \cu_m\right|^{-\frac1{2}}\left\| u - a \right\|_{{L}^2(\C_*(\cu_{m+3}))}.
\end{equation}
It remains to improve the norm on the left side from $L^2$ to $L^p$. It is clear from the construction in Section~\ref{s.dirichlet}, namely Lemma~\ref{l.propwws}(iii), and the fact that $w$ is harmonic that 
\begin{align} 
\label{e.wgradbound}
\left\| \nabla w \right\|_{L^\infty(\cu_m)} 
\leq
\left\| \nabla w \right\|_{\underline{L}^2(\cu_{m+1})} 
& \leq C\left\| \nabla u \indc_{\{\a\neq0\}} \right\|_{{L}^p(\cu_{m+1})} 
\\ \notag
& \leq C3^{-m} \inf_{a\in\R}  \left| \cu_m\right|^{-\frac1{2}}\left\| u - a \right\|_{{L}^2(\C_*(\cu_{m+3}))}. 
\end{align}
We used~\ref{e.meyercacc} again in the last line above. Now take $p' := \frac12(2+p)$ and apply the Sobolev-Poincar\'e inequality (Proposition~\ref{p.sobolev}) and then the H\"older inequality to deduce that, by taking $t(d,\lambda,\p)<\infty$ sufficiently large, 
\begin{align*} 
\label{}
\lefteqn{
 \left|\cu_m\right|^{-\frac1{p'}} \left\| u - w \right\|_{{L}^{p'}(\C_*(\cu_{m}))} 
 } \qquad & \\
 & \leq C 3^m \left( \frac1{|\cu_m|} \sum_{\cu \in\Pa, \ \cu\subseteq\cu_m} \size(\cu)^t \right)^{\frac1p - \frac1{p'}}   \left|\cu_m\right|^{-\frac1p}  \left\| \nabla (u-w) \indc_{\{\a\neq0\}} \right\|_{L^p(\C_*(\cu_{m}))}
  \\
 & \leq C \inf_{a\in\R}  \left| \cu_m\right|^{-\frac1{2}}\left\| u - a \right\|_{{L}^2(\C_*(\cu_{m+3}))}. 
 \end{align*}
The previous line and~\eqref{e.L2est} give us, by interpolation between $L^2$ and $L^{p'}$ the bound, for the exponent $p'':=\left( \frac12\left( \frac12 + \frac1{p'}\right) \right)^{-1}>2$, of
\begin{equation*} \label{}
 \left|\cu_m\right|^{-\frac1{p''}} \left\| u - w \right\|_{{L}^{p''}(\C_*(\cu_{m}))} \leq C3^{-m\alpha/2}  \inf_{a\in\R}  \left| \cu_m\right|^{-\frac1{2}}\left\| u - a \right\|_{{L}^2(\C_*(\cu_{m+3}))}. 
\end{equation*}
After redefining $\alpha$ to be slightly smaller, this yields~\eqref{e.tweaked}. 

\smallskip

\emph{Step 2.} We estimate the $L^2$ difference between $\left[ u \right]_\Pa$ and $w$ on the entire (continuum) cube $\cu_m$, using the estimate from the previous step. The claim is that, for every $m\in\N$ with $3^m\geq \X$ and every $u \in \A(\C_*(\cu_{m+3}))$,
\begin{equation} 
\label{e.coarsenharmonicapprox}
\inf_{w\in \overline{\A}(\cu_m)}  \left\| \left[ u \right]_\Pa - w \right\|_{\underline{L}^2(\cu_m)} 
\leq 
C3^{-m\alpha} \inf_{a\in\R}  \left| \cu_m\right|^{-\frac12}\left\| u - a \right\|_{{L}^2(\C_*(\cu_{m+3}))}.
\end{equation}
Taking $w$ as in the previous step and using that $ \left[ u \right]_\Pa$ is constant and equal to $u(\bar{z}(\cu))$ on every element $\cu$ of $\Pa$, we see that 
\begin{multline*} \label{}
\left\| \left[ u \right]_\Pa - w \right\|_{\underline{L}^2(\cu_m)}^2 
= \frac{1}{|\cu_m|} \sum_{\cu\in\Pa, \; \cu\subseteq\cu_m} \int_\cu \left| w(x) - u(\bar{z}(\cu)) \right|^2  \\
 \leq \frac{2}{|\cu_m|} \sum_{\cu\in\Pa, \; \cu\subseteq\cu_m} \left( |\cu| \cdot \left| w(\bar{z}(\cu)) - u(\bar{z}(\cu)) \right|^2 + \int_\cu \left| w(x) - w(\bar{z}(\cu)) \right|^2 \right).
\end{multline*}
The estimate for the second term inside the sum follows easily from~\eqref{e.wgradbound}:
\begin{multline*} \label{}
\frac{1}{|\cu_m|} \sum_{\cu\in\Pa, \; \cu\subseteq\cu_m}\int_\cu \left| w(x) - w(\bar{z}(\cu)) \right|^2\,dx \\
\leq \frac{1}{|\cu_m|} \sum_{\cu\in\Pa, \; \cu\subseteq\cu_m} \size(\cu)^2 \left\| \nabla w \right\|_{L^\infty(\cu_m)}^2 
\leq C3^{-2m} \inf_{a\in\R}  \left| \cu_m\right|^{-1}\left\| u - a \right\|_{{L}^2(\C_*(\cu_{m+3}))}^2.
\end{multline*}
The estimate for the first term follows from~\eqref{e.tweaked} and the H\"older inequality, after making the parameter $t(d,\lambda,\p)<\infty$ larger once again:
\begin{align*} \label{}
\lefteqn{
\frac{1}{|\cu_m|} \sum_{\cu\in\Pa, \; \cu\subseteq\cu_m} |\cu| \cdot \left| w(\bar{z}(\cu)) - u(\bar{z}(\cu)) \right|^2
} \qquad & \\
& \leq 
\left( \frac{1}{|\cu_m|} \sum_{\cu\in\Pa, \; \cu\subseteq\cu_m}  |\cu|^{\frac{p'}{p'-2}}  \right)^{\frac{p'-2}{p'}}
\left( \frac{1}{|\cu_m|} \sum_{\cu\in\Pa, \; \cu\subseteq\cu_m} \left| w(\bar{z}(\cu)) - u(\bar{z}(\cu)) \right|^{p'} \right)^{\frac{2}{p'}} \\
& \leq 
C \left( \frac{1}{|\cu_m|} \sum_{x\in \C_*(\cu_m)} \left| w(x) - u(x) \right|^{p'} \right)^{\frac{2}{p'}} \\
& \leq C3^{-2m\alpha} \inf_{a\in\R}  \left| \cu_m\right|^{-1}\left\| u - a \right\|_{{L}^2(\C_*(\cu_{m+3}))}^2.
\end{align*}
Combining the previous three displays yields~\eqref{e.coarsenharmonicapprox}. 

\smallskip

\emph{Step 3.} We compare $u$ and $\left[ u \right]_\Pa$. The claim is that there exists $p'(d,\lambda,\p)>2$ such that, 
for every $m\in\N$ with $3^m\geq \X$ and every $u\in \A\left( \C_*(\cu_{m+3})\right)$, 
\begin{equation} 
\label{e.coarsecmp}
\left| \cu_m \right|^{-\frac 1{p'}} \left\| u - \left[ u \right]_{\Pa} \right\|_{ {L}^{p'}(\C_*(\cu_{m}))}
\leq 
C3^{-m} \inf_{a\in\R} \left| \cu_m \right|^{-\frac 1{2}} \left\| u - a \right\|_{ {L}^2(\C_*(\cu_{m+3}))}.
\end{equation}
In fact, we can take $p':= \frac12(p+2)$. 
Then by~\eqref{e.coarseLs}, the H\"older inequality, and taking $t(d,\lambda,\p)<\infty$ to be large enough, we get
\begin{align*}
\lefteqn{
 \left|\cu_m\right|^{-1}\left\| u - \left[ u \right]_{\Pa} \right\|_{{L}^{p'}(\C_*(\cu_{m+1}))}^{p'}
} \qquad & \\ & 
\leq 
 \frac{C}{ \left| \cu_{m+1} \right|} \sum_{\cu \in \Pa,\, \cu \subseteq \cu_{m+1}}  \size(\cu)^{p'd} \int_{\cu}\left| \nabla u\indc_{\{\a\neq 0\}}\right|^{p'}(x)  \,dx
\\ & 
\leq C \left( \frac{1}{ \left| \cu_{m+1} \right|} \sum_{\cu\in\Pa, \cu\subseteq \cu_{m+1}}  \size(\cu)^{\frac{d(p'+1)p}{p-p'}}   \right)^{\frac{p-p'}{p}} 
 \left\| \nabla u \indc_{\{\a\neq 0\}} \right\|_{\underline{L}^p(\cu_{m+1})}^{p'}
 \\ & 
 \leq C \left\| \nabla u \indc_{\{\a\neq 0\}} \right\|_{\underline{L}^p(\cu_{m+1})}^{p'}.
\end{align*}
Combining this with~\eqref{e.meyercacc}, we get
\begin{equation*} \label{}
\left| \cu_m \right|^{-\frac 1{p'}} \left\| u - \left[ u \right]_{\Pa} \right\|_{ {L}^{p'}(\C_*(\cu_{m+1}))}
\leq
C3^{-m} \inf_{a\in\R} \left| \cu_m \right|^{-\frac 1{2}}  \left\| u-a\right\|_{ {L}^2(\C_*(\cu_{m+3}))}.
\end{equation*}
This yields~\eqref{e.coarsecmp}.

\smallskip

\emph{Step 4.} We complete the proof of the lemma by combining ingredients proved in the previous steps. According to~\eqref{e.coarsecmp} and the triangle inequality, for every $n\in\N$ with $3^n\geq \X$ and every $u\in \A\left( \C_*(\cu_{n+3})\right)$, 
\begin{align*}
\lefteqn{
 \inf_{a\in\R} \left| \cu_m \right|^{-\frac 1{2}} \left\| u - a \right\|_{ {L}^2(\C_*(\cu_{m}))}
}  \qquad & \\ & 
\leq 
 \inf_{a\in\R} \left| \cu_m \right|^{-\frac 1{2}} \left\| \left[ u \right]_\Pa - a \right\|_{ {L}^2(\C_*(\cu_{m}))}
+ \left| \cu_m \right|^{-\frac 1{2}} \left\| u - \left[ u \right]_\Pa \right\|_{ {L}^2(\C_*(\cu_{m}))} \\ & 
\leq 
 \inf_{a\in\R} \left\| \left[ u \right]_\Pa - a \right\|_{ \underline{L}^2(\cu_{m})}
 + C3^{-m} \inf_{a\in\R} \left| \cu_m \right|^{-\frac 1{2}} \left\| u - a \right\|_{ {L}^2(\C_*(\cu_{m+3}))}.
\end{align*}
An iteration of this inequality yields, for every $m,n\in\N$ with $3^m\geq \X$ and every  $u\in \A\left( \C_*(\cu_{m+3n})\right)$, the bound
\begin{align*}
\lefteqn{
 \inf_{a\in\R} \left| \cu_m \right|^{-\frac 1{2}} \left\| u - a \right\|_{ {L}^2(\C_*(\cu_{m}))}
}  \qquad & \\ & 
\leq 
 \inf_{a\in\R} \left\| \left[ u \right]_\Pa - a \right\|_{ \underline{L}^2(\cu_{m})}  +
 C \sum_{k=1}^{n-1} 3^{-\sum_{j=1}^k (m+3j)} \inf_{a\in\R} \left\| \left[ u \right]_\Pa - a \right\|_{ \underline{L}^2(\cu_{m+3k})}  \\ & \qquad 
+ C 3^{-\sum_{j=1}^n (m+3j)} \inf_{a\in\R} \left| \cu_{m+3n} \right|^{-\frac 1{2}} \left\| u - a \right\|_{ {L}^2(\C_*(\cu_{m+3n}))} \\ &
=  \inf_{a\in\R} \left\| \left[ u \right]_\Pa - a \right\|_{ \underline{L}^2(\cu_{m})}  +
 C \sum_{k=1}^{n-1} 3^{-k\left(m+\frac32(k+1)\right)} \inf_{a\in\R} \left\| \left[ u \right]_\Pa - a \right\|_{ \underline{L}^2(\cu_{m+3k})} 	
  \\ & \qquad 
  + C 3^{-n\left(m+\frac32(n+1)\right)} \inf_{a\in\R} \left| \cu_{m+3n} \right|^{-\frac 1{2}} \left\| u - a \right\|_{ {L}^2(\C_*(\cu_{m+3n}))}.
\end{align*}
Applying~\eqref{e.coarsenharmonicapprox}, we deduce that, for every $m,n\in\N$ with $3^m\geq \X$ and every  $u\in \A\left( \C_*(\cu_{m+3n})\right)$
\begin{align*}
\lefteqn{
\inf_{w\in \overline{\A}(\cu_m)}  \left\| \left[ u \right]_\Pa - w \right\|_{\underline{L}^2(\cu_m)} 
} \quad & \\ &
\leq 
C3^{-m\alpha} \inf_{a\in\R} \left\| \left[ u \right]_\Pa - a \right\|_{ \underline{L}^2(\cu_{m})} 
+ C3^{-m\alpha} \sum_{k=1}^{n-1} 3^{-k\left(m+\frac32(k+1)\right)} \inf_{a\in\R} \left\| \left[ u \right]_\Pa - a \right\|_{ \underline{L}^2(\cu_{m+3k})} 	
\\ & \qquad  
  + C 3^{-m\alpha-n\left(m+\frac32(n+1)\right)} \inf_{a\in\R} \left| \cu_{m+3n} \right|^{-\frac 1{2}} \left\| u - a \right\|_{ {L}^2(\C_*(\cu_{m+3n}))}.
\end{align*}
This is~\eqref{e.harmapproxcoarse}. 
\end{proof}

With the result of the previous lemma in mind, we next give an elementary real analysis lemma which formalizes the transfer of regularity from harmonic functions to functions which are well-approximated by harmonic functions on large scales. It is a variation of~\cite[Lemma 2.4]{AKM1}.

\begin{lemma}
\label{l.kiter}
Fix $A\geq 2$, $R\geq 2$, $\alpha>0$ and $u\in L^2(B_R)$. For each $k\in\N$ and $s\in (0,R]$, denote
\begin{equation*} \label{}
D_k(s):= \inf_{w\in \overline{\A}_k} \left\| u - w \right\|_{\underline{L}^2(B_s)}.
\end{equation*}
Assume that $X \in [1,A^{-1} R]$ and $E>0$ have the property that, for every $r\in \left[ X, A^{-1}R \right]$,
\begin{equation}
\label{e.harmapproxhypo}
\inf_{v \in \overline{\A}(B_r)} \left\| u - v \right\|_{\underline{L}^2(B_r)}
 \leq 
 r^{1-\alpha}\left( \sup_{Ar\leq s \leq R} \frac{D_0(s)}{s} + \frac{E}{r} \right). 
\end{equation}
Then, for each $k\in\N$, there exists a constant $C(k,A,\alpha,d,\lambda,\p)<\infty$ such that, for every $r\in \left[ X\vee C ,\frac12 R\right]$, 
\begin{equation}
\label{e.mesoregresult}
D_k(r) \leq C\left( \frac rR \right)^{k+1} D_k(R) + Cr^{1-\alpha} \left(  \frac{D_0(R)}{R} + \frac{E}{r} \right). 
\end{equation}
\end{lemma}
\begin{proof}
Fix $k\in\N$. Throughout we denote by $C$ and $c$ positive constants which depend only on $(k,A,\alpha,d,\lambda,\p)$ and may vary in each occurrence. 

\smallskip

\emph{Step 1.} We show that, for every $r\in \left[ X, \frac12R\right]$ and every $s\in \left(0,\frac12r\right]$, we have
\begin{equation}
\label{e.triangletransfer0}
D_k(s)  \leq 
C\left( \frac sr \right)^{k+1} \inf_{w \in \overline{\A}_k} \left\| u - w \right\|_{\underline{L}^1(B_r)} 
+ C \left( \frac sr \right)^{-\frac d2} \inf_{v \in \overline{\A}} \left\| u - v \right\|_{\underline{L}^2(B_r)}. 
\end{equation}
Select $v$ so that 
\begin{equation*} \label{}
\left\| u - v \right\|_{\underline{L}^2(B_r)}
=
\inf_{v' \in \overline{\A}} \left\| u - v' \right\|_{\underline{L}^2(B_r)}.
\end{equation*}
Since $v$ is an $\ahom$-harmonic function, we have that
\begin{equation*} \label{}
\inf_{w \in \A_k} \left\| v - w \right\|_{{L}^\infty(B_s)}
\leq  C\left( \frac sr \right)^{k+1} \inf_{w \in \overline{\A}_k} \left\| v - w \right\|_{\underline{L}^1(B_r)}.
\end{equation*}
Using this and the triangle inequality twice, we find that 
\begin{align*}
D_k(s) 
& = \inf_{w \in \overline{\A}_k} \left\| u - w \right\|_{\underline{L}^2(B_s)} \\
& \leq \inf_{w \in \overline{\A}_k} \left\| v - w \right\|_{\underline{L}^2(B_s)} + \left\| u - v \right\|_{\underline{L}^2(B_s)}\\
& \leq C\left( \frac sr \right)^{k+1} \inf_{w \in \overline{\A}_k} \left\| v - w \right\|_{\underline{L}^1(B_r)}
 + C \left( \frac{|B_r|}{|B_s|} \right)^{\frac12} \left\| u - v \right\|_{\underline{L}^2(B_r)} \\
& \leq C\left( \frac sr \right)^{k+1} \inf_{w \in \overline{\A}_k} \left\| u - w \right\|_{\underline{L}^1(B_r)}
+ C\left(  \left( \frac sr \right)^{k+1} + \left( \frac{r}{s}\right)^{\frac d2} \right) \left\| u - v \right\|_{\underline{L}^2(B_r)} \\
& \leq C\left( \frac sr \right)^{k+1} \inf_{w \in \overline{\A}_k} \left\| u - w \right\|_{\underline{L}^1(B_r)}
+ C \left( \frac sr \right)^{-\frac d2} \left\| u - v \right\|_{\underline{L}^2(B_r)} \\
& = C\left( \frac sr \right)^{k+1} \inf_{w \in \overline{\A}_k} \left\| u - w \right\|_{\underline{L}^1(B_r)}
+ C \left( \frac sr \right)^{-\frac d2} \inf_{v' \in \overline{\A}} \left\| u - v' \right\|_{\underline{L}^2(B_r)}.
\end{align*}
This is~\eqref{e.triangletransfer0}. Note that it also implies
\begin{equation}
\label{e.triangletransfer}
D_k(s)  \leq 
C\left( \frac sr \right)^{k+1} D_k(r)
+ C \left( \frac sr \right)^{-\frac d2} \inf_{v \in \overline{\A}} \left\| u - v \right\|_{\underline{L}^2(B_r)}. 
\end{equation}

\smallskip

\emph{Step 2.} We organize the rest of the argument. Denote
\begin{equation*} \label{}
\tilde{D}_k(r):= r^{-k} D_k(r) = r^{-k}\inf_{w \in \overline{\A}_k} \left\| u - w \right\|_{\underline{L}^2(B_r)}.
\end{equation*}
We also take $w_{k,r} \in \overline{\A}_k$ such that 
\begin{equation*} \label{}
\left\| u - w_{k,r} \right\|_{\underline{L}^2(B_r)} = \inf_{w \in \overline{\A}_k} \left\| u - w \right\|_{\underline{L}^2(B_r)} = D_k(r).
\end{equation*}
By~\eqref{e.triangletransfer}, there exists $\theta(k,d,\lambda,\p)\in \left(0 ,\frac12 \right]$ such that, for every $r\in \left[X,\frac12R \right]$, 
\begin{equation*} \label{}
\tilde{D}_k(\theta r) \leq \frac12 \tilde{D}_k(r) + Cr^{-k}  \inf_{v \in \overline{\A}} \left\| u - v \right\|_{\underline{L}^2(B_r)}.
\end{equation*}
Using the harmonic approximation hypothesis~\eqref{e.harmapproxhypo}, we get 
\begin{equation}
\label{e.readytoiterate}
\tilde{D}_k(\theta r) \leq \frac12 \tilde{D}_k(r) + Cr^{1-k-\alpha} \left( \sup_{Ar\leq s \leq R} \frac{D_0(s)}{s} + \frac{E}{r} \right). 
\end{equation}
We will complete the proof of the lemma by iterating~\eqref{e.readytoiterate}. We first must take care of the case $k=0$ before handling general $k\in\N$. 

\smallskip

\emph{Step 3.} We prove~\eqref{e.mesoregresult} for $k=0$. That is, we claim that, for every $r\in \left[ X\vee C, \frac12R \right]$, 
\begin{equation} 
\label{e.k=0}
\frac{D_0(r)}{r} \leq C\left( \frac{D_0(R)}{R} + \frac{E}{r^{1+\alpha}} \right). 
\end{equation}
By adding a constant to $u$, we may suppose that $w_{0,R} = 0$. Then 
\begin{align*} \label{}
\left\| w_{1,R} \right\|_{\underline{L}^2(B_R)}
\leq \left\| u \right\|_{\underline{L}^2(B_R)} + \left\| u - w_{1,R} \right\|_{\underline{L}^2(B_R)}
= D_0(R) + D_1(R) \leq 2D_0(R)
\end{align*}
and therefore (note that $w_{1,r}$ is affine so that $\nabla w_{1,r}$ is a constant vector)
\begin{equation} 
\label{e.firstslope}
\left| \nabla w_{1,R} \right| \leq \frac{C}{R} D_0(R). 
\end{equation}
Using the triangle inequality, we find that, for every $n\in\N$ with $\theta^nR \geq X\vee C$, 
\begin{align*} \label{}
\left\| w_{1,\theta^nR} - w_{1,\theta^{n+1}R} \right\|_{\underline{L}^2\left(B_{\theta^{n+1}R}\right)}
& \leq \left\| w_{1,\theta^nR} - u \right\|_{\underline{L}^2\left(B_{\theta^{n+1}R}\right)}
+ \left\| u - w_{1,\theta^{n+1}R} \right\|_{\underline{L}^2\left(B_{\theta^{n+1}R}\right)}
\\ &
\leq \left( \frac{|B_{\theta^nR}|}{|B_{\theta^{n+1}R}|} \right)^{\frac12} D_1\left( \theta^nR \right) + D_1\left( \theta^{n+1}R \right) \\ & 
= \theta^{-\frac d2} D_1(\theta^nR) +  D_1\left( \theta^{n+1}R \right). 
\end{align*}
Hence
\begin{align*} \label{}
 \left| \nabla w_{1,\theta^n R} - \nabla w_{1,\theta^{n+1} R} \right| 
& \leq \frac{1}{\theta^{n+1}R}\left\| w_{1,\theta^nR} - w_{1,\theta^{n+1}R} \right\|_{\underline{L}^2\left(B_{\theta^{n+1}R}\right)} \\ 
& \leq C \left( \tilde{D}_1(\theta^nR) +  \tilde{D}_1\left( \theta^{n+1}R \right) \right).
\end{align*}
Summing this and using~\eqref{e.firstslope}, we deduce that, for every $n\in\N$ with $\theta^nR\geq X\vee C$, 
\begin{align*} \label{}
\left| \nabla w_{1,\theta^{n+1}R} \right| 
& \leq \left| \nabla w_{1,R} \right|  + \sum_{k=0}^n  \left| \nabla w_{1,\theta^k R} - \nabla w_{1,\theta^{k+1} R} \right|  \\
& \leq \frac{C}{R} D_0(R) +  C \sum_{k=0}^{n+1} \tilde{D}_1(\theta^kR).
\end{align*}
Since the triangle inequality gives us
\begin{equation*} \label{}
D_0(\theta^{n+1}R) \leq \theta^{n+1}R \left| \nabla w_{1,\theta^{n+1}R} \right|  + D_1(\theta^{n+1}R),
\end{equation*}
we obtain
\begin{equation} 
\label{e.boomzcha}
\frac{D_0(\theta^{n+1}R)}{\theta^{n+1}R} \leq C \left( \frac{D_0(R)}{R}+  \sum_{k=0}^{n+1} \tilde{D}_1(\theta^kR) \right).
\end{equation}
By an iteration of~\eqref{e.readytoiterate}, we get
\begin{equation*} \label{}
\tilde{D}_1(\theta^{k} R) \leq 2^{-k} \tilde{D}_1(R) + C\left( \theta^kR\right)^{-\alpha} \left( \sup_{A \theta^k \leq s \leq R} \frac{D_0(s)}{s} + \left( \theta^kR\right)^{-1}E \right)
\end{equation*}
and thus 
\begin{equation*} \label{}
\sum_{k=0}^{n+1} \tilde{D}_1(\theta^kR) \leq C\tilde{D}_1(R) +  C\left( \theta^nR\right)^{-\alpha} \left( \sup_{A \theta^{n+1} \leq s \leq R} \frac{D_0(s)}{s} + \left( \theta^nR\right)^{-1}E \right).
\end{equation*}
Combining the above with~\eqref{e.boomzcha} and using $\tilde{D}_1(R) \leq R^{-1}D_0(R)$ yields 
\begin{equation*} \label{}
\frac{D_0(\theta^{n+1}R)}{\theta^{n+1}R} \leq  C\left( \frac{D_0(R)}{R} +\left( \theta^{n+1}R\right)^{-\alpha} \left( \sup_{ \theta^{n+1} \leq s \leq R} \frac{D_0(s)}{s} + \left( \theta^{n+1}R\right)^{-1}E \right) \right).
\end{equation*}
If we take $C$ sufficiently large, we obtain that $\theta^nR\geq X \vee C$ implies $C(\theta^nR)^{-\alpha} \le CX^{-\alpha} \leq \frac14\theta^{1+\frac d2}$ and we get from this and the previous display that 
\begin{equation*} \label{}
\frac{D_0(\theta^{n+1}R)}{\theta^{n+1}R} \leq  \frac{C D_0(R)}{R} +\frac12 \sup_{\theta^n \leq s \leq R} \frac{D_0(s)}{s} + \frac{E}{(\theta^{n+1}R)^{1+\alpha}}.
\end{equation*}
This and an easy induction argument gives us, for all such $n\in\N$,
\begin{equation*} \label{}
\frac{D_0(\theta^{n}R)}{\theta^{n}R} \leq  \frac{C D_0(R)}{R} +  \frac{E}{(\theta^{n+1}R)^{1+\alpha}}.
\end{equation*}
This implies~\eqref{e.k=0}. 

\smallskip

\emph{Step 3.} The proof of~\eqref{e.mesoregresult} for general $k\in\N$. In view of \eqref{e.k=0}, we can improve~\eqref{e.readytoiterate} to the bound 
\begin{equation}
\label{e.morereadytoiterate}
\tilde{D}_k(\theta r) \leq \frac12 \tilde{D}_k(r) + Cr^{1-k-\alpha} \left( \frac{D_0(R)}{R} + \frac{E}{r} \right). 
\end{equation}
The details of the proof of~\eqref{e.mesoregresult} for general $k\in\N$ now follows very closely from the argument in Step~3 in the proof of~\cite[Lemma 2.4]{AKM1}. We omit the details. 
\end{proof}

Following~\cite{AKM1,AKM2}, we next show that Theorem~\ref{t.homog} and Lemma~\ref{l.kiter} imply a form of higher regularity for coarsenings of~$\a$-harmonic functions on mesoscopic scales. The following lemma can be compared to~\cite[Theorem 2.1]{AKM1}.

\begin{lemma}
\label{l.mesoreg}
There exist exponents $s(d)>0$, $\delta(d,\p,\lambda)>0$ and a random variable $\X$ satisfying
\begin{equation*} \label{}
\X \leq \O_s\left( C(d,\p,\lambda) \right)
\end{equation*}
and, for each $k\in\N$, a constant $C(k,d,\p,\lambda)<\infty$ such that, for every $R \ge 2\X$, $u \in \A(\C_\infty\cap B_R)$ and $r\in \left[ \X, \frac12 R \right]$,
\begin{multline} 
\label{e.mesoreg}
\inf_{w\in \overline{\A}_k} \left|B_r \right|^{-\frac12} \left\| u - w \right\|_{{L}^2(\C_\infty\cap B_r)} 
\\
\leq
C \left( \frac rR \right)^{k+1} \inf_{w\in \overline{\A}_k}  \left|B_R \right|^{-\frac12}  \left\| u - w \right\|_{{L}^2(\C_\infty\cap B_R)} 
+ Cr^{-\delta} \left( \frac rR \right) \left| B_R \right|^{-\frac12} \left\| u \right\|_{{L}^2(\C_\infty\cap B_R)}.
\end{multline}
\end{lemma}
\begin{proof}
We take $\X$ as in the proof of Lemma~\ref{l.harmapproxcoarse} and fix $R \ge 2\X$ and $u \in \A(\C_\infty\cap B_R)$. Note that, due to the definition of~$\X$, for every $r\geq \X$ and $m\in\N$ such that $B_r \subseteq \cu_m$, we have that $\C_\infty\cap B_r \subseteq \C_*(\cu_{m+1})$. As in the statement of Lemma~\ref{l.kiter}, we denote
\begin{equation*} \label{}
D_k(s):= \inf_{w\in\overline{\A}_k} \left\|  \left[ u \right]_{\Pa} - w \right\|_{\underline{L}^2(B_s)}. 
\end{equation*}
Also define
\begin{equation*} \label{}
E:= \frac1R \inf_{a\in\R} R^{-\frac d2} \left\| u - a\right\|_{L^2(\C_\infty\cap B_R)}. 
\end{equation*}
Note that if $r \leq \frac1{27} d^{-\frac12}R$, then $\cu_{m+3}\subseteq B_R$. Thus~\eqref{e.harmapproxcoarse} implies that, for every $r\in \left[\X,cR\right]$, 
\begin{equation*}
\inf_{w\in \overline{\A}(B_r)} \left\| \left[ u \right]_{\Pa} - w \right\|_{\underline{L}^2(B_r)} 
\leq 
Cr^{1-\alpha} \sup_{ r\leq s \leq R } \left( \frac{D_0(s)}{s} + \frac{E}{r} \right)
\end{equation*}
An application of Lemma~\ref{l.kiter} therefore gives us that, for every $r\in \left[ \X,cR\right]$,
\begin{equation} 
\label{e.mesoregforcoarse0}
D_k(r) \leq C\left( \frac rR \right)^{k+1} D_k(cR) + Cr^{1-\alpha}  \left( \frac{D_0(cR)}{cR} + \frac{E}{r} \right).
\end{equation}
Using~\eqref{e.triangletransfer0}, we also have that 
\begin{equation*} \label{}
D_k(cR) \leq  C\inf_{w \in \overline{\A}_k} \left\| \left[ u \right]_{\Pa} - w \right\|_{\underline{L}^1(B_{cR})}
\end{equation*}
(the second $c$ is larger) and therefore we deduce that 
\begin{equation} 
\label{e.mesoregforcoarse}
D_k(r) \leq C\left( \frac rR \right)^{k+1} \inf_{w \in \overline{\A}_k} \left\| \left[ u \right]_{\Pa} - w \right\|_{\underline{L}^1(B_{cR})}  + Cr^{1-\alpha}  \left( \frac{D_0(cR)}{cR} + \frac{E}{r} \right).
\end{equation}
The bound~\eqref{e.mesoregforcoarse} is close to the desired result. What is left to do is to rewrite it in terms of~$u$ rather than~$\left[ u \right]_{\Pa}$. To this end, it is useful to define
\begin{equation*} \label{}
\tilde{D}_k(s):= \inf_{w\in\overline{\A}_k} \left| B_s \right|^{-\frac12} \left\| u - w \right\|_{L^2(\C_\infty\cap B_s)}.
\end{equation*}
The required approximations are presented in the following two steps and then the conclusion in the third and final step. As usual, $C$ denotes a positive constant depending only on $(k,d,\p,\lambda)$ which may vary. 

\smallskip

\emph{Step 1.} We claim there exists $C<\infty$ such that, for every $r\in\left[ \X, C^{-1}R \right]$, 
\begin{equation}
\label{e.finalapprox1}
\tilde{D}_k(r) \leq C \left( D_k(Cr) + \frac{1}{r} \tilde{D}_0(Cr) \right).  
\end{equation}
Take $m\in\N$ such that $3^{m-1} \leq r \leq 3^{m}$. Observe that the condition on $r$, if $C$ is large enough, gives us that $3^m \geq \X$ and that $\cu_{m+4} \subseteq B_{R}$. We may therefore apply~\eqref{e.coarsecmp} to deduce 
\begin{align*} \label{}
\tilde{D}_k(r) 
& \leq C \inf_{w\in\overline{\A}_k} \left| \cu_m \right|^{-\frac12} \left\| u - w \right\|_{L^2(\C_\infty\cap \cu_m)} 
\\ &
\leq C \inf_{w\in\overline{\A}_k} \left| \cu_{m} \right|^{-\frac12} \left\| u - w \right\|_{L^2(\C_*(\cu_{m+1}))} 
\\ &
\leq C \inf_{w\in\overline{\A}_k} \left| \cu_{m} \right|^{-\frac12} \left\| \left[ u \right]_\Pa - w \right\|_{L^2(\C_*(\cu_{m+1}))} 
+ C 3^{-m} \inf_{a\in\R} \left| \cu_m \right|^{-\frac12}  \left\| u - a \right\|_{L^2(\C_* (\cu_{m+4}))} 
\\ & 
\leq C \left( D_k(Cr) + \frac1r \tilde{D}_0(Cr) \right),
\end{align*}
which confirms~\eqref{e.finalapprox1}.

\smallskip

\emph{Step 2.} We claim that there exists $C<\infty$ such that
\begin{equation}
 \label{e.puttildeD}
 \inf_{w \in \overline{\A}_k} \left\| \left[ u \right]_{\Pa} - w \right\|_{\underline{L}^1(B_{cR})}
 \leq 
C\left( \tilde{D}_k(R) + \frac{D_0(R)}{R} \right). 
\end{equation}
With $m$ as in Step~1 for $r=cR$, we compute, for any $w\in\overline{\A}_k$, 
\begin{align*} \label{}
\lefteqn{
 \left\| \left[ u \right]_{\Pa} - w \right\|_{\underline{L}^1(B_{cR})} 
 } \qquad & \\
& \leq C \left| \cu_m \right|^{-\frac12} \left\| \left[ u \right]_\Pa - w \right\|_{L^1(\cu_m)} 
\\ &
=   \frac{C}{\left| \cu_{m} \right|} \sum_{\cu\in\Pa, \, \cu\subseteq \cu_m} \int_\cu\left| u(\bar{z}(\cu))- w(x) \right|\,dx 
\\ &
\leq \frac{C}{\left| \cu_{m} \right|} \sum_{\cu\in\Pa, \, \cu\subseteq \cu_m}
\left( 
\left| \cu\right| \cdot \left| u(\bar{z}(\cu))- w(\bar{z}(\cu)) \right| + \int_\cu\left| w(x)- w(\bar{z}(\cu)) \right| \,dx
\right)
\\ & 
\leq C\left| \cu_{m} \right|^{-\frac12}\left\| u -w \right\|_{L^2(\C_\infty\cap \cu_{m+1})} + \left\| \nabla w \right\|_{L^\infty(\cu_m)}.
\end{align*}
To get the last line, we use H\"older's inequality and make the exponent $t$ in the definition of $\X$ larger, if necessary. We now take $w\in \overline{\A}_k$ to achieve the infimum in the definition of $\tilde{D}_k(R)$. This yields that 
\begin{equation*} \label{}
\inf_{w\in\overline{\A}_k}  \left\| \left[ u \right]_{\Pa} - w \right\|_{\underline{L}^1(B_{cR})} 
\leq
C\tilde{D}_k(R) + C\left\| \nabla w \right\|_{L^\infty(\cu_m)}.
\end{equation*}
It remains to show that 
\begin{equation} 
\label{e.nabw}
\left\| \nabla w \right\|_{L^\infty(\cu_m)} \leq CR^{-1} \tilde{D}_0(R).
\end{equation}
To see this, we note that since $w$ is a harmonic polynomial, we have 
\begin{equation*} \label{}
\left\| \nabla w \right\|_{L^\infty(\cu_m)} \leq C 3^{-m} \inf_{a\in\R} \left\| w - a\right\|_{\underline{L}^1(\cu_m)}.
\end{equation*}
By the Holder inequality, we have, for any $a\in\R$,
\begin{align*}
\left\| w - a\right\|_{\underline{L}^1(\cu_m)}
& 
\leq \frac{C}{|\cu_m|} \sum_{\cu\in\Pa,\, \cu\subseteq \cu_m} \left\| w - a\right\|_{{L}^1(\cu)} 
\\ & 
\leq \frac{C}{|\cu_m|} \sum_{x\in \C_\infty\cap \cu_m} \left( \left| w(x) - a\right| + \size(\cu_\Pa(x))^{d+2} \left\| \nabla w\right\|_{L^\infty(\cu_m)}  \right) 
\\ & 
\leq C | \cu_m|^{-1} \left\| w - a \right\|_{{L}^1(\C_\infty\cap \cu_m)} + C  \left\| \nabla w\right\|_{L^\infty(\cu_m)} 
\end{align*}
Combining these after optimizing over $a\in\R$, we get 
\begin{equation*} \label{}
\left\| \nabla w \right\|_{L^\infty(\cu_m)} \leq C 3^{-m}\inf_{a\in\R} | \cu_m|^{-1} \left\| w - a \right\|_{{L}^1(\C_\infty\cap \cu_m)} + C3^{-m}  \left\| \nabla w\right\|_{L^\infty(\cu_m)}.
\end{equation*}
Since $3^m\geq cR \geq C$, we can absorb the second term on the right side to get
\begin{equation*} \label{}
\left\| \nabla w \right\|_{L^\infty(\cu_m)} \leq \frac{C}{R} \inf_{a\in\R} | \cu_m|^{-1} \left\| w - a \right\|_{{L}^1(\C_\infty\cap \cu_m)}.
\end{equation*}
Since $w$ achieves the infimum in the definition of $\tilde{D}_k(R)$, we have
\begin{align*} \label{}
\inf_{a\in\R} | \cu_m|^{-1} \left\| w - a \right\|_{{L}^1(\C_\infty\cap \cu_m)} &
 \leq C\inf_{a\in\R}|B_{R}|^{-\frac12} \left\| w - a\right\|_{\underline{L}^2(B_{R})} \\
& \leq 
\tilde{D}_0(R) + \tilde{D}_k(R)
\leq
2\tilde{D}_0(R). 
\end{align*}
This completes the proof of~\eqref{e.nabw} and thus of~\eqref{e.puttildeD}.

\smallskip

\emph{Step 3.} The conclusion. Combining~\eqref{e.mesoregforcoarse},~\eqref{e.finalapprox1} and~\eqref{e.puttildeD}, we obtain that, for every $r\in\left[ \X,C^{-1} R\right]$,   
\begin{equation} 
\label{e.mesoregapproach}
\tilde{D}_k(r) \leq C\left( \frac rR \right)^{k+1} \tilde{D}_k(R) + Cr^{1-\delta} \frac{\tilde{D}_0(R)}{R}. 
\end{equation}
By adjusting the constants $C$ we obtain the same inequality for all $r\in\left[ \X,\frac12R\right]$. This yields~\eqref{e.mesoreg} and completes the proof. 
\end{proof}

Notice that Lemma~\ref{l.mesoreg} already gives the statement of Theorem~\ref{t.reg} in the case $k=0$ and gives us the Lipschitz estimate~\eqref{e.lipschitzestimate}. To complete the proof of Theorem~\ref{t.reg}, we need to do an induction on $k$.

\begin{proof}[{Proof of Theorem~\ref{t.reg}}]
Now that we have proved Lemma~\ref{l.mesoreg}, the proof of Theorem~\ref{t.reg} closely follows the argument of~\cite[Proposition 3.1]{AKM2} with only very minor (mostly notational) modifications, using Theorem~\ref{t.homog} and Lemma~\ref{l.mesoreg} in place of~\cite[Proposition 3.2]{AKM2} and~\cite[Proposition 3.3]{AKM2}, respectively. We therefore refer the reader to~\cite{AKM2} and do not repeat the argument here.  
\end{proof}

\appendix

\section{Multiscale Poincar\'e inequality}

The purpose of this appendix is to recall a useful inequality introduced in~\cite{AKM1}, modified here for the discrete lattice, which allows for controlling the $L^2$ norm of a function by the spatial averages of its gradient. 

\smallskip

In this appendix, we will deviate from the notation for cubes introduced in Section~\ref{s.partition} and used in the rest of the paper by denoting 
\begin{equation*} \label{}
\cu_m:= \left( -\frac 12 3^m , \frac 12 3^m \right)^d \subseteq \Rd. 
\end{equation*}
Thus~$\cu_m$ is an open subset of~$\Rd$ and not just a collection of integer lattice points. We will also reserve the symbols $\int$ and $\fint$ to denote integration with respect to Lebesgue measure on subsets of~$\Rd$, with discrete sums denoted by~$\sum$. 

\smallskip

The inequality we consider here is a refinement of the usual Poincar\'e inequality which asserts that, for a constant $C(d)<\infty$ and every $u\in H^1(\cu_m)$, 
\begin{equation}
\label{e.usualPoincare}
\fint_{\cu_m} \left| u(x) - \left( u \right)_{\cu_m} \right|^2\,dx \leq C 3^{2m} \fint_{\cu_m} \left| \nabla u(x) \right|^2\,dx.
\end{equation}
The discrete version of this inequality can be written
\begin{equation}
\label{e.discreteusualPoincare}
\sum_{x\in \Zd \cap \cu_m} \left| u(x) - \left( u \right)_{\cu_m} \right|^2 \leq C 3^{2m} \sum_{x\in \Zd\cap \cu_m} \left| \nabla u \right|^2(x).
\end{equation}
Here we denote averages by~$\left( u \right)_U:= \fint_U u(x)\,dx$ or~$\left( u \right)_U:= |\Zd\cap U|^{-1} \sum_{x\in \Zd\cap U} u(x)$, depending on whether $u$ denotes a continuum or discrete function (which will always be clear from the context). It is not possible to improve the scaling of the constant~$C3^{2m}$ on the right side of these inequalities for general functions, as we see by considering the case that~$u$ is affine. However, if the gradient $\nabla u$ of $u$ has spatial averages on large scales which are very small compared to the (normalized) $L^2$ norm of~$\left| \nabla u \right|$, then it is possible to improve the scaling of the constant. In other words, we can improve the scaling in the Poincar\'e inequality if we use the weaker~$H^{-1}$ norm on the right side rather than the~$L^2$ norm. This is, of course, the situation we often find ourselves in when considering highly oscillating functions. The following result, which we call the \emph{multiscale Poincar\'e inequality}, was proved in~\cite{AKM1}. 

\begin{proposition}[{\cite[Proposition 6.1]{AKM1}}]
\label{p.mspoincare.old}
Fix $n,m\in\N$ with $n<m$. Then there exists a constant $C(d)<\infty$ such that, for every $u\in H^1(\cu_m)$,
\begin{multline}
\label{e.multiscalepoincare0}
\left\| u - \left( u \right)_{\cu_m} \right\|_{\underline{L}^2(\cu_m)} \\
\leq C 3^{n} \left\|\nabla u \right\|_{\underline{L}^2(\cu_m)}  
+  C \sum_{k=n}^{m-1} 3^k    \left( \frac{1}{\left| 3^k\Zd\cap \cu_m \right|} \sum_{y\in 3^k\Zd\cap \cu_m}\left| \left( \nabla u\right)_{y+\cu_{k}} \right|^2 \right)^{\frac12}.
\end{multline}
\end{proposition}

The purpose of this appendix is to explain how to use affine interpolation to derive from Proposition~\ref{p.mspoincare.old} the following discrete version of it. 

\begin{proposition}
\label{p.mspoincare}
Fix $n,m\in\N$ with $n\in \left[ \frac m2, m \right]$. Then there exists a constant $C(d)<\infty$ such that, for every $u: \Zd\cap \cu_m \to \R$,
\begin{multline}
\label{e.multiscalepoincare}
\left\| u - \left( u \right)_{\cu_m} \right\|_{\underline{L}^2(\Zd\cap \cu_m)} \\
\leq C 3^{n} \left\|\nabla u \right\|_{\underline{L}^2(\Zd\cap \cu_m)}  
+  C \sum_{k=n}^{m-1} 3^k    \left( \frac{1}{\left| 3^k\Zd\cap \cu_m \right|} \sum_{y\in 3^k\Zd\cap \cu_m}\left|  \frac{1}{\left| \cu_k \right|}  \left\langle \nabla u\right\rangle_{\Zd\cap(y+\cu_{k})} \right|^2 \right)^{\frac12}.
\end{multline}
\end{proposition}
\begin{proof}
We construct a smooth $\tilde{u} \in C^\infty(\cu_m)$ which is close to the discrete function $u$ by first extending $u$ to be constant on each cube of the form $z+ \cu_0$ with $z\in \Zd\cap \cu_m$ and then taking the convolution of it against a smooth approximation of the identity with supported contained in $B_{1/2}$. It follows that $\tilde{u}(z) = u(z)$ for each $z\in \Zd\cap\cu_m$ and, for each $z\in \Zd\cap \cu_m$,
\begin{equation*} \label{}
\sup_{x\in z+\cu_0} \left| \nabla \tilde{u}(x) \right| \leq C \sum_{y \in \Zd, |y-z|_\infty \leq 1} \left| \nabla u\right|\!(z). 
\end{equation*}
We then check from these facts, the discrete and continuum Stokes formulas and a similar calculation as in \eqref{particularcasei=1} that, for each $z\in\Zd\cap \cu_m$ and $k\in\N$ with $k<m$,
\begin{align*} \label{}
\left| 
\left( \nabla \tilde{u} \right)_{z+\cu_k} 
- \frac{1}{\left| \cu_k \right|} \left\langle \nabla {u} \right\rangle_{\Zd\cap(z+\cu_{k})} 
\right|
& \leq C \frac{1}{\left| \cu_k \right|} \sum_{y \in \Zd\cap \partial(z+\cu_k)}  \left| \nabla {u} \right|\!(y) \\
& \leq C 3^{-\frac k2} \left( \frac{1}{\left| \cu_k \right|} \sum_{y \in \Zd\cap \partial(z+\cu_k)} \left| \nabla {u} \right|^2\!(y)\right)^{\frac12}.
\end{align*}
Applying Proposition~\ref{p.mspoincare.old} to $\tilde{u}$ and using the above inequalities to rewrite the result in terms of $u$, we get~\eqref{e.multiscalepoincare}, as desired. 
\end{proof}

\small
\bibliographystyle{abbrv}
\bibliography{holes}

\end{document}